\documentclass[12pt,a4paper]{article}
\usepackage{a4wide}
\usepackage[utf8]{inputenc}
\usepackage{enumerate}
\usepackage{amsmath}
\usepackage{amsfonts}
\usepackage{amssymb}
\usepackage{bbm}
\usepackage{mathrsfs}
\usepackage{amsthm}
\usepackage{verbatim}
\usepackage{upgreek}
\usepackage{tikz-cd}
\usepackage{relsize}
\usepackage{enumitem}

\newcommand{\eps}{\varepsilon}
\newcommand{\ov}{\overline}
\newcommand{\id}{\textnormal{id}}
\newcommand{\mc}{\mathcal}

\newcommand{\mscr}{\mathscr}
\newcommand{\mf}{\mathfrak}
\newcommand{\msf}{\mathsf}
\newcommand{\I}{\mathbbm{1}}

\newcommand{\vp}{\varphi}
\newcommand{\Lvp}{\Lambda_{\vp}}

\newcommand{\Lvps}{\Lambda_{\psi}}
\newcommand{\Lhvps}{\Lambda_{\wh{\psi}}}

\newcommand{\md}{\operatorname{d}\!}
\newcommand{\cst}{\ifmmode \mathrm{C}^* \else $\mathrm{C}^*$\fi}

\newcommand{\NN}{\mathbb{N}}
\newcommand{\RR}{\mathbb{R}}

\newcommand{\CC}{\mathbb{C}}
\newcommand{\ZZ}{\mathbb{Z}}
\newcommand{\GG}{\mathbb{G}}
\newcommand{\HH}{\mathbb{H}}

\newcommand{\is}[2]{\left\langle#1\,\vline\,#2\right\rangle}

\newcommand{\ismaa}[2]{\langle#1\,|\,#2\rangle}

\newcommand{\wot}{\ifmmode \textsc{wot} \else \textsc{wot}\fi}
\newcommand{\sot}{\ifmmode \textsc{sot} \else \textsc{sot}\fi}
\newcommand{\sots}{\ifmmode \textsc{sot}^* \else \textsc{sot}$^*$\fi}
\newcommand{\ssot}{\ifmmode \sigma\textsc{-sot} \else $\sigma$-\textsc{sot }\fi}
\newcommand{\ssots}{\ifmmode \sigma\textsc{-sot}^* \else $\sigma$-\textsc{sot }$^*$\fi}
\newcommand{\swot}{\ifmmode \sigma\textsc{-wot} \else $\sigma$-\textsc{wot}\fi}

\newcommand{\Linf}{\operatorname{L}^{\infty}(\GG)}
\newcommand{\Linfd}{\operatorname{L}^{\infty}(\whG)}
\newcommand{\Lj}{\operatorname{L}^{1}(\GG)}
\newcommand{\Ljd}{\operatorname{L}^{1}(\whG)}
\newcommand{\Ljsharp}{\operatorname{L}^{1}_{\sharp}(\GG)}
\newcommand{\CG}{\mathrm{C}_0(\GG)}
\newcommand{\CGD}{\mathrm{C}_0(\whG)}
\newcommand{\CGDu}{\mathrm{C}_0^{u}(\whG)}

\newcommand{\vect}[2]{\begin{bmatrix}
#1 \\ 
#2 \end{bmatrix}}

\newcommand{\wh}{\widehat}
\newcommand{\whG}{\widehat{\GG}}

\newcommand{\hvp}{\widehat{\vp}}
\newcommand{\Lhvp}{\Lambda_{\hvp}}

\newcommand{\LdG}{\operatorname{L}^{2}(\GG)}
\newcommand{\LdIrr}{\operatorname{L}^{2}(\IrrG)}
\newcommand{\IrrG}{\Irr(\GG)}
\newcommand{\hpsi}{\widehat{\psi}}
\newcommand{\oxx}{\bar{\otimes}}

\newcommand{\lec}{\preccurlyeq}
\newcommand{\lecq}{\preccurlyeq_q}
\newcommand{\gec}{\succcurlyeq}
\newcommand{\gecq}{\mathbin{_q\!\gec}}
\renewcommand{\restriction}{\mathord{\upharpoonright}}
\newcommand{\rest}{\restriction}

\newcommand{\tp}{\!\!
{\scriptstyle
\text{
\raisebox{0.8pt}{
\textcircled{\raisebox{-1.7pt}{$\top$}}
} % \raisebox
} % \text
} % \scriptstyle
\!\!}
\newcommand{\stp}{\!\!\!
{\scriptscriptstyle
\text{
\raisebox{0.5pt}{
\textcircled{\raisebox{-1.2pt}{$\top$}}
} % \raisebox
} % \text
} % \scriptstyle
\!\!\!}

\newcommand{\WW}{{\reflectbox{$\mathrm{V}\!\!\mathrm{V}$} \!\!\!\;\!\!\!\raisebox{2.7pt}{\reflectbox{\smaller{\smaller{$\mathrm{V}$}}}}\:}
}

\newcommand{\WWl}{{\mathrm{V}\!\!\mathrm{V} \!\!\!\!\!\,\!\!\raisebox{2.7pt}{\reflectbox{\smaller{\smaller{$\mathrm{V}$}}}}\:\;
}}

\newcommand{\WWd}{{\mathrm{V}\!\!\mathrm{V} \!\!\!\!\!\,\!\!\raisebox{2.7pt}{\reflectbox{\smaller{\smaller{$\mathrm{V}$}}}}\:\;
\!\!\!\;\!\!\!\raisebox{2.7pt}{\reflectbox{\smaller{\smaller{$\mathrm{V}$}}}}\:
}}

\newcommand{\mrW}{\mathrm{W}}
\newcommand{\mrV}{\mathrm{V}}

\DeclareMathOperator{\lin}{span}
\DeclareMathOperator{\Irr}{Irr}

\DeclareMathOperator{\supp}{supp}
\DeclareMathOperator{\Tr}{Tr}
\DeclareMathOperator{\B}{B}
\DeclareMathOperator{\M}{M}

\DeclareMathOperator{\Dom}{Dom}

\DeclareMathOperator{\Mor}{Mor}

\DeclareMathOperator*{\esssup}{ess\,sup}
\DeclareMathOperator{\E}{E}
\DeclareMathOperator{\F}{F}
\DeclareMathOperator{\LL}{L}
\DeclareMathOperator{\HS}{HS}

\DeclareMathOperator{\Diag}{Diag}
\DeclareMathOperator{\Dec}{Dec}
\DeclareMathOperator{\Ind}{Ind}
\DeclareMathOperator{\Rep}{Rep}

\newtheorem{theorem}{Theorem}[section]
\newtheorem*{theoremb}{Theorem}
\newtheorem{proposition}[theorem]{Proposition}
\newtheorem{lemma}[theorem]{Lemma}
\theoremstyle{definition}
\newtheorem{corollary}[theorem]{Corollary}
\newtheorem*{remark}{Remark}
\newtheorem{definition}[theorem]{Definition}
\numberwithin{equation}{section}

\title{Coamenability of type I locally compact quantum groups}
\author{Jacek Krajczok\thanks{Email adress: jkrajczok@impan.pl}\\ Institute of Mathematics, Polish Academy of Sciences}
\date{}

\begin{document}
\maketitle
\begin{abstract}
We establish two conditions equivalent to coamenability for type I locally compact quantum groups. The first condition is concerned with the spectra of certain convolution operators on the space $\LL^2(\IrrG)$ of functions square integrable with respect to the Plancherel measure. The second condition involves spectra of character-like operators associated with direct integrals of irreducible representations. As examples we study special classes of quantum groups: classical, dual to classical, compact or given by a certain bicrossed product construction. 
\end{abstract}
\tableofcontents
\newpage

\section{Introduction}
Theory of topological quantum groups has begun with a paper of Woronowicz (\cite{Woronowiczsu2}) where he has constructed a deformation of the $SU(2)$ group. Later on, theory of general compact quantum groups was introduced in papers \cite{WoronowiczqLor, Woronowiczcqg}. A remarkable feature of compact quantum groups is the fact that they admit a representation theory closely resembling representation theory of classical compact groups: every representation decomposes as a direct sum of irreducible representations which are necessarily finite dimensional (we recommend \cite{NeshTu} as an introduction to the theory of compact quantum groups).\\
Locally compact quantum groups (with the starting point in the von Neumann algebraic world) were introduced by Kustermans and Vaes in \cite{KustermansVaes1} (see also \cite{Kustermans, KustermansVaes, Daele}). At present, their theory has reached a very satisfactory level and various authors have studied locally compact quantum groups focusing on questions coming from group theory, harmonic analysis or operator theory.\\
An intermediate step between compact and general locally compact quantum groups is formed by type I locally compact quantum groups. These are quantum groups whose full \cst-algebra (which is equal to $\CGDu$, see Section \ref{conventions}) is of type I. This class of quantum groups was studied previously in \cite{Caspers, Desmedt} (see also \cite{CaspersKoelink}) and more recently in \cite{VoigtYuncken2, VoigtYuncken}. Their distinguishing feature is a tractable representation theory: since there is a one to one correspondence between representations of a locally compact quantum group $\GG$ and these of its \cst-algebra $\CGDu$ (see \cite{Kustermans}), every representation of a type I locally compact quantum group decomposes as a direct sum of direct integrals of irreducible representations in a unique way (\cite{DixmierC}). Using slightly weaker assumption (namely that $\Linfd$ is a von Neumann algebra of type I) Desmedt was able to deduce an existence of the Plancherel measure for this class of quantum groups. A Plancherel measure is a measure on $\IrrG$, the spectrum of $\CGDu$, which together with certain fields of positive operators $(D_\pi)_{\pi\in \IrrG},(E_\pi)_{\pi\in \IrrG}$, allows us to express the Haar integrals for $\whG$. The need to include these fields of operators in the formulation of the Desmedt's theorem corresponds to the fact that $\whG$ can be non-unimodular, and its Haar integrals may be non-tracial.\\
Our main results are concerned with the notion of coamenability. It is an important property of locally compact quantum groups which has various equivalent formulations (see \cite{Brannan} for a nice survey). The one that is most convenient for us is the existence of a net of unit vectors $(\xi_i)_{i\in \mc{I}}$ in $\LdG$ such that 
\[
\|\mrW^{\GG} (\xi_i\otimes\eta)-\xi_i\otimes\eta\|\xrightarrow[i\in\mc{I}]{}0
\]
for all $\eta\in\LdG$, where $\mrW^{\GG}$ is the Kac-Takesaki operator (see Section \ref{conventions} or \cite{KustermansVaes, Daele}). We also would like to mention that coamenability of $\GG$ implies amenability of $\whG$, whereas the converse statement is an important open problem.\\
When $\GG$ is compact, the coamenability of $\GG$ can be expressed as a property of spectra of characters of representions of $\GG$ (\cite{Banica,NeshTu}) or of certain operators on $\ell^2(\IrrG)$ (\cite{HiaiIzumi, NeshTu}). Let us describe our main theorems which generalize these results; choose $\mu$, a Plancherel measure for $\GG$ and a measurable subset $\Omega\subseteq \IrrG$. Denote by $\mu_\Omega$ the measure $\mu$ restricted to $\Omega$. With the subset $\Omega$ we associate the representation $\sigma_\Omega=\int_{\Omega}^{\oplus} \pi \md\mu_{\Omega}(\pi)$. If $\int_\Omega \dim \md\mu<+\infty$ then we define the \emph{integral character} of $\sigma_\Omega$ as $\chi^{\int}(\int_{\Omega}^{\oplus} \pi \md\mu_{\Omega}(\pi))=\int_{\Omega} \chi(U^\pi) \md\mu_{\Omega}(\pi)\in\Linf$ (note that the integral character is not invariant under unitary equivalence and depends on the structure of a measurable field of representations -- see Definition \ref{defintrep}). Our first main theorem states that coamenability of $\GG$ is equivalent to a property of spectra of integral characters (see Theorem \ref{tw2} for the full statement).

\begin{theoremb}
Let $\GG$ be a second countable locally compact quantum group. Assume moreover that $\GG$ is type I and has only finite dimensional irreducible representations. Consider the following conditions:
\begin{enumerate}[label=\arabic*)]
\item $\GG$ is coamenable.
\item For any Plancherel measure $\mu$ and any measurable subset $\Omega\subseteq \IrrG$ such that $\int_{\Omega}\dim\md\mu<+\infty$ we have $\int_{\Omega}\dim\md\mu\in \sigma(\chi^{\int}(\int_{\Omega}^{\oplus}\pi\md\mu_{\Omega}(\pi)))$.
\end{enumerate}
We have $1)\Rightarrow 2)$. If all irreducible representations of $\GG$ are admissible then also $2)\Rightarrow 1)$.
\end{theoremb}

Let us now describe our second main theorem; take a measurable subset $\Omega\subseteq\IrrG$ and a finite dimensional nondegenerate representation $\kappa \colon \CGDu\rightarrow \B(\msf{H}_{\kappa})$. Assume that $\kappa$ is weakly contained in $\Lambda_{\whG}$ and that $\CGDu$ is separable and of type I. Consider the representation $\kappa\tp \sigma_\Omega\colon \CGDu\rightarrow \B( \msf{H}_{\kappa} \otimes \int_{\Omega}^{\oplus} \msf{H}_{\pi} \md\mu_{\Omega}(\pi))$. We can find a unique (up to measure zero) decreasing sequence of measurable subsets $\F^{n}_{\kappa\stp \sigma_\Omega}\subseteq\IrrG (n\in\NN)$ such that $\kappa\tp\sigma_\Omega$ is unitarily equivalent to $\bigoplus_{n=1}^{\infty} \int_{\F^{n}_{\kappa\stp \sigma_\Omega}}^{\oplus} \pi \md \mu_{\F^{n}_{\kappa\stp \sigma_\Omega}}(\pi)$. In Section \ref{secdecomp} we introduce a measurable function $\varpi^{\kappa,\Omega,\mu}\colon \F^1_{\kappa\stp \sigma_{\Omega}}\rightarrow \RR_{>0}$ such that we have equality of the integral characters
\[
\chi(U^{\kappa}) \chi^{\int}( \int_{\Omega}^{\oplus} \pi \md\mu_{\Omega}(\pi))=
\sum_{n=1}^{\infty} \chi^{\int}( \int_{\F^n_{\kappa\stp\sigma_\Omega}}^{\oplus} \zeta \md\, (\varpi^{\kappa,\Omega,\mu} \mu_{\F^{n}_{\kappa\stp \sigma_\Omega}})(\zeta)).
\]
Next, in Section \ref{secconv} we show that there exists a bounded operator $\mc{L}_{\kappa}$ which is given by
\[
\mc{L}_{\kappa} \colon \LL^2(\IrrG)\ni
\Tr(E^{2}_{\bullet})^{\frac{1}{2}} \chi_{\Omega} \mapsto
\Tr(E^{2}_{\bullet})^{\frac{1}{2}} \varpi^{\kappa,\Omega,\mu} \sum_{n=1}^{\infty} \chi_{\F^n_{\kappa\stp\sigma_{\Omega}}} \in \LL^2(\IrrG)
\]
for suitable subsets $\Omega\subseteq\IrrG$ -- roughly speaking, $\mc{L}_\kappa$ is the operator of tensoring by $\kappa$ on the level of $\LL^2(\IrrG)$. However, in order for this operator to be well defined, we need to introduce functions $\Tr(E^2_{\bullet})^{\frac{1}{2}}$ and $\varpi^{\kappa,\Omega,\mu}$ in the definition. For any positive function $\nu\in\LL^1(\IrrG)$ define $\mc{L}_{\nu}=\int_{\IrrG} \tfrac{\nu(\kappa)}{\dim(\kappa)} \mc{L}_{\kappa} \md\mu(\kappa)$. Our second main theorem relates spectra of operators $\mc{L}_{\nu}$ to the coamenability of $\GG$ (see Theorem \ref{tw4} for the full statement).
\begin{theoremb}
Let $\GG$ be a second countable locally compact quantum group. Assume moreover that $\GG$ is type I and has only finite dimensional irreducible representations.
Consider the following conditions:
\begin{enumerate}[label=\arabic*)]
\item $\GG$ is coamenable.
\item Let $\mu$ be any Plancherel measure and $\Omega\subseteq\IrrG$ a measurable subset such that $\int_\Omega\dim\md\mu<+\infty$. Define $\nu=\dim \chi_\Omega$. Then $\int_\Omega\dim\md\mu\in \sigma(\mc{L}_\nu)$.
\end{enumerate}
We have $1)\Rightarrow 2)$. If all irreducible representations of $\GG$ are admissible then also $2)\Rightarrow 1)$.
\end{theoremb}

The paper is organised as follows: in Section \ref{secPlancherel} we recall a theorem of Desmedt stating the existence and properties of the Plancherel measure. We also derive a more precise form of its uniqueness.\\
In Section \ref{secintrep} we introduce a notion of the integral representation: it is a representation of $\CGDu$ given by a direct integral $\int_X^{\oplus} \pi_x\md\mu_X(x)$, where $(\pi_x)_{x\in X}$ are representations of $\CGDu$ satisfying some additional assumptions. Moreover, in this section with an integral representation we associate the integral weight (given by a restriction of $\int_X^{\oplus} \Tr_x \md\mu_X(x)$) and the integral character $\chi^{\int}(\int_X^{\oplus} \pi_x\md\mu_X(x))$, which is an operator in $\Linf$ equal to $\int_X \chi(U^{\pi_x})\md\mu_X(x)$. We also discuss basic operations on integral representations.\\
In Section \ref{secintweight} we derive an important technical result which says that having (suitable) unitarily equivalent integral representations, we can twist a measure on one of them, so that their integral weights are transformed one to the other. Section \ref{secweightchar} is devoted to the connection between the integral character and the integral weight.\\
In Section \ref{secintwhG} we use the left (resp. right) invariance of $\hvp$ (resp. $\hpsi$), the Haar integrals of $\whG$, to deduce properties of $(D_\pi)_{\pi\in \IrrG}$ (resp. $(E_\pi)_{\pi\in \IrrG}$). In Section \ref{secdecomp} we start discussing a decomposition of the representations of the form $\kappa\tp\int_{\Omega}^{\oplus} \zeta\md\mu_{\Omega}(\zeta)$, where $\Omega\subseteq\IrrG$ is a measurable subset.\\
In the next section we establish properties related to the condition that the integral character of a given integral representation is square integrable with respect to $\psi$, i.e. $\chi^{\int}(\int_X^{\oplus}\pi_x\md\mu_X(x))\in\mf{N}_{\psi}$. In this section we also introduce an important isometry \\$T\colon \LL^2(\IrrG)\rightarrow \LdG$.\\
In Section \ref{secconv} we introduce operators $\mc{L}_\kappa$, which loosely speaking are given by tensoring with $\kappa$: $\int_{\Omega}^{\oplus} \zeta\md\mu_\Omega(\zeta)\mapsto \kappa\tp \int_{\Omega}^{\oplus} \zeta\md\mu_\Omega(\zeta)$ on the level of $\LL^2(\IrrG)$. However, for this operation to be well defined, we need to take into consideration results from Section \ref{secintweight}. We also show that operator $\mc{L}_{\kappa}$ is related to the character $\chi(U^{\kappa})$ via the isometry $T$.\\
In Section \ref{secconj} we consider properties of taking the conjugate representation, considered as a map $\IrrG\rightarrow \IrrG$. We also derive a theorem which relates coamenability of $\GG$ to the spectra of integral characters.\\
In the penultimate section we introduce $\mf{A}$, a \cst-subalgebra of $\Linf$ which is generated by (suitable) integral characters. We also deduce a result which says that coamenability of $\GG$ is equivalent to properties of the spectra of operators $\mc{L}_{\nu}=\int_{\IrrG} \tfrac{\nu(\kappa)}{\dim (\kappa)} \mc{L}_{\kappa}\md\mu(\kappa)$.\\
The remainder of the paper is devoted to discussing examples: we consider situation when $\GG$ is compact, classical, dual to classical or given by a special bicrossed product. In these examples we calculate a Plancherel measure and derive formulas for the actions of $\mc{L}_\kappa$. Moreover, we show that when $\GG$ is compact, our results on coamenability of $\GG$ corresponds to previously known results, and when $\whG$ is classical we arrive at a variation of the Kesten criterion.

\section{Conventions}\label{conventions}
Throughout the paper, $\GG$ stands for a locally compact quantum group in the sense of Kustermans and Vaes and $\whG$ is its dual. This means in particular that we have a von Neumann algebra $\Linf$ which is represented on the Hilbert space $\LdG$, together with a coproduct which is a normal faithful $\star$-homomorphism $\Delta_{\GG}\colon \Linf\rightarrow\Linf\oxx\Linf$ (we will write $\Delta$ instead of $\Delta_{\GG}$ if there is no risk of confusion). We denote by $\vp,\psi,(\sigma^{\vp}_t)_{t\in\RR},$ $(\sigma^{\psi}_t)_{t\in\RR},\Lvp,\Lambda_{\psi}$ the left and right Haar weights of $\GG$ together with their modular groups and the GNS maps. The scaling constant, the scaling group and the unitary antipode of $\GG$ will be denoted as usual by $\nu, \,(\tau_t)_{t\in\RR}$ and $R$. We remark that the GNS Hilbert spaces for $\vp$ and $\psi$ may be identified -- we will denote the resulting space by $\LdG$. The modular conjugations related to $\vp$, $\psi$ are denoted respectively by $J$ and $J^{\psi}$. The modular element of $\GG$ will be denoted by $\delta$. We remark that the same convention will be used also for classical groups: if $\GG$ is a classical locally compact quantum group with the left and the right Haar measure denoted by respectively $\mu_L$ and $\mu_R$, then $\delta=\tfrac{\md\mu_R}{\md\mu_L}$. Objects related to $\whG$ will be accordingly decorated with hats.\\
We will frequently use the Kac-Takesaki operator $\mrW^{\GG}\in \Linf\bar\otimes \LL^\infty(\whG)$, which is a unitary on $\LdG\otimes\LdG$ satisfying
\[
(\mrW^{\GG})^* (\Lvp(x)\otimes\Lvp(y))=(\Lvp\otimes\Lvp)(\Delta(y)(x\otimes\I))\quad(x,y\in\mf{N}_\vp).
\]
We will also several times use the unitary operator related to the right Haar weight: $\mrV^{\GG}\in \LL^\infty(\whG)'\bar\otimes \Linf$ given by
\[
\mrV^{\GG} (\Lvps(x)\otimes\Lvps(y))=(\Lvps\otimes\Lvps)(\Delta(x)(\I\otimes y))\quad(x,y\in\mf{N}_\psi).
\]
We think of $\Linf$ as a space of essentialy bounded measurable functions on the quantum group $\GG$. Consequently, the predual of $\Linf$ will be denoted by $\Lj$. It is a Banach algebra together with a convolution product given by $\omega\star\nu=(\omega\otimes\nu)\circ\Delta\,(\omega,\nu\in\Lj)$. One introduces the space of functions on $\GG$ which are continuous and vanish at infinity via
\[
\CG=\{(\id\otimes\omega)\mrW^{\GG}\,|\, \omega\in \Ljd\}^{-\|\cdot\|}.
\]
This is a \cst-algebra which is weakly dense in $\Linf$ and one can check that the coproduct restricts to a nondegenerate morphism $\CG\rightarrow \M(\CG\otimes\CG)$. Moreover, the operator $\mrW^{\GG}$ belongs to the \cst-algebra $\M(\CG\otimes \mathrm{C}_0(\whG))$ and satisfies
\[
(\Delta_{\GG}\otimes\id)\mrW^{\GG}=\mrW^{\GG}_{13} \mrW^{\GG}_{23},\quad
(\id\otimes\Delta_{\whG})\mrW^{\GG}=\mrW^{\GG}_{13} \mrW^{\GG}_{12}.
\]
There are plenty of relations that connect the above objects. For the convenience of the reader, we gather here these which we use in the paper, and refer to \cite{KustermansVaes1, KustermansVaes, Daele} for their proofs; for $x\in \Linf$ and $t\in\RR$ we have
\[
\begin{split}
\tau(x)=\nabla_{\hvp}^{it} x \nabla_{\hvp}^{-it},\quad
\chi(\mrV^{\GG})&=(\hat{J}\otimes\hat{J}) {\mrW^{\GG}}^* (\hat{J}\otimes\hat{J}),\quad
R(x)=\hat{J}x^*\hat{J},\\
\vp\circ \sigma^{\psi}_t=\nu^t \vp,\quad
\psi\circ \sigma^{\vp}_t&=\nu^{-t}\psi,\quad
\psi\circ R=\vp,\quad
\tau_t(\delta)=\delta\\
\hat{J}\Lambda_{\vp}(x)=\Lambda_{\psi}(R(x^*)),&\quad
\nabla_{\hvp}^{it}\Lvp(x)=\Lvp(\tau_t(x) \delta^{-it})\\
(R\otimes \hat{R})\mrW^{\GG}=\mrW^{\GG},\quad
(\tau_t&\otimes \hat{\tau}_t)\mrW^{\GG}=\mrW^{\GG},\quad
\mrW^{\whG}=\chi(\mrW^{\GG})^*,
\end{split}
\]
moreover the groups of automorphisms $(\sigma^{\vp}_t)_{t\in\RR},(\sigma^{\psi}_t)_{t\in\RR},(\tau_t)_{t\in\RR}$ commute.\\
We also have $\CGDu$, the universal version of the \cst-algebra $\CGD$. It is equipped with a surjective morphism $\Lambda_{\GG}\in\Mor(\CGDu,\CGD)$ and a unitary ${\WW}^{\GG}\in\M(\CG\otimes\CGDu)$ which satisfies $(\id\otimes\Lambda_{\whG}){\WW}^{\GG}=\mrW^{\GG}$ (we will write $\mrW,\mrV,\WW$ etc.~if there is no risk of confusion). An important property of this \cst-algebra is the fact that there is a one to one correspondence between representations of $\GG$ and nondegenerate representations of $\CGDu$ given by $(\id\otimes\pi){\WW}^{\GG}\leftrightarrow \pi$. The \cst-algebra $\CGDu$ comes together with its universal unitary antipode $\hat{R}^{u}$ which satisfies relations similar to those satisfied by $\hat{R}$.\\
For a normal functional $\omega\in\Lj$ we define elements $\lambda(\omega)=(\omega\otimes\id)\mrW^{\GG}\in \CGD$ and $\lambda^u(\omega)=(\omega\otimes\id){\WW}^{\GG}\in\CGDu$. Define the following subspace in $\Lj$:
\[
\begin{split}
\LL^1_{\sharp}(\GG)
&=\{\omega\in\LL^1(\GG)\,|\, \exists\, \omega^{\sharp}\in\LL^1(\GG):
\lambda(\omega)^*=\lambda(\omega^{\sharp})\}.
\end{split}
\]
One can check that $\LL^1_{\sharp}(\GG)\ni\omega\mapsto \omega^{\sharp}\in \LL^1_{\sharp}(\GG)$ is a well defined antilinear involution together with which $\LL^1_{\sharp}(\GG)$ becomes a $\star$-algebra.
For $\omega\in\Lj$ we define a normal functional $\ov{\omega}\in\Lj$ via $\ov{\omega}(x)=\ov{\omega(x^*)}\;(x\in\Linf)$.\\
For any Hilbert space $\msf{H}$, we define its conjugate space $\ov{\msf{H}}=\{\ov{\xi}\,|\,\xi\in\msf{H}\}$ which is a Hilbert space with inner product given by $\ismaa{\ov{\xi}}{\ov{\eta}}=\ismaa{\eta}{\xi}\,(\xi,\eta\in\msf{H})$. We denote by $\jmath_{\msf{H}}$ the canonical linear, antimultiplicative map $\B(\msf{H})\rightarrow\B(\ov{\msf{H}})$ given by $\jmath_{\msf{H}}(T)\ov{\xi}=\ov{T^*\xi}\,(\xi\in\msf{H},T\in\B(\msf{H}))$.\\
For any $\pi\colon \CGDu\rightarrow \B(\msf{H}_\pi)$, a nondegenerate representation of $\CGDu$ we define a conjugate representation $\pi^c\colon \CGDu\rightarrow \B(\msf{H}_\pi)$ via
\[
\pi^c=\jmath_{\msf{H}_\pi}\circ \pi\circ \hat{R}^{u}
\]
(see e.g. \cite{SoltanWoronowicz}).\\
To ease the notation, we will write $\sup$ rather than $\esssup$ whenever we take an essential supremum over a measure space. Similarly, $\supp$ will stand for the essential support of a function. Moreover, whenever we have a measure space $(X,\mf{M}_X,\mu)$ and a measurable subset $Y\subseteq X$, we will write $\mu_Y$ for the restricted measure on $Y$.\\
We say that a measure $\mu_X$ on $X$ is standard if there exists a measurable subset $Y\subseteq X$ such that $\mu_X(Y^{c})=0$ and $(Y,\mf{M}_Y,\mu_Y)$ is a standard measure space. For any \cst-algebra $\msf{A}$, we will denote by $\Irr(\msf{A})$ its spectrum.\\
If $\pi,\sigma$ are (nondegenerate) representations of some \cst-algebra, then we write $\pi\approx\sigma,\pi\approx_q \sigma,\pi\simeq \sigma$ for respectively: weak equivalence, quasi-equivalence and unitary equivalence. We will also write $\pi\lec \sigma,\pi\lec_q\sigma,\pi\subseteq \sigma$ for the corresponding relations of containment. Necessary information about quasi-equivalence is gathered in the appendix.\\
For the theory of direct integrals we refer the reader to \cite{DixmierC, DixmiervNA, Lance}. We will use the following terminology (see \cite[Definition 2, page 182]{DixmiervNA} and ): for a measurable field of Hilbert spaces $(\msf{H}_x)_{x\in X}$, operators of the form $\int_X^{\oplus} T_x\md\mu_X(x)$ are called \emph{decomposable}. The set of decomposable operators form a von Neumann algebra which will be denoted $\Dec(\int_X^{\oplus} \msf{H}_x\md\mu_X(x))$. Decomposable operators $\int_X^{\oplus} T_x\md\mu_X(x)$ with $T_x\in \CC \I_{\msf{H}_x}\,(x\in X)$ are called \emph{diagonalisable}. Similarly, they form a von Neumann algebra which will be denoted by $\Diag(\int_X^{\oplus} \msf{H}_x\md\mu_X(x))$.\\
Several times we will use the following useful notation: if $x,y$ are elements of some metric space $(X,d)$ and $\eps>0$ is a positive number then $x\approx_{\eps} y$ means $d(x,y)\le \eps$.

\section{The Plancherel measure}\label{secPlancherel}
In this section we present a theorem of P. Desmedt which establishes the existence of the Plancherel measure for locally compact quantum group under the assumption that $\Linfd$ is of type I. We will also prove a more precise version of the uniqueness result.\\
Recall that for a direct integral of Hilbert spaces $\int_{X}^{\oplus}\msf{H}_x\md\mu_X(x)$ an operator of the form $\int_{X}^{\oplus} f(x) \I_{\msf{H}_x}\md\mu_X(x)$, where $f$ is a scalar valued function, is called \emph{diagonalisable}. Existence of the Plancherel measure follows from the following more general result (\cite[Theorem 3.4.5]{Desmedt}).

\begin{theorem}\label{Desmedtcstar}
Let $A$ be a separable \cst-algebra. Let $\phi$ be a lower semi-continuous densely defined approximately KMS-weight on $A$ such that $\pi_\phi(A)''$ is a von Neumann algebra of type I. We equip $\Irr(A)$, the spectrum of $A$, with the Mackey-Borel $\sigma$-algebra. There exists $\mu$, a standard measure on $\Irr(A)$, a measurable field of Hilbert spaces $(K_\sigma)_{\sigma\in\operatorname{Irr}(A)}$, a measurable field of representations $(\pi_\sigma)_{\sigma\in \operatorname{Irr}(A)}$ of $A$ on $K_\sigma$ such that $\pi_\sigma\in\sigma$ for every $\sigma\in\Irr(A)$, a measurable field $(D_\sigma)_{\sigma\in \operatorname{Irr}(A)}$ of self-adjoint, strictly positive operators and a unitary operator $\mc{P}\colon \msf{H}_\phi\rightarrow \int_{\operatorname{Irr}(A)}^{\oplus} K_\sigma\otimes \ov{K_\sigma}\md\mu(\sigma)$ with the following properties:
\begin{enumerate}[label=\arabic*)]
\item For every $x\in \mf{N}_\phi$ and $\mu$-almost every $\sigma\in\Irr(A)$ operator $\pi_\sigma(x) D_{\sigma}^{-1}$ is bounded and its closure $\pi_\sigma(x)\cdot D_{\sigma}^{-1}$ is Hilbert-Schmidt.
\item\label{item1} For all $a,b\in\mf{N}_\phi$ we have the Parseval formula:
\[
\is{\Lambda_\phi(a)}{\Lambda_\phi(b)}=
\int_{\Irr(A)} \Tr\bigl(
\bigl(\pi_\sigma(a)\cdot D_{\sigma}^{-1}\bigr)^{*}
\bigl(\pi_\sigma(b)\cdot D_{\sigma}^{-1}\bigr)
\bigr)\md\mu(\sigma),
\]
and $\mc{P}$ is the isometric extension of
\[
\Lambda_\phi(\mf{N}_\phi)\ni \Lambda_\phi(x)\mapsto
\int_{\Irr(A)}^{\oplus} 
\pi_\sigma(x)\cdot D_{\sigma}^{-1}\md\mu(\sigma)
\in
\int_{\Irr(A)}^{\oplus} \HS(K_\sigma)\md\mu(\sigma),
\]
\item Denote by $J_\phi$ the modular conjugation of $\phi$ and by $J_\sigma$ the map $K_\sigma\otimes \ov{K_\sigma}\ni \xi\otimes\ov{\eta}\mapsto \eta\otimes\ov{\xi}\in K_\sigma\otimes\ov{K_\sigma}$. Denote by $\rho_\phi$ the representation of $A^{op}$ on $H_\phi$: $\rho_\phi(x)=J_\phi \pi_\phi(x^*) J_\phi\,(x\in A)$ and by $\check{\pi}_\sigma$ the representation of $A^{op}$ on $\ov{\msf{H}_\sigma}$ given by $\check{\pi}_\sigma(a)=\jmath_{\msf{H}_\sigma}(\pi_\sigma(a))$. Operator $\mc{P}$ transforms $J_\phi$ into $\int_{\Irr(A)}^{\oplus} J_\sigma \md\mu(\sigma)$, $\pi_\phi$ into $\int_{\Irr(A)}^{\oplus} (\pi_\sigma\otimes\I)\md\mu(\sigma)$, $\rho_\phi$ into $\int_{\Irr(A)}^{\oplus} (\I\otimes\check{\pi}_\sigma)\md\mu(\sigma)$, $\pi_\phi(A)''$ into $\int_{\Irr(A)}^{\oplus} \B(K_\sigma)\otimes\CC \md\mu(\sigma)$, $\pi_\phi(A)'$ into $\int_{\Irr(A)}^{\oplus} \CC\otimes\B(\ov{K_\sigma}) \md\mu(\sigma)$ and $\pi_\phi(A)''\cap \pi_\pi(A)'$ into the algebra of diagonalisable operators.
\item
For $x\in A_+$ the function $\Irr(A)\ni \sigma \mapsto \Tr(\pi_\sigma(x)\cdot {D_{\sigma}}^{-2})$ is lower semi-continuous and we have
\[
\phi(x)=
\int_{\Irr(A)} \Tr(\pi_\sigma(x) \cdot {D_{\sigma}}^{-2})
\md\mu(\sigma).
\]
\item
The weight $\tilde{\phi}$, lift of $\phi$ to $\pi_\phi(A)''$, is tracial if and only if almost all $D_{\sigma}$ are multiples of the identity.
\item\label{item2}
Suppose that there exists a standard measure $\mu'$ and $\mu'$-measurable fields $(K'_\sigma)_{\sigma\in\Irr(A)}$, $(\pi'_\sigma)_{\sigma\in\Irr(A)}$, $(D'_\sigma)_{\sigma\in\Irr(A)}$ having the same properties as fields without a prime. Then $\mu$ and $\mu'$ are equivalent, and we have for $\mu$-almost all $\sigma\in \Irr(A)$ that
\[
D'_\sigma=\sqrt{\frac{\md\mu'}{\md\mu}(\sigma)} T_\sigma D_\sigma {T_\sigma}^{-1},
\]
where $T_\sigma$ is an intertwiner between $\pi_\sigma$ and $\pi'_\sigma$ for almost all $\sigma$.
\end{enumerate}
\end{theorem}

\begin{remark}$ $
\begin{itemize}
\item
In Desmedt's paper in point \ref{item2} a slightly different equation appears, but it seems that it is not the correct one.
\item 
It follows from the construction that $\mu$ is standard and we can choose a measurable field of Hilbert spaces $(K_\sigma)_{\sigma\in \operatorname{Irr}(A)}$ to be canonical, i.e.~to reduce to $\CC^n$ on each component $\{\pi\in \Irr(A)\,|\, \dim(\pi)=n\}$ (see \cite[Section 8.6.1]{DixmierC}).
\item
We can get a more precise form of point \ref{item2}, which is stated in the next lemma. For the sake of simplicity we will assume that irreducible representations of $A$ are finite dimensional.
\item
Let us remark that we will use this result only in the case of type I \cst-algebras, for which the proof simplifies.
\end{itemize}
\end{remark}

\begin{lemma}\label{lemat13}
In the situation from the previous theorem, assume that all irreducible representations of $A$ are finite dimensional, $\mu'$ is a standard measure on $\Irr(A)$, $(K'_\sigma)_{\sigma\in\Irr(A)}$ is a measurable family of Hilbert spaces, $(\pi'_{\sigma})_{\sigma\in\Irr(A)}$ is a measurable family of representations such that $\pi'_\sigma\in \sigma$ for $\mu'$-every $\sigma$. Assume moreover that there exists a unitary operator $\mc{P}'\colon \msf{H}_\phi\rightarrow\int_{\Irr(A)}^{\oplus} K'_\sigma\otimes \ov{K'_\sigma} \md\mu'(\sigma)$. If
\begin{enumerate}[label=\arabic*)]
\item operator $\mc{P}'$ transforms $\pi_\phi$ into $\int_{\Irr(A)}^{\oplus} (\pi'_\sigma\otimes\msf{1})\md\mu'(\sigma)$,
\item operator $\mc{P}'$ transforms $\pi_\phi(A)''\cap\pi_\phi(A)'$ into the algebra of diagonalisable operators
\end{enumerate}
then measures $\mu,\mu'$ are equivalent. If moreover there exists a measurable family of strictly positive self-adjoint operators $(D'_\sigma)_{\sigma\in \Irr(A)}$ and 
\begin{enumerate}[label=\arabic*)]
\item[3)] operator $\mc{P}'$ transforms $\rho_\phi$ into $\int_{\Irr(A)}^{\oplus} (\I\otimes\check{\pi}'_\sigma)\md\mu(\sigma)$,
\item[4)] we have the equality
\[
\mc{P}'\Lambda_{\phi}(x)=\int_{\Irr(A)}^{\oplus} \pi'_{\sigma}(x) {D'}_{\sigma}^{-1} \md\mu'(\sigma)
\] 
for all $x$ in a subspace $X\subseteq\mf{N}_\phi$ such that the unit operator belongs to the \wot -- sequential closure of $X$,
\end{enumerate}
then for $\mu$-almost all $\sigma\in\Irr(A)$ there exists a unitary intertwiner $T_\sigma\colon K_\sigma\rightarrow K'_\sigma$ such that
\[
D'_\sigma=\sqrt{\tfrac{\md\mu'}{\md\mu}(\sigma)} T_\sigma D_\sigma T_\sigma^{-1}.
\]
\end{lemma}

\begin{proof}
Let us define a unitary operator
\[
\mc{U}=\mc{P}'\circ\mc{P}^{-1}\colon 
\int_{\Irr(A)}^{\oplus} K_\sigma\otimes \ov{K_\sigma} \md\mu(\sigma)
\rightarrow
\int_{\Irr(A)}^{\oplus} K'_\sigma\otimes \ov{K'_\sigma} \md\mu'(\sigma).
\]
It transforms diagonalisable operators into diagonalisable operators. Consider the following representations of $A$:
\[
\int_{\Irr(A)}^{\oplus} \pi_\sigma \otimes \I \md\mu(\sigma),\quad
\int_{\Irr(A)}^{\oplus}\pi'_\sigma \otimes \I\md\mu'(\sigma).
\]
We would like to use \cite[Proposition 8.2.4]{DixmierC}. In order to do that, we need to check that $\mc{U}$ is a morphism between these representations. Let $a \in A$. Thanks to properties of $\mc{P},\mc{P}'$ we have
\[
\begin{split}
&\quad\;\mc{U}\bigl(\int_{\Irr(A)}^{\oplus} \pi_\sigma \otimes \I \md\mu(\sigma)(a)\bigr)\mc{U}^{-1}=
\mc{P}'\mc{P}^{-1}\bigl(\int_{\Irr(A)}^{\oplus} \pi_\sigma(a) \otimes \I\md\mu(\sigma)\bigr)\mc{P}\mc{P}'^{-1}\\
&=
\mc{P}' \pi_{\phi}(a) \mc{P}'^{-1}=
\bigl(\int_{\Irr(A)}^{\oplus} \pi'_\sigma \otimes \I \md\mu'(\sigma)\bigr)(a).
\end{split}
\]
Now, \cite[Proposition 8.2.4]{DixmierC} gives us subsets $N$, $N'\subseteq\Irr(A)$ which are correspondingly of $\mu$ and $\mu'$-measure $0$, Borel isomorphism $\eta\colon \Irr(A)\setminus N\rightarrow\Irr(A)\setminus N'$ which maps $\mu$ into a measure $\tilde{\mu}'$ equivalent to $\mu'$ and a family $\sigma\mapsto V(\sigma)$ such that for each $\sigma$, $V(\sigma) \colon K_\sigma\otimes\ov{K_\sigma}\rightarrow K'_{\eta(\sigma)}\otimes\ov{K'_{\eta(\sigma)}}$ is a unitary map and a vector field $(\xi_\sigma)_{\sigma\in \Irr(A)\setminus N}$ is measurable with respect to $(K_\sigma\otimes \ov{K_\sigma})_{\sigma\in \Irr(A)\setminus N}$ if and only if $(V(\sigma)\xi_\sigma)_{\eta(\sigma)\in \Irr(A)\setminus N'}$ is measurable with respect to $(K'_{\eta(\sigma)}\otimes \ov{K'_{\eta(\sigma)}})_{\eta(\sigma)\in \Irr(A)\setminus N'}$. Such a family is called $\eta$-isomorphism \cite[A 70]{DixmierC}. For $\sigma\in \Irr(A)\setminus N$ operator $V(\sigma)$ is a unitary morphism between $\pi_\sigma\otimes\I$ and $\pi'_{\eta(\sigma)}\otimes\I$, moreover
\[
\mc{U}=\bigl(\int_{\Irr(A)}^{\oplus} K_\sigma'\otimes \ov{K'_\sigma}\md\tilde{\mu}'(\sigma)\rightarrow
\int_{\Irr(A)}^{\oplus} K_\sigma'\otimes \ov{K'_\sigma}\md\mu'(\sigma)
\bigr)\circ \int_{\Irr(A)}^{\oplus} V(\sigma)\md\mu(\sigma).
\]
Fix $\ov{\zeta}\in \ov{K_\sigma},\ov{\zeta}'\in\ov{K'_{\eta(\sigma)}}$ and define a bounded operator $S^\sigma_{\ov{\zeta}',\ov{\zeta}}\in \B(K_\sigma,K'_{\eta(\sigma)})$ via equality
\[
\ismaa{\xi'}{S^\sigma_{\ov{\zeta}',\ov{\zeta}} \xi}=
\ismaa{\xi'\otimes \ov{\zeta}'}{V(\sigma) \xi\otimes\ov{\zeta}}\quad
(\xi\in K_\sigma,\,\xi'\in K'_\sigma).
\]
For $a\in A$ and arbitrary $\xi,\xi'$ we have
\[\begin{split}
&\quad\;
\ismaa{\xi'}{S^{\sigma}_{\ov{\zeta}',\ov{\zeta}}\pi_\sigma(a)\xi}=
\ismaa{\xi'\otimes \ov{\zeta}'}{V(\sigma)( \pi_\sigma(a)\xi\otimes \ov{\zeta})}\\
&=
\ismaa{\pi'_{\eta(\sigma)}(a^*)\xi'\otimes \ov{\zeta}'}{V(\sigma) (\xi\otimes\ov{\zeta})}=
\ismaa{\pi'_{\eta(\sigma)}(a^*)\xi'}{
S^{\sigma}_{\ov{\zeta}',\ov{\zeta}}\xi}=
\ismaa{\xi'}{\pi'_{\eta(\sigma)}(a)
S^{\sigma}_{\ov{\zeta}',\ov{\zeta}}\xi}.
\end{split}\]
This means that $S^\sigma_{\ov{\zeta}',\ov{\zeta}}$ is a morphism between $\pi_{\sigma}$ and $\pi'_{\eta(\sigma)}$. It is clear that there exist $\ov{\zeta},\,\ov{\zeta}'$ for which $S^{\sigma}_{\ov{\zeta}',\ov\zeta}$ is non-zero. Consequently, as there are no nontrivial morphisms between nonequivalent irreducible representations, $\eta$ needs to be identity on $\Irr(A)\setminus (N\cup N')$. Therefore $\mu=\tilde{\mu}'$ on this set. This proves the first part of the lemma.\\
Assume now that $\mc{P}'$ transforms $\rho_\phi$ into $\int_{\Irr(A)}^{\oplus} (\I\otimes\check{\pi}'_\sigma)\md\mu(\sigma)$, and we have a family $(D'_\sigma)_{\sigma\in\Irr(A)}$ which meets conditions stated in the lemma. Thanks to the Schur's lemma we have
\[
S^\sigma_{\ov{\zeta}',\ov{\zeta}}=q(\ov{\zeta}',\ov{\zeta}) T_\sigma
\]
for a unitary intertwiner $T_\sigma\in \B(K_\sigma,K'_\sigma)$ and a bounded sesquilinear form $q$. We know how forms like this looks: there exists an operator $\tilde{T}_\sigma\in \B(\ov{K}_\sigma,\ov{K'_\sigma})$ such that
\[
\ismaa{\xi'\otimes \ov{\zeta}'}{V(\sigma) \xi\otimes\ov{\zeta}}=
\ismaa{\xi'}{S^\sigma_{\ov{\zeta}',\ov{\zeta}} \xi}=
\ismaa{\xi'\otimes \ov{\zeta}'}{
(T_\sigma\otimes \tilde{T}_\sigma)(\xi\otimes\ov{\zeta})}.
\]
Operator $\tilde{T}_\sigma$ is a morphism between $\check{\pi}_\sigma$ and $\check{\pi}'_\sigma$. Indeed, take $a,b\in A$. Then we have
\[\begin{split}
&\quad\;
\ismaa{\xi'\otimes \ov{\zeta}'}{
(T_\sigma\otimes \tilde{T}_\sigma)(\pi_\sigma(a)\xi\otimes\check{\pi}_\sigma(b)\ov{\zeta})}=
\ismaa{\xi'\otimes \ov{\zeta}'}{
V(\sigma)(\pi_\sigma(a)\xi\otimes\check{\pi}_\sigma(b)\ov{\zeta})}\\
&=
\ismaa{\xi'\otimes \ov{\zeta}'}{
(\pi_\sigma(a)\otimes \check{\pi}_\sigma(b))V(\sigma)(\xi\otimes\ov{\zeta})}=
\ismaa{\xi'\otimes\ov{\zeta}'}{
(\pi'_\sigma(a)\otimes \check{\pi}'_\sigma(b))
(T_\sigma\otimes\tilde{T}_\sigma)
(\xi\otimes\ov{\zeta})}.
\end{split}\]
Taking $a$ to be an approximate identity shows that $\tilde{T}_\sigma$ is morphism between $\check{\pi}_\sigma$ and $\check{\pi}'_\sigma$. The calculation
\[
\jmath(T_\sigma)\check{\pi}'_\sigma(a)\ov{\xi}=\jmath(T_\sigma) \ov{\pi'_\sigma(a^*)\xi}=\ov{T_\sigma^* \pi'_\sigma(a^*)\xi}=
\ov{\pi_\sigma(a^*) T_\sigma^*\xi}=
\check{\pi}(a) \jmath(T_\sigma)\ov{\xi}\quad(\ov{\xi}\in \ov{K'_\sigma},
\; a\in A)
\]
implies that $\jmath(T_\sigma)$ is a unitary morphism $\check{\pi}'_\sigma\rightarrow \check{\pi}_\sigma$. Schur's lemma shows that $\tilde{T}_\sigma=z_\sigma\jmath(T^{-1}_\sigma)$ for a certain $z_\sigma\in \CC$. Since
\[
1=\|V(\sigma)\|=\|T_\sigma\otimes\tilde{T}_\sigma\|=\|\tilde{T}_\sigma\|
=|z_\sigma|
\]
we know that $\tilde{T}_\sigma=z_\sigma \jmath(T^{-1}_\sigma)$ is a unitary operator. Let us see how $V(\sigma)$ acts on $\HS(K_\sigma)=K_\sigma\otimes\ov{K_\sigma}$. For every $\xi\otimes \ov{\zeta}\in K_\sigma\otimes\ov{K_\sigma}$ we have
\[
\begin{split}
&\quad\;V(\sigma) (|\xi\rangle\langle \zeta|)=
V(\sigma)(\xi\otimes\ov{\zeta})=
(T_\sigma \xi)\otimes (z_\sigma\jmath(T^{-1}_\sigma) \ov{\zeta})=
z_\sigma(T_\sigma\xi\otimes \ov{ T_\sigma \zeta})\\
&=
z_\sigma|T_\sigma \xi\rangle\langle T_\sigma \zeta |=
z_\sigma T_\sigma (|\xi\rangle\langle \zeta|) T_\sigma^{-1}
\end{split}
\]
Let us make use of our knowledge about operator $\mc{P}'$. For $a$ in a subspace $X\subseteq \mf{N}_\phi$ we have
\[
\begin{split}
&\quad\;\int_{\Irr(A)}^{\oplus}\pi'_\sigma(a) {D'}_\sigma^{-1}\md\mu'(\sigma)=
\mc{U} \int_{\Irr(A)}^{\oplus}\pi_\sigma(a) D_\sigma^{-1}\md\mu(\sigma)\\
&=
\int_{\Irr(A)}^{\oplus} \sqrt{\tfrac{\md\mu}{\md\mu'}(\sigma)} V(\sigma)(\pi_\sigma(a) D_\sigma^{-1})\md\mu'(\sigma),
\end{split}
\]
which implies
\[
\pi'_\sigma(a) {D'}_\sigma^{-1}=
\sqrt{\tfrac{\md\mu}{\md\mu'}(\sigma)} V(\sigma)(\pi_\sigma(a) D_\sigma^{-1})=
z_\sigma\sqrt{\tfrac{\md\mu}{\md\mu'}(\sigma)} T_\sigma(\pi_\sigma(a) D_\sigma^{-1})T_\sigma^{-1}
\]
for almost all $\sigma\in\Irr(A)$. Let $(a_n)_{n\in \NN}$ be a sequence in $X$ converging to $\I$ in \wot. Putting $a_n$ in the above equality and letting $n$ go to the infinity we get
\[
{D'_\sigma}^{-1}=z_\sigma \sqrt{\tfrac{\md\mu}{\md\mu'}(\sigma)} T_\sigma D_\sigma^{-1}T_\sigma^{-1}
\]
(recall that $\sigma$ is finite dimensional). Since both ${D'}_\sigma^{-1}$ and $\sqrt{\tfrac{\md\mu}{\md\mu'}(\sigma)} T_\sigma D_\sigma^{-1}T_\sigma^{-1}$ are positive operators, we must have $z_\sigma=1$. We will arive at the desired equation once we take an inverse of both sides.
\end{proof}

Using Theorem \ref{Desmedtcstar} we can derive a result concerning quantum groups (see \cite[Theorem 3.4.1]{Desmedt}, also cf.~\cite[Theorem 1.6.1]{Caspers}). Recall that $\Irr(\CGDu)$ stands for the spectrum of the \cst-algebra $\CGDu$. We again remark that we will use theorems \ref{PlancherelL}, \ref{PlancherelR} only for quantum groups $\GG$ with \cst-algebra $\CGDu$ of type I.

\begin{theorem}\label{PlancherelL}
Let $\GG$ be a locally compact quantum group such that $\Linfd$ is a von Neumann algera of type I and the \cst-algebra $\CGDu$ is separable. There exists a standard measure $\mu$ on $\IrrG(=\Irr(\CGDu))$, a measurable field of Hilbert spaces $(\msf{H}_\pi)_{\pi\in \IrrG}$, measurable field of representations, measurable field of strictly positive self-adjoint operators $(D_\pi)_{\pi\in \IrrG}$ and a unitary operator $\mc{Q}_L\colon \msf{H}_{\hvp}\rightarrow \int_{\IrrG}^{\oplus} \HS(\msf{H}_\pi)\md\mu(\pi)$ such that:
\begin{enumerate}[label=\arabic*)]
\item For all $\alpha\in \Lj$ such that $\lambda(\alpha)\in\mf{N}_{\hvp}$ and $\mu$-almost every $\pi\in\IrrG$ the operator $(\alpha\otimes\id) (U^{\pi}) D_{\pi}^{-1}$ is bounded, and its closure $(\alpha\otimes\id) (U^{\pi})\cdot D_{\pi}^{-1}$ is Hilbert-Schmidt.
\item For all $\alpha,\beta\in\Lj$ such that $\lambda(\alpha),\lambda(\beta)\in \mf{N}_{\hvp}$ we have the Parseval formula:
\[
\ismaa{\Lhvp(\lambda(\alpha))}{\Lhvp(\lambda(\beta))}=
\int_{\IrrG} \Tr\bigl(
\bigl((\alpha\otimes\id) (U^{\pi})\cdot D_{\pi}^{-1}\bigr)^{*}
\bigl((\id\otimes\beta) (U^{\pi})\cdot D_{\pi}^{-1}\bigr)
\bigr)\md\mu(\pi),
\]
and $\mc{Q}_L$ is an isometric extension of
\[
\Lhvp(\lambda(\Lj)\cap \mf{N}_{\hvp})\ni \Lhvp(\lambda(\alpha))\mapsto
\int_{\IrrG}^{\oplus} 
(\alpha\otimes\id) (U^{\pi})\cdot D_{\pi}^{-1}\md\mu(\pi)
\in
\int_{\IrrG}^{\oplus} \HS(\msf{H}_\pi)\md\mu(\pi),
\]
\item
$\mc{Q}_L$ satisfies the following equations:
\[
\mc{Q}_L (\omega\otimes\id)\mrW=
\bigl(\int_{\IrrG}^{\oplus} (\omega\otimes\id)U^{\pi}\otimes\I_{\ov{\msf{H}_{\pi}}}\md\mu(\pi)\bigr)\mc{Q}_L
\]
and
\[
\mc{Q}_L (\omega\otimes\id)\chi(\mrV)=
\bigl(\int_{\IrrG}^{\oplus} \I_{\msf{H}_\pi}\otimes \pi^{c}((\omega\otimes\id){\WW})\md\mu(\pi)\bigr)\mc{Q}_L
\]
for every $\omega\in\LL^1(\GG)$.
\item
If $\alpha\in \LL^1(\GG)$ is such that $\lambda(\alpha)\in\LL^{\infty}(\whG)_+$, then we have
\[
\hvp(\lambda(\alpha))=
\int_{\IrrG} \Tr((\alpha\otimes\id)(U^{\pi}) \cdot {D_{\pi}}^{-2})
\md\mu(\pi).
\]
\item
Haar integrals on $\whG$ are tracial if and only if almost all $D_{\pi}$ are multiples of the identity.
\item Operator $\mc{Q}_L$ transforms $\LL^{\infty}(\whG)\cap \LL^{\infty}(\whG)'$ into diagonalisable operators.
\item
Assume that all irreducible representations of $\GG$ are finite dimensional. Let $\mu'$, $(\msf{H}'_{\pi'})_{\pi'\in\IrrG}$, $(D'_{\pi'})_{\pi'\in\IrrG}$ be other objects of the above type (with $\pi,\pi'\in[\pi']=[\pi]$). Assume that there exists a unitary operator $\mc{Q}'_L\colon \msf{H}_{\hvp}\rightarrow \int_{\IrrG}^{\oplus} \HS(\msf{H}'_{\pi'})\md\mu'(\pi')$. If
\begin{enumerate}
\item[7.1)] point $3)$ is satisfied for $\mu'$, $(\msf{H}'_{\pi'})_{\pi'\in\IrrG}$, $(D'_{\pi'})_{\pi'\in\IrrG}$,
\item[7.2)] we have
\[
\mc{Q}'_L \Lambda_{\hvp}(\lambda(\alpha))=
\int_{\IrrG}^{\oplus}
(\alpha\otimes\id)U^{\pi'} {D'}_{\pi'}^{-1}
\md\mu'(\pi')
\]
for all $\lambda(\alpha)$ in a subspace $X\subseteq\mf{N}_{\hvp}$ such that $\I$ belongs to the \wot -- sequential closure of $X$,
\item[7.3)] $\mc{Q}'_L$ transforms $\LL^{\infty}(\whG)\cap\LL^{\infty}(\whG)'$ into diagonalisable operators
\end{enumerate}
then $\mu$ and $\mu'$ are equivalent and for $\mu$-almost all $\pi\in\IrrG$ there exists a unitary intertwiner $T_\pi\colon \msf{H}_\pi\rightarrow\msf{H}'_{\pi'}$ such that
\[
D'_{\pi'}=\sqrt{\tfrac{\md\mu'}{\md\mu}(\pi)} T_\pi D_{\pi} {T_\pi}^{-1}.
\]
Moreover, objects with primes satisfies all the properties $1)$--$6)$.
\item We can assume that $(\msf{H}_\pi)_{\pi\in\IrrG}$ is the canonical measurable field of Hilbert spaces.
\end{enumerate}
\end{theorem}

\begin{remark}$ $
\begin{itemize} 
\item We are abusing notation by using the same letter $\pi$ for both: a class of representations and chosen representative.
\item There are minor differences between this version of theorem and the version which appears in Desmedt's dissertation: 
\begin{itemize}
\item we prefer to state this result in a slightly more flexible way: rather than using functionals from $\mc{I}$, we prefer to take arbitrary $\alpha\in\Lj$ and assume that $\lambda(\alpha)\in\mf{N}_{\hvp}$ (see appendix \ref{appendixqg} for the definition of $\mc{I}$),
\item we believe that it is necessary (and worthwhile) to include point $6)$ in the statement of the theorem.
\end{itemize}
Due to these differences we include a proof. 
\end{itemize}
\end{remark}

\begin{proof}
We use Theorem \ref{Desmedtcstar} for $A=\CGDu$ and $\phi=\hvp^{u}$. In this way we get the objects that appear in the theorem.\\
Observe that for an arbitrary $\alpha\in \LL^1(\GG)$ we have
\[
\begin{split}
\mf{N}_{\hvp^u}\ni (\alpha\otimes\id) {\WW}
&\Leftrightarrow
\hvp^{u} \bigl( ((\alpha\otimes\id) {\WW})^*
((\alpha\otimes\id) {\WW})\bigr)<\infty\\
&\Leftrightarrow
\hvp( \lambda(\alpha)^* \lambda(\alpha))<\infty\\
&\Leftrightarrow
\lambda(\alpha)\in \mf{N}_{\hvp}.
\end{split}
\]
Points $1),2)$ follow from
\[
\pi((\alpha\otimes \id){\WW})=(\alpha\otimes\id)U^{\pi}
\]
for $\alpha\in \Lj,\pi\in\IrrG$ and points $1),2)$ of Theorem \ref{Desmedtcstar} combined with the equality
\[
\Lambda_{\hvp^{u}}((\alpha\otimes\id){\WW})
=\Lambda_{\hvp}(\lambda(\alpha))
\]
for $\alpha\in\Lj$ such that $\lambda(\alpha)\in\mf{N}_{\hvp}\Leftrightarrow (\alpha\otimes\id){\WW}\in\mf{N}_{\hvp^{u}}$. Observe also that\\
$\Lambda_{\hvp}(\lambda(\Lj)\cap\mf{N}_{\hvp})$ is a dense subspace of $\LdG$ (Lemma \ref{lemat28}).\\
Let us justify point $3)$: the first part follows directly from Theorem \ref{Desmedtcstar}, the second one from calculation
\[
\begin{split}
&\quad\;\check{\pi}((\alpha\otimes\id){\WW})=
\jmath_{\msf{H}_\pi}(\pi((\alpha\otimes\id){\WW}))=
\pi^{c}\circ \hat{R}^{u} ((\alpha\otimes\id){\WW})=
\pi^{c}((\alpha\circ R\otimes\id){\WW}),\\
&
\rho_{\hvp^{u}}((\alpha\otimes\id){\WW})=
\hat{J} \lambda(\alpha)^* \hat{J}=
(\alpha\circ R\otimes\id)( (\hat{J}\otimes\hat{J}) (\mrW)^*
(\hat{J}\otimes\hat{J}) )=
(\alpha\circ R\otimes\id) \chi(\mrV)
\end{split}
\]
and equality $\alpha\circ R\circ R=\alpha$ which holds for all $\alpha\in\Lj$. \\
Now we turn to the point $4)$: we can find $x\in \CGDu_+$ such that $\Lambda_{\whG}(x)=\lambda(\alpha)$ (because $\Lambda_{\whG}$ is a $\star$-homomorphism). Therefore
\[
\hvp(\lambda(\alpha))={\hvp}^{u}(x)=
\int_{\IrrG} \Tr(\pi(x)\cdot D_\pi^{-2})\md\mu(\pi).
\]
Theorem 3.4.8 in \cite{Desmedt} says that support of $\mu$ contains representations which factor through $\CGD$ (we can use this result because its proof does not rely on the property which we now wish to derive). This means that for all $\pi$ in the support of $\mu$ we can find a representation of $\CGD$, $\pi'$ such that $\pi=\pi'\circ \Lambda_{\whG}$, consequently
\[
\begin{split}
&\quad\;\pi(x)=\pi'(\Lambda_{\whG}(x))=\pi'(\lambda(\alpha))=
(\alpha\otimes\id)(\id\otimes\pi')\mrW
\\
&=(\alpha\otimes\id)(\id\otimes\pi'\circ \Lambda_{\whG})\WW=
(\alpha\otimes\id)(\id\otimes\pi)\WW=
(\alpha\otimes\id)U^{\pi}
\end{split}
\]
and
\[
\hvp(\lambda(\alpha))=\int_{\IrrG} \Tr((\alpha\otimes\id)U^{\pi}\cdot D_\pi^{-2})\md\mu(\pi).
\]
 Points $5),6),8)$ are immediate. Let us now justify point $7)$.\\
Lemma \ref{lemat13} implies that measures $\mu,\mu'$ are equivalent and for each $\pi$ there exists a unitary intertwiner $T_\pi\colon \msf{H}_\pi \rightarrow \msf{H}'_{\pi'}$ such that
\[
D'_{\pi'}=\sqrt{\tfrac{\md \mu'}{\md\mu}(\pi)} T_\pi D_\pi T_\pi^{-1}.
\]
Moreover, the proof of Lemma \ref{lemat13} implies that $\mc{Q}'_L \mc{Q}_L^{-1}$ is given by a composition
\[
\bigl(\int_{\IrrG}^{\oplus}
\HS(\msf{H}'_{\pi})\md \mu (\pi)
\rightarrow
\int_{\IrrG}^{\oplus}
\HS(\msf{H}'_{\pi})\md \mu' (\pi)\bigr)\circ
\int_{\IrrG}^{\oplus} V(\pi)\md\mu(\pi),
\]
where $V(\pi)\colon \HS(\msf{H}_\pi)\rightarrow \HS(\msf{H}'_\pi)$ is the map given by a $S\mapsto T_\pi S T_\pi^{-1}$. Take $\alpha\in\Lj$ such that $\lambda(\alpha)\in\mf{N}_{\hvp}$. Then we have
\[\begin{split}
&\quad\;
\mc{Q}'_L \Lambda_{\hvp}(\lambda(\alpha))=
\mc{Q}'_L \mc{Q}_L^{-1}
\int_{\IrrG}^{\oplus} (\alpha\otimes\id)U^{\pi}\,D_\pi^{-1} \md\mu(\pi)\\
&=
\int_{\IrrG}^{\oplus}
\sqrt{\tfrac{\md\mu}{\md\mu'}(\pi)}\; T_\pi (\alpha\otimes\id)U^{\pi} D_\pi^{-1} T_\pi^{-1} \md\mu'(\pi)\\
&=
\int_{\IrrG}^{\oplus}
\sqrt{\tfrac{\md\mu}{\md\mu'}(\pi)}\; (\alpha\otimes\id)U^{\pi'} T_\pi D_\pi^{-1} T_\pi^{-1} \md\mu'(\pi)\\
&=
\int_{\IrrG}^{\oplus}
(\alpha\otimes\id)U^{\pi'} {D'}_\pi^{-1} \md\mu'(\pi).
\end{split}\]
In the above calculation we have used the property that $T_\pi$ is an intertwiner. This proves point $7)$.
\end{proof}

Using Theorem \ref{Desmedtcstar} for $A=\CGDu$ and $\phi=\wh{\psi}^u$ we can get a right version of the above theorem (cf.~\cite[Theorem 1.6.3]{Caspers}).

\begin{theorem}\label{PlancherelR}
Let $\GG$ be a locally compact quantum group such that $\Linfd$ is a von Neumann algebra of type I and the \cst-algebra $\mathrm{C}_0^{u}(\whG)$ is separable. There exists a standard measure $\mu^R$ on $\IrrG(=\Irr(\CGDu))$, a measurable field of Hilbert space $(\msf{K}_\pi)_{\pi\in \IrrG}$, measurable field of representations, measurable field of strictly positive self-adjoint operators $(E_\pi)_{\pi\in \IrrG}$ and a unitary operator $\mc{Q}_R\colon \msf{H}_{\wh{\psi}}\rightarrow \int_{\IrrG}^{\oplus} \HS(\msf{K}_\pi)\md\mu^{R}(\pi)$ such that:
\begin{enumerate}[label=\arabic*)]
\item For all $\alpha\in \LL^1(\GG)$ such that $\lambda(\alpha)\in\mf{N}_{\hpsi}$ and $\mu^R$-almost every $\pi\in\IrrG$ operator $(\alpha\otimes\id) (U^{\pi}) E_{\pi}^{-1}$ is bounded and its closure $(\alpha\otimes\id) (U^{\pi})\cdot E_{\pi}^{-1}$ is Hilbert-Schmidt.
\item For all $\alpha,\beta\in\LL^1(\GG)$ such that $\lambda(\alpha),\lambda(\beta)\in\mf{N}_{\hpsi}$ we have the Parseval formula
\[
\ismaa{\Lhvps(\lambda(\alpha))}{\Lhvps(\lambda(\beta))}=
\int_{\IrrG} \Tr\bigl(
\bigl((\alpha\otimes\id) (U^{\pi})\cdot E_{\pi}^{-1}\bigr)^{*}
\bigl((\beta\otimes\id) (U^{\pi})\cdot E_{\pi}^{-1}\bigr)
\bigr)\md\mu^{R}(\pi),
\]
and $\mc{Q}_R$ is an isometric extension of
\[
\begin{split}
\hat{J}J\Lhvps(\lambda(\LL^1(\GG))\cap \mf{N}_{\hpsi})\ni &
\hat{J}J\Lhvps(\lambda(\alpha))\mapsto\\
&\mapsto
\int_{\IrrG}^{\oplus} 
(\alpha\otimes\id) (U^{\pi})\cdot E_{\pi}^{-1}\md\mu^R(\pi)
\in
\int_{\IrrG}^{\oplus} \HS(\msf{K}_\pi)\md\mu^R(\pi),
\end{split}
\]
\item
$\mc{Q}_R$ satisfies the following equations:
\[
\mc{Q}_R\hat{J}J (\omega\otimes\id)\mrW=
\bigl(\int_{\IrrG}^{\oplus} (\omega\otimes\id)U^{\pi}\otimes\I_{\ov{\msf{H}_{\pi}}}\md\mu^R(\pi)\bigr)\mc{Q}_R \hat{J}J
\]
and
\[
\mc{Q}_R \hat{J}J(\omega\otimes\id)\chi(\mrV)=
\bigl(\int_{\IrrG}^{\oplus} \I_{\msf{H}_\pi}\otimes \pi^{c}((\omega\otimes\id){\WW})\md\mu^R(\pi)\bigr)\mc{Q}_R\hat{J}J
\]
for every $\omega\in\LL^1(\GG)$.
\item
If $\beta\in \LL^1(\GG)$ is such that $\lambda(\beta)\in\LL^{\infty}(\whG)_+$, then
\begin{equation}\label{eq42}
\hpsi(\lambda(\beta))=
\int_{\IrrG} \Tr((\beta\otimes\id)(U^{\pi}) \cdot {E_{\pi}}^{-2})
\md\mu^R(\pi).
\end{equation}
\item
Haar inegrals on $\whG$ are tracial if and only if almost all $E_{\pi}$ are multiples of the identity.
\item Operator $\mc{Q}_R$ transforms $\LL^{\infty}(\whG)\cap \LL^{\infty}(\whG)'$ into diagonalisable operators.
\item
Assume that all irreducible representations of $\GG$ are finite dimensional. Let ${\mu'}^R$, $(\msf{K}'_{\pi'})_{\pi'\in\IrrG}$, $(E'_{\pi'})_{\pi'\in\IrrG}$ be other objects of the above type (with $\pi,\pi'\in[\pi']=[\pi]$). Assume that there exists a unitary operator $\mc{Q}'_R\colon \msf{H}_{\hpsi}\rightarrow \int_{\IrrG}^{\oplus} \HS(\msf{K}'_{\pi'})\md{\mu'}^R(\pi')$. If
\begin{enumerate}
\item[7.1)] point $3)$ is satisfied for ${\mu'}^R$, $(\msf{K}'_{\pi'})_{\pi'\in\IrrG}$, $(E'_{\pi'})_{\pi'\in\IrrG}$,
\item[7.2)] we have
\[
\mc{Q}'_R \hat{J}J\Lambda_{\hpsi}(\lambda(\alpha))=
\int_{\IrrG}^{\oplus}
(\alpha\otimes\id)U^{\pi'} {E'}_{\pi'}^{-1}
\md{\mu'}^R(\pi')
\]
for all $\lambda(\alpha)$ in a subspace $X\subseteq\mf{N}_{\hpsi}$ such that $\I$ belongs to the \wot -- sequential closure of $X$,
\item[7.3)] $\mc{Q}'_R$ transforms $\LL^{\infty}(\whG)\cap\LL^{\infty}(\whG)'$ into diagonalisable operators
\end{enumerate}
then $\mu^R$ and ${\mu'}^R$ are equivalent and for $\mu^R$-almost all $\pi\in\IrrG$ there exists a unitary intertwiner $T_\pi\colon \msf{K}_\pi\rightarrow\msf{K}'_{\pi'}$ such that
\[
E'_{\pi'}=\sqrt{\tfrac{\md{\mu'}^R}{\md\mu^R}(\pi)} T_\pi E_{\pi} {T_\pi}^{-1}.
\]
Moreover, objects with primes satisfies all the properties $1)$--$6)$.
\item
We can choose $\mu^R=\mu$ and $\msf{K}_\pi=\msf{H}_\pi$ (and the same field of representations as in Theorem \ref{PlancherelL}).
\end{enumerate}
\end{theorem}

\begin{proof}
We use the same argument as in the proof of Theorem \ref{PlancherelL}. We use Theorem \ref{Desmedtcstar} for $A=\mathrm{C}^{u}_{0}(\whG)$ and $\phi=\wh{\psi}_u$. In this way we get a measure $\mu^R$, measurable field of Hilbert spaces, representations, operators $(E_\pi)_{\pi\in\IrrG}$, and a unitary operator $\mc{Q}_{R,0}$. These objects satisfy the conditions of the theorem, except for the commutation rules (point $3)$) and that $\mc{Q}_{R,0}$ maps vector $\Lambda_{\hpsi}(\lambda(\alpha))$ to $\int_{\IrrG}^{\oplus} (\alpha\otimes\id)U^{\pi}\cdot E_\pi^{-1}\md\mu^R(\pi)$.  Let us define a unitary operator $\mc{Q}_R$ as $\mc{Q}_R=\mc{Q}_{R,0}\circ J\hat{J}$. Now, this operator meets conditions $2)$ and $3)$ (we justify point $3)$ in the same way, as in the proof of Theorem \ref{PlancherelL}). The subspace $\hat{J}J\Lambda_{\hpsi}(\lambda(\LL^1(\GG))\cap \mf{N}_{\hpsi})$ is dense in $\LdG$, which is shown in Lemma \ref{lemat28}.\\
Since conjugation by $J$ preserves $\Linfd$ and $\Linfd'$ and moreover $\hat{J}\Linfd \hat{J}=\Linfd'$, point $6)$ holds also for $\mc{Q}_R$. Points $4)$ and $7)$ can be obtain as in the left version of the theorem. Point 5 is clear, point 8 is proven in the Caspers' dissertation \cite{Caspers}.
\end{proof}

Whenever we say that we choose a Plancherel measure $\mu$ for $\GG$, we will in fact mean that we choose all of the objects
\[
\mu,(\msf{H}_\pi)_{\pi\in\IrrG},(D_\pi)_{\pi\in\IrrG},(E_\pi)_{\pi\in\IrrG},\mc{Q}_L,\mc{Q}_R
\]
and a measurable family of representations. Unless said otherwise, we will always choose the canonical measurable field of Hilbert spaces. In the examples section we describe the Plancherel measure for special classes of locally compact quantum groups: compact, classical, dual to classical and certain class of quantum groups constructed via the bicrossed product.

\begin{remark}
If $\whG$ is unimodular, we can take $E_\pi=D_\pi$ and $\mc{Q}_R=\mc{Q}_L\circ J\hat{J}$.
\end{remark}

The above theorem says that two Plancherel measures are equivalent. One can ask the following question: having a measure equivalent to the Plancherel measure, can we find operators $D_\pi$ which satisfy conditions of Theorem \ref{PlancherelL}? The following propositions states that this is indeed the case: 

\begin{proposition}\label{stw14}
Assume that $\GG$ is a locally compact quantum group such that $\CGDu$ is separable and of type I. Let $\mu,(\msf{H}_\pi)_{\pi\in\IrrG},(D_\pi)_{\pi\in\IrrG},(E_\pi)_{\pi\in\IrrG},\mc{Q}_L,\mc{Q}_R$ be objects given by theorems \ref{PlancherelL}, \ref{PlancherelR}. Assume that $\mu'$ is a measure equivalent to $\mu$. Define $\msf{H}'_\pi=\msf{H}_\pi$ (the same structure of measurable field of Hilbert spaces and of measurable field of representations),
\[
\begin{split}
\mc{Q}'_L\colon& \msf{H}_{\hvp}\xrightarrow[]{\mc{Q}_L} \int_{\IrrG}^{\oplus} \HS(\msf{H}_\pi)\md\mu(\pi)\rightarrow
\int_{\IrrG}^{\oplus} \HS(\msf{H}'_\pi)\md\mu'(\pi)\\
\mc{Q}'_R\colon &\msf{H}_{\hpsi}\xrightarrow[]{\mc{Q}_R} \int_{\IrrG}^{\oplus} \HS(\msf{H}_\pi)\md\mu(\pi)\rightarrow
\int_{\IrrG}^{\oplus} \HS(\msf{H}'_\pi)\md\mu'(\pi),
\end{split}
\]
where $\int_{\IrrG}^{\oplus} \HS(\msf{H}_\pi)\md\mu(\pi)\rightarrow
\int_{\IrrG}^{\oplus} \HS(\msf{H}'_\pi)\md\mu'(\pi)$ is the canonical unitary \\$\int_{\IrrG}^{\oplus} \xi_\pi\md\mu(\pi)\mapsto
\int_{\IrrG}^{\oplus} \xi_\pi \sqrt{\tfrac{\md\mu}{\md\mu'}(\pi)}\md\mu'(\pi)$ and
\[
D'_\pi=\sqrt{\tfrac{\md\mu'}{\md\mu}(\pi)} D_\pi,\quad
E'_\pi=\sqrt{\tfrac{\md\mu'}{\md\mu}(\pi)} E_\pi\quad(\pi\in\IrrG).
\]
Then objects with primes satisfy theorems \ref{PlancherelL}, \ref{PlancherelR}.
\end{proposition}

\begin{proof}
It is clear that $\mc{Q}'_L,\mc{Q}'_R$ are unitary operators -- they are given by a composition of unitaries. Operators $D'_\pi,E'_\pi$ form a measurable field of positive invertible operators. Points $1)$ of theorems \ref{PlancherelL}, \ref{PlancherelR} are satified since measures $\mu$, $\mu'$ are equivalent. For $\alpha,\beta\in\Lj$ such, that $\alpha,\beta\in\mf{N}_{\hvp}$ we have
\[
\begin{split}
&\quad\;\int_{\IrrG} \Tr\bigl(
\bigl((\alpha\otimes\id) (U^{\pi})\cdot D_{\pi}^{-1}\bigr)^{*}
\bigl((\beta\otimes\id) (U^{\pi})\cdot D_{\pi}^{-1}\bigr)
\bigr)\md\mu(\pi)\\
&=\int_{\IrrG} \Tr\bigl(
\bigl((\alpha\otimes\id) (U^{\pi})\cdot {D'}_{\pi}^{-1}\bigr)^{*}
\bigl((\beta\otimes\id) (U^{\pi})\cdot {D'}_{\pi}^{-1}\bigr)
\bigr)\md\mu'(\pi),
\end{split}
\]
and
\[
\mc{Q}'_L \Lambda_{\hvp}(\lambda(\alpha))=
\int_{\IrrG}^{\oplus} \sqrt{\tfrac{\md\mu}{\md\mu'}(\pi)}(\alpha\otimes\id)U^\pi\cdot D_\pi^{-1} \md\mu'(\pi)=
\int_{\IrrG}^{\oplus} (\alpha\otimes\id)U^\pi\cdot {D'}_\pi^{-1} \md\mu'(\pi)
\]
which shows that point $2)$ of Theorem \ref{PlancherelL} is satisfied. In a similar manner we check that point $2)$ of Theorem \ref{PlancherelR} holds. Recall that operators of the form $\int_{\IrrG}^{\oplus} T_\pi \md\mu(\pi)\; (T_\pi\in \B(\HS(\msf{H}_\pi)))$ are called decomposable. Point $3)$ is true since the mapping $\int_{\IrrG}^{\oplus}\HS(\msf{H}_\pi)\md\mu(\pi)\rightarrow \int_{\IrrG}^{\oplus}\HS(\msf{H}'_\pi)\md\mu'(\pi)$ transforms decomposable operators as follows: $\int_{\IrrG}^{\oplus} T_\pi\md\mu(\pi)\mapsto$\\ $\mapsto \int_{\IrrG}^{\oplus} T_\pi \md\mu'(\pi)$. The rest is straightforward.
\end{proof}

\section{Integral representations}\label{secintrep}
From now on we assume that $\whG$ is \emph{second countable}, which means that the \cst-algebra $\CGDu$ is separable (see Lemma \ref{lemat37} for equivalent conditions).\\
Let $(X,\mf{M}_X,\mu_X)$ be a $\sigma$-finite measure space, $(\msf{H}_x)_{x\in X}$ a measurable field od Hilbert spaces, and $(\pi_x)_{x\in X}$ a measurable field of representations of $\CGDu$. Define $\pi_X=\int_X^{\oplus}\pi_x\md\mu_X(x)$. Then $( \pi_x(\CGDu)'')_{x\in X}$ is a measurable field of von Neumann algebras (\cite[Definition 1, page 197]{DixmiervNA}): there exists a dense subset $\{a_i\}_{i\in \NN}$ in $\CGDu$. Then $\{\int_X^{\oplus} \pi_x(a_i)\md\mu_X(x)\}_{i\in \NN}$ is a countable family of decomposable operators, and for almost all $x\in X$ family $\{\pi_x(a_i)\}_{i\in \NN}$ generates $\pi_x(\CGDu)''$. It follows that we can form a direct integral von Neumann algebra 
\[
\int_X^{\oplus} \pi_x(\CGDu)''\md\mu_X(x)=
\bigr\{\int_X^{\oplus} T_x \md\mu_X(x)\,\big|\, \textnormal{for almost all } x\in X,\;T_x\in \pi_x(\CGDu)''\bigr\}
\]
and we have
\[
\pi_X(\CGDu)''\subseteq \int_X^{\oplus} \pi_x(\CGDu)''\md\mu_X(x).
\]
In general this inclusion is strict.
\begin{definition}\label{defintrep}
An \emph{integral representation} of $\GG$ is a sextuple\footnote{
We could consider a more general situation of integral representations equipped with a measurable field of faithful normal functionals $(\psi_x)_{x\in X}$, where $\psi_x\in \B(\msf{H}_x)_*$. We could then define integral character and weight via respectively $\int_X (\id\otimes\psi_x) U^{\pi_x}\md\mu_X(x)$ and $\int_X^{\oplus} \psi_x\md\mu_X(x)$. Most results from sections \ref{secintrep}, \ref{secintweight}, \ref{secweightchar} generalize to this setting. Definition \ref{defintrep} corresponds to the case $\psi_x=\Tr_{\msf{H}_x}\,(x\in X)$.
} $(X,\mf{M}_X,\mu_X,(\msf{H}_x)_{x\in X}$, $\{(e^i_x)_{x\in X}\}_{i\in \NN}, (\pi_x)_{x\in X})$, where
\begin{enumerate}[label=\arabic*)]
\item $(X,\mf{M}_X,\mu_X)$ is a $\sigma$-finite measure space,
\item $(X,(\msf{H}_x)_{x\in X},\{(e^i_x)_{x\in X}\}_{i\in \NN})$ is a measurable field of Hilbert spaces such that $\dim(\msf{H}_x)<+\infty$ for all $x\in X$,
\item $(\pi_x)_{x\in X}$ is a measurable field of nondegenerate representations of $\CGDu$.
\end{enumerate}
\end{definition}
With every integral representation we can associate a nondegenerate representation of $\CGDu$: 
\[
\pi_X=\int_X^{\oplus}\pi_x\md\mu_X(x)\colon \CGDu\rightarrow
\B(\int_X^{\oplus} \msf{H}_x\md\mu_X(x)).
\]
To ease the notation we will write that an integral representation is $(X,\mu_X,(\pi_x)_{x\in X})$ or even $\pi_X$ if there is no risk of confusion. With an integral representation we will associate three more notions: an \emph{integral dimension}
\[
\dim^{\int}(\pi_X)=\dim^{\int}(\int_X^{\oplus} \pi_x\md\mu_X(x))=\int_X \dim(\pi_x) \md\mu_X.
\]
and an \emph{integral character} (whenever the integral dimension is finite)
\begin{equation}\label{eq25}
\chi^{\int}(\int_X^{\oplus}\pi_x\md\mu_X(x))=
\int_X (\id\otimes\Tr_x)U^{\pi_x}\md\mu_X(x)\in \Linf.
\end{equation}
(Symbol $\Tr_x$ denotes the trace over $\msf{H}_x$). We stress that the integral dimension and the integral character of a given integral representation in general are not preserved by unitary equivalence. Moreover, we define the \emph{integral weight} $\Psi^{\int_X}$ as the restriction of the weight $\int_X^{\oplus}\Tr_x\md\mu_X(x)$ to $\pi_X(\CGDu)''$ (\cite[Definition 1, page 223]{DixmiervNA}). It is a faithful, normal tracial weight, though not necessarily semifinite. If the integral dimension of $\pi_X$ is finite, the integral weight is a bounded functional. Let us introduce few more convenient definitions: 

\begin{definition}
Let $(X,\mu_X,(\pi_x)_{x\in X})$ be an integral representation of $\GG$ and let $n\in\NN$. We define a measurable subset
\[
X\rest_n=\{x\in X\,|\, \dim(\pi_x)=n\}.
\]
\end{definition}

\begin{definition}
The family of integral representations will be denoted by $\Rep^{\int}(\GG)$. Moreover, $\Rep^{\int}_{<+\infty}(\GG)$ will stand for the family of integral representations with finite integral dimension.
\end{definition}

Let us prove a certain technical lemma which will be of use later on. 

\begin{lemma}\label{lemat22}
Let $(X,\mu_X,(\pi_x)_{x\in X})$ be an integral representation of $\GG$. We have $\pi_X\lec \Lambda_{\whG}$ if and only if $\pi_x\lec \Lambda_{\whG}$ for $\mu_X$-almost all $x\in X$.
\end{lemma}

Recall that symbol $\lec$ denotes weak containment: for nondegenerate representations $\rho,\rho'$ we have $\rho\lec\rho'$ if and only if $(\rho'(a)=0)\Rightarrow (\rho(a)=0)$ for all $a$.\\
We remark that the above lemma could be phrased in more general setting, with $\Lambda_{\whG}$ replaced by an arbitrary nondegenerate representation.

\begin{proof}
Assume $\pi_X\lec \Lambda_{\whG}$, then we can define a representation $\rho\colon \CGD\rightarrow \B(\int_{X}^{\oplus}\msf{H}_x\md\mu_X(x))$ via
\[
\rho(\Lambda_{\whG}(a))=\pi_X(a)\quad(a\in \CGDu).
\]
For every $\Lambda_{\whG}(a)\in \CGD$ operator $\rho(\Lambda_{\whG}(a))$ is decomposable, therefore we can use \cite[Lemma 8.3.1]{DixmierC} -- for every $x\in X$ there exists a representation $\pi'_x\colon \CGD\rightarrow \B(\msf{H}_x)$ such that
\[
\int_X^{\oplus} \pi_x(a)\md\mu_X(x)=\pi_X(a)=
\rho(\Lambda_{\whG}(a))=
\int_X^{\oplus} \pi'_x(\Lambda_{\whG}(a)) \md\mu_X(x)\quad(a\in \CGDu).
\]
Let $\{a_i\}_{i\in\NN}$ be a countable dense subset of $\CGDu$, fix $i\in\NN$. Thanks to the above equality we can find subset of full measure $X_i\subseteq X$ such that $\pi_x(a_i)=\pi'_x(\Lambda_{\whG}(a_i))$ for all $x\in X_i$. Intersection $X_0=\bigcap_{i=1}^{\infty} X_i$ is also of full measure and we have $\pi_x(a_i)=\pi'_x(\Lambda_{\whG}(a_i))$ for all $(i,x)\in \NN\times X_0$. Density of $\{a_i\}_{i=1}^{\infty}$ gives us $\pi_x(a)=\pi'_x(\Lambda_{\whG}(a))$ for all $a\in\CGDu$ and $x\in X_0$. Consequently we get $\pi_x\lec \Lambda_{\whG}\;(x\in X_0)$.\\
One the other hand, assume that $\pi_x\lec \Lambda_{\whG}$ for almost all $x\in X$ and take $a\in \CGDu$ such that $\Lambda_{\whG}(a)=0$. Then we have
\[
\pi_X(a)=\int_X^{\oplus}\pi_x(a)\md\mu_X(x)=0,
\]
and it follows that $\pi_X\lec \Lambda_{\whG}$.
\end{proof}

\subsection{Operations on integral representations}
Let $(X,\mu_X,(\pi_x)_{x\in X}),(Y,\mu_Y,(\gamma_y)_{y\in Y})\in\Rep^{\int}(\GG)$ be two integral representations of $\GG$. With these representations we can associate three new integral representations.

\subsubsection{Direct sum}
We define a direct sum of the above integral representations as
\[
(X\sqcup Y, \mu_{X\sqcup Y},(\kappa_{z})_{z\in X\sqcup Y}),
\]
where $\mu_{X\sqcup Y}$ is the obvious measure on the disjoint union of measure spaces. We equip $(X\sqcup Y,\mu_{X\sqcup })$ with the natural measurable field of Hilbert spaces. Moreover, we define $\kappa_z=\pi_x$ if $z=x\in X$ and $\kappa_z=\gamma_y$ if $z=y\in Y$. We have an identification (unitary isomorphism)
\[
\int_{X\sqcup Y}^{\oplus} \msf{H}_z \md\mu_{X\sqcup Y}(z)=
\int_{X}^{\oplus} \msf{H}_x \md\mu_{X}(x)\oplus
\int_{ Y}^{\oplus} \msf{H}_y \md\mu_{ Y}(y)
\]
which gives
\[
\int_{X\sqcup Y}^{\oplus} \kappa_z \md\mu_{X\sqcup Y}(z)=
\int_{X}^{\oplus} \pi_x \md\mu_{X}(x)\oplus
\int_{ Y}^{\oplus} \gamma_y \md\mu_{ Y}(y).
\]
Therefore we have expressions for the integral dimension and weight:
\[\begin{split}
\dim^{\int}(\int_{X\sqcup Y}^{\oplus} \kappa_z \md\mu_{X\sqcup Y}(z))&=
\dim^{\int}(\int_{X}^{\oplus} \pi_x \md\mu_{X}(x))+
\dim^{\int}(\int_{ Y}^{\oplus} \gamma_y \md\mu_{ Y}(y))\\
\Psi^{\int_{X\sqcup Y}}&=\Psi^{\int_X}\oplus\Psi^{\int_Y},
\end{split}\]
and if the integral dimension is finite we also have
\[
\chi^{\int}(\int_{X\sqcup Y}^{\oplus} \kappa_z \md\mu_{X\sqcup Y}(z))=
\chi^{\int}(\int_{X}^{\oplus} \pi_x \md\mu_{X}(x))+
\chi^{\int}(\int_{ Y}^{\oplus} \gamma_y \md\mu_{ Y}(y)).
\]

\subsubsection{Tensor product}
First, we define the measure space of the tensor product to be $X\times Y$ with the usual $\sigma$-algebra and measure $\mu_{X\times Y}=\mu_X\times\, \mu_Y$. Define a structure of measurable field of Hilbert spaces $(\msf{H}_{x}\otimes\msf{H}_{y})_{(x,y)\in X\times Y}$ as follows: as fundamental vector fields we take $\{(e^i_x\otimes f^j_y)_{(x,y)\in X\times Y}\}_{(i,j)\in\NN\times\NN}$, where $\{e_i\}_{i\in\NN}$ $\{f_j\}_{j\in\NN}$ are fundamental vector fields on $X,Y$. Then, if $(\xi_x)_{x\in X},\,(\eta_y)_{y\in Y}$ are measurable vector fields on $X,Y$ it follows that $(\xi_x\otimes\eta_y)_{(x,y)\in X\times Y}$ is a measurable vector field on $X\times Y$ (see e.g.~\cite{DixmiervNA}). We have a measurable family of representations $(\pi_x\tp\gamma_y)_{(x,y)\in X\times Y}$. Indeed, let $a\in \CGDu$ be arbitrary and let $(\iota_k)_{k\in \NN}$ be an approximate unit for $\CGDu\otimes\CGDu$. For any $k\in\NN$ an operator $\Delta^u(a)\iota_k$ belongs to $\CGDu\otimes\CGDu$, therefore it can be written as
\[
\Delta^u(a)\iota_k=
\lim_{p\to\infty} E_{p,k}
\]
for certain $E_{p,k}$ in the algebraic tensor product $\CGDu\otimes_{\operatorname{alg}}\CGDu \;(p,k\in\NN)$. Then, for each $(i,j),(i',j')\in\NN\times\NN$ the function 
\[
\begin{split}
X\times Y \ni (x,y)&\mapsto
\ismaa{ e^i_x\otimes f^j_y}{
\pi_x\tp\gamma_y(a) (e^{i'}_x\otimes f^{j'}_y)}\\
&=
\lim_{k\to\infty}\ismaa{ e^i_x\otimes f^j_y}{
(\pi_x\otimes\gamma_y)(\Delta^{u}(a) \iota_k) (e^{i'}_x\otimes f^{j'}_y)}\\
&=
\lim_{k\to\infty}\lim_{p\to\infty}
\ismaa{ e^i_x\otimes f^j_y}{
(\pi_x\otimes\gamma_y)(E_{p,k}) (e^{i'}_x\otimes f^{j'}_y)}
\in \CC
\end{split}
\]
is measurable -- it is a pointwise limit of a sequence of measurable functions. We define the tensor product of $\pi_X,\gamma_Y$ to be
\[
(X\times Y,\mu_X\otimes\mu_Y,(\pi_x\tp \gamma_y)_{(x,y)\in X\times Y}).
\]
We have
\[
\int_{X\times Y}\dim\md\mu_{X\times Y}=
\bigl(\int_X\dim\md\mu_X\bigr)
\bigl(\int_{Y}\dim\md\mu_{Y}\bigr),
\]
therefore the integral dimension $\int_{X\times Y}\dim\md\mu_{X\times Y}$ is finite if and only if $\int_{X}\dim\md\mu_{X},\int_{Y}\dim\md\mu_{Y}$ are finite. Assume this is the case, then we can consider the integral character:
\[
\begin{split}
&\quad\;
\chi^{\int}\bigl(\int_{X\times Y}^{\oplus}
\pi_x\tp \gamma_y \md\mu_{X\times Y}(x,y)\bigr)\\
&=
\int_{X\times Y} \chi(U^{\pi_x}\tp U^{\gamma_y})
\md\mu_{X\times Y}(x,y)\\
&=
\bigl(\int_{X}\chi(U^{\pi_x})\md\mu_X(x)\bigr)
\bigl(\int_{Y}\chi(U^{\gamma_y})\md\mu_Y(y)\bigr)\\
&=
\bigl(\chi^{\int}(\int_{X}^{\oplus}\pi_x \md\mu_{X}(x)\bigr)
\bigl(\chi^{\int}(\int_{Y}^{\oplus}\gamma_y \md\mu_{Y}(y)\bigr).
\end{split}
\]
\subsubsection{Conjugate representation}\label{conjrep}
Define formally a new measure space $(\ov{X},\mf{M}_{\ov{X}},\mu_{\ov{X}})$ to be equal $(X,\mf{M}_X,\mu_X)$. Points in $\ov{X}$ will be denoted by $\ov{x}$. Next, define a measurable family of Hilbert spaces $((\msf{H}_{\ov{x}})_{\ov{x}\in \ov{X}},\{(e^i_{\ov{x}})_{\ov{x}\in \ov{X}}\}_{i\in\NN})$ such that $\msf{H}_{\ov{x}}=\ov{\msf{H}_x}$ and $e^i_{\ov{x}}=\ov{e^i_x}$. Take a measurable family of representations $(\pi_{\ov{x}})_{\ov{x}\in \ov{X}}$ to be $\pi_{\ov{x}}=\jmath_{\msf{H}_{x}}\circ \pi_x\circ \hat{R}^u$. For all $a\in \CGDu,i,j\in\NN,\ov{x}\in \ov{X}$ we have
\[
\ismaa{e^i_{\ov{x}}}{\pi_{\ov{x}}(a) e^j_{\ov{x}}}=
\ismaa{\ov{e^i_x}}{\jmath_{\msf{H}_x}( \hat{R}^u(a)) \ov{e^j_x}}=
\ismaa{e^j_x}{\hat{R}^u(a) e^i_x},
\]
therefore we indeed get a measurable family of representations. This way we define an integral representation of $\GG$:
\[
(\ov{X}, \mu_{\ov{X}},(\pi_{\ov{x}})_{\ov{x}\in\ov{X}}).
\]
 We have $\int_{\ov{X}}\dim \md\mu_{\ov{X}}=\int_X\dim\md\mu_X$ so if this expression is finite we can consider the integral character. Provided irreducible representations of $\GG$ are admissible we have (see Definition \ref{defadmissible} and Lemma \ref{lemat32})
\[
\chi^{\int}(\int_{\ov{X}}^{\oplus}\pi_{\ov{x}} \md\mu_{\ov{X}}(\ov{x}))=
\int_{\ov{X}}\chi(U^{\ov{\pi_x}}) \md\mu_{\ov{X}}(\ov{x})=
\int_X \chi(U^{\pi_x})^* \md\mu_X(x)=
\bigl(\chi^{\int}(\int_X^{\oplus}\pi_x\md\mu_X(x))\bigr)^*.
\]
Note that we have a unitary equivalence between $\pi_{\ov X}$ and $\jmath_{\msf{H}_{\pi_X}}\circ \pi_X \circ \hat{R}^{u}$. Indeed, define a map 
\[
U\colon \int_{\ov{X}}^{\oplus} \msf{H}_{\ov x}\md\mu_{\ov X}(\ov x)\ni
\int_{\ov{X}}^{\oplus}\xi_{\ov x}\md\mu_{\ov X}(\ov x)\mapsto
\ov{ \int_X^{\oplus} \ov{\xi_{\ov x}}\md\mu_X(x)}
\in \ov{\int_X^{\oplus} \msf{H}_x\md\mu_X(x)}.
\]
This map is well defined: for $i\in \NN$ and $\int_{\ov{X}}^{\oplus}\xi_{\ov x}\md\mu_{\ov X}(\ov x)\in \int_{\ov{X}}^{\oplus} \msf{H}_{\ov x}\md\mu_{\ov X}(\ov x)$ we have
\[
X\ni x \mapsto\ismaa{e^i_x}{ \ov{\xi_{\ov{x}}}}=
\ov{\ismaa{\ov{e^i_x}}{ \xi_{\ov{x}}}}\in \CC
\]
which is measurable. Next,
\[
\bigl\|
\int_{X}^{\oplus}\ov{\xi_{\ov x}}\md\mu_{ X}( x)
\bigr\|^2=
\int_{X} \|\ov{\xi_{\ov{x}}}\|^2 \md\mu_X(x)=
\int_{\ov X} \|\xi_{\ov{x}}\|^2 \md\mu_{\ov X} (\ov x)=
\bigl\|\int_{\ov X}^{\oplus} \xi_{\ov{x}} \md\mu_{\ov X}(\ov x)
\bigr\|^2,
\]
so $U$ is an isometry. It is clear that $U$ is surjective, so $U$ is a unitary map. Let us check that $U$ is an intertwiner: for $a\in \CGDu$ we have
\[\begin{split}
&\quad\;
\bigl( U \int_{\ov X}^{\oplus} \pi_{\ov x}\md\mu_{\ov X}(\ov x)(a)\bigr)
\int_{\ov X}^{\oplus} \xi_{\ov{x}} \md\mu_{\ov X}(\ov x)=
U\int_{\ov X}^{\oplus} \ov{ \pi_{x}(\hat{R}^{u}(a))^* \ov{\xi_{\ov x}}}\md\mu_{\ov X}(\ov x)\\
&=
\ov{ \int_X^{\oplus} \pi_x(\hat{R}^{u}(a)^*) \ov{\xi_{\ov x}}
\md\mu_X(x)}=
\bigl(\jmath_{\msf{H}_{\pi_X}}\circ\pi_X\circ \hat{R}^{u}(a) \bigr)
\ov{\int_X^{\oplus} \ov{\xi_{\ov x}} \md\mu_X(x)}\\
&=
\bigl((\jmath_{\msf{H}_{\pi_X}}\circ\pi_X\circ \hat{R}^{u}(a)) U\bigr) 
\int_{\ov X}^{\oplus} \xi_{\ov{x}} \md\mu_{\ov X}(\ov x).
\end{split}\]
This proves that $U$ is a unitary intertwiner between $\pi_{\ov X}$ and $\jmath_{\msf{H}_{\pi_X}}\circ\pi_X \circ \hat{R}^{u}$.
\subsection{Measure class associated with an integral representation}
Let $(X,\mu_X,(\pi_x)_{x\in X})$ be an integral representation of $\GG$. Define $\pi_X$ as usual: $\pi_X=\int_X^{\oplus} \pi_x\md\mu_X(x)$. This is a nondegenerate representation of $\CGDu$, hence we can associate with it a measure class $[\mu_{\pi_X}]$ on $\IrrG$, the spectrum of $\CGDu$. We have $\pi_X\approx_q \int_{\IrrG}^{\oplus} \pi \md\mu_{\pi_X}(\pi)$ (\cite[Definition 8.4.3.]{DixmierC}). In the case when $\CGDu$ is of type $I$, we have the Plancherel measure $\mu$ and $[\mu_{\pi_X}]\ll [\mu]$ holds if and only if $\pi_X\lec_q \Lambda_{\whG}$ ($\lec_q$ denotes quasi-containment, for more information about this relation see the appendix). This observation stems from the fact that the measure class associated with $\Lambda_{\whG}$ is $[\mu]$. Indeed, representation $\Lambda_{\whG}\colon \CGDu\rightarrow \CGD\subseteq \B(\LL^2(\GG))\colon \Lambda_{\whG}(\lambda^u(\omega))=\lambda(\omega)$ is unitarily equivalent to
\[
\int_{\IrrG}^{\oplus} \pi\otimes\I_{\ov{\msf{H}_\pi}}\md\mu(\pi)
\simeq \int_{\IrrG}^{\oplus} (\dim (\pi))\cdot \pi \md\mu(\pi)
\approx_q \int_{\IrrG}^{\oplus}\pi\md\mu(\pi)
\]
(point $3)$ Theorem \ref{PlancherelL}). We have used the property that quasi-equivalence does not see multiplicities of representations and respects direct sums. It is clear that the measure class associated with $\int_{\IrrG}^{\oplus}\pi\md\mu(\pi)$ is $[\mu]$.

\section{Integral weights}\label{secintweight}
Recall that we assume that $\whG$ is second countable. From now on we moreover assume that $\GG$ is \emph{type I} (which means that $\CGDu$ is a \cst-algebra of type I) and irreducible representations of $\GG$ are finite dimensional.\\
The next proposition is an important technical result on which most of the further reasoning will rely. Allow us first to explain it in a less formal way.\\
Let $\cdots\subseteq X_2\subseteq X_1\subseteq\IrrG$ be a sequence of measurable subsets and let $\mu_{X_1}$ be a $\sigma$-finite measure on $X_1$. Out of it, we can construct the disjoint union measure space $X=\bigsqcup_{i=1}^{\infty} X_i$, measure $\mu_X$ on $X$ and an integral representation $(X,\mu_X,(\pi_x)_{x\in X})$ in a obvious way ($(\pi_x)_{x\in \IrrG}$ is a measurable family of representations such that each $\pi_x$ belongs to the class $x$). Let $(Y,\mu_Y,(\gamma_y)_{y\in Y})$ be another integral representation and assume that representations $\int_X^{\oplus}\pi_x\md\mu_X(x)$ and $\int_Y^{\oplus} \gamma_y \md\mu_Y(y)$ are equivalent. Let $\mc{O}$ be a unitary operator implementing this equivalence. On the von Neumann algebras  generated by the images of $\int_X^{\oplus}\pi_x\md\mu_X(x)$ and $\int_Y^{\oplus} \gamma_y \md\mu_Y(y)$ we have weights given by the direct integrals of traces (i.e.~the integral weights $\Psi^{\int_X},\Psi^{\int_Y}$ introduced in the Section \ref{secintrep}). Operator $\mc{O}$ transforms these von Neumann algebras one to the other, however it may happen that it does not transform the corresponding weights. The next proposition tells us that we can always rescale the measure $\mu_{X_1}$ (using a function $\varpi\colon X_1\rightarrow \RR_{>0}$) in such a way that $\mc{O}$ will transform the (rescaled) weights, one onto the other.

\begin{proposition}\label{treq}
Let $\cdots\subseteq X_2\subseteq X_1\subseteq\IrrG$ be a family of measurable subsets (perhaps empty for some $n$) and let $\mu_{X_1}$ be a $\sigma$-finite measure on $X_1$. Assume moreover that $\sum_{i=1}^{\infty} \chi_{X_i}(x)<+\infty$ for each $x\in \IrrG$. Define a measurable space $X=\bigsqcup_{i=1}^{\infty} X_i$ and measures: on $X_i$, let $\mu_{X_i}$ be the restriction of $\mu_{X_1}$ ($i\in \NN$) and on the whole $X$ via
\[
\mu_X(\Omega)=\sum_{i=1}^{\infty} \mu_{X_i}(\Omega\cap X_i)\quad(\Omega\in\mc{B}(X)),
\]
i.e. $\mu_X=\bigsqcup_{i=1}^{\infty} \mu_{X_i}$. Define a structure of measurable field of Hilbert spaces on $X_i,\,X(i\in\NN)$, which comes from the (canonical) structure on $\IrrG$. Let $(\pi_x)_{x\in X}$ be a measurable field of representations such that $\pi_x\in x$. We get an integral representation $(X,\mu_X,(\pi_x)_{x\in X})$.\\
Assume that we have a second integral representation $(Y,\mu_Y,(\gamma_y)_{y\in Y})$ and a unitary intertwiner
\[
\mc{O}\colon \msf{H}_\pi=\int_X^{\oplus} \msf{H}_x\md\mu_X(x)\rightarrow
\msf{H}_\gamma=\int_Y^{\oplus} \msf{K}_y\md\mu_Y(y),
\]
where 
\[
\pi=\int_X^{\oplus} \pi_x \md\mu_X(x),\quad
\gamma=\int_Y^{\oplus} \gamma_y \md\mu_Y(y).
\]
There exists a unique measurable function $\varpi\colon X_1\rightarrow \RR_{>0}$ such that if we define an integral representation $(X,\tilde{\mu}_{X},(\pi_{x})_{x\in X})$ in the same way as $(X,\mu_X,(\pi_x)_{x\in X})$ except $\tilde{\mu}_X =\bigsqcup_{i=1}^{\infty} \varpi|_{X_i} \mu_{X_i}$, then
\[
\dim^{\int}(\int_Y^{\oplus}\gamma_y\md\mu_Y(y))=
\dim^{\int}(\int_X^{\oplus} \pi_x\md\tilde{\mu}_X(x))
\]
and we have an equality of weights on $\pi(\CGDu)''$:
\[
\Psi^{\int_Y} (\mc{O} \cdot \mc{O}^*)=
(\Psi^{\int_X})^{\sim},
\]
where $(\Psi^{\int_X})^{\sim}$ is the integral weight associated with $(X,\tilde{\mu}_{X},(\pi_{x})_{x\in X})$.
\end{proposition}

The above result can be written as
\[
\int_Y \Tr_y ( (\mc{O} \int_X^{\oplus}T_x \md\mu_X(x) \mc{O}^*)|_{y}) \md\mu_Y(y)=
\int_X \Tr_x (T_x) \md\tilde{\mu}_X(x)
\]
for $\int_X^{\oplus} T_x\md\mu_X(x)\in \pi(\CGDu)''_+$ (cf.~remarks at the beginning of the Section \ref{secintrep}). Proof of this proposition owes much to the proof of \cite[Theorem 3.4.5]{Desmedt}.\\
Recall that every operator of the form $\int_X^{\oplus} T_x\md\mu_X(x)$ is called decomposable, and if furthermore $T_x\in \CC\I_{\msf{H}_x}\,(x\in X)$ then we say that it is diagonalisable.

\begin{proof}
The weight $\Psi^{\int_X}$ is semifinite -- we will prove this in Lemma \ref{lemat36}, at the end of this section.\\
First we will show existence of $\varpi$. Let us introduce the notation $A=\CGDu$. The universal property of $A^{**}$ tells us that nondegenerate representations of $A$ extend uniquely to normal representations of $A^{**}$ (which will be denoted by a bar). Proposition 8.6.4 in \cite{DixmierC} says that for each $i$ the commutant of $\int_{X_i}^{\oplus} \pi_x\md\mu_{X_i}(x) (A)$ is the algebra of diagonalisable operators $\Diag (\int_{X_i}^{\oplus} \msf{H}_x \md\mu_{X_i}(x))$, therefore
\begin{equation}\label{eq41}
\bigl(\int_{X_i}^{\oplus} \pi_x\md\mu_{X_i}(x) (A)\bigr)''=
\Dec (\int_{X_i}^{\oplus} \msf{H}_x \md\mu_{X_i}(x)).
\end{equation}
Consequently, arbitrary decomposable operator on $\int_{X_i}^{\oplus} \msf{H}_x \md\mu_{X_i}(x)$ is in the image of the representation $\int_{X_{i}}^{\oplus} \pi_x \md\mu_{X_{i}}$ extended to $A^{**}$. For each $i\in \NN$ the isometry
\[
P_{i}\colon 
\int_{X_i}^{\oplus}\msf{H}_x\md\mu_{X_i}(x)\ni
\int_{X_i}^{\oplus}\xi_x\md\mu_{X_i}(x)
\mapsto
\int_{X_1}^{\oplus}\xi_x \chi_{X_i}(x)\md\mu_{X_1}(x)\in
\int_{X_1}^{\oplus}\msf{H}_x\md\mu_{X_1}(x)
\]
satisfies
\[
\pi_i(a)
\int_{X_i}^{\oplus}\xi_x\md\mu_{X_i}(x)=
P_i^* \pi_1(a) P_i\int_{X_i}^{\oplus}\xi_x\md\mu_{X_i}(x)\quad(
a\in A, \int_{X_i}^{\oplus}\xi_x\md\mu_{X_i}(x)\in \int_{X_i}^{\oplus}\msf{H}_x\md\mu_{X_i}(x))
\]
therefore due to \swot-continuity we have
\[
\ov{\pi_i}(a)=P_i^* \ov{\pi_1}(a) P_i\quad(a\in A^{**}).
\]
\swot-continuity of the representation $\ov{\pi}$ allows us to conclude that
\[
\mc{O} \ov{\pi}(a) \mc{O}^*\in 
\bigl(\int_Y^{\oplus} \gamma_y\md\mu_Y(y)(A)\bigr)''
\subseteq \Dec\bigl(\int_Y^{\oplus}\msf{K}_y\md\mu_Y(y)\bigr)
\]
for all $a\in A^{**}$. Define two weights on $\pi_1(A)''=\ov{\pi}_1(A^{**})$: 
\[\begin{split}
&\Psi_1\colon \ov{\pi_1}(a)\mapsto
\Psi^{\int_Y}(\mc{O} \ov{\pi}(a) \mc{O}^*),\\
&\Psi_2\colon \ov{\pi_1}(a)\mapsto
\Psi^{\int_X} (\ov{\pi}(a)).
\end{split}
\]
The weight $\Psi_1$ is normal, faithful and tracial, while $\Psi_2$ is a n.s.f.~weight. Moreover $\Psi_2$ is tracial, hence its modular automorphism group is trivial. Therefore we know (\cite[Proposition 5.2]{VaesRN}, see also \cite[Theorem 5.4]{PedersenTakesaki}) that there exists a positive self-adjoint operator $\Pi$ affiliated with $\pi_1(A)''$ such that
\[
\Psi_1=\Psi_2(\Pi^{\frac{1}{2}}\cdot \Pi^{\frac{1}{2}})
\]
(weight $\Psi_2(\Pi^{\frac{1}{2}}\cdot \Pi^{\frac{1}{2}})$ is defined in \cite{VaesRN}). Since
\[
\pi_1(A)''\subseteq \Dec(\int_{X_1}^{\oplus} \pi_x\md\mu_{X_1}(x))\subseteq
 \Diag(\int_{X_1}^{\oplus} \pi_x\md\mu_{X_1}(x))'
\]
we know that the operator $\Pi$ is decomposable (\cite[Theorem 1.8]{Lance}): there exists a measurable field of closed densely defined operators $(\Pi_x)_{x\in X_1}$ such that $\Pi=\int_{X_1}^{\oplus} \Pi_x\md\mu_{X_1}(x)$. Since $\dim(\msf{H}_x)<+\infty$ for each $x\in X$, we know that operators $\Pi_x$ are bounded. Moreover, we know that for almost all $x\in X_1$ operator $\Pi_x$ is positive and invertible (it follows e.g. from \cite[Theorem 1.10]{Lance}). Since both $\Psi_1,\Psi_2$ are tracial, operator $\Pi$ is in fact affiliated with $\mc{Z}(\pi_1(A)'')$ (this is a consequence of Proposition 2.5 and the formula for $\Delta'^{is}$ \cite[Lemma 2.1]{VaesRN}). It follows that each $\Pi_x$ belongs to $\B(\msf{H}_x)'=\CC \I_{\msf{H}_x}$, hence we can introduce a measurable function $\varpi\colon X_1\rightarrow \RR_{>0}$ such that $\Pi_x=\sqrt{\varpi(x)}\I_{\msf{H}_x}$.\\
To move further we need to identify a GNS construction for $\Psi_2$. For $x\in X$ consider the Hilbert space $(\pi_x(\CGDu)'',\langle\cdot|\cdot\rangle_{x})$ where the scalar product is given by
\[
\is{T}{S}_{x}=(\sum_{i=1}^{\infty} \chi_{X_i}(x))\Tr_x(T^*S)\quad(T,S\in \pi_x(\CGDu)'').
\]
It is a complete space since $\dim(\msf{H}_x)<+\infty$. Denote this space by $\msf{L}_x$. Choose a dense subset of $\CGDu$, $\{a_i\}_{i\in \NN}$. We can construct a measurable field of Hilbert spaces $(\msf{L}_x)_{x\in X}$ with fundamental fields $\{(\pi_x(a_i))_{x\in X}\}_{i\in \NN}$. Indeed, it is clear that for each $x\in X$ set $\{\pi_x(a_i)\}_{i\in \NN}$ generates $\msf{L}_x$ and that for $i,j\in \NN$ the function
\[
X\ni x \mapsto \is{\pi_x(a_i)}{\pi_x(a_j)}_{x}=
(\sum_{k=1}^{\infty} \chi_{X_k}(x))
\Tr_x(\pi_x(a_i^* a_j))\in \CC
\]
is measurable. In order to distinguish elements in $\int_{X_1}^{\oplus}\pi_x(\CGDu)''\md\mu_{X_1}(x)$ and $\int_{X_1}^\oplus \msf{L}_x\md\mu_{X_1}(x)$, a vector in the second space will be denoted with an underline: $\underline{\int_{X_1}^{\oplus}} T_x\md\mu_{X_1}(x)$. Consider the map
\[
\eta_{\Psi_2}\colon \mf{N}_{\Psi_2}\ni
\int_{X_1}^{\oplus} T_x \md\mu_{X_1}(x)\mapsto
\underline{\int_{X_1}^{\oplus}} T_x \md\mu_{X_1}(x)\in
 \int_{X_1}^{\oplus} \msf{L}_x\md\mu_{X_1}(x).
\]
We have
\[
\int_{X_1} \|T_x\|_{x}^{2} \md\mu_{X_1}(x)=
\int_X \Tr_x( T_x^*T_x) \md\mu_X(x)=
\Psi_2\bigl(\bigl( \int_{X_1}^{\oplus}T_x\md\mu_{X_1}(x)\bigr)^*
\bigl( \int_{X_1}^{\oplus}T_x\md\mu_{X_1}(x)\bigr)\bigr)<+\infty,
\]
hence it is a well defined linear map (to get measurability of $\underline{\int_{X_1}^{\oplus}} T_x \md\mu_{X_1}(x)$ we use separability of $\msf{H}_\pi$). It is clear that $\eta_{\Psi_2}$ has trivial kernel. We can define a representation
\[
\pi_{\Psi_2}\colon \pi_1(\CGDu)''\rightarrow \B(\int_{X_1}^{\oplus}\msf{L}_x\md\mu_{X_1}(x))
\]
via
\[
\pi_{\Psi_2}(\int_{X_1}^{\oplus} S_x\md\mu_{X_1}(X)) \underline{\int_{X_1}^{\oplus}} T_x\md\mu_{X_1}(x)=
\underline{\int_{X_1}^{\oplus}} S_xT_x\md\mu_{X_1}(x).
\]
It is a GNS construction for the weight $\Psi_2$: we know that such a construction exists, let $\tilde{\eta}$ be its generalized vector. There exists the unitary operator mapping $\eta_{\Psi_2}(a)$ to $\tilde{\eta}(a)$. This way we get a $\ssot\times\|\cdot\|$ closedness of $\eta_{\Psi_2}$. Note that since the weight $\Psi_2$ is tracial, we have $\nabla_{\Psi_2}=\I$ and consequently
\[
J_{\Psi_2} \eta_{\Psi_2}(a)=
\eta_{\Psi_2}(a^*)\quad (a\in \mf{N}_{\Psi_2}).
\]
The GNS construction for the weight $\Psi_3=\Psi_2(\Pi^{\frac{1}{2}}\cdot \Pi^{\frac{1}{2}})$ is introduced in \cite{VaesRN}. Let us recall its details. It is built out of the GNS construction for $\Psi_2$; for $\ov{\pi}_1(a)\in\pi_1(\CGDu)''$ such that $\ov{\pi}_1(a)\Pi$ is closable and its closure $\ov{\pi}_1(a)\cdot \Pi$ is in $\mf{N}_{\Psi_2}$ we define
\[
\eta_{\Psi_3}(\ov{\pi}_1(a))=\eta_{\Psi_2}(\ov{\pi}_1(a)\cdot\Pi).
\]
We can introduce the approximate unit $(e_n)_{n\in\NN}$ in $\pi_1(\CGDu)''$ (defined in \cite{VaesRN}). Fix any $\ov{\pi}_1(a)\in \mf{N}_{\Psi_3}$. For $n\in\NN$ we have $\ov{\pi}_1(a) (\Pi e_n)\in\mf{N}_{\Psi_2}$ (\cite[Lemma 3.2]{VaesRN}) and
\begin{equation}\begin{split}\label{eq11}  
\eta_{\Psi_3}(\ov{\pi}_1(a) e_n)=\eta_{\Psi_2}(\ov{\pi}_1(a) (\Pi e_n))=
\underline{\int_{X_1}^{\oplus}} \ov{\pi}_1(a)_x (\Pi e_n)_x \md\mu_{X_1}(x).
\end{split}\end{equation}
We can write
\[
\eta_{\Psi_3}(\ov{\pi}_1(a))=\underline{\int_{X_1}^{\oplus}} v_x\md\mu_{X_1}(x)
\]
for some vectors $v_x\in \msf{L}_x$ -- we now wish to show that $v_x=\ov{\pi}_1(a)_x\Pi_x$. Thanks to the properties of $e_n$ we get
\[\begin{split} 
\eta_{\Psi_3}(\ov{\pi}_1(a) e_n)=
J_{\Psi_3} \pi_{\Psi_3}(\sigma^{\Psi_3}_{-i/2}(e_n)) J_{\Psi_3} \eta_{\Psi_3}(\ov{\pi}_1(a))=
J_{\Psi_3} \pi_{\Psi_3}(\sigma^{\Psi_2}_{-i/2}(e_n)) J_{\Psi_3} \eta_{\Psi_3}(\ov{\pi}_1(a))
\end{split}\]
Since $J_{\Psi_3}=J_{\Psi_2}$, we can write
\begin{equation}\begin{split}\label{eq12}  
&\quad\;
\eta_{\Psi_3}(\ov{\pi}_1(a) e_n)=
J_{\Psi_2} \pi_{\Psi_3}(\sigma^{\Psi_2}_{-i/2}(e_n)) \underline{\int_{X_1}^{\oplus}}
v_x^*\md\mu_{X_1}(x)=
J_{\Psi_2}  \underline{\int_{X_1}^{\oplus}}
\sigma^{\Psi_2}_{-i/2}(e_n)_x v_x^*\md\mu_X(x)\\
&=
\underline{\int_{X_1}^{\oplus}}
v_x\sigma^{\Psi_2}_{i/2}(e_n)_x\md\mu_{X_1}(x).
\end{split}\end{equation}
Equations \eqref{eq11}, \eqref{eq12} give us
\[
\ov{\pi}_1(a)_x (\Pi e_n)_x=
v_x \sigma^{\Psi_2}_{i/2}(e_n)_x
\]
for almost all $x\in X_1$. We have $e_n\xrightarrow[n\to\infty]{\ssot}\I$ and $\sigma^{\Psi_2}_{i/2}(e_n)\xrightarrow[n\to\infty]{\ssot} \I$ so there exist a subsequence such that $(e_{n_k})_x\xrightarrow[n\to\infty]{\ssot}\I_x$ and $\sigma^{\Psi_2}_{i/2}(e_{n_k})_x\xrightarrow[k\to\infty]{\ssot} \I_x$ almost everywhere.\\ Consequently
\[
\eta_{\Psi_3}(\ov{\pi}_1(a))=\underline{\int_{X_1}^{\oplus}} \ov{\pi}_1(a)_x \Pi_x \md\mu_{X_1}(x)=
\underline{\int_{X_1}^{\oplus}} \ov{\pi}_1(a)_x \sqrt{\varpi(x)} \md\mu_{X_1}(x)
\]
for all $\ov{\pi}_1(a)\in\mf{N}_{\Psi_3}$ (in particular this vector is integrable). Moreover we have
\begin{equation}\begin{split}
&\quad\;\Psi^{\int_Y}(\mc{O} \ov{\pi}(a^* b) \mc{O}^*)=
\Psi_1(\ov{\pi}_1(a^* b))=
\Psi_3(\ov{\pi}_1(a^* b))=
\ismaa{\eta_{\Psi_3}(\ov{\pi}_1(a))}{\eta_{\Psi_3}(\ov{\pi}_1(b))}\\=
&\int_{X_1}(\sum_{i=1}^{\infty}\chi_{X_i}(x)) \Tr_x( \ov{\pi}_1(a^*b)_x )\varpi(x)\md\mu_{X_1}(x)=
\int_{X} \Tr_x( \ov{\pi}(a^*b)_x )\varpi(x)\md\mu_X(x)
\end{split}
\end{equation}
for $\ov{\pi}_1(a),\ov{\pi}_1(b)\in \mf{N}_{\Psi_3}$ (note that we have extended the domain of $\varpi$ to $X$ in the obvious way). This way we proved
\begin{equation}\label{eq14}
\Psi^{\int_Y}(\mc{O}a^*b\,\mc{O}^*)=
(\Psi^{\int_X})^{\sim}(a^*b),
\end{equation}
for all $a,b\in \mf{N}_{\Psi^{\int_Y} (\mc{O}\cdot \mc{O}^*)}\subseteq \pi(\CGDu)''$, where $(\Psi^{\int_X})^{\sim}$ is the weight as in the statement of the theorem. Since these weights are tracial and n.s.f., \cite[Proposition 3.15]{TakesakiII} gives us
\[
\Psi^{\int_Y}(\mc{O}\cdot\mc{O}^*)=(\Psi^{\int_X})^{\sim}.
\]
We need to check the equality of integral dimensions. We get this result once we substitute the unit into equation \eqref{eq14}:
\[\begin{split}
&\quad\;
\int_Y\dim(\gamma_y)\md\mu_Y(y)=
\Psi^{\int_Y}(\ov{\gamma}(\I))=
\Psi^{\int_Y}(\mc{O} \ov{\pi}(\I)\mc{O}^*)\\
&=
\int_X \dim(\pi_x) \varpi(x)\md\mu_X(x)=
\int_X\dim(\pi_x)\md\tilde{\mu}_X.
\end{split}\]
Now, that existence of $\varpi$ is established, let us show that $\varpi$ is the only function (up to a difference on a $\mu_X$-measure zero set) for which the equality of the rescaled weights holds.\\
Assume this is not true, and we have two measurable functions $\varpi_1,\varpi_2\colon X_1\rightarrow \RR_{>0}$ such that
\begin{equation}\label{eq6}
\bigoplus_{n=1}^{\infty} \int_{X_n}^{\oplus}\varpi_1(x)\Tr_x \md\mu_{X_n}(x)=
\bigoplus_{n=1}^{\infty} \int_{X_n}^{\oplus}\varpi_2(x)\Tr_x 
\md \mu_{X_n}(x).
\end{equation}
If $\varpi_1\neq \varpi_2$, then without loss of generality we can assume that there exists a measurable subset $Y\subseteq X_1$ such that $\sum_{n=1}^{\infty}\int_{Y\cap X_n} \varpi_1 \dim \md\mu_{X_n}<+\infty$ and $\varpi_1>\varpi_2$ on $Y$. We have \cite[Proposition 8.6.4, A 80]{DixmierC} 
\[
\bigl(\int_{X_1}^{\oplus} \pi_x \md\mu_X(x)(A)\bigr)''=\Dec(\int_{X}^{\oplus} \msf{H}_x \md\mu_X(x)),
\]
hence there exists a positive operator $a\in A^{**}$ such that
\[
\ov{\pi}_1(a)=
\int_{X_1}^{\oplus} \chi_Y(x) \I_{\msf{H}_x} \md\mu_{X_1}(x)\in
\bigl(\int_{X_1}^{\oplus} \pi_x \md\mu_X(x)(A)\bigr)''.
\]
If we apply the weights that appear in equation \eqref{eq6} to this operator, we get
\[\begin{split}
+\infty&>
\bigl(\bigoplus_{n=1}^{\infty} \int_{X_n}^{\oplus}\varpi_1(x)\Tr_x \md\mu_{X_n}(x)\bigr)(\ov{\pi}(a))=
\sum_{n=1}^{\infty}
\int_{X_n} \varpi_1(x) \chi_Y(x) \Tr_x(\I_{\msf{H}_x}) \md\mu_{X_n}(x)\\
&=
\sum_{n=1}^{\infty}
\int_{Y\cap X_n} \varpi_1\dim \md\mu_{X_n}>
\sum_{n=1}^{\infty}
\int_{Y\cap X_n} \varpi_2\dim \md\mu_{X_n}\\
&=
\bigl(\bigoplus_{n=1}^{\infty} \int_{X_n}^{\oplus}\varpi_2(x)\Tr_x \md\mu_{X_n}(x)\bigr)(\ov{\pi}(a)),
\end{split}\]
which gives us a contradiction.
\end{proof}

\begin{lemma}\label{lemat36}
The integral weight $\Psi^{\int_X}$ from the previous proposition is semifinite.
\end{lemma}

\begin{proof}
Let $(V_n)_{n\in\NN}$ be an increasing family of measurable subsets of $X_1$ such that $\mu_{X_1}(V_n)<+\infty\,(n\in\NN)$ and $\bigcup_{n\in\NN}V_n=X_1$. Define a family of measurable sets $\{\Omega_n\}_{n\in\NN}$ via
\[
\Omega_n=\{x\in X_1\,|\, x\in V_n,\, \dim(x)\le n,\,\sum_{i=1}^{\infty} \chi_{X_i}(x)\le n\}\quad(n\in\NN).
\]
Take any $a\in \CGDu^{**}_+$. We can write $\ov{\pi_X}(a)=\bigoplus_{i=1}^{\infty}\int_{X_i}^{\oplus} T_{x}\md\mu_{X_i}(x)$ for some positive operators $(T_{x})_{x\in X_1}$. Since the measurable field of Hilbert spaces under consideration is canonical (i.e.~reduces to $\CC^n$ on each component $X_1\rest_n$), we have equality \eqref{eq41} and consequently for each $n\in\NN$ we can find an operator $b_n\in \CGDu^{**}_+$ such that
\[
\ov{\pi_X}(b_n)=
\bigoplus_{i=1}^{\infty}\int_{X_i}^{\oplus}\chi_{\Omega_n}(x) T_{x}\md\mu_{X_i}(x).
\]
Clearly, we have $\ov{\pi_X}(b_n)\xrightarrow[n\to\infty]{\swot}\ov{\pi_X}(a)$ and 
\[
\Psi^{\int_X} (\ov{\pi_X}(b_n)^*\,
\ov{\pi_X}(b_n))=
\sum_{i=1}^{\infty} \int_{X_i} 
\chi_{\Omega_n}(x) \Tr_x(T_{x}^* T_x)\md\mu_{X_i}(x)\le
n^2 \mu(V_n) \|\ov{\pi_X}(a)\|^2<+\infty.
\]
It follows that $\Psi^{\int_X}$ is semifinite.
\end{proof}

\section{Connections between the integral weight and the integral character}\label{secweightchar}

\begin{proposition}
Let $\pi_X\in \Rep^{\int}_{<+\infty}(\GG)$ be an integral representation with finite integral dimension. We have the following:
\[
\omega(\chi^{\int}(\int_X^{\oplus}\pi_x\md\mu_X(x)))=
\Psi^{\int_X}(\pi_X((\omega\otimes\id){\WW}))
\]
for any $\omega\in\Lj$.
\end{proposition}

\begin{proof}
The above equality follows from a direct calculation:
\[
\begin{split}
&\quad\;\omega(\chi^{\int}(\int_X^{\oplus}\pi_x\md\mu_X(x)))=
\int_X \omega((\id\otimes\Tr_x)U^{\pi_x})\md\mu_X(x)=
\int_X \Tr_{x} ( (\omega\otimes\id)U^{\pi_x})\md\mu_X(x)\\
&=
\int_X \Tr_x(\pi_x ((\omega\otimes\id){\WW}))\md\mu_X(x)=
\Psi^{\int_X}\bigl( \int_X^{\oplus} \pi_x((\omega\otimes\id){\WW})\md\mu_X(x)\bigr).
\end{split}
\]
\end{proof}

Second result tells us that for equivalent representations, equality of integral characters holds if and only if the integral weights are equal (after composing with appropriate isomorphism).

\begin{proposition}\label{trchar}
Take $\pi_X,\gamma_Y\in\Rep^{\int}_{<+\infty}(\GG)$ and assume that we have a unitary intertwiner
\[
\mc{O}\colon \int_X^{\oplus} \msf{H}_x\md\mu_X(x)\rightarrow
\int_Y^{\oplus} \msf{K}_y\md\mu_Y(y).
\]
Then integral characters are equal: $\chi^{\int}(\pi_X)=\chi^{\int}(\gamma_Y)$ if and only if
\[
\Psi^{\int_X}=\Psi^{\int_Y}(\mc{O} \cdot \mc{O}^*).
\]
\end{proposition}

\begin{proof}
Assume that the integral characters are equal. Due to the previous proposition we have
\[
\begin{split}
&\quad\;\Psi^{\int_X}( \pi_X((\omega\otimes\id){\WW}) )=
\omega(\chi^{\int}(\pi_X))
=\omega(\chi^{\int}(\gamma_Y))\\
&=
\Psi^{\int_Y} (\gamma_Y((\omega\otimes\id){\WW}))=
\Psi^{\int_Y} (\mc{O}\,\pi_X((\omega\otimes\id){\WW})\mc{O}^*)
\end{split}
\]
for all $\omega\in\Lj$. Norm density of $\{(\omega\otimes\id){\WW}\,|\,\omega\in \Lj\}$ in $\CGDu$ gives us
\[
\Psi^{\int_X}( \pi_X(a) )= \Psi^{\int_Y}(\mc{O} \pi_X(a) \mc{O}^*)\quad(a\in \CGDu),
\]
and normality of the integral weights (which in this case are bounded functionals) gives us the claim. We get an implication in the other direction, once we reverse this argument.
\end{proof}

\section{Results concerning the Haar integrals on $\whG$}\label{secintwhG}
Recall that we assume that $\whG$ is second countable and $\GG$ is a type I locally compact quantum group whose irreducible representations are finite dimensional. 
Let us fix an arbitrary Plancherel measure for $\GG$ (along with the operator $\mc{Q}_L$, etc.).\\
The right invariance of $\wh{\psi}$ in the \cst-algebraic version states the following: if $\theta\in \CGD^*_+$ is a positive functional, and $a\in \CGDu_+$ is such that $\wh{\psi}(a)<+\infty$ then we have
\[
\theta(\I_{\whG})\wh{\psi}(a)=\wh{\psi}((\id\otimes\theta)\Delta_{\whG}(a)).
\]
For $\omega\in \LL^1(\GG)$ such that $\lambda(\omega)\in\CGD_+$ we have
\[
\wh{\psi}(\lambda(\omega))=
\int_{\IrrG} \Tr((\omega\otimes\id)U^{\pi}\,E_{\pi}^{-2}) \md\mu(\pi),
\]
where $\lambda(\omega)=(\omega\otimes\id)\mrW\in \mathrm{C}_0(\whG)$ (see equation \eqref{eq42}). Assume moreover that $\wh{\psi}(\lambda(\omega))<+\infty$. Then for any $\theta\in \CGD^*_+$ we have
\[
\begin{split}
&\quad\;
\theta(\I_{\whG})\int_{\IrrG} \Tr((\omega\otimes\id)U^{\pi}\,E_{\pi}^{-2}) \md\mu(\pi)=
\theta(\I_{\whG})\wh{\psi}(\lambda(\omega))=
\wh{\psi}((\id\otimes\theta)\Delta_{\whG} (\lambda(\omega)))\\
&=
\wh{\psi}((\omega\otimes\id\otimes\theta)(\mrW_{13} \mrW_{12}))=
\wh{\psi}\bigl(
(\omega\bigl((\id\otimes\theta)\mrW\cdot\bigr)\otimes\id)\mrW
\bigr)\\
&=
\wh{\psi}(\lambda(\omega((\id\otimes\theta)\mrW \cdot)
))\overset{\star}{=}
\int_{\IrrG}\Tr(
(\omega((\id\otimes\theta)\mrW\cdot)\otimes\id)U^{\pi}E_{\pi}^{-2}
)\md\mu(\pi)\\
&=
\int_{\IrrG}\Tr\bigl(
(\omega\otimes\id\otimes\theta)
( \mrW_{13}U^{\pi}_{12})
E_{\pi}^{-2}
\bigr)\md\mu(\pi)\\
&=
\int_{\IrrG}\Tr\bigl(
(\omega\otimes\pi\otimes \theta\circ \Lambda_{\whG})
({\WW}_{13} {\WW}_{12})
E_{\pi}^{-2}
\bigr)\md\mu(\pi)\\
&=
\int_{\IrrG}\Tr\bigl(
(\omega\otimes\pi\otimes \theta\circ \Lambda_{\whG})
(\id\otimes\Delta^{u}_{\whG}){\WW}
E_{\pi}^{-2}
\bigr)\md\mu(\pi)
\end{split}
\]
Equality marked $\overset{\star}{=}$ follows from the fact that
\[
\lambda(\omega((\theta\otimes\id)\mrW)\cdot)=
(\id\otimes\theta)\Delta_{\whG} (\lambda(\omega))\in \mf{M}_{\hat{\psi}}^+,
\]
so we can use the formula \eqref{eq42} for $\beta$ replaced with $\omega((\theta\otimes\id)\mrW \cdot)$. If $\theta\circ \Lambda_{\whG}=\theta'\circ \kappa$ for a certain nondegenerate representation $\kappa\in\Mor( \CGDu,\mc{K}(\msf{H}_{\kappa}))$ and $\theta'\in \B(\msf{H}_\kappa)^*_+$, we get
\[
\begin{split}
&\quad\;\theta'(\I_{\kappa})
\int_{\IrrG} \Tr((\omega\otimes\id)U^{\pi}\,E_{\pi}^{-2}) \md\mu(\pi)=
\int_{\IrrG}\Tr\bigl(
(\omega\otimes\pi\otimes \theta'\circ \kappa)
(\id\otimes\Delta^{u}_{\whG}){\WW}
E_{\pi}^{-2}
\bigr)\md\mu(\pi)\\
&=\int_{\IrrG}\Tr\bigl(
(\omega\otimes\theta'\otimes \id) (\id\otimes\kappa\tp\pi){\WW}
E_{\pi}^{-2}
\bigr)\md\mu(\pi).
\end{split}
\]
In particular, if $\dim(\kappa)<+\infty$, $\theta'=\omega_{\xi_i}$ for an orthonormal basis $\{\xi_i\}_{i=1}^{\dim(\kappa)}$ and we take a sum over $i$ we arrive at
\[
\begin{split}
&\quad\;\dim(\kappa)
\int_{\IrrG} \Tr((\omega\otimes\id)U^{\pi}\,E_{\pi}^{-2}) \md\mu(\pi)\\
&=
\int_{\IrrG}\Tr_{\kappa\stp \pi}\bigl(
(\omega\otimes\id\otimes \id) (\id\otimes\kappa\tp\pi){\WW}
(\I_\kappa\otimes E_{\pi}^{-2})
\bigr)\md\mu(\pi).
\end{split}
\]
If $\kappa\lec \Lambda_{\whG}$ then $\kappa=\kappa'\circ \Lambda_{\whG}$ for a certain nondegenerate representation $\kappa'\colon \CGD\rightarrow \B(\msf{H}_{\kappa})$ (this situation occurs when $\kappa\in\supp(\mu)$, \cite[Theorem 3.4.8]{Desmedt}). For $\theta'\in \B(\msf{H}_{\kappa})^*_+$ we have
\[
\theta'\circ \kappa=\theta'\circ \kappa'\circ \Lambda_{\whG}=\theta\circ\Lambda_{\whG}
\]
where $\theta=\theta'\circ\kappa'\in \CGD^*_+$. To sum up, we have derived two results:

\begin{lemma}\label{lemat7}
Assume that $\omega\in \Lj$ is a functional such that $\lambda(\omega)\in \mf{M}_{\wh{\psi}}^+$. Then for any $\theta\in \CGD^*_+$ we have
\[
\theta(\I_{\whG})\int_{\IrrG} \Tr((\omega\otimes\id)U^{\pi}\,E_{\pi}^{-2}) \md\mu(\pi)=
\int_{\IrrG}\Tr\bigl(
(\omega\otimes\pi\otimes\theta\circ \Lambda_{\whG})
(\id\otimes\Delta^{u}_{\whG}){\WW}
E_{\pi}^{-2}
\bigr)\md\mu(\pi).
\]
If $\kappa\lec \Lambda_{\whG}$ and $\dim(\kappa)<+\infty$ then
\[
\dim(\kappa)
\int_{\IrrG} \Tr((\omega\otimes\id)U^{\pi}\,E_{\pi}^{-2}) \md\mu(\pi)=
\int_{\IrrG}\Tr_{\kappa\stp \pi}\bigl(
\kappa\tp\pi(
(\omega\otimes\id){\WW})
(\I_{\kappa}\otimes E_{\pi}^{-2})
\bigr)\md\mu(\pi).
\]
Above we have equalities of (finite) nonnegative numbers.
\end{lemma}

Now we conduct a similar reasoning for $\hvp$: assume that $\lambda(\omega)\in \mf{M}_{\wh{\vp}}^+$.
\[
\begin{split}
&\quad\;
\theta(\I_{\whG})\int_{\IrrG} \Tr((\omega\otimes\id)U^{\pi}\,D_{\pi}^{-2}) \md\mu(\pi)=
\theta(\I_{\whG})\wh{\vp}(\lambda(\omega))=
\wh{\vp}((\theta\otimes\id)\Delta_{\whG} (\lambda(\omega)))\\
&=
\wh{\vp}((\omega\otimes\theta\otimes\id)(\mrW_{13} \mrW_{12}))=
\wh{\vp}\bigl(
(\omega\bigl(\cdot(\id\otimes\theta)\mrW\bigr)\otimes\id)\mrW
\bigr)\\
&=
\wh{\vp}(\lambda(\omega(\cdot\,(\id\otimes\theta)\mrW\,)
))=
\int_{\IrrG}\Tr\bigl(
(\omega(\cdot\,(\id\otimes\theta) \mrW)\otimes\id)U^{\pi}D_{\pi}^{-2}
\bigr)\md\mu(\pi)\\
&=
\int_{\IrrG}\Tr\bigl(
(\omega\otimes\theta\otimes\id)
(U^{\pi}_{13}\mrW_{12} )
D_{\pi}^{-2}
\bigr)\md\mu(\pi)\\
&=
\int_{\IrrG}\Tr\bigl(
(\omega\otimes\theta\circ \Lambda_{\whG}\otimes\pi)
({\WW}_{13} \WW_{12})
D_{\pi}^{-2}
\bigr)\md\mu(\pi)\\
&=
\int_{\IrrG}\Tr\bigl(
(\omega\otimes\theta\circ \Lambda_{\whG}\otimes\pi)
(\id\otimes\Delta^{u}_{\whG}){\WW}
D_{\pi}^{-2}
\bigr)\md\mu(\pi)
\end{split}
\]
If $\theta\circ\Lambda_{\whG}=\theta'\circ\kappa$ we can further write
\[
\begin{split}
&\quad\;
\int_{\IrrG}\Tr\bigl(
(\omega\otimes\theta'\circ\kappa\otimes\pi)
(\id\otimes\Delta^{u}_{\whG}){\WW}
D_{\pi}^{-2}
\bigr)\md\mu(\pi)\\
&=
\int_{\IrrG}\Tr\bigl(
(\omega\otimes\id\otimes\theta')
(\id\otimes\pi\tp\kappa){\WW}
D_{\pi}^{-2}
\bigr)\md\mu(\pi)
\end{split}
\]
and we proceed as before. In the end we get a left version of the previous lemma:

\begin{lemma}\label{linttr}
Assume that $\omega\in \Lj$ is such that $\lambda(\omega)\in \mf{M}_{\wh{\vp}}^+$. Then for any  $\theta\in \CGD^*_+$ we have
\[
\theta(\I_{\whG})\int_{\IrrG} \Tr((\omega\otimes\id)U^{\pi}\,D_{\pi}^{-2}) \md\mu(\pi)=
\int_{\IrrG}\Tr\bigl(
(\omega\otimes \theta\circ \Lambda_{\whG}\otimes\pi)
(\id\otimes \Delta^{u}_{\whG}){\WW}\;
D_{\pi}^{-2}
\bigr)\md\mu(\pi).
\]
If $\kappa\lec \Lambda_{\whG}$ and $\dim(\kappa)<+\infty$ then
\[
\dim(\kappa)
\int_{\IrrG} \Tr((\omega\otimes\id)U^{\pi}\,D_{\pi}^{-2}) \md\mu(\pi)=
\int_{\IrrG}\Tr_{\pi\stp\kappa}\bigl(
\pi\tp\kappa(
(\omega\otimes \id){\WW})
(D_{\pi}^{-2}\otimes\I_\kappa)
\bigr)\md\mu(\pi).
\]
Above we have an equality of (finite) nonnegative numbers.
\end{lemma}

We can also derive an $\LL^2$-version of Lemma \ref{lemat7}:

\begin{lemma}\label{lemat16}
Let $\omega,\nu$ be a functionals such that $\lambda(\omega),\lambda(\nu)\in \lambda(\LL^1_{\sharp}(\GG))\cap \mf{N}_{\hpsi}$. If $\kappa\lec \Lambda_{\whG}$ and $\dim(\kappa)<+\infty$ then
\[
\begin{split}
\dim(\kappa)\hpsi (\lambda(\omega)^*\lambda(\nu))
&=\dim(\kappa)
\int_{\IrrG} \Tr(\pi(\lambda^{u}(\omega)^*\lambda^{u}(\nu))
\,E_{\pi}^{-2}) \md\mu(\pi)\\
&=
\int_{\IrrG}\Tr_{\kappa\stp\pi}\bigl(
\kappa\tp\pi(
\lambda^{u}(\omega)^*\lambda^{u}(\nu))
(\I_\kappa\otimes E_{\pi}^{-2})
\bigr)\md\mu(\pi).
\end{split}
\]
In particular, the above intergrals are convergent.
\end{lemma}

\begin{proof}
Let $\omega,\nu$ be functionals as in the lemma. We have
\[
\lambda(\omega+\nu)^*\lambda(\omega+\nu)=
\lambda((\omega^{\sharp}+\nu^{\sharp})\star(\omega+\nu))\in
\mf{M}_{\hpsi}^+,
\]
therefore we can use Lemma \ref{lemat7}:
\[
\begin{split}
+\infty&>\dim(\kappa)
\int_{\IrrG} \Tr(\pi(\lambda^{u}((\omega^{\sharp}+\nu^{\sharp})\star(\omega+\nu)))
\,E_{\pi}^{-2}) \md\mu(\pi)\\
&=
\int_{\IrrG}\Tr_{\kappa\stp\pi}\bigl(
\kappa\tp\pi(
\lambda^{u}((\omega^{\sharp}+\nu^{\sharp})\star(\omega+\nu))
(\I_\kappa\otimes E_{\pi}^{-2})
\bigr)\md\mu(\pi).
\end{split}
\]
We can subtract integrals with $\omega^{\sharp}\star\omega$ and $\nu^{\sharp}\star\nu$, because these are finite and equal:
\begin{equation}\label{eq15}
\begin{split}
&\dim(\kappa)
\int_{\IrrG} \Tr(\pi(\lambda^{u}(\omega^{\sharp}\star \nu+\nu^{\sharp}\star\omega))
\,E_{\pi}^{-2}) \md\mu(\pi)\\
&=
\int_{\IrrG}\Tr_{\kappa\stp\pi}\bigl(
\kappa\tp\pi(
\lambda^{u}(\omega^{\sharp}\star\nu+\nu^{\sharp}\star\omega)
(\I_\kappa\otimes E_{\pi}^{-2})
\bigr)\md\mu(\pi).
\end{split}
\end{equation}
Analogous reasoning with $\omega+i\nu$ in place of $\omega+\nu$ gives us
\[
(\omega+i\nu)^{\sharp}\star(\omega+i\nu)=
\omega^{\sharp}\star\omega+
\nu^{\sharp}\star\nu+
i\omega^{\sharp}\star\nu-i \nu^{\sharp}\star\omega
\]
therefore (after subtracting terms with $\omega^{\sharp}\star\omega, \nu^\sharp\star\nu$ and dividing by $i$)
\begin{equation}\label{eq16}
\begin{split}
&\dim(\kappa)
\int_{\IrrG} \Tr(\pi(\lambda^{u}(\omega^{\sharp}\star \nu-\nu^{\sharp}\star\omega))
\,E_{\pi}^{-2}) \md\mu(\pi)\\
&=
\int_{\IrrG}\Tr_{\kappa\stp\pi}\bigl(
\kappa\tp\pi(
\lambda^{u}(\omega^{\sharp}\star\nu-\nu^{\sharp}\star\omega)
(\I_\kappa\otimes E_{\pi}^{-2})
\bigr)\md\mu(\pi).
\end{split}
\end{equation}
If we add both sides of equalities \eqref{eq15}, \eqref{eq16} (and divide by $2$) we get
\[
\begin{split}
&\quad\;
\dim(\kappa)
\int_{\IrrG} \Tr(\pi(\lambda^{u}(\omega)^*\lambda^{u}( \nu))
\,E_{\pi}^{-2}) \md\mu(\pi)\\
&=\dim(\kappa)
\int_{\IrrG} \Tr(\pi(\lambda^{u}(\omega^{\sharp}\star \nu))
\,E_{\pi}^{-2}) \md\mu(\pi)\\
&=
\int_{\IrrG}\Tr_{\kappa\stp\pi}\bigl(
\kappa\tp\pi(
\lambda^{u}(\omega^{\sharp}\star\nu)
(\I_\kappa\otimes E_{\pi}^{-2})
\bigr)\md\mu(\pi)\\
&=
\int_{\IrrG}\Tr_{\kappa\stp\pi}\bigl(
\kappa\tp\pi(
\lambda^{u}(\omega)^*\lambda^{u}(\nu))
(\I_\kappa\otimes E_{\pi}^{-2})
\bigr)\md\mu(\pi).
\end{split}
\]
\end{proof}

Clearly there is also a version of this result for $\hvp$.

\section{Decomposition of the tensor product}\label{secdecomp}
Fix any Plancherel measure $\mu$ for $\GG$. Let us start with the definition of the representation associated with a subset $\Omega\subseteq\IrrG$:

\begin{definition}
Let $\Omega\subseteq \IrrG$ be a measurable subset. We will denote by $\sigma_\Omega$ the integral representation with measure space $(\Omega,\mc{B}(\Omega),\mu_\Omega)$, where $\mc{B}(\Omega)$ is the Borel $\sigma$-algebra and $\mu_\Omega$ is the restriction of $\mu$ to $\Omega$. As a measurable field of Hilbert spaces and field of representations we take the canonical fields on $\IrrG$ restricted to $\Omega$. The symbol $\sigma_\Omega$ will also stand for the representation of $\CGDu$ given by $\int_\Omega^{\oplus}\pi\md\mu_\Omega(\pi)$.
\end{definition}
 
Take any measurable subset $\Omega\subseteq\IrrG$ and a finite dimensional nondegenerate representation $\kappa\colon \CGDu\rightarrow \B(\msf{H}_\kappa)$. We can form the tensor product representation $\kappa\tp\sigma_\Omega$. Because $\CGDu$ is a \cst-algebra of type I, we have the following decomposition
\[
\kappa\tp\sigma_\Omega\simeq \bigoplus_{n\in\NN\cup\{\aleph_0\}}
n\cdot
\int_{\IrrG}^{\oplus} \zeta \md\mu_n(\zeta)
\]
for certain disjoint measures $\{\mu_n\,|\,n\in\NN\cup\{\aleph_0\}\}$ (\cite[Theorem 8.6.6]{DixmierC}). We know that the measure class associated with $\sigma_\Omega$ is $[\chi_\Omega\,\mu]$, it is clear that $[\chi_\Omega\, \mu]\ll[\mu]$. Next, due to \cite[Proposition 3.14]{SoltanWoronowicz} we know that the representation $\kappa\tp\Lambda_{\whG}$ is unitarily equivalent to the direct sum of $\dim(\kappa)$ copies of $\Lambda_{\whG}$, hence quasi-equivalent to $\Lambda_{\whG}$. In particular, the measure class associated with $\kappa\tp\Lambda_{\whG}$ is $[\mu]$. Thanks to the properties of quasi-containment (Proposition \ref{stw10}) we have the following: 
\[
[\mu_{\kappa\stp\sigma_\Omega}]=
[\mu_{\kappa}]\tp[\chi_\Omega\,\mu]\ll
[\mu_{\kappa}]\tp[\mu]=
[\mu_{\kappa\stp \Lambda_{\whG}}]=[\mu_{\Lambda_{\whG}}]=[\mu],
\]
and therefore we know that the measures $\{\mu_n\,|\,n\in\NN\cup\{\aleph_0\}\}$ are absolutely continuous with respect to $\mu$. Taking equivalent measures we can assume that
\[
\mu_n=\chi_{\E^n_{\kappa\stp\sigma_{\Omega}}} \,\mu
\]
for certain measurable pairwise disjoint subsets $\E^n_{\kappa\stp\sigma_\Omega}\subseteq\IrrG$ (defined uniquely up to measure $0$). We get
\[
\kappa\tp\sigma_\Omega \simeq \bigoplus_{n\in\NN\cup\{\aleph_0\}}n\cdot
\int_{\E^n_{\kappa\stp\sigma_\Omega}}^{\oplus}\zeta\md\mu_{\E^n_{\kappa\stp\sigma_\Omega}}(\zeta)=\bigoplus_{n\in\NN\cup\{\aleph_0\}} n\cdot \sigma_{\E^n_{\kappa\stp \sigma_{\Omega}}},
\]
where $\mu_{\E^n_{\kappa\stp\sigma_\Omega}}$ is the measure $\mu$ restricted to $\E^n_{\kappa\stp\sigma_\Omega}$. To be in the situation of Proposition \ref{treq}, define subsets
\[
\F^{n}_{\kappa\stp\sigma_\Omega}=
\bigcup_{k\in \{n,n+1,\dotsc\}\cup\{\aleph_0\}} \E^{k}_{\kappa\stp\sigma_\Omega}\quad(n\in\NN),
\]
then we get another decomposition
\[
\kappa\tp\sigma_\Omega\simeq \bigoplus_{n=1}^{\infty}
\int_{\F^n_{\kappa\stp\sigma_{\Omega}}}^{\oplus} \zeta \md\mu_{\F^n_{\kappa\stp\sigma_{\Omega}}}(\zeta)=
\bigoplus_{n=1}^{\infty} \sigma_{\F^{n}_{\kappa\stp \sigma_\Omega}}.
\]
We remark here that in Lemma \ref{lemat8} we will prove $\E^{\aleph_0}_{\kappa\stp\sigma_{\Omega}}=\emptyset$.\\
We will use the same notation of sets $\E,\F$ for arbitrary representation (not just $\kappa\tp\sigma_\Omega$) whose measures are absolutely continuous with respect to $\mu$.

\subsection{Results concerning sets $\F^n_{\kappa\stp\sigma_\Omega}$}

In this section we derive a couple of results concerning sets $\F^n_{\kappa\stp\sigma_\Omega}$. We start with aquiring some information about the decomposition of $\kappa\tp\Lambda_{\whG}$:

\begin{lemma}
For an arbitrary nondegenerate finite dimensional representation $\kappa\colon\CGDu\rightarrow\B(\msf{H}_\kappa)$ we have
\[
\sum_{n=1}^{\infty}\chi_{\F^n_{\kappa\stp\Lambda_{\whG}}}=\dim(\kappa)
\sum_{n=1}^{\infty} n\chi_{\IrrG\rest_n},\quad
\E^{\aleph_0}_{\kappa\stp\Lambda_{\whG}}=\emptyset.
\]
\end{lemma}

Recall that $\Omega\rest_n=\{\pi\in\Omega\,|\, \dim(\pi)=n\}$ for any $\Omega\subseteq \IrrG$.

\begin{proof}
By Theorem \ref{PlancherelL} we have the following equivalence
\[
\Lambda_{\whG}\simeq \int_{\IrrG}^{\oplus} (\dim \pi)\cdot \pi \md\mu(\pi),
\]
so
\[
\E^n_{\Lambda_{\whG}}=\IrrG\rest_n,\quad
\F^n_{\Lambda_{\whG}}=\bigcup_{k=n}^{\infty}\IrrG\rest_k
\quad(n\in\NN),\quad 
\E^{\aleph_0}_{\Lambda_{\whG}}=\emptyset. 
\]
Regular representation $W^{\GG}$ corresponds to the representation $\Lambda_{\whG}\in\Mor(\CGDu,\mc{K}(\LL^2(\GG)))$. We know that $\kappa\tp \Lambda_{\whG}\simeq \dim(\kappa)\cdot \Lambda_{\whG}$ (regular representation is right absorbing \cite[Proposition 3.14]{SoltanWoronowicz}). Therefore
\[
\kappa \tp \Lambda_{\whG}\simeq \dim(\kappa) \cdot\Lambda_{\whG}\simeq
\bigoplus_{n=1}^{\infty} \dim(\kappa)n\cdot \int_{\IrrG\rest_n}^{\oplus}
\pi\md\mu_{\IrrG\rest_n}(\pi)
\]
and $\E^n_{\kappa\stp \Lambda_{\whG}}=\IrrG\rest_{\tfrac{n}{\dim(\kappa)}}$ if $\tfrac{n}{\dim(\kappa)}\in \NN$ and $\E^n_{\kappa\stp \Lambda_{\whG}}=\emptyset$ otherwise. Consequently
\[
\sum_{i=1}^{\infty}\chi_{\F^i_{\kappa\stp\Lambda_{\whG}}}=
\sum_{n=1}^{\infty} n \chi_{\E^n_{\kappa\stp \Lambda_{\whG}}}=\dim(\kappa)
\sum_{n=1}^{\infty} n\chi_{\IrrG\rest_n}
\]
and $\E^{\aleph_0}_{\kappa\stp\Lambda_{\whG}}=\emptyset$.
\end{proof}

Recall that for a measurable subset $\Omega\subseteq\IrrG$ we have introduced the integral representation $\sigma_\Omega=\int_{\Omega}^{\oplus} \pi\md\mu_{\Omega}(\pi)$.

\begin{lemma}\label{lemat3}
For any finite dimensional nondegenerate representation $\kappa\colon\CGDu\rightarrow\B(\msf{H}_\kappa)$ we have
\[
1\le \sum_{n=1}^{\infty}\chi_{\F^n_{\kappa\stp\sigma_{\IrrG}}}
\le
\dim(\kappa)\sum_{n=1}^{\infty} n\chi_{\IrrG\rest_n}
\]
almost everywhere and $\E^{\aleph_0}_{\kappa\stp\sigma_{\IrrG}}=\emptyset$.
\end{lemma}

\begin{proof}
The first inequality follows from the quasi-equivalence
\[
\Lambda_{\whG}\simeq \int_{\IrrG}^{\oplus} (\dim \pi)\cdot \pi \md\mu(\pi)\approx_q
\int_{\IrrG}^{\oplus} \pi \md\mu(\pi)=\sigma_{\IrrG}
\]
which gives $\kappa\tp \sigma_{\IrrG}\approx_q \kappa\tp \Lambda_{\whG}\simeq \dim(\kappa)\cdot \Lambda_{\whG}$ and $\F^1_{\kappa\stp \sigma_{\IrrG}}=\F^1_{\kappa\stp \Lambda_{\whG}}=\IrrG$. Next, $\sigma_{\IrrG}$ is (equivalent to) a subrepresentation of $\Lambda_{\whG}$, so that $\kappa\tp\sigma_{\IrrG}$ is (equivalent to) a subrepresentation of $\kappa\tp\Lambda_{\whG}$. It follows that
\[
\sum_{n=1}^{\infty}\chi_{\F^n_{\kappa\stp\sigma_{\IrrG}}}\le
\sum_{n=1}^{\infty}\chi_{\F^n_{\kappa\stp\Lambda_{\whG}}}=
\dim(\kappa)\sum_{n=1}^{\infty} n \chi_{\IrrG\rest_n}.
\]
and $\E^{\aleph_0}_{\kappa\stp\sigma_{\IrrG}}=\emptyset$.
\end{proof}

We get a corollary for an arbitrary $\Omega\subseteq \IrrG$:

\begin{lemma}\label{lemat8}
Let $\kappa\colon\CGDu\rightarrow\B(\msf{H}_\kappa)$ be a finite dimensional nondegenerate representation, and $\Omega\subseteq\Omega'\subseteq\IrrG$ be measurable subsets. We have
\[
\sum_{n=1}^{\infty}n \chi_{\E^{n}_{\kappa\stp\sigma_{\Omega}}}=
\sum_{n=1}^{\infty}\chi_{\F^n_{\kappa\stp\sigma_{\Omega}}}\le
\sum_{n=1}^{\infty}\chi_{\F^n_{\kappa\stp\sigma_{\Omega'}}}\le
\dim(\kappa) \sum_{n=1}^{\infty} n \chi_{\IrrG\rest_n}
\]
almost everywhere and $\E^{\aleph_0}_{\kappa\stp\sigma_{\Omega}}=\emptyset$.
\end{lemma}

\subsection{Functions $\varpi$}
Let us fix a nondegenerate finite dimensional representation $\kappa\colon \CGDu\rightarrow \B(\msf{H}_\kappa)$ and a measurable subset $\Omega\subseteq\IrrG$. As usual, associate with $\Omega$ the integral representation $\sigma_\Omega$. Having lemmas from the previous subsection we can make use of Proposition \ref{treq}: take as the first representation
\[
\pi=\bigoplus_{i=1}^{\infty}
\int_{\F^i_{\kappa\stp\sigma_{\Omega}}}^{\oplus} \zeta \md\mu_{\F^i_{\kappa\stp\sigma_{\Omega}}}(\zeta)
\]
and define the second one to be
\[
\gamma=\int_{\Omega}^{\oplus} \kappa\tp x\md\mu_{\Omega}(x).
\]
Due to Lemma \ref{lemat8} we know that $\sum_{i=1}^{\infty} \chi_{\F^i_{\kappa\stp\sigma_{\Omega}}}<+\infty$. Sets $\F^i_{\kappa\stp\sigma_{\Omega}}$ correspond to the sets $X_i$ from Proposition \ref{treq}. Proposition \ref{treq} gives us a measurable function
\[
\varpi^{\kappa,\Omega,\mu}\colon \F^1_{\kappa\stp \sigma_\Omega}\rightarrow \RR_{>0}
\]
satisfying
\[
\begin{split}
&\quad\;
\Psi^{\int_\Omega}(\mc{O}\cdot\,\mc{O}^*)=\sum_{m=1}^{\infty} (\int_{\Omega\restriction_m}\otimes \Tr_m) (\mc{O} \cdot \mc{O}^*)\\
&=
\sum_{n=1}^{\infty} 
(\int_{(\bigsqcup_{i=1}^{\infty} \F^i_{\kappa\stp \sigma_\Omega})\restriction_n} \cdot \md\tilde{\mu}_{(\bigsqcup_{i=1}^{\infty} \F^i_{\kappa\stp \sigma_\Omega})\restriction_n} \otimes \Tr_n)=
(\Psi^{\int_{\bigsqcup_{i=1}^{\infty} \F^i_{\kappa\stp \sigma_\Omega}}})^{\sim}
\end{split}
\]
where $\tilde{\mu}_{(\bigsqcup_{i=1}^{\infty} \F^i_{\kappa\stp \sigma_\Omega})\restriction_n}$ is the measure $\mu_{(\bigsqcup_{i=1}^{\infty} \F^i_{\kappa\stp \sigma_\Omega})\restriction_n}$ multiplied by $\varpi^{\kappa,\Omega,\mu}$ (in the natural sense). Note that now on $\Omega$ we consider the field of representations $(\kappa\tp x)_{x\in \Omega}$ and $\mc{O}$ is a unitary operator given by appropriate compositions. \\
In the examples section we describe the function $\varpi^{\kappa,\Omega,\mu}$ in the case of quantum groups which are compact, classical, dual to classical or constructed via certain bicrossed product.\\
We would like to find the function $\varpi^{\kappa,\Omega,\mu}$ or at least derive some bounds for it. Our first result in this direction tells us how the function $\varpi$ changes once we change the Plancherel measure:

\begin{lemma}\label{lematvarpi}
Let $\Omega\subseteq\IrrG$ be a measurable subset, $\kappa\colon\CGDu\rightarrow\B(\msf{H}_\kappa)$ a nondegenerate finite dimensional representation and $f\colon \IrrG\rightarrow \RR_{>0}$ a measurable function. Take another Plancherel measure $\mu'=f\mu$. Assume that $c_1\le f\le c_2$ on $\Omega$. Then
\[
c_1\,\varpi^{\kappa,\Omega,\mu}\le f\varpi^{\kappa,\Omega,\mu'}\le c_2\,\varpi^{\kappa,\Omega,\mu}
\]
almost everywhere on $\F^1_{\kappa\stp\sigma_{\Omega}}$.
\end{lemma}

\begin{proof}
For $a\in \CGDu^{**}_+$ let
\[
\begin{split}
\overline{(\int_{\bigsqcup_{i=1}^{\infty}\F^i_{\kappa\stp\sigma_{\Omega}}}^{\oplus} x 
\md\mu_{\bigsqcup_{i=1}^{\infty}\F^i_{\kappa\stp\sigma_{\Omega}}}(x)})(a)&=
\int_{\bigsqcup_{i=1}^{\infty}\F^i_{\kappa\stp\sigma_{\Omega}}}^{\oplus} S_x
\md\mu_{\bigsqcup_{i=1}^{\infty}\F^i_{\kappa\stp\sigma_{\Omega}}}(x),\\
\overline{(\int_{\Omega}^{\oplus} \kappa\tp x 
\md\mu_{\Omega}(x)})(a)&=
\int_{\Omega}^{\oplus} T_x
\md\mu_{\Omega}(x)
\end{split}
\]
for certain almost everywhere positive operators $S_x,T_x$. We also have
\[
\begin{split}
\overline{(\int_{\bigsqcup_{i=1}^{\infty}\F^i_{\kappa\stp\sigma_{\Omega}}}^{\oplus} x 
\md\mu'_{\bigsqcup_{i=1}^{\infty}\F^i_{\kappa\stp\sigma_{\Omega}}}(x)})(a)&=
\int_{\bigsqcup_{i=1}^{\infty}\F^i_{\kappa\stp\sigma_{\Omega}}}^{\oplus} S_x
\md\mu'_{\bigsqcup_{i=1}^{\infty}\F^i_{\kappa\stp\sigma_{\Omega}}}(x),\\
\overline{(\int_{\Omega}^{\oplus} \kappa\tp x 
\md\mu'_{\Omega}(x)})(a)&=
\int_{\Omega}^{\oplus} T_x
\md\mu'_{\Omega}(x)
\end{split}
\]
and
\[
\begin{split}
&\quad\;\int_{\IrrG} \bigl( \sum_{i=1}^{\infty} \chi_{\F^i_{\kappa\stp\sigma_\Omega}}(\zeta)\bigr) \Tr_\zeta(S_\zeta)
\varpi^{\kappa,\Omega,\mu'}(\zeta) f(\zeta)\md\mu(\zeta)\\
&=
\int_{\IrrG} \bigl( \sum_{i=1}^{\infty} \chi_{\F^i_{\kappa\stp\sigma_\Omega}}(\zeta)\bigr) \Tr_\zeta(S_\zeta)
\varpi^{\kappa,\Omega,\mu'}(\zeta) \md\mu'(\zeta)\\
&=
\int_\Omega \Tr_{\kappa\stp x}(T_x) \md\mu'_\Omega(x)\le
c_2\int_\Omega \Tr_{\kappa\stp x}(T_x) \md\mu_\Omega(x)\\
&=
c_2\int_{\IrrG} \bigl( \sum_{i=1}^{\infty} \chi_{\F^i_{\kappa\stp\sigma_\Omega}}(\zeta)\bigr) \Tr_\zeta(S_\zeta)
\varpi^{\kappa,\Omega,\mu}(\zeta) \md\mu(\zeta).
\end{split}
\]
Since $a$ is arbitrary, $S_\zeta$ also is 
arbitrary on $\F^1_{\kappa\stp\sigma_\Omega}$ and we get $f\varpi^{\kappa,\Omega,\mu'}\le c_2\varpi^{\kappa,\Omega,\mu}$. The second inequality can be derived in a similar fashion.
\end{proof}

The next proposition gives us an upper bound on $\varpi^{\kappa,\Omega,\mu}$ and is crucial for further reasoning.

\begin{proposition}\label{stw4}
Let $\kappa\colon\CGDu\rightarrow\B(\msf{H}_\kappa)$ be such a nondegenerate representation that $\kappa\lec \Lambda_{\whG}$ and $\dim(\kappa)<+\infty$. Let $\Omega\subseteq\IrrG$ be a measurable subset. The inequality
\[
\bigl( \sum_{i=1}^{\infty} \chi_{\F^i_{\kappa\stp\sigma_{\Omega}}}\bigr) (\pi)\varpi^{\kappa,\Omega,\mu}(\pi)\le
(\sup_{\pi'\in \Omega}\|E_{\pi'}^2\|)
\dim(\kappa) \|E_\pi^2\|^{-1}
\]
holds for almost every $\pi\in \IrrG$.
\end{proposition}

\begin{proof}
It is enough to consider the case $\sup_{\pi\in\Omega}\|E^2_{\pi}\|<+\infty$. Define
\[
K_\pi\colon \HS(\msf{H}_\pi)\ni S \mapsto
S E_\pi\in\HS(\msf{H}_\pi)\quad(\pi\in\IrrG),
\]
it is a bounded positive operator. Moreover, consider the operator 
\[
\int_{\IrrG}^{\oplus}
 \bigl( \sum_{i=1}^{\infty} \chi_{\F^i_{\kappa\stp\sigma_{\Omega}}}\bigr) (\pi)^{\frac{1}{2}}
\varpi^{\kappa,\Omega,\mu}(\pi)^{\frac{1}{2}} K_\pi
\md\mu(\pi)
\]
with domain consisting of those $\int_{\IrrG}^{\oplus}T_\pi\md\mu(\pi)\in \int_{\IrrG}^{\oplus} \HS(\msf{H}_\pi)\md\mu(\pi)$ for which
\[
\int_{\IrrG}\bigl\|
 \bigl( \sum_{i=1}^{\infty} \chi_{\F^i_{\kappa\stp\sigma_{\Omega}}}\bigr) (\pi)^{\frac{1}{2}}
\varpi^{\kappa,\Omega,\mu}(\pi)^{\frac{1}{2}} K_\pi(T_\pi)
\bigr\|_{\HS(\msf{H}_\pi)}^{2}\md\mu(\pi)<+\infty.
\]
It is an unbounded positive self-adjoint operator (for the theory of direct integrals of unbounded operators we refer to \cite{Lance}). For $\lambda(\omega)\in\lambda(\LL^1_{\sharp}(\GG))\cap\mf{N}_{\hpsi}$ we have
\[
\begin{split}
&\quad\;
\int_{\IrrG}\bigl\|
 \bigl( \sum_{i=1}^{\infty} \chi_{\F^i_{\kappa\stp\sigma_{\Omega}}}\bigr) (\pi)^{\frac{1}{2}}
\varpi^{\kappa,\Omega,\mu}(\pi)^{\frac{1}{2}} K_\pi
\bigl( \pi(\lambda^u(\omega))E_\pi^{-1}\bigr)
\bigr\|_{\HS(\msf{H}_\pi)}^2\md\mu(\pi)\\
&=
\int_{\IrrG}
 \bigl( \sum_{i=1}^{\infty} \chi_{\F^i_{\kappa\stp\sigma_{\Omega}}}\bigr) (\pi)
\Tr_{\pi}\bigl(\pi(\lambda^u(\omega)^*\lambda^u(\omega)) 
\bigr)\varpi^{\kappa,\Omega,\mu}(\pi)
\md\mu(\pi)\\
&\overset{\star}{=}
\int_{\Omega} \Tr_{ \kappa\stp\pi} \bigl(
\kappa\tp\pi ( \lambda^u(\omega)^*\lambda^u(\omega))
\bigr)\md\mu(\pi)\\
&=
\int_{\Omega} \Tr_{ \kappa\stp\pi} \bigl(
\kappa\tp\pi ( \lambda^u(\omega)^*\lambda^u(\omega))
(\I_\kappa\otimes E_\pi^{-2})(\I_\kappa\otimes E_\pi^{2})
\bigr)\md\mu(\pi)\\
&\le
(\sup_{\pi\in \Omega}\|E_\pi^2\|)
\int_{\IrrG}\Tr_{ \kappa\stp\pi} \bigl(
\kappa\tp\pi ( \lambda^u(\omega)^*\lambda^u(\omega))
(\I_\kappa\otimes E_\pi^{-2})
\bigr)\md\mu(\pi)\\
&\overset{\star\star}{=}
(\sup_{\pi\in \Omega}\|E_\pi^2\|)
\dim(\kappa)\int_{\IrrG}
\Tr_{\pi} (\pi(\lambda^{u}(\omega)^* \lambda^{u}(\omega)) E_\pi^{-2})
\md\mu(\pi)\\
&=
\bigl\langle \int_{\IrrG}^{\oplus} \pi(\lambda^u(\omega)) E_\pi^{-1} \md\mu(\pi) \big|
(\sup_{\pi\in\Omega}\|E_\pi^2\|) \dim(\kappa)
\int_{\IrrG}^{\oplus} \pi(\lambda^u(\omega))E_\pi^{-1}\md\mu(\pi)
\bigr\rangle<+\infty,
\end{split}
\]
hence vector $\int_{\IrrG}^{\oplus} \pi(\lambda^u(\omega))E_\pi^{-1}\md\mu(\pi)$ belongs to the domain and we have an inequality
\[\begin{split}
&\quad\;
\big\|
\int_{\IrrG}^{\oplus}
 \bigl( \sum_{i=1}^{\infty} \chi_{\F^i_{\kappa\stp\sigma_{\Omega}}}\bigr) (\pi)^{\frac{1}{2}}
\varpi^{\kappa,\Omega,\mu}(\pi)^{\frac{1}{2}} K_\pi
\md\mu(\pi)
\int_{\IrrG}^{\oplus} \pi(\lambda^u(\omega))E_\pi^{-1}\md\mu(\pi)
\bigr\|^2\\
&\le
(\sup_{\pi\in\Omega}\|E_\pi^2\|) \dim(\kappa)
\bigl\|\int_{\IrrG}^{\oplus} \pi(\lambda^u(\omega))E_\pi^{-1}\md\mu(\pi)
\bigr\|^2.
\end{split}\]
In equation $\overset{\star}{=}$ we have used the definition of the function $\varpi^{\kappa,\Omega,\mu}$ and in $\overset{\star\star}{=}$ we have used Lemma \ref{lemat16}.\\
From Lemma \ref{lemat28} and unitarity of $\mc{Q}_R$ it follows that the subspace
\[
\bigl\{
\int_{\IrrG}^{\oplus}
\pi(\lambda^u(\omega)) E_\pi^{-1}\md\mu(\pi)\big|
\lambda(\omega)\in\lambda(\LL^1_{\sharp}(\GG))\cap\mf{N}_{\hpsi}\bigr\}.
\]
is dense in $\int_{\IrrG}^{\oplus}\HS(\msf{H}_\pi)\md\mu(\pi)$, therefore we can use Lemma \ref{lemat24}: the operator
\[
\int_{\IrrG}^{\oplus}
 \bigl( \sum_{i=1}^{\infty} \chi_{\F^i_{\kappa\stp\sigma_{\Omega}}}\bigr) (\pi)^{\frac{1}{2}}
\varpi^{\kappa,\Omega,\mu}(\pi)^{\frac{1}{2}} K_\pi
\md\mu(\pi)
\]
is bounded. The same reasoning gives
\[\begin{split}
&\quad\;
\bigl\langle \int_{\IrrG}^{\oplus} \pi(\lambda^u(\omega)) E_\pi^{-1} \md\mu(\pi) \big|
\int_{\IrrG}^{\oplus}
 \bigl( \sum_{i=1}^{\infty} \chi_{\F^i_{\kappa\stp\sigma_{\Omega}}}\bigr) (\pi)
\varpi^{\kappa,\Omega,\mu}(\pi) K_\pi^2
\md\mu(\pi)
\int_{\IrrG}^{\oplus} \pi(\lambda^u(\omega))E_\pi^{-1}\md\mu(\pi)
\bigr\rangle\\
&\le
\bigl\langle \int_{\IrrG}^{\oplus} \pi(\lambda^u(\omega)) E_\pi^{-1} \md\mu(\pi) \big|
(\sup_{\pi\in\Omega}\|E_\pi^2\|) \dim(\kappa)
\int_{\IrrG}^{\oplus} \pi(\lambda^u(\omega))E_\pi^{-1}\md\mu(\pi)
\bigr\rangle.
 \end{split}\]
Consequently, the operator
\[
\int_{\IrrG}^{\oplus}
\bigl(
(\sup_{\pi\in\Omega}\|E_\pi^2\|)\dim(\kappa)\I_\pi-
 \bigl( \sum_{i=1}^{\infty} \chi_{\F^i_{\kappa\stp\sigma_{\Omega}}}\bigr) (\pi)
\varpi^{\kappa,\Omega,\mu}(\pi) K_\pi^2\bigr)
\md\mu(\pi)
\]
is bounded and positive. In particular
\[
 \bigl( \sum_{i=1}^{\infty} \chi_{\F^i_{\kappa\stp\sigma_{\Omega}}}\bigr) (\pi)
\varpi^{\kappa,\Omega,\mu}(\pi) K_\pi^2\le
(\sup_{\pi'\in\Omega}\|E_{\pi'}^2\|)\dim(\kappa)\I_\pi
\]
for almost all $\pi\in \IrrG$ (as an inequality of operators on the Hilbert space $\HS(\msf{H}_\pi)$). Let $\xi_\pi\in\msf{H}_\pi$ be a eigenvector of $E_\pi$ with highest eigenvalue, and $P_\pi$ the projection onto $\CC \xi_\pi$. The previous inequality gives us
\[\begin{split}
&
 \bigl( \sum_{i=1}^{\infty} \chi_{\F^i_{\kappa\stp\sigma_{\Omega}}}\bigr) (\pi)
\varpi^{\kappa,\Omega,\mu}(\pi) \|E_\pi^2\|=
\bigl\langle
P_{\pi} \big|
 \bigl( \sum_{i=1}^{\infty} \chi_{\F^i_{\kappa\stp\sigma_{\Omega}}}\bigr) (\pi)
\varpi^{\kappa,\Omega,\mu}(\pi) K_\pi^2 (P_\pi)\bigr\rangle
\\
&\le
\bigl\langle P_\pi \big|
(\sup_{\pi'\in\Omega}\|E_{\pi'}^2\|)\dim(\kappa)\I_\pi(P_\pi)
\bigr\rangle=
(\sup_{\pi'\in\Omega}\|E_{\pi'}^2\|)\dim(\kappa).
\end{split}\]
\end{proof}

The next lemma allows us to compute $\varpi^{\kappa,\Omega,\mu}$ if the set $\Omega$ can be written as a disjoint sum of two sets. We will also use it in Section \ref{secconv} in order to define the operator $\mc{L}_\kappa$.

\begin{lemma}\label{lemat6}
Let $\kappa\colon\CGDu\rightarrow\B(\msf{H}_\kappa)$ be such a representation that $\kappa\lec \Lambda_{\whG}$ and $\dim(\kappa)<+\infty$. Let $\Omega_1,\Omega_2$ be two disjoint measurable subsets of $\IrrG$ and $\Omega=\Omega_1\cup\Omega_2$. Then the equality
\[
\varpi^{\kappa,\Omega_1,\mu}(\pi)
\sum_{i=1}^{\infty} \chi_{\F^i_{\kappa\stp\sigma_{\Omega_1}}} (\pi)
+
\varpi^{\kappa,\Omega_2,\mu}(\pi)
\sum_{i=1}^{\infty} \chi_{\F^i_{\kappa\stp\sigma_{\Omega_2}}}(\pi)
=
\varpi^{\kappa,\Omega,\mu}(\pi)
\sum_{i=1}^{\infty} \chi_{\F^i_{\kappa\stp\sigma_{\Omega}}}(\pi)
\]
holds for almost all $\pi\in\IrrG$.
\end{lemma}

\begin{proof}
We have unitary intertwiners:
\[
\mc{O}_k\colon \bigoplus_{i=1}^{\infty}\int^{\oplus}_{\F^{i}_{\kappa\stp\sigma_{\Omega_k}}}\msf{H}_{\zeta} \md\mu_{\F^i_{\kappa\stp\sigma_{\Omega_k}}}(\zeta)\rightarrow
\int_{\Omega_k}^{\oplus} \msf{H}_{\kappa\stp x}\md\mu_{\Omega_k}(x)
\quad(k\in\{1,2\})
\]
and
\[
\mc{O}\colon \bigoplus_{i=1}^{\infty}\int^{\oplus}_{\F^{i}_{\kappa\stp\sigma_{\Omega}}}\msf{H}_{\zeta} \md\mu_{\F^i_{\kappa\stp\sigma_{\Omega}}}(\zeta)\rightarrow
\int_{\Omega}^{\oplus} \msf{H}_{\kappa\stp x}\md\mu_{\Omega}(x).
\]
Define
\[
\pi_k=\bigoplus_{i=1}^{\infty}\int_{\F^{i}_{\kappa\stp\sigma_{\Omega_k}}}^{\oplus}\zeta 
\md\mu_{\F^i_{\kappa\stp\sigma_{\Omega_k}}}(\zeta)\quad(k\in\{1,2\}),\quad
\pi=\bigoplus_{i=1}^{\infty}\int_{\F^{i}_{\kappa\stp\sigma_{\Omega}}}^{\oplus}\zeta 
\md\mu_{\F^i_{\kappa\stp\sigma_{\Omega}}}(\zeta)
\]
and let $P_k$ be the canonical projection
\[
P_k\colon \int_{\Omega}^{\oplus} \msf{H}_{\kappa\stp x}\md\mu_{\Omega}(x)\rightarrow 
\int_{\Omega_k}^{\oplus} \msf{H}_{\kappa\stp x}\md\mu_{\Omega_k}(x)\quad
(k\in\{1,2\})
\]
corresponding to the inclusion $\Omega_k\subseteq \Omega$. For $a\in \CGDu$ we have
\[
P_1^*\bigl(\int_{\Omega_1}^{\oplus}\kappa\tp x \md\mu_{\Omega_1}(x)
\bigr)(a)P_1+
P_2^*\bigl(\int_{\Omega_2}^{\oplus}\kappa\tp x \md\mu_{\Omega_2}(x)
\bigr)(a)P_2=
\bigl(\int_{\Omega}^{\oplus}\kappa\tp x \md\mu_{\Omega}(x)\bigr)(a)
\]
and moreover
\[
\begin{split}
&\quad\;P_1^*\mc{O}_1\bigl(\bigoplus_{i=1}^{\infty}\int_{\F^{i}_{\kappa\stp\sigma_{\Omega_1}}}^{\oplus}\pi 
\md\mu_{\F^i_{\kappa\stp\sigma_{\Omega_1}}}(\pi)
\bigr)(a)\mc{O}_1^*P_1+
P_2^*\mc{O}_2\bigl(\bigoplus_{i=1}^{\infty}\int_{\F^{i}_{\kappa\stp\sigma_{\Omega_2}}}^{\oplus}\pi 
\md\mu_{\F^i_{\kappa\stp\sigma_{\Omega_2}}}(\pi)
\bigr)(a)\mc{O}_2^*P_2\\
&=
\mc{O}\bigl(\bigoplus_{i=1}^{\infty}\int_{\F^{i}_{\kappa\stp\sigma_{\Omega}}}^{\oplus}\pi 
\md\mu_{\F^i_{\kappa\stp\sigma_{\Omega}}}(\pi)
\bigr)(a)\mc{O}^*.
\end{split}
\]
From the $\swot$ continuity in $a$ we get this equality for an arbitrary $a\in \CGDu^{**}$ (and extended representations). For an arbitrary bounded $0\le T=\int_{\IrrG}^{\oplus}T_\pi\md\mu(\pi)$ let $a\in \CGDu^{**}_+$ be such an element that $(\int_{\IrrG}^{\oplus}\pi\md\mu(\pi))^{-}(a)=T$ (recall that bar denotes the extension of the representation to the bidual). Such an element $a$ exists since $(\int_{\IrrG}^{\oplus} \pi\md\mu(\pi))(\CGDu)''=\Dec(\int_{\IrrG}^{\oplus} \msf{H}_\pi \md\mu(\pi))$ \cite[Proposition 8.6.4, A 80]{DixmierC}. We have
\[
\begin{split}
&\quad\;
\int_{\IrrG} \bigl(
\Tr_\pi(T_\pi) \varpi^{\kappa,\Omega_1,\mu}(\pi)
\sum_{i=1}^{\infty} \chi_{\F^i_{\kappa\stp\sigma_{\Omega_1}}} (\pi)
+
\Tr_\pi(T_\pi)\varpi^{\kappa,\Omega_2,\mu}(\pi)
\sum_{i=1}^{\infty} \chi_{\F^i_{\kappa\stp\sigma_{\Omega_2}}}
(\pi)
\bigr)\md\mu(\pi)\\
&=
\int_{\Omega_1} \Tr_{\kappa\stp x}(\mc{O}_1\ov{\pi_1}(a)\mc{O}^*_1|_{\kappa\stp x} ) \md\mu_{\Omega_1}(x)
+
\int_{\Omega_2} \Tr_{\kappa\stp x}(\mc{O}_2\ov{\pi_2}(a)\mc{O}^*_2|_{\kappa\stp x}) \md\mu_{\Omega_2}(x)\\
&=
\int_{\Omega} \Tr_{\kappa\stp x}(\mc{O}\ov{\pi}(a)\mc{O}^*|_{\kappa\stp x}) \md\mu_{\Omega}(x)\\
&=
\int_{\IrrG}
\Tr_\pi(T_\pi) \varpi^{\kappa,\Omega,\mu}(\pi)
\sum_{i=1}^{\infty}\chi_{\F^i_{\kappa\stp\sigma_{\Omega}}}
(\pi) \md\mu(\pi).
\end{split}
\]
Since $T$ was arbitrary, we get the claim.
\end{proof}

\section{Square integrable integral characters}\label{secsquare}
Recall that we assume that $\whG$ is second countable and $\GG$ is a type I locally compact quantum group with finite dimensional irreducible representations. Fix any Plancherel measure for $\GG$.\\
In this section we will exhibit relations between: on the one hand, conditions similar to $\chi^{\int}(\int_\Omega^{\oplus}\pi\md\mu_{\Omega}(\pi))\in\mf{N}_\psi$, and on the other hand, $\int_\Omega \Tr(E^2_\bullet)\md\mu<+\infty$. \\
We will frequently use the following orthogonality relations (\cite[Lemma 2.1.2, Lemma 2.1.4, Theorem 2.1.5]{Caspers}):

\begin{proposition}\label{Casp}
Let $\xi,\xi',\eta,\eta'\in\int_{\IrrG}^{\oplus} \msf{H}_\pi \md\mu(\pi)$ be square integrable vector fields and let $E=\int_{\IrrG}^{\oplus} E_\pi \md\mu(\pi), D=\int_{\IrrG}^{\oplus} D_\pi\md\mu(\pi)$.
\begin{enumerate}[label=\arabic*)]
\item
Assume that $\eta,\eta'\in\Dom(E)$ and fields $(\xi_\pi\otimes \ov{E_\pi \eta_\pi})_{\pi\in \IrrG}, (\xi'_\pi\otimes \ov{E_\pi \eta'_\pi})_{\pi\in \IrrG}$ are square integrable. Then, there exist \swot-convergent integrals
\[
\int_{\IrrG} (\id\otimes\omega_{\xi_\pi,\eta_\pi})U^{\pi} \md\mu(\pi),\quad
\int_{\IrrG} (\id\otimes\omega_{\xi'_\pi,\eta'_\pi})U^{\pi} \md\mu(\pi)
\in\mf{N}_{\psi}\subseteq\Linf.
\]
Moreover, we have the following orthogonality relation:
\[
\begin{split}
\psi\bigl(\bigl(
\int_{\IrrG} (\id\otimes\omega_{\xi_\pi,\eta_\pi})U^{\pi} \md\mu(\pi)
\bigr)^*&\bigl(
\int_{\IrrG} (\id\otimes\omega_{\xi'_\pi,\eta'_\pi})U^{\pi} \md\mu(\pi)
\bigr)\bigr)\\
&=
\int_{\IrrG}\is{\xi'_\pi}{\xi_\pi}
\is{E_\pi \eta_\pi}{E_\pi \eta'_\pi}\md\mu(\pi).
\end{split}
\]
\item
If $\eta,\eta'\in\Dom(D)$ and fields $(\xi_\pi\otimes \ov{D_\pi \eta_\pi})_{\pi\in \IrrG}, (\xi'_\pi\otimes \ov{D_\pi \eta'_\pi})_{\pi\in \IrrG}$ are square integrable, then there exist \swot-convergent integrals
\[
\int_{\IrrG} (\id\otimes\omega_{\xi_\pi,\eta_\pi})(U^{\pi *}) \md\mu(\pi),\quad
\int_{\IrrG} (\id\otimes\omega_{\xi'_\pi,\eta'_\pi})(U^{\pi *}) \md\mu(\pi)
\in\mf{N}_{\vp}\subseteq\Linf.
\]
We have also
\[
\begin{split}
\vp\bigl(\bigl(
\int_{\IrrG} (\id\otimes\omega_{\xi_\pi,\eta_\pi})(U^{\pi *}) \md\mu(\pi)
\bigr)^*&\bigl(
\int_{\IrrG} (\id\otimes\omega_{\xi'_\pi,\eta'_\pi})(U^{\pi *}) \md\mu(\pi)
\bigr)\bigr)\\
&=
\int_{\IrrG}\is{\xi'_\pi}{\xi_\pi}
\is{D_\pi \eta_\pi}{D_\pi \eta'_\pi}\md\mu(\pi).
\end{split}
\]
\end{enumerate}
\end{proposition}

\subsection{Approach via sesquilinear form}
In the next subsection it will turn out that for a subset $\Omega\subseteq\IrrG$, the condition $\int_\Omega \Tr(E^2_\bullet)\md\mu<+\infty$ is directly related to $\chi^{\int}(\int_{\Omega}^{\oplus} \pi\md\mu(\pi))\in\mf{N}_{\psi}$. Our aim now is to derive the following equality:
\[
\int_\Omega \Tr(E^2_\pi)\md\mu(\pi)=
\sum_{k=1}^{\infty}\psi(
(\int_{\IrrG} (\id\otimes\omega_{\xi^k_\pi})U^\pi \md\mu(\pi))^*
(\int_{\IrrG} (\id\otimes\omega_{\xi^k_\pi})U^\pi \md\mu(\pi))),
\]
which gives necessary and sufficient requirements for this condition to hold (precise formulation of this result can be found as Proposition \ref{stw6}). First, we will derive however a more general statement, working with unbounded sesquilinear forms. This approach owes much to the proof of \cite[Theorem 2.1.6]{Caspers}.\\
Fix a vector $\xi=\bigoplus_{k=1}^{\infty}\xi^k \in \bigoplus_{k=1}^{\infty} \int_{\IrrG}^{\oplus} \msf{H}_\pi \md\mu(\pi)$ such that for each $k\in \NN$ the vector field $(\xi^k_\pi)_{\pi\in\IrrG}$ is bounded and the function $\|\xi^k_\bullet\|$ is bounded from below on its support. For $\eta\in \bigoplus_{k\in\NN}\int_{\IrrG}^{\oplus} \msf{H}_\pi\md\mu(\pi)$ and $k\in\NN$ the integral $\int_{\IrrG} (\id\otimes\omega_{\xi^k_\pi,\eta^k_\pi})U^\pi \md\mu(\pi)$ is well defined: measurability is not a problem, moreover
\[
\begin{split}
&\quad\;\int_{\IrrG} \|(\id\otimes\omega_{\xi^k_\pi,\eta^k_\pi}U^\pi\| \md\mu(\pi)\le
\int_{\IrrG} \|\xi^k_\pi\| \|\eta^k_\pi\|\md\mu(\pi)\\
&\le
\bigl(
\int_{\IrrG} \|\xi^k_\pi\|^2\md\mu(\pi)\bigr)^{\frac{1}{2}}
\bigl(
\int_{\IrrG} \|\eta^k_\pi\|^2\md\mu(\pi)\bigr)^{\frac{1}{2}}<+\infty.
\end{split}
\]
It follows that we can consider a subspace
\[
\begin{split}
&\Dom(q)=\bigl\{
\eta=\bigoplus_{k=1}^{\infty} \int_{\IrrG}^{\oplus}\eta^k_\pi \md\mu(\pi)
\in \bigoplus_{k=1}^{\infty} \int_{\IrrG}^{\oplus} \msf{H}_\pi \md\mu(\pi) \big|\\
&\big|\sum_{k=1}^{\infty}
\psi(
(\int_{\IrrG} (\id\otimes\omega_{\xi^k_\pi,\eta^k_\pi})U^\pi \md\mu(\pi))^*
(\int_{\IrrG} (\id\otimes\omega_{\xi^k_\pi,\eta^k_\pi})U^\pi \md\mu(\pi)))<+\infty\bigr\}.
\end{split}
\]
and define an (unbounded) sesquilinear form $q$ with the domain $\Dom(q)$, acting as follows:
\[
q(\eta,\eta')=
\sum_{k=1}^{\infty}
\psi(
(\int_{\IrrG} (\id\otimes\omega_{\xi^k_\pi,\eta^k_\pi})U^\pi \md\mu(\pi))^*
(\int_{\IrrG} (\id\otimes\omega_{\xi^k_\pi,{\eta'}^k_\pi})U^\pi \md\mu(\pi)))
\quad(\eta,\eta'\in\Dom(q)).
\]
It is clear that for $\eta,\eta'\in \Dom(q)$ the series in the definition of $q(\eta,\eta')$ is convergent. Note that $q$ depends on the choice of vector $\xi$.

\begin{theorem}\label{tw5}
Let $\bigoplus_{k\in\NN}\xi^k$ be a vector as above and let $q$ be an unbounded sesquilinear form associated with it. It is densely defined, closed and positive. Moreover $\Dom(q)=\Dom(\bigoplus_{k\in\NN} \int_{\IrrG}^{\oplus} \|\xi^k_\pi\| E_\pi\md\mu(\pi))$ and
\[
q(\eta,\eta')=\bigl\langle \bigl(\bigoplus_{k\in\NN} \int_{\IrrG}^{\oplus} \|\xi^k_\pi\| E_\pi\md\mu(\pi)\bigr)\, \eta\big|
\bigl(\bigoplus_{k\in\NN} \int_{\IrrG}^{\oplus} \|\xi^k_\pi\| E_\pi\md\mu(\pi)\bigr)\,\eta'\bigr\rangle
\]
for $\eta,\eta'\in\Dom(q)$.
\end{theorem}

\begin{proof}
First, let us show that $q$ is closed: let $(\eta(n))_{n\in\NN}$ be a sequence in $\Dom(q)$ such that $\eta(n)\xrightarrow[n\to\infty]{} \eta\in\bigoplus_{k\in\NN}\int_{\IrrG}^{\oplus}\msf{H}_\pi\md\mu(\pi)$ and $q(\eta(n)-\eta(m))\xrightarrow[n,m\to\infty]{}0$. We want to show that $\eta\in \Dom(q)$ and $q(\eta(n)-\eta)\xrightarrow[n\to\infty]{}0$. We have
\[
\sum_{k=1}^{\infty}
\psi(
(\int_{\IrrG} (\id\otimes\omega_{\xi^k_\pi,\eta^k_\pi(n)-\eta^k_\pi(m)})U^\pi \md\mu(\pi))^*
(\int_{\IrrG} (\id\otimes\omega_{\xi^k_\pi,\eta^k_\pi(n)-\eta^k_\pi(m)})U^\pi \md\mu(\pi)))
\xrightarrow[n,m\to\infty]{}0,
\]
therefore the sequence
\[
\bigl(\bigoplus_{k=1}^{\infty}\Lambda_{\psi}( \int_{\IrrG} (\id\otimes\omega_{\xi^k_\pi,\eta^k_\pi(n)})U^\pi \md\mu(\pi) )\bigr)_{n\in\NN}
\]
is Cauchy. Let us see what happens on individual $k$'s. For each $k\in\NN$ the following holds
\[
\begin{split}
&\quad\;\bigl\|
\int_{\IrrG} (\id\otimes\omega_{\xi^k_\pi,\eta^k_\pi(n)})U^\pi \md\mu(\pi) -
\int_{\IrrG} (\id\otimes\omega_{\xi^k_\pi,\eta^k_\pi})U^\pi \md\mu(\pi)\bigr\|\\
&=
\bigl\|
\int_{\IrrG} (\id\otimes\omega_{\xi^k_\pi,\eta^k_\pi(n)-\eta^k_\pi})U^\pi \md\mu(\pi)\bigr\|\le
\int_{\IrrG} \|\xi^k_\pi\| \|\eta^k_\pi(n)-\eta^k_\pi\|\md\mu(\pi)\\
&\le
\|\xi^k\| \|\eta^k(n)-\eta^k\|\xrightarrow[n\to\infty]{}0,
\end{split}
\]
therefore due to closedness of $\Lambda_\psi$ we arrive at
\[
\int_{\IrrG} (\id\otimes\omega_{\xi^k_\pi,\eta^k_\pi})U^\pi \md\mu(\pi)\in \Dom(\Lambda_\psi)
\]
and
\[
\Lambda_{\psi}( \int_{\IrrG} (\id\otimes\omega_{\xi^k_\pi,\eta^k_\pi(n)})U^\pi \md\mu(\pi) )\xrightarrow[n\to\infty]{}
\Lambda_{\psi}( \int_{\IrrG} (\id\otimes\omega_{\xi^k_\pi,\eta^k_\pi})U^\pi \md\mu(\pi) )
\]
for each $k\in\NN$. It follows that
\[
\bigoplus_{k=1}^{\infty}\Lambda_{\psi}( \int_{\IrrG} (\id\otimes\omega_{\xi^k_\pi,\eta^k_\pi(n)})U^\pi \md\mu(\pi) )
\xrightarrow[n\to\infty]{}
\bigoplus_{k=1}^{\infty}\Lambda_{\psi}( \int_{\IrrG} (\id\otimes\omega_{\xi^k_\pi,\eta^k_\pi})U^\pi \md\mu(\pi) )
\]
has to hold. We get
\[\begin{split} 
&\quad\;\sum_{k=1}^{\infty}
\psi(
(\int_{\IrrG} (\id\otimes\omega_{\xi^k_\pi,\eta^k_\pi})U^\pi \md\mu(\pi))^*
(\int_{\IrrG} (\id\otimes\omega_{\xi^k_\pi,\eta^k_\pi})U^\pi \md\mu(\pi)))\\
&=
\sum_{k=1}^{\infty}\bigl\|
\Lambda_\psi
(\int_{\IrrG} (\id\otimes\omega_{\xi^k_\pi,\eta^k_\pi})U^\pi \md\mu(\pi))\bigr\|^2=
\bigl\|
\bigoplus_{k=1}^{\infty}\Lambda_{\psi}( \int_{\IrrG} (\id\otimes\omega_{\xi^k_\pi,\eta^k_\pi})U^\pi \md\mu(\pi) )\bigr\|^2
<+\infty,
\end{split}\]
hence $\eta\in\Dom(q)$. Moreover
\[\begin{split} 
&\quad\;q(\eta(n)-\eta)\\
&=
\sum_{k=1}^{\infty} 
\psi(
(\int_{\IrrG} (\id\otimes\omega_{\xi^k_\pi,\eta^k_\pi(n)-\eta^k_\pi})U^\pi \md\mu(\pi))^*
(\int_{\IrrG} (\id\otimes\omega_{\xi^k_\pi,\eta^k_\pi(n)-\eta^k_\pi})U^\pi \md\mu(\pi)))\\
&=
\sum_{k=1}^{\infty}
\bigl\|
\Lambda_\psi(\int_{\IrrG} (\id\otimes\omega_{\xi^k_\pi,\eta^k_\pi(n)})U^\pi \md\mu(\pi))-
\Lambda_\psi(\int_{\IrrG} (\id\otimes\omega_{\xi^k_\pi,\eta^k_\pi})U^\pi \md\mu(\pi))
\bigr\|^2\\
&=
\bigl\|
\bigoplus_{k=1}^{\infty}\Lambda_\psi(\int_{\IrrG} (\id\otimes\omega_{\xi^k_\pi,\eta^k_\pi(n)})U^\pi \md\mu(\pi))-
\bigoplus_{k=1}^{\infty}
\Lambda_\psi(\int_{\IrrG} (\id\otimes\omega_{\xi^k_\pi,\eta^k_\pi})U^\pi \md\mu(\pi))
\bigr\|^2\\
&\quad\;\xrightarrow[n\to\infty]{}0,
\end{split}\]
which proves that $q$ is closed. It follows directly from the definition that $q$ is positive ($q(\eta,\eta)\ge 0$) and symmetric ($q(\eta,\eta')=\ov{q(\eta',\eta)}\; (\eta,\eta'\in\Dom(q))$). The form $q$ is densely defined, which means that $\Dom(q)$ is a dense subspace. Indeed, if $(\zeta_\pi)_{\pi\in\IrrG}$ is a square integrable vector field such that $\int_{\IrrG} \|E_\pi \zeta_\pi\|^2\md\mu(\pi)<+\infty$ then from the ortogonality relations (Proposition \ref{Casp}) we get
\[\begin{split}
&\quad\;\psi(
(\int_{\IrrG} (\id\otimes\omega_{\xi^k_\pi,\zeta_\pi})U^\pi \md\mu(\pi))^*
(\int_{\IrrG} (\id\otimes\omega_{\xi^k_\pi,\zeta_\pi})U^\pi \md\mu(\pi)))\\
&=
\int_{\IrrG} \|\xi^k_\pi\|^2 \|E_\pi \zeta_\pi\|^2\md\mu(\pi)<+\infty
 \end{split}\]
for each $k\in \NN$. Consequently, for such a vector field and any $k_0\in\NN$ the vector
\[
\bigoplus_{k\in\NN} \delta_{k,k_0}\int_{\IrrG}^{\oplus} \zeta_\pi\md\mu(\pi)
\]
belongs to the domain of $q$. It is clear that such vectors span a dense subspace, therefore $q$ is densely defined. We can use \cite[Theorem 2.23]{Kato} - there exists an (unbounded) positive self-adjoint operator $A$ on $\bigoplus_{k\in\NN}\int_{\IrrG}^\oplus \msf{H}_\pi\md\mu(\pi)$ such that
\[
\Dom(A)=\Dom(q),\quad q(\eta,\eta')=\ismaa{A\eta}{A\eta'}\quad(\eta,\eta'\in\Dom(q)).
\]
Consider now the positive self-adjoint unbounded operator $\bigoplus_{k\in\NN}\int_{\IrrG}^{\oplus}\|\xi^k_\pi\| E_\pi\md\mu(\pi)$. For $\eta$ in its domain, for each $k$ the vector fields
\[
(\xi^k_\pi)_{\pi\in\IrrG},\; (\eta^k_\pi)_{\pi\in\IrrG}\;,(\xi^k_\pi\otimes \ov{E_\pi\eta^k_\pi})_{\pi\in\IrrG},\;
(E_\pi\eta^k_\pi \supp_{\xi^k}(\pi))_{\pi\in\IrrG}
\]
are square integrable. Indeed, it is clear for the first two ones, and moreover
\[
\int_{\IrrG} \|\xi^k_\pi\otimes \ov{E_\pi\eta^k_\pi}\|^2\md\mu(\pi)=
\int_{\IrrG} \|\|\xi^k_\pi\|E_\pi \eta^k_\pi\|^2\md\mu(\pi)<+\infty
\]
because $\eta^k\in\Dom(\int_{\IrrG}^{\oplus}\|\xi^k_\pi\|E_\pi\md\mu(\pi))$ and
\[
\int_{\IrrG} \chi_{\supp \xi^k}(\pi)\|E_\pi\eta^k_\pi\|^2\md\mu(\pi)\le
\bigl(\sup_{\pi\in \supp\|\xi^k_{\bullet}\|} \|\xi^k_\pi\|^{-2}\bigr)
\int_{\IrrG} \|\xi^k_\pi\|^2\|E_\pi\eta^k_\pi\|^2\md\mu(\pi)<+\infty.
\]
Here we use that fact that $\|\xi^k_{\bullet}\|$ is bounded from below on its support. We can therefore make use of the orthogonality relations:
\[
\begin{split}
&\quad\;\bigl\|
\bigl(\bigoplus_{k\in\NN}\int_{\IrrG}^{\oplus}\|\xi^k_\pi\| E_\pi\md\mu(\pi)\bigr)\eta\bigr\|^2\\
&=
\sum_{k=1}^{\infty}
\int_{\IrrG} \|\xi^k_\pi\|^2 \|E_\pi \eta^k_\pi\|^2\md\mu(\pi)\\
&=
\sum_{k=1}^{\infty}
\psi(
(\int_{\IrrG} (\id\otimes\omega_{\xi^k_\pi,\eta^k_\pi})U^\pi \md\mu(\pi))^*
(\int_{\IrrG} (\id\otimes\omega_{\xi^k_\pi,{\eta}^k_\pi})U^\pi \md\mu(\pi)))\\
&=q(\eta,\eta)=\ismaa{A\eta}{A\eta}=\|A\eta\|^2,
\end{split}
\]
in particular $\Dom(\bigoplus_{k\in\NN}\int_{\IrrG}^{\oplus}\|\xi^k_\pi\| E_\pi\md\mu(\pi))\subseteq\Dom(q)=\Dom(A)$. Lemma \ref{lemat2} gives us the equality
\[
\bigoplus_{k\in\NN}\int_{\IrrG}^{\oplus}\|\xi^k_\pi\| E_\pi\md\mu(\pi)
= A.
\]
\end{proof}

We can treat the above result as a generalization of the orthogonality relations.\\

Now we will make use of this result in a specific situation. Take $\Omega\subseteq \IrrG$, a measurable subset such that $\int_\Omega \dim \md\mu<+\infty$ (we treat $\dim$ as a function $\IrrG\ni \pi\mapsto \dim(\pi)\in\NN$). Let $\{\xi^k\}_{k=1}^{\infty}$ be a measurable field of orthonormal bases. Define vector fields $\{\xi^{\Omega,k}\}_{k=1}^{\infty}$ via $\xi^{\Omega,k}_{\pi}=\chi_{\Omega}(\pi) \xi^{k}_{\pi}$. We have
\[
\sum_{k=1}^{\infty}\|\xi^{\Omega,k}\|^2=
\sum_{k=1}^{\infty} \int_{\IrrG} \|\xi^{\Omega,k}_\pi\|^2\md\mu(\pi)=
\sum_{k=1}^{\infty} \mu(\{\pi\in \IrrG\,|\,\dim(\pi)\ge k\})=
\int_\Omega\dim\md\mu<+\infty,
\]
therefore $\xi^{\Omega}=\bigoplus_{k\in\NN}\xi^{\Omega,k}\in\bigoplus_{k\in\NN} \int_{\IrrG}^{\oplus}\msf{H}_\pi\md\mu(\pi)$. Of course $\|\xi^{\Omega,k}_\pi\|=1$ on the support of $\xi^{\Omega,k}$, hence we can use Theorem \ref{tw5}: for $\eta=\bigoplus_{k=1}^{\infty}\eta^k\in\bigoplus_{k\in\NN} \int_{\IrrG}^{\oplus}\msf{H}_\pi\md\mu(\pi)$ we have
\[
\sum_{k=1}^{\infty}\psi(
(\int_{\IrrG} (\id\otimes\omega_{\xi^{\Omega,k}_\pi,\eta^k_\pi})U^\pi \md\mu(\pi))^*
(\int_{\IrrG} (\id\otimes\omega_{\xi^{\Omega,k}_\pi,\eta^k_\pi})U^\pi \md\mu(\pi)))<+\infty
\]
if and only if
\[
\sum_{k=1}^{\infty} \int_{\IrrG} \|\xi^{\Omega,k}_\pi\|^2 \|E_\pi \eta^k_\pi\|^2\md\mu(\pi)=
\int_{\Omega} \sum_{k=1}^{\dim(\pi)} \|E_\pi\eta^k_\pi\|^2 \md\mu(\pi)<+\infty
\]
(we have used the monotone convergence theorem). Moreover, these numbers are both equal to $q(\eta,\eta)$. In particular we can take $\eta=\xi^{\Omega}$. Then $\int_\Omega \Tr(E^2_\pi)\md\mu(\pi)<+\infty$ if and only if 
\[
\sum_{k=1}^{\infty}\psi(
(\int_{\IrrG} (\id\otimes\omega_{\xi^{\Omega,k}_\pi})U^\pi \md\mu(\pi))^*
(\int_{\IrrG} (\id\otimes\omega_{\xi^{\Omega,k}_\pi})U^\pi \md\mu(\pi)))<+\infty
\]
and we have arrived at the following proposition:

\begin{proposition}\label{stw6}
Let $\Omega$ be a measurable subset of $\IrrG$ such that $\int_\Omega \dim\md\mu<+\infty$. Let $\{(\xi^k_\pi)_{\pi\in\IrrG}\}_{k=1}^{\infty}$ be a measurable field of orthonormal bases. Define vector fields $\{\xi^{\Omega,k}\}_{k=1}^{\infty}$ via $\xi^{\Omega,k}_\pi=\chi_{\Omega}(\pi) \xi^{k}_{\pi}$. We have
\[
\int_\Omega \Tr(E^2_\pi)\md\mu(\pi)=
\sum_{k=1}^{\infty}\psi(
(\int_{\IrrG} (\id\otimes\omega_{\xi^{\Omega,k}_\pi})U^\pi \md\mu(\pi))^*
(\int_{\IrrG} (\id\otimes\omega_{\xi^{\Omega,k}_\pi})U^\pi \md\mu(\pi))),
\]
the above numbers can be equal to $+\infty$.
\end{proposition}

\subsection{Equivalence of $\chi^{\int}(\int_\Omega^{\oplus}\pi\md\mu(\pi))\in\mf{N}_\psi$ and $\int_\Omega \Tr(E^2_{\bullet})\md\mu<+\infty$}

Recall that $\mc{Q}_R$ is the unitary operator given by Theorem \ref{PlancherelR}. The first result which we wish to prove in this subsection says that $\Lambda_{\psi}(\chi^{\int}(\int_{\Omega}^{\oplus}\pi\md\mu(\pi)))$ gets mapped to $\int_{\IrrG}^{\oplus}\chi_\Omega(\pi) E_\pi\md\mu(\pi)$ by $\mc{Q}_R$. Similarly to the previous subsection we will first derive more general result. Let $\bigoplus_{k=1}^{\infty} \int_{\IrrG}^{\oplus} \xi^k_\pi\md\mu(\pi)$, $\bigoplus_{k=1}^{\infty} \int_{\IrrG}^{\oplus} \eta^k_\pi\md\mu(\pi)$ be vectors in $\bigoplus_{k=1}^{\infty}\int_{\IrrG}^{\oplus}
\msf{H}_\pi\md\mu(\pi)$. In particular this means that
\[
\sum_{k=1}^{\infty} \int_{\IrrG} \|\xi^k_\pi\|^2\md\mu(\pi)=
\int_{\bigsqcup_{k=1}^{\infty}\IrrG} \|\xi^k_\pi\|^2\md\mu^{\sqcup}(k,\pi)<+\infty,
\]
in other words, the function
\[
\bigsqcup_{k=1}^{\infty}\IrrG\ni (k,\pi)\mapsto
\|\xi^k_\pi\|\in \CC
\]
is square integrable (with respect to the natural measure $\mu^{\sqcup}$ on the disjoint union). The same holds for $\eta$. We can calculate the scalar product:
\[
\ismaa{\|\xi^{\bullet}_{\bullet}\|}{
\|\eta^{\bullet}_{\bullet}\|}=
\int_{\bigsqcup_{k=1}^{\infty} \IrrG}
\|\xi^k_\pi\| \|\eta^k_\pi\|\md\mu^{\sqcup}(k,\pi)=
\sum_{k=1}^{\infty} \int_{\IrrG}
\|\xi^k_\pi\| \|\eta^k_\pi\|\md\mu(\pi)<+\infty
\]
It follows that the operator $\sum_{k=1}^{\infty}\int_{\IrrG}(\id\otimes\omega_{\xi^{k}_{\pi},\eta^{k}_\pi})U^{\pi}\md\mu(\pi)$ is well defined and belongs to $\Linf$: for each $k\in\NN$ we have
\[
\int_{\IrrG} \| (\id\otimes\omega_{\xi^{k}_{\pi},\eta^{k}_\pi})U^{\pi} \|\md\mu(\pi)\le
\int_{\IrrG} \|\xi^k_\pi\|\|\eta^k_\pi\|\md\mu(\pi)<+\infty,
\]
and
\[
\sum_{k=1}^{\infty} \bigl\|
\int_{\IrrG}(\id\otimes\omega_{\xi^{k}_{\pi},\eta^{k}_\pi})U^{\pi}
\md\mu(\pi) \bigr\|\le
\sum_{k=1}^{\infty}
\int_{\IrrG} \|\xi^k_\pi\|\|\eta^k_\pi\|\md\mu(\pi)<+\infty.
\]
Moreover, thanks to the monotone convergence theorem we have
\[
\sum_{k=1}^{\infty}
\int_{\IrrG} \|\xi^k_\pi\|\|\eta^k_\pi\|\md\mu(\pi)=
\int_{\IrrG} \sum_{i=1}^{\infty}\|\xi^k_\pi\|\|\eta^k_\pi\|\md\mu(\pi)<+\infty,
\]
consequently $\sum_{i=1}^{\infty} \|\xi^k_\pi\|\|\eta^k_\pi\|<+\infty$ for almost all $\pi$ and the operator
\[
\sum_{k=1}^{\infty} |\eta^k_\pi\rangle\langle \xi^k_\pi|
\in\B(\msf{H}_\pi)
\]
is well defined for almost all $\pi\in\IrrG$. 

\begin{proposition}\label{stw8}
Let $\bigoplus_{k=1}^{\infty} \int_{\IrrG}^{\oplus} \xi^k_\pi\md\mu(\pi)$, $\bigoplus_{k=1}^{\infty} \int_{\IrrG}^{\oplus} \eta^k_\pi\md\mu(\pi)$ be vectors in \\$\bigoplus_{k=1}^{\infty}\int_{\IrrG}^{\oplus}
\msf{H}_\pi\md\mu(\pi)$. Assume that $\sum_{k=1}^{\infty}\int_{\IrrG}(\id\otimes\omega_{\xi^{k}_{\pi},\eta^{k}_\pi})U^{\pi}\md\mu(\pi)\in\mf{N}_\psi$. Then, the vector field $(\sum_{k=1}^{\infty} | \xi^k_\pi\rangle\langle \eta^k_\pi| E_\pi)_{\pi\in\IrrG}$ is measurable, square integrable and we have equality in $\int_{\IrrG}^{\oplus}\HS(\msf{H}_\pi)\md\mu(\pi)$:
\[
\mc{Q}_R\, \Lambda_{\psi}(
\sum_{k=1}^{\infty}\int_{\IrrG}(\id\otimes\omega_{\xi^{k}_{\pi},\eta^{k}_\pi})U^{\pi}\md\mu(\pi))=
\int_{\IrrG}^{\oplus} 
\sum_{k=1}^{\infty} | \eta^k_\pi\rangle\langle \xi^k_\pi| E_\pi\md\mu(\pi).
\]
\end{proposition}

\begin{proof}
We have
\[
\mc{Q}_R\, \Lambda_{\psi}(
\sum_{k=1}^{\infty}\int_{\IrrG}(\id\otimes\omega_{\xi^{k}_{\pi},\eta^{k}_\pi})U^{\pi}\md\mu(\pi))=
\int_{\IrrG}^{\oplus} T_\pi\md\mu(\pi)
\]
for certain operators $T_\pi\in \HS(\msf{H}_\pi)$. We want to show that $T_\pi=\sum_{k=1}^{\infty} | \eta^k_\pi\rangle\langle \xi^k_\pi| E_\pi$ for almost all $\pi$. For $\omega\in(\ov{\mc{I}_R}\cap\Ljsharp)^{\sharp}$ we have
\begin{equation}\begin{split} \label{eq19}
&\quad\;
\bigl\langle
\int_{\IrrG}^{\oplus}T_\pi\md\mu(\pi)\big|
\int_{\IrrG}^{\oplus} (\omega\otimes\id)U^\pi E_\pi^{-1}\md\mu(\pi)
\bigr\rangle\\
&=
\bigl\langle
\mc{Q}_R \Lambda_{\psi}(
\sum_{k=1}^{\infty}\int_{\IrrG}(\id\otimes\omega_{\xi^{k}_{\pi},\eta^{k}_\pi})U^{\pi}\md\mu(\pi)) \big|
\mc{Q}_R\xi_R(\ov{\omega^{\sharp}})\bigr\rangle\\
&=\ov{\omega^{\sharp}}(( 
\sum_{k=1}^{\infty}\int_{\IrrG}(\id\otimes\omega_{\xi^{k}_{\pi},\eta^{k}_\pi})U^{\pi}\md\mu(\pi))^*)\\
&=
\ov{\sum_{k=1}^{\infty}\int_{\IrrG}
(\omega^{\sharp}\otimes\omega_{\xi^k_\pi,\eta^k_\pi})U^\pi\md\mu(\pi)}\\
&=
\sum_{k=1}^{\infty}\int_{\IrrG}
(\omega\otimes\omega_{\eta^k_\pi,\xi^k_\pi})U^\pi\md\mu(\pi)
\end{split}
\end{equation}
(we have used Lemma \ref{lemat25} to get $\xi_R(\ov{\omega^{\sharp}})=\hat{J}J\Lambda_{\hpsi}(\lambda(\omega))$). We know that $\int_{\IrrG}^{\oplus} \pi \md\mu(\pi)(\CGDu)''$ is the von Neumann algebra of decomposable operators on $\int_{\IrrG}^{\oplus}\msf{H}_\pi\md\mu(\pi)$, i.e.~operators of the form $\int_{\IrrG}^{\oplus} T_\pi\md\mu(\pi)$. In particular, arbitrary decomposable operator can be approximated in \ssot by operators of the form $\int_{\IrrG}^{\oplus} \pi(a) \md\mu(\pi)$ with $a\in \CGDu$. Thanks to the Kaplansky theorem we can approximate with bounded nets. Next, the set of functionals $\omega$ as above is norm dense in $\Lj$ (Lemma \ref{lemat28}), therefore we can approximate any decomposable operator by operators of the form $\int_{\IrrG}^{\oplus} \pi (\lambda^u(\omega))\md\mu(\pi)$. Since the Hilbert space $\int_{\IrrG}^{\oplus}\msf{H}_\pi\md\mu(\pi)$ is separable, the $\sigma$-strong operator topology is metrizable on bounded subsets (\cite[Proposition I.6.3]{Davidson}). Therefore we can approximate using bounded sequences.\\
Choose arbitrary measurable subset $V\subseteq\IrrG$ of finite measure and a measurable family of operators $(K_\pi)_{\pi\in V}$ on $(\msf{H}_\pi)_{\pi\in\IrrG}$ such that
\[
\sup_{\pi'\in V}( \dim(\pi') \|T_{\pi'}\| \|E_{\pi'}^{-1}\|)<+\infty,
\quad
\sup_{\pi\in V}\|K_\pi\|<+\infty.
\]
Reasoning in the previous paragraph implies that we can find a sequence $(\omega_n)_{n\in\NN}$ in\\ $(\ov{\mc{I}_R}\cap\Ljsharp)^{\sharp}$ such that
\[
\sup_{n\in\NN} \bigl\|
\int_{\IrrG}^{\oplus} (\omega_n\otimes\id)U^\pi \md\mu(\pi)
\bigr\|=
\sup_{n\in\NN}\sup_{\pi\in\IrrG}\|
(\omega_n\otimes\id)U^\pi\|<+\infty
\]
and
\[
\int_{\IrrG}^{\oplus} (\omega_n\otimes\id)U^\pi \md\mu(\pi)
\xrightarrow[n\to\infty]{\ssot}
\int_{\IrrG}^{\oplus} \chi_V(\pi) K_\pi\md\mu(\pi).
\]
Next, we can find a subsequence $(n_p)_{p\in\NN}$ such that
\begin{equation}\label{eq18}
(\omega_{n_p}\otimes\id)U^{\pi}
\xrightarrow[p\to\infty]{\ssot}
\chi_V(\pi)K_\pi
\end{equation}
for almost all $\pi\in \IrrG$ (\cite[Proposition 4, page 183]{DixmiervNA}). Since $\sum_{k=1}^{\infty} \| \int_{\IrrG}^{\oplus}\xi^k_\pi\md\mu(\pi)\|^2<+\infty$ and similarly for $\eta$, the functional $\sum_{k=1}^{\infty} \ismaa{\int_{\IrrG}^{\oplus}\eta^k_\pi\md\mu(\pi)}{\cdot\;\int_{\IrrG}^{\oplus}\xi^k_\pi\md\mu(\pi)}$ is well defined and normal. Consequently, due to equation \eqref{eq19} we have
\[\begin{split} 
&\quad\;\sum_{k=1}^{\infty}\int_{V} \ismaa{\eta^k_\pi}{K_\pi \xi^k_\pi}\md\mu(\pi)\\
&=\sum_{k=1}^{\infty} 
\bigl\langle\int_{\IrrG}^{\oplus} \eta^k_\pi \md\mu(\pi)\big|
\int_{\IrrG}^{\oplus} \chi_V(\pi) K_\pi\md\mu(\pi)
\int_{\IrrG}^{\oplus} \xi^k_\pi \md\mu(\pi)\bigr\rangle\\
&=
\lim_{p\to\infty}
\sum_{k=1}^{\infty} 
\bigl\langle\int_{\IrrG}^{\oplus} \eta^k_\pi \md\mu(\pi)\big|
\int_{\IrrG}^{\oplus} (\omega_{n_p}\otimes\id)U^\pi \md\mu(\pi)
\int_{\IrrG}^{\oplus} \xi^k_\pi \md\mu(\pi)\bigr\rangle\\
&=
\lim_{p\to\infty}
\sum_{k=1}^{\infty} 
\int_{\IrrG} (\omega_{n_p}\otimes\omega_{\eta^k_\pi,\xi^k_\pi} )U^\pi \md\mu(\pi)\\
&=
\lim_{p\to\infty}
\bigl\langle\int_{\IrrG}^{\oplus}T_\pi\md\mu(\pi)\big|
\int_{\IrrG}^{\oplus} (\omega_{n_p}\otimes\id)U^{\pi}
E_{\pi}^{-1} \md\mu(\pi)\bigr\rangle\\
&=
\lim_{p\to\infty} \int_{\IrrG} \Tr( T_\pi^* (\omega_{n_p}\otimes\id)U^{\pi} E_\pi^{-1})\md\mu(\pi).
\end{split}\]
Since
\[
\sum_{k=1}^{\infty} \int_V \ismaa{\eta^k_\pi}{K_\pi \xi^k_\pi} \md\mu(\pi)=
\int_V \sum_{k=1}^{\infty} \ismaa{\eta^k_\pi}{K_\pi \xi^k_\pi} \md\mu(\pi)
\]
and
\[
\begin{split}
\sum_{k=1}^{\infty} \ismaa{\eta^k_\pi}{K_\pi \xi^k_\pi}&=
\sum_{k=1}^{\infty}\sum_{j=1}^{\dim(\pi)} \ismaa{\eta^k_\pi}{K_\pi\zeta^j_\pi}
\ismaa{\zeta^j_\pi}{ \xi^k_\pi}
=
\sum_{j=1}^{\dim(\pi)}
\bigl\langle\zeta^j_\pi\big|
\sum_{k=1}^{\infty} |\xi^k_\pi\rangle\langle \eta^k_\pi| K_\pi\; \zeta^j_\pi\bigr\rangle\\
&=
\Tr\bigl( \sum_{k=1}^{\infty} |\xi^k_\pi\rangle\langle \eta^k_\pi| K_\pi \bigr),
\end{split}
\]
(for any orthonormal basis $\{\zeta^j_\pi\}_{j=1}^{\dim(\pi)}$ in $\msf{H}_\pi$) we get
\[\begin{split}
\int_V \Tr( (\sum_{k=1}^{\infty} | \xi^k_\pi
\rangle\langle \eta^k_\pi | ) K_\pi)=
\lim_{p\to\infty}
\int_V \Tr( T_\pi^* (\omega_{n_p}\otimes\id)U^\pi E_\pi^{-1})
\md\mu(\pi).
 \end{split}\]
Observe that \eqref{eq18} implies that for almost all $\pi\in V$ we have
\[
\lim_{p\to\infty} \Tr( T_\pi^* (\omega_{n_p}\otimes\id)U^\pi E_\pi^{-1})=
\Tr( T_\pi^* K_\pi E_\pi^{-1}).
\]
Moreover
\[
|\Tr( T_\pi^* (\id\otimes\omega_{n_p})U^\pi E_\pi^{-1})|\le
\sup_{\pi'\in V}( \dim(\pi') \|T_{\pi'}\| \|E_{\pi'}^{-1}\|)\;
\sup_{n\in\NN}\sup_{\pi'\in\IrrG}\|(\id\otimes\omega_n)U^{\pi'}\|<+\infty
\]
for all $p\in \NN$ and almost all $\pi\in V$. Since $\mu(V)<+\infty$, the above function is integrable and we can make use of the dominated convergence theorem:
\[
\int_V \Tr( (\sum_{k=1}^{\infty} | \xi^k_\pi
\rangle\langle \eta^k_\pi | ) K_\pi)=
\int_V \lim_{p\to\infty}\Tr( T_\pi^* (\omega_{n_p}\otimes\id)U^\pi E_\pi^{-1})
\md\mu(\pi)=
\int_V \Tr( T_\pi^* K_\pi E_\pi^{-1})
\md\mu(\pi).
\]
Subset $V$ and operators $(K_\pi)_{\pi\in\IrrG}$ are arbitrary (within the made assumptions), hence
\[
\sum_{k=1}^{\infty} | \xi^k_\pi
\rangle\langle \eta^k_\pi |=E_\pi^{-1}T_\pi^* 
\]
and
\[
T_\pi=
\sum_{k=1}^{\infty} | \eta^k_\pi
\rangle\langle \xi^k_\pi | E_\pi
\]
for almost all $\pi\in\IrrG$.
\end{proof}

As a consequence we get a result for integral characters:

\begin{proposition}\label{stw9}
Let $\Omega\subseteq\IrrG$ be such a subset that $\int_\Omega\dim\md\mu<+\infty$, assume moreover that the integral character $\chi^{\int}(\int_{\Omega}^{\oplus}\pi\md\mu(\pi))$ is in $\mf{N}_\psi$. Then
\[
\mc{Q}_R \Lambda_{\psi}(
\chi^{\int}(\int_{\Omega}^{\oplus}\pi\md\mu(\pi)))=
\int_{\IrrG}^{\oplus} \chi_{\Omega}(\pi)E_\pi \md\mu(\pi)
\]
and consequently
\[
\|\Lambda_{\psi}(
\chi^{\int}(\int_{\Omega}^{\oplus}\pi\md\mu(\pi)))\|^2=
\int_{\Omega}\Tr(E^2_\pi)\md\mu(\pi).
\]
\end{proposition}

\begin{proof}
Let $\{(\zeta^k_\pi)_{\pi\in\IrrG}\}_{k=1}^{\infty}$ be any field of orthonormal bases, define $\xi^k_\pi=\eta^k_\pi=\chi_{\Omega}(\pi) \zeta^k_\pi$. We are in the situation from the previous proposition: for any $k\in\NN$ we have
\[
\int_{\IrrG} \|\xi^k_\pi\|^2\md\mu(\pi)=
\mu(\{\pi\in\Omega\,|\,\dim(\pi)\le k\})<+\infty,
\]
and
\[
\sum_{k=1}^{\infty} \bigl\|\int_{\IrrG}^{\oplus}\xi^k_\pi\md\mu(\pi)
\bigr\|^2=
\sum_{k=1}^{\infty} \mu(\{\pi\in\Omega\,|\,
\dim(\pi)\le k\})=
\int_\Omega\dim\md\mu<+\infty,
\]
which gives us $\bigoplus_{k=1}^{\infty} \int_{\IrrG}^{\oplus}\xi^k_\pi\md\mu(\pi)\in\bigoplus_{k=1}^{\infty} \int_{\IrrG}^{\oplus}\msf{H}_\pi\md\mu(\pi)$. The operator from Proposition \ref{stw8} is
\[
\sum_{k=1}^{\infty} \int_{\IrrG} (\id\otimes\omega_{\xi^k_\pi})U^{\pi}\md\mu(\pi)=
\int_\Omega (\id\otimes\Tr)U^{\pi}\md\mu(\pi)=
\chi^{\int}(\int_{\Omega}^{\oplus}\pi\md\mu(\pi))
\]
which we assume to be in $\mf{N}_\psi$. It follows that Proposition \ref{stw8} gives us
\[
\begin{split}
&\quad\;
\mc{Q}_R \Lambda_{\psi}(
\chi^{\int}(\int_{\Omega}^{\oplus}\pi\md\mu(\pi)))=
\mc{Q}_R\, \Lambda_{\psi}(
\sum_{i=1}^{\infty}\int_{\IrrG}(\id\otimes\omega_{\xi^{k}_{\pi}})U^{\pi}\md\mu(\pi))\\
&=
\int_{\IrrG}^{\oplus} 
\sum_{k=1}^{\infty} | \xi^k_\pi\rangle\langle \xi^k_\pi| E_\pi\md\mu(\pi)=
\int_{\IrrG}^{\oplus} \chi_{\Omega}(\pi)E_\pi \md\mu(\pi).
\end{split}
\]
In particular
\[
\|\Lambda_{\psi}(
\chi^{\int}(\int_{\Omega}^{\oplus}\pi\md\mu(\pi)))\|^2=
\int_{\Omega}\Tr(E^2_\pi)\md\mu(\pi).
\]
\end{proof}

In the previous proposition we have assumed $\chi^{\int}(\int_{\Omega}^{\oplus}\pi\md\mu(\pi))\in\mf{N}_\psi$ and ended up with the conclusion that $\int_\Omega\Tr(E^2_\bullet)<+\infty$. The next proposition tells us in particular that the reverse implication also holds.

\begin{proposition}\label{stw11}
Let $\Omega\subseteq \IrrG$ be a measurable subset such that $\int_\Omega \dim \md\mu<+\infty$. Let $\{\{\xi^n_\pi\}_{\pi\in \IrrG}\}_{n=1}^{\infty}$ be a measurable field of orthonormal bases and let $\{V_p\}_{p\in\NN}$ be an increasing family of subsets of $\IrrG$ such that $\bigcup_{p\in\NN}V_p=\IrrG$ and $\lim_{p\to\infty} p \mu(\Omega\setminus V_p)=0$. Then 
\[
\chi^{\int}(\int_\Omega^{\oplus} \pi \md\mu_\Omega(\pi))=
\sum_{n=1}^{\infty}
\chi^{\int}(\int_{\Omega\rest_n}^{\oplus} \pi \md\mu_{\Omega\rest_n}(\pi))=
\lim_{p\to\infty}\sum_{n=1}^{p}
\int_{\Omega}\chi_{V_p}(\pi) (\id\otimes\omega_{\xi^n_\pi})U^\pi \md\mu_{\Omega}(\pi).
\]
Moreover, if $\int_{\Omega} \Tr(E^2_{\bullet})\md\mu<+\infty$ then also
\[
\chi^{\int}(\int_\Omega^{\oplus} \pi \md\mu_\Omega(\pi)),\quad
\chi^{\int}(\int_{\Omega\rest_n}^{\oplus} \pi \md\mu_{\Omega\rest_n}(\pi)),
\quad
\int_\Omega \chi_{V_p}(\pi)(\id\otimes\omega_{\xi^n_\pi})U^\pi \md\mu_\Omega(\pi)
\in\mf{N}_\psi\quad(n,p\in\NN)
\]
and
\[\begin{split}
\Lambda_{\psi}(\chi^{\int}(\int_\Omega^{\oplus} \pi \md\mu_\Omega(\pi)))&=
\sum_{n=1}^{\infty} \Lambda_{\psi}(\chi^{\int}(\int_{\Omega\rest_n}^{\oplus} \pi \md\mu_{\Omega\rest_n}(\pi)))\\
&=
\lim_{p\to\infty}\sum_{n=1}^{p} \Lambda_{\psi}(\int_\Omega
\chi_{V_p}(\pi) (\id\otimes\omega_{\xi^n_\pi})U^\pi \md\mu_\Omega(\pi)).
\end{split}\]
The above sums and limits are norm-convergent.
\end{proposition}

\begin{proof}
For any $N\in\NN$ we have
\[
\begin{split}
&\quad\;\sum_{n=1}^{N}
\chi^{\int}(\int_{\Omega\rest_n}^{\oplus} \pi \md\mu_{\Omega\rest_n}(\pi))\\
&=
 \sum_{n,m=1}^{N}
\int_{\Omega\rest_n} (\id\otimes\omega_{\xi^m_\pi})U^\pi \md\mu_{\Omega\rest_n}(\pi)\\
&=
 \sum_{m=1}^{N}
(\int_{\Omega}(\id\otimes\omega_{\xi^m_\pi})U^\pi \md\mu_{\Omega}(\pi)-
\sum_{n=N+1}^{\infty}\int_{\Omega\rest_n} (\id\otimes\omega_{\xi^m_\pi})U^\pi \md\mu_{\Omega\rest_n}(\pi))\\
&=
 \sum_{m=1}^{N}
(\int_{\Omega}\chi_{V_N}(\pi)(\id\otimes\omega_{\xi^m_\pi})U^\pi \md\mu_{\Omega}(\pi)-
\sum_{n=N+1}^{\infty}
\int_{\Omega}(\id\otimes\omega_{\xi^m_\pi})U^\pi \md\mu_{\Omega}(\pi)\\
&\quad\quad\quad\quad\quad\quad
\quad\quad\quad\quad\quad\quad\quad\quad\quad\quad\quad+
\int_{\Omega}\chi_{\IrrG\setminus V_N}(\pi) (\id\otimes\omega_{\xi^m_\pi})U^\pi \md\mu_{\Omega}(\pi)).
\end{split}
\]
The two last terms disappear in the limit $N\to\infty$:
\[
\begin{split}
&\bigl\|\sum_{m=1}^{N}
\sum_{n=N+1}^{\infty}\int_{\Omega\rest_n} (\id\otimes\omega_{\xi^m_\pi})U^\pi \md\mu_{\Omega\rest_n}(\pi)\bigr\|\le
\sum_{n=N+1}^{\infty} N \mu(\Omega\rest_n)\le
\sum_{n=N+1}^{\infty} n \mu(\Omega\rest_n)\xrightarrow[N\to\infty]{}0,\\
&\bigl\|
\sum_{m=1}^{N}
\int_{\Omega}\chi_{\IrrG\setminus V_N}(\pi) (\id\otimes\omega_{\xi^m_\pi})U^\pi \md\mu_{\Omega}(\pi)\bigr\|\le
N\mu(\Omega\setminus V_N)\xrightarrow[N\to\infty]{}0.
\end{split}
\]
Consequently, we get the convergence in norm:
\[
\chi^{\int}(\int_\Omega^{\oplus} \pi \md\mu_\Omega(\pi))=
\lim_{N\to\infty} \sum_{n=1}^{N}
\chi^{\int}(\int_{\Omega\rest_n}^{\oplus} \pi \md\mu_{\Omega\rest_n}(\pi))=
\lim_{N\to\infty} \sum_{n=1}^{N}
\int_{\Omega}\chi_{V_N}(\pi) (\id\otimes\omega_{\xi^n_\pi})U^\pi \md\mu_{\Omega}(\pi),
\]
(the first equality follows from the definition of the integral character). Due to Proposition \ref{Casp} we know that the integrals 
\[
\int_\Omega \chi_{V_N}(\pi)(\id\otimes\omega_{\xi^n_\pi})U^\pi \md\mu_\Omega(\pi),\quad
\chi^{\int}(\int_{\Omega\rest_n}^{\oplus}\pi\md\mu_{\Omega\rest_n}(\pi))
\]
belong to $\mf{N}_\psi$ (here we use $\int_\Omega \Tr(E^2_{\bullet})\md\mu<+\infty$). For $N>N'\in\NN$ we have
\[
\begin{split}
&\quad\;\bigl\|\sum_{n=1}^{N}\Lambda_{\psi}(
\int_{\Omega}\chi_{V_N}(\pi)
 (\id\otimes\omega_{\xi^n_\pi})U^\pi \md\mu_{\Omega}(\pi))-
\sum_{n=1}^{N'}\Lambda_{\psi}(
\int_{\Omega}\chi_{V_{N'}}(\pi)
 (\id\otimes\omega_{\xi^n_\pi})U^\pi \md\mu_{\Omega}(\pi))\bigr\|^2\\
&=
\bigl\|
\sum_{n=N'+1}^{N}
\Lambda_{\psi}(
\int_{\Omega}\chi_{V_{N'}}(\pi)
 (\id\otimes\omega_{\xi^n_\pi})U^\pi \md\mu_{\Omega}(\pi))+
\sum_{n=1}^{N}
\Lambda_{\psi}(
\int_{\Omega}\chi_{V_N\setminus V_{N'}}(\pi)
 (\id\otimes\omega_{\xi^n_\pi})U^\pi \md\mu_{\Omega}(\pi))\bigr\|^2\\
&=
\sum_{n,n'=N'+1}^{N} \int_{\Omega} \chi_{V_{N'}}(\pi)\ismaa{\xi^n_\pi}{\xi^{n'}_\pi}
\ismaa{E_\pi\xi^n_\pi}{E_\pi\xi^{n'}_\pi}\md\mu_\Omega(\pi)\\
&\quad\quad\quad\quad\quad\quad+
\sum_{n,n'=1}^{N} \int_{\Omega} \chi_{V_N\setminus V_{N'}}(\pi)\ismaa{\xi^n_\pi}{\xi^{n'}_\pi}
\ismaa{E_\pi\xi^n_\pi}{E_\pi\xi^{n'}_\pi}\md\mu_\Omega(\pi)\\
&=
\sum_{n=N'+1}^{N} \int_\Omega \chi_{V_{N'}}(\pi) \ismaa{\xi^n_\pi}{E_\pi^2 \xi^n_\pi}\md\mu_\Omega(\pi)+
\sum_{n=1}^{N} \int_\Omega \chi_{V_N\setminus V_{N'}}(\pi)\ismaa{\xi^n_\pi}{E_\pi^2 \xi^n_\pi}\md\mu_\Omega(\pi)\\
&\le
\sum_{n=N'+1}^{N}\sum_{m=N'+1}^{\infty} \int_{\Omega\rest_m} 
\ismaa{\xi^n_\pi}{E_\pi^2 \xi^n_\pi}\md\mu(\pi)+
\int_{\Omega\setminus V_{N'}} \Tr(E^2_\pi)\md\mu(\pi)
\\
&\le
\sum_{m=N'+1}^{\infty} \int_{\Omega\rest_m} \Tr(E_\pi^2 )\md\mu_{\Omega}(\pi)+\int_{\Omega\setminus V_{N'}}\Tr(E^2_\pi)\md\mu(\pi)
\xrightarrow[N'\to\infty]{} 0.
\end{split}
\]
As the Hilbert space $\LL^2(\GG)$ is complete, we get the existence of the norm limit
\[
\lim_{N\to\infty} \sum_{n=1}^{N}\Lambda_{\psi}(\int_{\Omega} \chi_{V_N}(\pi)(\id\otimes\omega_{\xi^n_\pi})U^\pi \md\mu_{\Omega}(\pi)).
\]
The $\ssot\times \|\cdot\|$ closedness of $\Lambda_\psi$ implies that $\chi^{\int}(\int_{\Omega}^{\oplus}\pi\md\mu_{\Omega}(\pi))\in\mf{N}_{\psi}$ and
\[
\lim_{N\to\infty} \sum_{n=1}^{N}\Lambda_{\psi}(\int_{\Omega} \chi_{V_N}(\pi)(\id\otimes\omega_{\xi^n_\pi})U^\pi \md\mu_{\Omega}(\pi))=\Lambda_{\psi}(
\chi^{\int}(\int_{\Omega}^{\oplus} \pi\md\mu_{\Omega}(\pi))).
\]
We check the last equality of the statement in a similar fashion: convergence
\[
\begin{split}
&\quad\;\bigl\| \sum_{n=N'}^{N} \Lambda_{\psi}(\chi^{\int} (\int_{\Omega\rest_n}^{\oplus}
\pi\md\mu_{\Omega\rest_n}(\pi)))\bigr\|^2
=\bigl\|\sum_{n=N'}^{N}\sum_{m=1}^{n}
\Lambda_{\psi}(\int_{\Omega\rest_n} (\id\otimes\omega_{\xi^m_\pi})U^\pi \md\mu_{\Omega}(\pi))\bigr\|^2\\
&=
\sum_{n,n'=N'}^{N}\sum_{m=1}^{n}\sum_{m'=1}^{n'} \int_{\IrrG} 
\chi_{\Omega\rest_n\cap \Omega\rest_{n'}}(\pi)\ismaa{\xi^m_\pi}{\xi^{m'}_\pi}
\ismaa{E_\pi\xi^m_\pi}{E_\pi\xi^{m'}_\pi}\md\mu(\pi)\\
&=
\sum_{n=N'}^{N}\sum_{m=1}^{n} \int_{\Omega\rest_n} \ismaa{\xi^m_\pi}{E_\pi^2 \xi^m_\pi}\md\mu_{\Omega\rest_n}(\pi)=
\sum_{n=N'}^{N}\int_{\Omega\rest_n} \Tr(E^2_\pi)\md\mu_{\Omega\rest_n}(\pi)\xrightarrow[N',N\to\infty]{} 0.
\end{split}
\]
and $\ssot\times \|\cdot\|$ closedness of $\Lambda_\psi$ give us
\[
\lim_{N\to\infty} \sum_{n=1}^{N}\Lambda_{\psi}(
\chi^{\int}(\int_{\Omega\rest_n}^{\oplus} \pi\md\mu_{\Omega\rest_n}(\pi))
)=\Lambda_{\psi}(
\chi^{\int}(\int_{\Omega}^{\oplus} \pi\md\mu_{\Omega}(\pi))).
\]
\end{proof}

The next result concerns integral characters related to the tensor product $\kappa\tp\sigma_\Omega$.

\begin{lemma}\label{lemat9}
Let $\kappa\colon\CGDu\rightarrow\B(\msf{H}_\kappa)$ be a nondegenerate representation such that $\kappa\lec \Lambda_{\whG}$ and $\dim(\kappa)<+\infty$, let $\Omega\subseteq \IrrG$ be a measurable subset such that $\int_\Omega \dim \md\mu<+\infty$ and $\sup_{\pi\in\Omega}\|E^2_\pi\|<+\infty$. Let $\tilde{\mu}=(\varpi^{\kappa,\Omega,\mu}+\chi_{\IrrG\setminus \F^{1}_{\kappa\stp\sigma_\Omega}})\mu$ be an equivalent Plancherel measure. Then
\begin{equation}\label{eq10}
\dim(\kappa)\int_{\Omega}\dim \md\mu=
\sum_{i=1}^{\infty} \int_{\F^{i}_{\kappa\stp\sigma_\Omega}} 
\dim \md\tilde{\mu}<+\infty
\end{equation}
and
\[
\begin{split}
 \chi(U^\kappa)\chi^{\int} ( \int_{\Omega}^{\oplus} \pi \md\mu_\Omega(\pi))&=
\chi^{\int} ( \int_{\Omega}^{\oplus} \kappa\tp\pi \md\mu_\Omega(\pi))=
\chi^{\int}( \int_{\bigsqcup_{i=1}^{\infty} \F^{i}_{\kappa\stp\sigma_\Omega} }^{\oplus} \zeta \md\tilde{\mu}_{\bigsqcup_{i=1}^{\infty} 
\F^{i}_{\kappa\stp\sigma_\Omega} }(\zeta))\\
&=\sum_{i=1}^{\infty}\chi^{\int}( \int_{\F^{i}_{\kappa\stp\sigma_\Omega} }^{\oplus} \zeta \md\tilde{\mu}_{\F^{i}_{\kappa\stp\sigma_\Omega} }(\zeta)).
\end{split}
\]
Moreover
\[
\sum_{i'=1}^{\infty}\chi^{\int}( \int_{\F^{i'}_{\kappa\stp\sigma_\Omega} }^{\oplus} \zeta \md\tilde{\mu}_{\F^{i'}_{\kappa\stp\sigma_\Omega} }(\zeta)),\quad
\chi^{\int}( \int_{\F^{i}_{\kappa\stp\sigma_\Omega} }^{\oplus} \zeta \md\tilde{\mu}_{\F^{i}_{\kappa\stp\sigma_\Omega} }(\zeta))\in\mf{N}_\psi
\quad(i\in\NN)
\]
and
\[
\Lambda_{\psi}\bigl(
\sum_{i=1}^{\infty}\chi^{\int}( \int_{\F^{i}_{\kappa\stp\sigma_\Omega} }^{\oplus} \zeta \md\tilde{\mu}_{\F^{i}_{\kappa\stp\sigma_\Omega} }(\zeta))
\bigr)=
\sum_{i=1}^{\infty}\Lambda_{\psi}\bigl(\chi^{\int}( \int_{\F^{i}_{\kappa\stp\sigma_\Omega} }^{\oplus} \zeta \md\tilde{\mu}_{\F^{i}_{\kappa\stp\sigma_\Omega} }(\zeta))\bigr).
\]
The above series are norm convergent. 
\end{lemma}

\begin{proof}
The first equality was proven in Proposition \ref{treq}, the second one follows from the equality of the integral weights (after composition with an appropriate unitary operator) and Proposition \ref{trchar}.\\
Let $\tilde{E}_\pi$ be an operator associated with Plancherel measure $\tilde{\mu}$: $\tilde{E}_\pi=\sqrt{\varpi^{\kappa,\Omega,\mu}(\pi)}E_\pi$. Then
\begin{equation}\label{eq3}
\begin{split}
\Tr(\tilde{E}_{\pi}^2)=\varpi^{\kappa,\Omega,\mu}(\pi)\Tr(E_\pi^2)&\le 
\dim(\kappa)(\sup_{\pi'\in\Omega} \|E^2_{\pi'}\|) \|E^2_\pi\|^{-1}
\bigl(\sum_{i=1}^{\infty} \chi_{\F^i_{\kappa\stp\sigma_{\Omega}}}\bigr)(\pi)^{-1}\Tr(E^2_\pi)\\
&\le
\dim(\kappa)(\sup_{\pi'\in\Omega} \|E^2_{\pi'}\|) 
\bigl(\sum_{i=1}^{\infty} \chi_{\F^i_{\kappa\stp\sigma_{\Omega}}}\bigr)(\pi)^{-1}\dim(\pi)
\end{split}
\end{equation}
for almost all $\pi\in \F^{1}_{\kappa\stp \sigma_\Omega}$ (Proposition \ref{stw4}). Therefore for any $i\in \NN$ we have
\[
\int_{\F^i_{\kappa\stp\sigma_\Omega}}\Tr(\tilde{E}^2_\pi)\md\tilde{\mu}(\pi)\le
\dim(\kappa)(\sup_{\pi'\in\Omega} \|E^2_{\pi'}\|) 
\int_{\F^i_{\kappa\stp\sigma_\Omega}}\dim\md\tilde{\mu}<+\infty
\]
We can make use of Proposition \ref{stw11} to get
\[
\chi^{\int}(\int_{\F^i_{\kappa\stp\sigma_\Omega}}^{\oplus} \pi \md\tilde{\mu}_{\F^i_{\kappa\stp\sigma_\Omega}}(\pi))=
\int_{\F^i_{\kappa\stp\sigma_\Omega}}\varpi^{\kappa,\Omega,\mu}(\pi)\chi( U^\pi) \md\mu_{\F^i_{\kappa\stp\sigma_\Omega}}(\pi)\in\mf{N}_{\psi}\quad
(i\in\NN).
\]
Since $\sup_{\Omega}\|E^2_\bullet\|<+\infty$ and $\int_\Omega\dim\md\mu<+\infty$ then $\chi^{\int}(\int_\Omega^{\oplus}\pi\md\mu(\pi))\in\mf{N}_{\psi}$ (Proposition \ref{stw11}) and as $\mf{N}_{\psi}$ is an left ideal we also have
\[
\chi(U^{\kappa})\chi^{\int}(\int_\Omega^{\oplus}\pi\md\mu_\Omega(\pi))=\sum_{i=1}^{\infty}\chi^{\int}(\int_{\F^i_{\kappa\stp\sigma_\Omega}}^{\oplus} \pi \md\tilde{\mu}_{\F^i_{\kappa\stp\sigma_\Omega}}(\pi))\in\mf{N}_\psi.
\]
Let us now show that
\[
\bigl(\sum_{i=1}^{a} \Lambda_{\psi}(
\int_{\F^i_{\kappa\stp\sigma_\Omega}}\varpi^{\kappa,\Omega,\mu}(\pi)\chi( U^\pi) \md\mu_{\F^i_{\kappa\stp\sigma_\Omega}}(\pi)
)\bigr)_{a\in\NN}
\]
is a Cauchy sequence in $\LdG$. Let $\{(\xi^n_\pi)_{\pi\in\IrrG}\,|\, n\in\NN\}$ be a measurable field of orthonormal bases. Again due to Proposition \ref{stw11} we have
\[
\Lambda_{\psi}(
\int_{\F^i_{\kappa\stp\sigma_\Omega}}\varpi^{\kappa,\Omega,\mu}(\pi)\chi( U^\pi) \md\mu_{\F^i_{\kappa\stp\sigma_\Omega}}(\pi)
)=
\sum_{n=1}^{\infty}
\Lambda_{\psi}(
\int_{\F^i_{\kappa\stp\sigma_\Omega}}
(\id\otimes\omega_{\xi^{n}_\pi})(U^{\pi})\,
 \md\tilde{\mu}_{\F^i_{\kappa\stp\sigma_\Omega}}(\pi)
\quad(i\in\NN),
\]
hence 
\[
\begin{split}
&\quad\;\bigl\|\sum_{i=a}^{a'} \Lambda_{\psi}(
\int_{\F^i_{\kappa\stp\sigma_\Omega}}\varpi^{\kappa,\Omega,\mu}(\pi)\chi( U^\pi) \md\mu_{\F^i_{\kappa\stp\sigma_\Omega}}(\pi)
)\bigr\|^2\\
&=
\sum_{i,i'=a}^{a'}\sum_{n,n'=1}^{\infty}
\ismaa{
\Lambda_{\psi}(
\int_{\F^i_{\kappa\stp\sigma_\Omega}}
(\id\otimes\omega_{\xi^{n}_\pi})(U^{\pi})\,
 \md\tilde{\mu}_{\F^i_{\kappa\stp\sigma_\Omega}}(\pi)
)}
{
\Lambda_{\psi}(
\int_{\F^{i'}_{\kappa\stp\sigma_\Omega}}
(\id\otimes\omega_{\xi^{n'}_\pi})(U^{\pi }) \md\tilde{\mu}_{\F^{i'}_{\kappa\stp\sigma_\Omega}}(\pi)
)}\\
&=
\sum_{i,i'=a}^{a'}\sum_{n,n'=1}^{\infty}
\int_{\IrrG}
\chi_{\F^i_{\kappa\stp\sigma_\Omega}\cap\F^{i'}_{\kappa\stp\sigma_\Omega}}(\pi)
\ismaa{\xi^{n}_\pi}{\xi^{n'}_\pi}
\ismaa{\tilde{E}_\pi\xi^{n}_\pi}{\tilde{E}_\pi\xi^{n'}_\pi}
\md\tilde{\mu}(\pi)\\
&=
\sum_{i,i'=a}^{a'}
\int_{\IrrG}
\chi_{\F^i_{\kappa\stp\sigma_\Omega}\cap\F^{i'}_{\kappa\stp\sigma_\Omega}}(\pi)
\Tr(\tilde{E}^2_\pi)
\md\tilde{\mu}(\pi)
\end{split}
\]
for $a'>a$. The last equality follows from the monotone convergence theorem. Due to the inequality \eqref{eq3} and the fact that
\[
\begin{split}
\F^i_{\kappa\stp\sigma_\Omega}=\bigcup_{n\ge i}\E^n_{\kappa\stp\sigma_\Omega},&
\quad
\chi_{\F^i_{\kappa\stp\sigma_\Omega}\cap\F^j_{\kappa\stp\sigma_\Omega}}=
\sum_{n=\max\{ i,j\}}^{\infty}\chi_{\E^n_{\kappa\stp\sigma_\Omega}}
\quad(i,j\in\NN),\\
\sum_{j=1}^{\infty}\chi_{\F^j_{\kappa\stp\sigma_\Omega}}&=
\sum_{n=1}^{\infty}n\chi_{\E^n_{\kappa\stp\sigma_\Omega}}
\end{split}
\]
we can further write
\[
\begin{split}
&\quad\;\bigl\|\sum_{i=a}^{a'} \Lambda_{\psi}(
\int_{\F^i_{\kappa\stp\sigma_\Omega}}\varpi^{\kappa,\Omega,\mu}(\pi)\chi( U^\pi ) \md\mu_{\F^i_{\kappa\stp\sigma_\Omega}}(\pi)
)\bigr\|^2\\
&\le
\dim(\kappa) (\sup_{\pi'\in \Omega} \|E^2_{\pi'}\|)
\sum_{i,i'=a}^{a'}\int_{\F^1_{\kappa\stp\sigma_\Omega}}
\bigl(\sum_{n=\max\{i,i'\}}^{\infty} \chi_{\E^{n}_{\kappa\stp\sigma_\Omega}}\bigr)
\bigl(\sum_{n=1}^{\infty}n \chi_{\E^{n}_{\kappa\stp\sigma_\Omega}}\bigr)^{-1}\dim\md\tilde{\mu}\\
&=
\dim(\kappa)(\sup_{\pi'\in \Omega} \|E^2_{\pi'}\|)
\sum_{i,i'=a}^{a'}
\sum_{m=1}^{\infty}\int_{\E^{m}_{\kappa\stp\sigma_\Omega}}
\bigl(\sum_{n=\max\{i,i'\}}^{\infty} \chi_{\E^{n}_{\kappa\stp\sigma_\Omega}}\bigr)
\bigl(\sum_{n=1}^{\infty}n \chi_{\E^{n}_{\kappa\stp\sigma_\Omega}}\bigr)^{-1}\dim\md\tilde{\mu}\\
&=
\dim(\kappa)(\sup_{\pi'\in \Omega} \|E^2_{\pi'}\|)
\sum_{i,i'=a}^{a'}
\sum_{m=1}^{\infty}\int_{\E^{m}_{\kappa\stp\sigma_\Omega}}
[m\ge i,i']\dim(\pi)
\;\tfrac{1}{m}\md\tilde{\mu}(\pi)\\
&=
\dim(\kappa)(\sup_{\pi'\in \Omega} \|E^2_{\pi'}\|)
\sum_{i,i'=a}^{a'}\;
\sum_{m=\max\{i,i'\}}^{\infty}\tfrac{1}{m}\int_{\E^{m}_{\kappa\stp\sigma_\Omega}}
\dim(\pi)
\md\tilde{\mu}(\pi)\\
&=
\dim(\kappa)(\sup_{\pi'\in \Omega} \|E^2_{\pi'}\|)
\sum_{i=a}^{a'}\;\bigl(
(i-a)\sum_{m=i}^{\infty}\tfrac{1}{m}\int_{\E^{m}_{\kappa\stp\sigma_\Omega}}
\dim(\pi)
\md\tilde{\mu}(\pi)\\
&\quad\quad\quad\quad\quad\quad\quad\quad\quad
\quad\quad\quad\quad\quad\quad\quad\quad\quad+
\sum_{i'=i}^{a'}
\sum_{m=i'}^{\infty}\tfrac{1}{m}\int_{\E^{m}_{\kappa\stp\sigma_\Omega}}
\dim(\pi)
\md\tilde{\mu}(\pi)
\bigr)\\
&\le
\dim(\kappa)(\sup_{\pi'\in \Omega} \|E^2_{\pi'}\|)
\sum_{i=a}^{\infty}\;\bigl(
(i-a)\sum_{m=i}^{\infty}\tfrac{1}{m}\int_{\E^{m}_{\kappa\stp\sigma_\Omega}}
\dim(\pi)
\md\tilde{\mu}(\pi)\\
&\quad\quad\quad\quad\quad\quad\quad\quad\quad
\quad\quad\quad\quad\quad\quad\quad\quad\quad+
\sum_{i'=i}^{\infty}
\sum_{m=i'}^{\infty}\tfrac{1}{m}\int_{\E^{m}_{\kappa\stp\sigma_\Omega}}
\dim(\pi)
\md\tilde{\mu}(\pi)
\bigr).
\end{split}
\]
Let us check that both terms of this sum converge to $0$ as $a\to\infty$:
\[
\begin{split}
0&\le\sum_{i=a}^{\infty}
(i-a)\sum_{m=i}^{\infty}\tfrac{1}{m}\int_{\E^{m}_{\kappa\stp\sigma_\Omega}}
\dim(\pi)
\md\tilde{\mu}(\pi)\le
2\sum_{i=a}^{\infty}\sum_{m=i}^{\infty}\tfrac{i}{m}\int_{\E^{m}_{\kappa\stp\sigma_\Omega}}
\dim(\pi)
\md\tilde{\mu}(\pi)\\
&\le
2\sum_{i=a}^{\infty}\sum_{m=i}^{\infty}\int_{\E^{m}_{\kappa\stp\sigma_\Omega}}
\dim(\pi)
\md\tilde{\mu}(\pi)=
2\sum_{m=a}^{\infty}(m-a+1)
\int_{\E^{m}_{\kappa\stp\sigma_\Omega}}
\dim(\pi)
\md\tilde{\mu}(\pi)\\
&\le
6\sum_{m=a}^{\infty}m
\int_{\E^{m}_{\kappa\stp\sigma_\Omega}}
\dim(\pi)
\md\tilde{\mu}(\pi)\xrightarrow[a\to\infty]{}0,
\end{split}
\]
the last convergence follows from the equation \eqref{eq10}. Similarly,
\[
\begin{split}
0&\le\sum_{i=a}^{\infty}\sum_{i'=i}^{\infty}
\sum_{m=i'}^{\infty}\tfrac{1}{m}\int_{\E^{m}_{\kappa\stp\sigma_\Omega}}
\dim(\pi)
\md\tilde{\mu}(\pi)=
\sum_{i'=a}^{\infty} (i'-a+1)
\sum_{m=i'}^{\infty}\tfrac{1}{m}\int_{\E^{m}_{\kappa\stp\sigma_\Omega}}
\dim(\pi)
\md\tilde{\mu}(\pi)\\
&\le
3\sum_{i'=a}^{\infty}
\sum_{m=i'}^{\infty}\tfrac{i'}{m}\int_{\E^{m}_{\kappa\stp\sigma_\Omega}}
\dim(\pi)
\md\tilde{\mu}(\pi)
\le
3\sum_{i'=a}^{\infty}
\sum_{m=i'}^{\infty}\int_{\E^{m}_{\kappa\stp\sigma_\Omega}}
\dim(\pi)
\md\tilde{\mu}(\pi)
\xrightarrow[a\to\infty]{}0,
\end{split}
\]
where the last convergence follows from previously made calculations. This reasoning shows that the sequence
\[
\bigl(\sum_{i=1}^{a} \Lambda_{\psi}(
\int_{\F^i_{\kappa\stp\sigma_\Omega}}\varpi^{\kappa,\Omega,\mu}(\pi)\chi( U^\pi) \md\mu_{\F^i_{\kappa\stp\sigma_\Omega}}(\pi)
)\bigr)_{a\in\NN}
\]
is a Cauchy sequence, consequently due to $\ssot\times\|\cdot\|$ closedness of $\Lambda_{\psi}$ we have
\[
\Lambda_{\psi} \bigl(
\sum_{i=1}^{\infty} 
 \int_{\F^{i}_{\kappa\stp\sigma_\Omega} } 
\varpi^{\kappa,\Omega,\mu}(\zeta)\chi(U^\zeta) \md\mu_{\F^{i}_{\kappa\stp\sigma_\Omega} }(\zeta)
\bigr)=
\sum_{i=1}^{\infty}\Lambda_{\psi}\bigl(
 \int_{\F^{i}_{\kappa\stp\sigma_\Omega} } 
\varpi^{\kappa,\Omega,\mu}(\zeta)\chi(U^\zeta) \md\mu_{\F^{i}_{\kappa\stp\sigma_\Omega} }(\zeta)
\bigr).
\]
\end{proof}

At the end of this section we define the following operator:
\[
T\colon \LL^{2}(\IrrG)\supseteq \mc{D}(T)\ni f\mapsto \Lambda_{\psi} (
\int_{\IrrG} f(\pi)
\Tr(E_\pi^{2})^{-\frac{1}{2}}
 \chi(U^\pi) \md\mu(\pi))\in \LL^2(\GG),
\]
where $\mc{D}(T)$ is a subspace of those $f\in \LL^2(\IrrG)$ for which
\[
\mu(\supp f),\;\int_{\IrrG}|f|^2\md\mu,\;
\int_{\IrrG}|f|^2\Tr(E^2_\pi)^{-1}\md\mu<+\infty
\]
and $|\{n\in\NN\,|\, \IrrG \restriction_n \cap \supp f\neq \emptyset\}|<\infty$. The subspace $\mc{D}(T)$ from now on will be called the original domain of $T$.

\begin{lemma}\label{lemat34}
Operator $T$ is well defined and extends to an isometry $\LL^2(\IrrG)\rightarrow \LL^2(\GG)$.
\end{lemma}

We will denote the isometry $\LL^2(\IrrG)\rightarrow \LdG$ from the above lemma also by $T$.

\begin{proof}
Let us check that the integral which appears in the definition of $T$ is well defined: let $\{(\xi^{\pi}_k)_{\pi\in\IrrG}\}_{k=1}^{\infty}$ be a measurable field of orthonormal bases. Fix $f\in\mc{D}(T),n\in\NN$ and define
\[
\xi=\xi'=(\chi_{\supp f}(\pi) \xi^{\pi}_n)_{\pi\in\IrrG},\quad
\eta=\eta'=(f(\pi)\Tr(E^{2}_{\pi})^{-\frac{1}{2}}\xi^{\pi}_n)_{\pi\in\IrrG}.
\]
These fields satisfy assumptions of Proposition \ref{Casp}: due to the assumptions on $f$, they are square integrable. As
\[
\begin{split}
\int_{\IrrG} \| E_\pi \eta_\pi\|^2\md\mu(\pi)&=
\int_{\IrrG} |f(\pi)|^2 \Tr(E^{2}_{\pi})^{-1} \|E_\pi \xi^{\pi}_n\|^2
\md\mu(\pi)\\
&\le
\int_{\IrrG} |f(\pi)|^2 \Tr(E^{2}_{\pi})^{-1} \Tr(E^{2}_\pi)
\md\mu(\pi)
=\int_{\IrrG} |f|^2 \md\mu<+\infty,
\end{split}
\]
$\eta$ belongs to $\Dom(E)$. Moreover
\[
\int_{\IrrG} \|\xi_\pi\otimes \ov{E_\pi \eta_\pi} \|^2\md\mu(\pi)=
\int_{\IrrG} \chi_{\supp f}(\pi) \|E_\pi \eta_\pi\|^2\md\mu(\pi)\le\int_{\IrrG}|f|^2\md\mu<+\infty
\]
due to previous calculations. Consequently for any $n\in\NN$ there exist integrals
\[
\int_{\IrrG} f(\pi) \Tr(E^{2}_{\pi})^{-\frac{1}{2}}
(\id\otimes\omega_{\xi^n_\pi})U^{\pi} \md\mu(\pi)\in\mf{N}_{\psi},
\]
and if we take a sum over $n$ (which is finite) we get existence of
\[
\sum_{n=1}^{\infty}
\int_{\IrrG} f(\pi) \Tr(E^{2}_{\pi})^{-\frac{1}{2}}
(\id\otimes\omega_{\xi^n_\pi})U^{\pi} \md\mu(\pi)=
\int_{\IrrG} f(\pi) \Tr(E^{2}_{\pi})^{-\frac{1}{2}}
\chi(U^{\pi}) \md\mu(\pi)\in \mf{N}_{\psi}.
\]
This shows that $T$ is a well defined on the dense domain $\mc{D}(T)$. Let us make use of the orthogonality relations to show that $T$ is an isometry:
\[
\begin{split}
&\quad\;\psi\bigl(\bigl(
\int_{\IrrG} 
f(\pi) \Tr(E^{2}_{\pi})^{-\frac{1}{2}} \chi(U^{\pi})
 \md\mu(\pi)
\bigr)^*\bigl(
\int_{\IrrG}  
f(\pi) \Tr(E^{2}_{\pi})^{-\frac{1}{2}} \chi(U^{\pi})
 \md\mu(\pi)
\bigr)\bigr)\\
&=
\sum_{n,m=1}^{\infty}
\psi\bigl(\bigl(
\int_{\IrrG} 
f(\pi) \Tr(E^{2}_{\pi})^{-\frac{1}{2}} (\id\otimes\omega_{\xi^n_\pi})(U^{\pi})
 \md\mu(\pi)
\bigr)^*\\
&\quad\quad\quad\quad\quad\quad\quad\quad\quad\quad\quad\quad\bigl(
\int_{\IrrG}  
f(\pi) \Tr(E^{2}_{\pi})^{-\frac{1}{2}} (\omega_{\xi^{\pi}_m}\otimes\id)(U^{\pi})
 \md\mu(\pi)
\bigr)\bigr)\\
&=
\sum_{n,m=1}^{\infty}
\int_{\IrrG}|f(\pi)|^2 \Tr(E^{2}_{\pi})^{-1}
\is{\xi^{\pi}_n}{\xi^{\pi}_m}
\is{E_\pi \xi^{\pi}_n}{E_\pi \xi^{\pi}_m} \md\mu(\pi)\\
&=
\int_{\IrrG} |f(\pi)|^2 \Tr(E^2_{\pi})^{-1} \Tr(E^{2}_{\pi}) 
\md\mu(\pi)=\|f\|^2.
\end{split}
\]
\end{proof}

Define $\mc{H}$ to be an image of $T$:
\[
\mc{H}=T(\LL^2(\IrrG))\subseteq \LL^2(\GG).
\]
Since $T$ is an isometry, $\mc{H}$ is a closed linear subspace. In Section \ref{cstA} we will show that $\mc{H}$ is the subspace of integral characters on the $\LL^2$ level (see Proposition \ref{stw1} for precise formulation).

\section{Convolution operators}\label{secconv}
\subsection{Operator $\mc{L}_\kappa$}
Let us choose a Plancherel measure and a finite dimensional nondegenerate representation $\kappa\colon\CGDu\rightarrow\B(\msf{H}_\kappa)$ such that $\kappa\lec\Lambda_{\whG}$. Define a dense subspace in $\LL^2(\IrrG)$:
\[
\mc{F}=\lin\{\Tr(E^2_{\bullet})^{\frac{1}{2}}\chi_\Omega\,|\,\Omega\subseteq_{meas}\IrrG:\;\mu(\Omega)<+\infty,\;\sup_{\pi\in\Omega}(\dim(\pi)+\|E^2_\pi\|)<+\infty\}
\]
and the following map
\begin{equation}\label{eq24}
\mc{L}_\kappa \colon \mc{F}\ni
\Tr(E^2_\bullet)^{\frac{1}{2}}\chi_\Omega\mapsto
\Tr(E^2_\bullet)^{\frac{1}{2}}\varpi^{\kappa,\Omega,\mu}\sum_{i=1}^{\infty} \chi_{\F^i_{\kappa\stp\sigma_{\Omega}}}\in \LL^2(\IrrG).
\end{equation}

\begin{proposition}\label{stw3}
$\mc{L}_\kappa$ is a well defined linear operator.
\end{proposition}

\begin{proof}
Lemma \ref{lemat6} says that if $\Omega_1,\Omega_2$ are measurable disjoint subsets in $\IrrG$ then
\[
\varpi^{\kappa,\Omega_1\cup\Omega_2,\mu}\sum_{i=1}^{\infty} \chi_{\F^i_{\kappa\stp\sigma_{\Omega_1\cup\Omega_2}}}=
\varpi^{\kappa,\Omega_1,\mu}\sum_{i=1}^{\infty} \chi_{\F^i_{\kappa\stp\sigma_{\Omega_1}}}
+
\varpi^{\kappa,\Omega_2,\mu}\sum_{i=1}^{\infty} \chi_{\F^i_{\kappa\stp\sigma_{\Omega_2}}}
\]
almost everywhere. Consequently, we can define linear map $\mc{L}_\kappa$ as above, although now we only know that for $f\in \mc{F}$, $\mc{L}_\kappa (f)$  is in the linear space of measurable maps on $\IrrG$. Let $\{V_p\}_{p\in\NN}$ be any increasing family of measurable subsets which are of finite measure and $\sup_{\pi\in V_p} (\dim(\pi)+\|E_\pi^2\|^{-1})\le p$. Let $\sum_{k=1}^{n} c_k \Tr(E^2_\bullet)^{\frac{1}{2}}\chi_{\Omega_k}\in \mc{F}$. Due to Proposition \ref{stw4} we have $\chi_{V_p} \mc{L}_\kappa \bigl(c_k \Tr(E^2_\bullet)^{\frac{1}{2}}\chi_{\Omega_k} \bigr)\in \LL^2(\IrrG)$ for each $p\in \NN$ and $k\in\{1,\dotsc,n\}$. Indeed:
\[\begin{split} 
&\quad\;
\int_{\IrrG} \bigl|
\chi_{V_p} \mc{L}_\kappa \bigl(c_k \Tr(E^2_\bullet)^{\frac{1}{2}}\chi_{\Omega_k} \bigr)\bigr|^2\md\mu\\
&=
|c_k|^2 \int_{V_p} \Tr(E^2_\pi) \bigl(
\bigl( \sum_{i=1}^{\infty} \chi_{\F^i_{\kappa\stp\sigma_{\Omega_k}}}\bigr) (\pi)\varpi^{\kappa,\Omega_k,\mu}(\pi)\bigr)^2\md\mu(\pi)\\
&\le
|c_k|^2 ( \sup_{\pi'\in \Omega_k}\|E^2_{\pi'}\|)^2
\dim(\kappa)^2\int_{V_p} \Tr(E^2_\pi) \|E^2_\pi\|^{-2}\md\mu(\pi)\\
&\le
|c_k|^2 ( \sup_{\pi'\in \Omega_k}\|E^2_{\pi'}\|)^2
\dim(\kappa)^2
\int_{V_p}\dim(\pi) \|E^2_\pi\|^{-1}\md\mu(\pi)\\
&\le
|c_k|^2 ( \sup_{\pi'\in \Omega_k}\|E^2_{\pi'}\|)^2
\dim(\kappa)^2 p^2 \mu(V_p)<+\infty.
\end{split}\]
The following holds:
\[
\begin{split}
&\quad\;\bigl\| \chi_{V_p}\mc{L}_\kappa\bigl( \sum_{k=1}^{n} c_k \Tr(E^2_\bullet)^{\frac{1}{2}}\chi_{\Omega_k} \bigr)\bigr\|^2\\
&=
\sum_{k,k'=1}^{n}\ov{c_k} c_{k'}
\int_{V_p}\Tr(E^2_\bullet)
\bigl( 
\varpi^{\kappa,\Omega_k,\mu}\sum_{i=1}^{\infty} \chi_{\F^i_{\kappa\stp\sigma_{\Omega_k}}}\bigr)
\bigl(
 \varpi^{\kappa,\Omega_{k'},\mu}\sum_{i'=1}^{\infty} \chi_{\F^{i'}_{\kappa\stp\sigma_{\Omega_{k'}}}} \bigr)
\md\mu(\pi).
\end{split}
\]
Let us fix $k,k'\in\{1,\dotsc,n\}$ and turn to giving a bound for the above integrals. Thanks to Proposition \ref{stw4} we have
\[
\begin{split}
&\quad\;\int_{V_p}
\Tr(E^2_\pi)
\bigl( \varpi^{\kappa,\Omega_k,\mu}\sum_{i=1}^{\infty} \chi_{\F^i_{\kappa\stp\sigma_{\Omega_k}}}\bigr)(\pi)
\bigl( 
\varpi^{\kappa,\Omega_{k'},\mu}\sum_{i'=1}^{\infty} \chi_{\F^{i'}_{\kappa\stp\sigma_{\Omega_{k'}}}} \bigr)(\pi)
\md\mu(\pi)\\
&\le
\dim(\kappa)(\sup_{\pi\in\Omega_k}\|E^2_\pi\|)\int_{V_p}
\Tr(E^2_\pi)\|E^2_\pi\|^{-1} 
\varpi^{\kappa,\Omega_{k'},\mu}(\pi)\sum_{i=1}^{\infty} \chi_{\F^i_{\kappa\stp\sigma_{\Omega_{k'}}}}(\pi)\md\mu(\pi)\\
&\le
\dim(\kappa)(\sup_{\pi\in\Omega_k}\|E^2_\pi\|)\sum_{i=1}^{\infty}
\int_{\F^i_{\kappa\stp\sigma_{\Omega_{k'}}}}\dim
\md\,(\varpi^{\kappa,\Omega_{k'},\mu} \mu)\\
&=
\dim(\kappa)^2(\sup_{\pi\in\Omega_k}\|E^2_\pi\|)
\int_{\Omega_{k'}}\dim
\md\mu<+\infty.
\end{split}
\] 
In the last equality we have invoked equality \eqref{eq10}.Thanks to the above inequality we get
\[
\begin{split}
&\quad\;\bigl\| \chi_{V_p}\mc{L}_\kappa\bigl( \sum_{k=1}^{n} c_k \chi_{\Omega_k} \bigr)\bigr\|^2\le
\sum_{k,k'=1}^{n} |c_k c_{k'}| 
(\sup_{\pi\in\Omega_k}\|E^2_\pi\|)\dim(\kappa)^2 \mu(\Omega_{k'})<+\infty
\quad(p\in\NN).
\end{split}
\]
As this inequality does not depend on $p$, we arrive at the claim of the proposition:
\[
\begin{split}
&\quad\;\int_{\IrrG}\bigl|
\mc{L}_\kappa\bigl( \sum_{k=1}^{n} c_k 
\Tr(E^2_\bullet)^{\frac{1}{2}}\chi_{\Omega_k} \bigr)\bigr|^2
\md\mu=
\lim_{p\to\infty}
\int_{\IrrG}\chi_{V_p}\bigl|
\mc{L}_\kappa\bigl( \sum_{k=1}^{n} c_k 
\Tr(E^2_\bullet)^{\frac{1}{2}}\chi_{\Omega_k} \bigr)\bigr|^2
\md\mu\\
&\le
\sum_{k,k'=1}^{n} |c_k c_{k'}|
(\sup_{\pi\in\Omega_k}\|E^2_\pi\|) \dim(\kappa)^2 \mu(\Omega_k)<+\infty.
\end{split}
\]
\end{proof}

Let us introduce a notion of admissibility of representations (\cite[Definition 3.1]{DasDawsSalmi}, see also \cite[Definition 2.2]{Soltanqb}):

\begin{definition}\label{defadmissible}
Let $U=(U_{i,j})_{i,j=1}^{N}\in \M(\CG\otimes \mc{K}(\msf{H}))$ be a finite dimensional representation of $\GG$. We say that $U$ is \emph{admissible} if the element $U^t=(U_{j,i})_{i,j=1}^{N}$ is invertible in $\M(\CG\otimes \mc{K}(\msf{H}))$.
\end{definition}

Recall that we have defined $\mc{H}$, a closed subspace in $\LdG$ via $\mc{H}=T(\LL^2(\IrrG))$. In the next theorem we show that the operator $\mc{L}_\kappa$ is unitarily equivalent to the restricted character $\chi(U^{\kappa})|_{\mc{H}}$.

\begin{theorem}\label{tw1}
Choose a nondegenerate finite dimensional representation $\kappa\colon\CGDu\rightarrow\B(\msf{H}_\kappa)$ such that $\kappa\lec\Lambda_{\whG}$. Operator $\mc{L}_\kappa$ extends to a bounded operator on $\LdIrr$ such that
\[
T\mc{L}_\kappa f=\chi(U^{\kappa})T f\quad(f\in \LL^2(\IrrG)).
\]
In particular $\|\mc{L}_\kappa\|\le\dim(\kappa)$, and if $\kappa$ is \emph{admissible} then ${\mc{L}_\kappa}^*=\mc{L}_{\kappa^{c}}$.
\end{theorem}

In the above theorem, $\kappa^{c}$ means the conjugate representation $\jmath_{\msf{H}_{\kappa}}\circ \kappa\circ\hat{R}^{u}$.

\begin{proof}
Take $f=\Tr(E^2_\bullet)^{\frac{1}{2}}\chi_\Omega\in\mc{F}\subseteq\LL^2(\IrrG)$ for a measurable subset $\Omega\subseteq \supp(\mu)\subseteq\IrrG$ with finite measure. Then $f$ belongs to the original domain of $T$, $\mc{D}(T)$. We have
\[
\begin{split}
\star&=\chi(U^\kappa) T f\\
&=
\chi(U^\kappa)
\Lambda_{\psi}(\int_{\IrrG} f(\pi)\Tr(E^2_\pi)^{-\frac{1}{2}} \chi(U^{\pi})\md\mu(\pi))\\
&=
\chi(U^\kappa)
\Lambda_{\psi}(\int_{\IrrG} \chi_\Omega(\pi) \chi(U^{\pi})\md\mu(\pi))\\
&=
\Lambda_{\psi}(\chi(U^{\kappa}) \chi^{\int}(\int_{\Omega}^{\oplus}
\pi\md\mu_{\Omega}(\pi))).
\end{split}
\]
We have previously derived the following equations:
\begin{equation}\label{eq7}
\dim(\kappa)\int_\Omega \dim \md\mu_\Omega(\pi)
=
\sum_{i=1}^{\infty}
\int_{\F^{i}_{\kappa\stp\sigma_\Omega}} \dim\;\varpi^{\kappa,\Omega,\mu}\md\mu_{\F^i_{\kappa\stp\sigma_\Omega}}(\pi)
<+\infty
\end{equation}
and
\begin{equation}
\label{chareq}
\begin{split}
&\quad\;
 \chi(U^\kappa)\chi^{\int} ( \int_{\Omega}^{\oplus} \pi \md\mu_\Omega(\pi))
=\sum_{i=1}^{\infty}\chi^{\int}( \int_{\F^{i}_{\kappa\stp\sigma_\Omega} }^{\oplus} \zeta \md\,
(\varpi^{\kappa,\Omega,\mu}\mu)_{\F^{i}_{\kappa\stp\sigma_\Omega} }(\zeta)).
\end{split}
\end{equation}
Due to the equality \eqref{chareq} we get
\begin{equation}\label{eq20}
\begin{split}
\star=\chi(U^\kappa)T f&=
\Lambda_{\psi}\bigl(
\sum_{i=1}^{\infty}\chi^{\int}( \int_{\F^{i}_{\kappa\stp\sigma_\Omega} }^{\oplus} \pi \md\,(
\varpi^{\kappa,\Omega,\mu}\mu)_{\F^{i}_{\kappa\stp\sigma_\Omega} }(\pi))
\bigr)\\
&=
\Lambda_{\psi}\bigl(
\sum_{i=1}^{\infty}\int_{\F^i_{\kappa\stp\sigma_\Omega}}\chi(U^\pi)
\varpi^{\kappa,\Omega,\mu}(\pi)\md\mu(\pi)\bigr)\\
&=
\Lambda_{\psi}\bigl(
\int_{\IrrG}\bigl(\sum_{i=1}^{\infty}\chi_{\F^i_{\kappa\stp\sigma_\Omega}}(\pi)\bigr)\chi(U^\pi)
\varpi^{\kappa,\Omega,\mu}(\pi)\md\mu(\pi)\bigr)\\
&=
\Lambda_{\psi}\bigl(
\int_{\IrrG}\Tr(E^2_\pi)^{\frac{1}{2}}\bigl(\sum_{i=1}^{\infty}\chi_{\F^i_{\kappa\stp\sigma_\Omega}}(\pi)\bigr)
\varpi^{\kappa,\Omega,\mu}(\pi)
\chi(U^\pi)
\Tr(E^2_\pi)^{-\frac{1}{2}}\md\mu(\pi)\bigr)
\end{split}
\end{equation}
We would like to write that the above vector is a result of the action of the operator $T$ on a function
\[
g=\Tr(E^2_\bullet)^{\frac{1}{2}}\varpi^{\kappa,\Omega,\mu}\sum_{i=1}^{\infty} \chi_{\F^{i}_{\kappa\stp\sigma_\Omega}}.
\]
However, we cannot do this right away -- we do not know whether $g$ belongs to the original domain of $T$. Observe that Proposition \ref{stw3} implies that the function $g$ is in $\LL^2(\IrrG)$. Let us introduce an increasing family of subsets of $\IrrG$, $\{V_p\}_{p\in\NN}$, such that
\[
\mu(V_p)<+\infty,\quad
V_p\subseteq \bigcup_{p'=1}^{p}\IrrG\rest_{p'},\quad
\sup_{\pi\in V_p}\Tr(E^2_\pi)^{-1}<+\infty\quad(p\in\NN)
\]
and $\bigcup_{p=1}^{\infty} V_p=\IrrG$. It is clear that such a family exists - one simply has to take an intersection of appropriate subsets. For any  $p\in\NN$ we have
\[
\mu(\supp(\chi_{V_p} g)),\quad
\int_{\IrrG} |\chi_{V_p}g|^2\md\mu,\quad
\int_{\IrrG} |\chi_{V_p}g|^2\Tr(E^2_\bullet)^{-1}\md\mu<+\infty,
\]
moreover every representation in $\supp(\chi_{V_p}g)$ has dimension $\le p$. It follows that function $\chi_{V_p}g$ is in the original domain of $T$ and we have
\[
\begin{split}
T(\chi_{V_p}g)&=
\Lambda_{\psi}\bigl(\int_{\IrrG} \chi_{V_p}(\pi) g(\pi) \Tr(E^2_\pi)^{-\frac{1}{2}} \chi(U^\pi)\md\mu(\pi)\bigr)\\
&=
\Lambda_{\psi}\bigl(\int_{V_p} \varpi^{\kappa,\Omega,\mu}(\pi)
(\sum_{i=1}^{\infty} \chi_{\F^i_{\kappa\stp\sigma_{\Omega}}})(\pi) \chi(U^\pi)\md\mu(\pi)\bigr)\quad(p\in\NN).
\end{split}
\]
It is clear that $\chi_{V_p}g\xrightarrow[p\to\infty]{}g$, therefore by continuity of $T$ we have
\[
T(g)=
\lim_{p\to\infty}T(\chi_{V_p}g)=
\lim_{p\to\infty}
\Lambda_{\psi}\bigl(\int_{V_p} \varpi^{\kappa,\Omega,\mu}(\pi)
(\sum_{i=1}^{\infty} \chi_{\F^i_{\kappa\stp\sigma_{\Omega}}})(\pi) \chi(U^\pi)\md\mu(\pi)\bigr)
\]
(in particular this limit exists). We also have
\[
\begin{split}
&\quad\;
\bigl\|\int_{V_p} \varpi^{\kappa,\Omega,\mu}(\pi)
(\sum_{i=1}^{\infty} \chi_{\F^i_{\kappa\stp\sigma_{\Omega}}})(\pi) \chi(U^\pi)\md\mu(\pi)-
\int_{\IrrG} \varpi^{\kappa,\Omega,\mu}(\pi)
(\sum_{i=1}^{\infty} \chi_{\F^i_{\kappa\stp\sigma_{\Omega}}})(\pi) \chi(U^\pi)\md\mu(\pi)\bigr\|\\
&=
\bigl\|
\int_{\IrrG\setminus V_p}\varpi^{\kappa,\Omega,\mu}(\pi)
(\sum_{i=1}^{\infty} \chi_{\F^i_{\kappa\stp\sigma_{\Omega}}})(\pi) \chi(U^\pi)\md\mu(\pi)\bigr\|\\
&\le
\int_{\IrrG\setminus V_p}
\varpi^{\kappa,\Omega,\mu}(\pi)
(\sum_{i=1}^{\infty} \chi_{\F^i_{\kappa\stp\sigma_{\Omega}}})(\pi) \dim(\pi)\md\mu(\pi)\xrightarrow[p\to\infty]{}0,
\end{split}
\]
therefore closedness of $\Lambda_{\psi}$ implies
\[
\begin{split}
&\quad\;\star=\chi(U^\kappa)Tf=
\Lambda_{\psi}\bigl(\int_{\IrrG} \varpi^{\kappa,\Omega,\mu}(\pi)
(\sum_{i=1}^{\infty} \chi_{\F^i_{\kappa\stp\sigma_{\Omega}}})(\pi) \chi(U^\pi)\md\mu(\pi)\bigr)\\
&=
\lim_{p\to\infty}
\Lambda_{\psi}\bigl(\int_{V_p} \varpi^{\kappa,\Omega,\mu}(\pi)
(\sum_{i=1}^{\infty} \chi_{\F^i_{\kappa\stp\sigma_{\Omega}}})(\pi) \chi(U^\pi)\md\mu(\pi)\bigr)\\
&=\lim_{p\to\infty} T(\chi_{V_p}g)=T(g)=
T\bigl(
\Tr(E^2_\bullet)^{\frac{1}{2}}\varpi^{\kappa,\Omega,\mu}\sum_{i=1}^{\infty} \chi_{\F^{i}_{\kappa\stp\sigma_\Omega}}\bigr)=
T\mc{L}_{\kappa}(f).
\end{split}
\]
So far we have checked this equality for very special $f\in \mc{F}$, namely those of the form $f=\Tr(E^2_\bullet)^{\frac{1}{2}}\chi_{\Omega}$. However, by linearity of $\mc{L}_\kappa$ we know that this equality holds for every $f\in \mc{F}$. Because $T$ is an isometric map we get
\[
\sup_{f\in \mc{F}:\, \|f\|=1}\|\mc{L}_\kappa f\|=
\sup_{f\in \mc{F}:\, \|f\|=1}\|T\mc{L}_\kappa f\|=
\sup_{f\in \mc{F}:\, \|f\|=1}\|\chi(U^\kappa)Tf\|\le
\dim(\kappa).
\]
It follows that $\mc{L}_\kappa$ is a bounded operator and can be extended to the whole $\LL^2(\IrrG)$. Denote this extension with the same symbol, by continuity we have
\[
\chi(U^\kappa)T f=
T\mc{L}_\kappa f
\quad(f\in\LL^2(\IrrG)).
\]
Since $T$ is an isometry, we can write $\mc{L}_\kappa=T^* \chi(U^{\kappa})T$. Therefore if $\kappa$ is an admissible representation, we have
\[
{\mc{L}_\kappa}^*=(T^* \chi(U^{\kappa}) T)^*=T^* \chi(U^{\kappa^{c}}) T
= \mc{L}_{\kappa^{c}}.
\]
\end{proof}

\begin{corollary}\label{wniosek}
The above result shows in particular that
\[
\chi(U^\kappa)\mc{H}\subseteq \mc{H}
\]
for any finite dimensional nondegenerate representation $\kappa$ which is weakly contained in $\Lambda_{\whG}$.
\end{corollary}

\subsection{Operator $\mc{L}_\nu$}
For any irreducible representation $\kappa\in\supp(\mu)$ we have $\dim(\kappa)<+\infty$ and $\kappa\lec \Lambda_{\whG}$ (\cite[Theorem 3.4.8]{Desmedt}). We have introduced a bounded linear map
\[
\mc{L}_\kappa\colon\LdIrr\rightarrow\LL^2(\IrrG),
\]
therefore for $\nu\in\LL^1(\IrrG)$ we can define a linear operator
\begin{equation}\label{eq46}
\mc{L}_\nu=\int_{\IrrG}\tfrac{\nu(\kappa)}{\dim(\kappa)}\mc{L}_{\kappa} \md\mu(\kappa).
\end{equation}
The above integral converges in \swot$ $ -- it follows from the bound $\|\mc{L}_\kappa\|\le\dim(\kappa)$ and the fact that for $\xi,\eta\in\LdG$ the function
\[
\IrrG\ni \kappa\mapsto \ismaa{\xi}{\chi(U^{\kappa})\eta}\in\CC
\]
is measurable. We obviously have $\|\mc{L}_\nu\|\le\|\nu\|_1$. We get another corollary from Theorem \ref{tw1}:

\begin{corollary}\label{wniosek2}
Let $\Omega\subseteq \IrrG$ be a measurable subset such that $\int_{\Omega}\dim\md\mu<+\infty$. Let $\nu=\sum_{m=1}^{\infty}m\chi_{\Omega\rest_m}=\dim\chi_\Omega$ be a function in $\LL^1(\IrrG)$. Then
\[
T\mc{L}_\nu f=\chi^{\int}(\int_{\Omega}^{\oplus}\pi\md\mu_{\Omega}(\pi))T f\quad(f\in \LL^2(\IrrG)).
\]
\end{corollary}

\begin{proof}
The above result is a direct consequence of Theorem \ref{tw1}: for $f\in \LdIrr,g\in \LdG$ the following holds
\[\begin{split} 
\is{g}{T\lambda_{\nu} f}&=
\int_{\IrrG}\tfrac{\nu(\kappa)}{\dim(\kappa)} 
\is{g}{T\mc{L}_{\kappa}f} \md\mu(\kappa)=
\int_{\IrrG}\chi_{\Omega}(\kappa)
\is{g}{\chi(U^{\kappa}) Tf} \md\mu(\kappa)\\
&=
\ismaa{g}{\chi^{\int}(\int_{\Omega}^{\oplus}\pi\md\mu_{\Omega}(\pi) T f}.
\end{split}\]
\end{proof}

\section{Conjugation $\IrrG\rightarrow \IrrG$}\label{secconj}
Recall that $\GG$ is a second countable type I locally compact quantum group whose all irreducible representations are finite dimensional.\\
In this section we will be concerned with the conjugation map defined on the level of $\IrrG$. Our first important result is that the Plancherel measure is equivalent to the Plancherel measure composed with conjugation. Next, we will be able to derive relations between traces of appropriate powers of $E_\pi$ and $D_{\ov{\pi}}$ (see propositions \ref{stw13}, \ref{stw5} for precise formulation). At the end of the section we prove a theorem which connects coamenability of $\GG$ with spectra of integral characters and is one of the two main theorems of the paper.\\

For $\pi$, a nondegenerate representation of $\CGDu$ we define its conjugate representation as $\pi^c=\jmath_{\msf{H}_\pi}\circ\pi\circ\hat{R}^u$. If $\pi$ is irreducible then so is $\pi^c$. This operation also preserves equivalences, therefore it can be transfered to the level of classes: we get a map $\IrrG\rightarrow\IrrG$ which will be denoted by $\pi\mapsto\ov{\pi}$ -- since this moment $\ov{\pi}$ denotes a class or a representative chosen according to our measurable field of representations on the canonical measurable field of Hilbert spaces. We obviously have $\ov{\ov{\pi}}=\pi$. Note that we do not have equality of $\ov{\pi}$ and $\pi^c=\jmath_{\msf{H}_\pi}\circ\pi\circ \hat{R}^{u}$ (e.g.~the first representation acts on $\CC^{\dim(\pi)}$, the second one on $\ov{\CC^{\dim(\pi)}}$), though these representations are unitarily equivalent. Let $\rho_{\ov{\pi}}\colon \B(\ov{\CC^{\dim(\pi)}})\rightarrow \B(\CC^{\dim(\pi)})$ be an isomorphism given by conjugation with a unitary intertwiner. This means that the following equality holds: $\rho_{\ov{\pi}}\circ\jmath_{\msf{H}_\pi}\circ \pi\circ \hat{R}^{u}=\ov{\pi}$. Note that since $\ov{\ov{\pi}}=\pi$ we have
\[
\pi=
\rho_{\pi}\circ \jmath_{\msf{H}_{\ov{\pi}}}\circ \ov{\pi}\circ \hat{R}^u=
\rho_{\pi}\circ \jmath_{\msf{H}_{\ov{\pi}}}\circ 
\rho_{\ov{\pi}}\circ\jmath_{\msf{H}_\pi}\circ \pi\circ \hat{R}^{u}
\circ \hat{R}^u=
\rho_{\pi}\circ \jmath_{\msf{H}_{\ov{\pi}}}\circ 
\rho_{\ov{\pi}}\circ\jmath_{\msf{H}_\pi}\circ \pi,
\]
which implies $\id=
\rho_{\pi}\circ \jmath_{\msf{H}_{\ov{\pi}}}\circ 
\rho_{\ov{\pi}}\circ\jmath_{\msf{H}_\pi}\,(\pi\in \IrrG)$. As usual, let us fix a Plancherel measure $\mu$.

\begin{lemma}
The map $\IrrG\ni \pi \mapsto \ov{\pi}\in \IrrG$ is measurable.
\end{lemma}

\begin{proof}
Since $\GG$ is type I, the $\sigma$-algebra on $\IrrG$ is equal to the Mackey-Borel structure. Let $\Omega\subseteq\IrrG$ be a measurable subset. Then the set $\{\sigma\in \Irr(\CGDu)\rest_n\,|\, [\sigma]\in \Omega\}$ is measurable for each $n\in\NN$ \cite[Section 3.8.1]{DixmierC}. We need to show that the set $\{\sigma\in \Irr(\CGDu)\rest_n\,|\, [\sigma]\in \ov{\Omega}\}$ is also measurable.\\
Choose any isomorphism $\rho_n\colon \B(\ov{\CC^n})\rightarrow \B(\CC^n)$ which is given by conjugating with a unitary operator $\ov{\CC^n}\rightarrow \CC^n$. On $\operatorname{Irrep}(\CGDu)_n$ the Mackey-Borel structure is defined as the smallest $\sigma$-algebra such that the functions $\sigma\mapsto \is{\xi}{\sigma(x)\zeta}\,(x\in\CGDu,\xi,\zeta\in\CC^n)$ are measurable. It follows that the function $j\colon \sigma\mapsto \jmath_{\ov{\msf{H}_\sigma}}\circ\rho_n^{-1}\circ \sigma\circ \hat{R}^{u}$ on the level of $\Irr(\CGDu)\rest_n$ is measurable. Since we have
\[\begin{split}
&\quad\;j^{-1}(\{
\sigma\in \Irr(\CGDu)\rest_n\,|\,[\sigma]\in\Omega
\})=
\{\rho_n\circ \jmath_{\msf{H}_\sigma}\circ\sigma\circ \hat{R}^{u}
\in \Irr(\CGDu)\rest_n\,|\,
[\sigma]\in \Omega\}\\
&=
\{\sigma
\in \Irr(\CGDu)\rest_n\,|\,
[\sigma]\in \ov{\Omega}\},
\end{split}\]
the map $\IrrG\ni[\sigma]\mapsto[\sigma^c] \in \IrrG$ is measurable.
\end{proof}

\begin{proposition}\label{stw12}
The measure $\mu'\colon \mscr{B}(\IrrG)\ni \Omega\mapsto \mu(\ov{\Omega})\in \RR_{\ge 0}\cup \{+\infty\}$ is equivalent to $\mu$.
\end{proposition}

\begin{proof} 
Since $\IrrG$ is a standard measurable space, the measure $\mu'$ is standard. With measure space $(\IrrG,\mu')$ we associate the standard measurable field of Hilbert spaces -- same as for $(\IrrG,\mu)$. Recall that the operator $\hat{J}$ acts as follows:
\[
\hat{J}\Lvp(x)=\Lambda_{\psi}(R(x)^*)\quad(x\in\mf{N}_{\vp})
\]
therefore by duality
\[
J\Lhvp(x)=\Lambda_{\hpsi}(\hat{R}(x)^*)\quad(x\in\mf{N}_{\hvp}).
\]
In particular for $\lambda(\alpha)\in\lambda(\LL^1_{\sharp}(\GG))\cap\mf{N}_{\hvp}$ we have
\[
J\Lhvp(\lambda(\alpha))=
\Lambda_{\hpsi}(\lambda(\alpha\circ R)^*)=
\Lambda_{\hpsi}(\lambda((\alpha\circ R)^{\sharp})).
\]
We have a canonical antilinear map\footnote{We need to introduce operator $\rho_\pi$ in the definition of $\mc{V}_1$ in order for the operators to act on the appropriate spaces -- nevertheless, it is rather artificial and stems from our choice to work with the canonical measurable field of Hilbert spaces.}
\[\begin{split}
\mc{V}_1\colon \int_{\IrrG}^{\oplus} &\HS(\msf{H}_\pi)\md\mu(\pi)\ni \int_{\IrrG}^\oplus T_\pi\md\mu(\pi)\mapsto\\
&\mapsto\int_{\IrrG}^\oplus \rho_{\pi}\circ\jmath_{\msf{H}_{\ov\pi}}(T_{\ov\pi})^*\md\mu'(\pi)\in
\int_{\IrrG}^{\oplus} \HS(\msf{H}_\pi)\md\mu'(\pi).
\end{split}\]
It is well defined: if $\tilde{T}_\pi=T_\pi$ for $\mu$-almost all $\pi$ then $\jmath_{\msf{H}_{\ov\pi}}(\tilde{T}_{\ov\pi})=\jmath_{\msf{H}_{\ov\pi}}(T_{\ov\pi})$ for $\mu'$-almost all $\pi$. Map $\mc{V}_1$ is isometric:
\[\begin{split} 
&\quad\;\bigl\|\int_{\IrrG}^\oplus \rho_{\pi}\circ\jmath_{\msf{H}_{\ov\pi}}(T_{\ov\pi})^*\md\mu'(\pi)\bigr\|^2=
\int_{\IrrG}\Tr(\jmath_{\msf{H}_{\ov\pi}}(T_{\ov\pi})
\jmath_{\msf{H}_{\ov\pi}}(T_{\ov\pi})^*)\md\mu'(\pi)\\
&=
\int_{\IrrG} \Tr( T_{\ov\pi}^* T_{\ov\pi})\md\mu'(\pi)=
\int_{\IrrG}\Tr(T_{\pi}^* T_{\pi})\md\mu(\pi)=
\bigl\|\int_{\IrrG}^\oplus T_{\pi}\md\mu(\pi)\bigr\|^2.
\end{split}\]
The same argument shows that the map
\[\begin{split}
\int_{\IrrG}^{\oplus} &\HS(\msf{H}_\pi)\md\mu'(\pi)\ni \int_{\IrrG}^\oplus T_\pi\md\mu'(\pi)\\
&\mapsto
\int_{\IrrG}^\oplus 
\rho_\pi\circ\jmath_{\msf{H}_{\ov\pi}}(T_{\ov\pi})^*\md\mu(\pi)
\in
\int_{\IrrG}^{\oplus} \HS(\msf{H}_\pi)\md\mu(\pi).
\end{split}\]
is well defined, it is clear that it is an inverse to $\mc{V}_1$. Consequently, $\mc{V}_1$ is an antiunitary. Since
\[\begin{split}
&\quad\;
\rho_{\pi}\circ\jmath_{\msf{H}_{\ov\pi}}( (\alpha\otimes\id)U^{\ov\pi})^*=
\rho_{\pi}\circ\jmath_{\msf{H}_{\ov\pi}}\circ
\rho_{\ov{\pi}}\circ\jmath_{\msf{H}_\pi}
\bigl(((\alpha\circ R\otimes\id)U^{\pi})^* \bigr)\\
&=
(((\alpha\circ R)\otimes\id)U^{\pi})^*=
((\alpha\circ R)^{\sharp}\otimes\id)U^\pi
\end{split}\]
and $
(\alpha\circ R)^{\sharp}=\alpha^{\sharp}\circ R$ for $\alpha\in\LL^1_{\sharp}(\GG),\,\pi\in\IrrG$, the operator $\mc{V}_1$ satisfies
\[
\mc{V}_1 \bigl(\int_{\IrrG}^{\oplus} 
((\alpha\circ R)^{\sharp}\otimes\id)U^\pi\, E_{\pi}^{-1} \md\mu(\pi)\bigr)=
\int_{\IrrG}^{\oplus} 
(\alpha\otimes\id)U^{\pi}
\rho_{\pi}\circ
\jmath_{\msf{H}_{\ov\pi}}(E_{\ov\pi}^{-1})\md\mu'(\pi)
\]
for $\alpha\circ R\in\LL^1_{\sharp}(\GG)\cap \ov{\mc{I}_R}$.\\
If we precompose this map with $\mc{Q}_R$, $\hat{J}J$, $J$ and $\mc{Q}_L^{-1}$ we get a unitary operator
\[
\mc{V}=\mc{V}_1\circ\mc{Q}_R\circ (\hat{J}J)\circ J\circ\mc{Q}_L^{-1}
\colon \int_{\IrrG}^{\oplus}\HS(\msf{H}_\pi) \md\mu(\pi)
\rightarrow
\int_{\IrrG}^{\oplus}\HS(\msf{H}_\pi) \md\mu'(\pi)
\]
satisfying
\[
\begin{split}
&\quad\;\mc{V} \bigl(
\int_{\IrrG}^{\oplus} (\alpha\otimes\id)U^\pi D_\pi^{-1}\md\mu(\pi)\bigr)
=\mc{V}_1 \circ \mc{Q}_R \circ (\hat{J}J)\circ J (\Lhvp(\lambda(\alpha)))\\
&=
\mc{V}_1\circ \mc{Q}_R (\hat{J}J \Lambda_{\hpsi} (\hat{R}(\lambda(\alpha))^*)=
\mc{V}_1\circ \mc{Q}_R (\hat{J}J \Lambda_{\hpsi} (\lambda(\alpha^{\sharp}\circ R)))\\
&=
\mc{V}_1\bigl(\int_{\IrrG}^{\oplus} ((\alpha\circ R)^{\sharp}\otimes\id)U^\pi E_\pi^{-1}\md\mu(\pi)\bigr)
=
\int_{\IrrG}^{\oplus} (\alpha\otimes\id)U^\pi 
\rho_\pi\circ\jmath_{\msf{H}_{\ov\pi}}(E_{\ov\pi}^{-1})\md\mu'(\pi)
\end{split}
\]
for $\alpha\in\LL^1_{\sharp}(\GG)$ such that $\lambda(\alpha)\in\mf{N}_{\hvp}$ and $\alpha\circ R\in \ov{\mc{I}_R}$.\\
Operator $\mc{V}$ maps diagonalisable operators to diagonalisable operators\footnote{Recall that an operator is diagonalisable if it can be written as $\int_{\IrrG}^{\oplus} f(\pi) \I_{\HS(\msf{H}_\pi)} \md\mu(\pi)$ for a scalar valued function $f$.}. Indeed: we know that operators $\mc{Q}_L,\mc{Q}_R$ transforms diagonalisable operators into $\LL^{\infty}(\whG)\cap{\LL^{\infty}(\whG)}'$ (point 6) of theorems \ref{PlancherelL}, \ref{PlancherelR}). Next, we have $
\hat{J}\LL^{\infty}(\whG)\hat{J}=\LL^{\infty}(\whG)'$,
consequently, conjugation by $\hat{J}$ preserves $\LL^{\infty}(\whG)\cap\LL^{\infty}(\whG)'$. We need to check that the operator $\mc{V}_1$ transforms diagonalisable operators to diagonalisable operators -- it follows directly from the definition: we have
\[\begin{split}
&\quad\;
\mc{V}_1^* \bigl( \int_{\IrrG}^{\oplus} C(\pi)\I_{\HS(\msf{H}_\pi)}\md\mu'(\pi)
\bigr) \mc{V}_1 \bigl(
\int_{\IrrG}^{\oplus} T_\pi\md\mu(\pi)\bigr)\\
&=
\mc{V}_1^* \bigl(
\int_{\IrrG}^{\oplus} C(\pi)\rho_\pi\circ\jmath_{\msf{H}_{\ov\pi}}(T_{\ov\pi})^*\md\mu'(\pi)\bigr)\\
&=
\mc{V}_1^* \bigl(
\int_{\IrrG}^{\oplus} \rho_\pi\circ\jmath_{\msf{H}_{\ov\pi}}(\ov{C(\pi)}T_{\ov\pi})^*\md\mu'(\pi)\bigr)\\
&=
\int_{\IrrG}^{\oplus} \ov{C(\ov\pi)} T_\pi\md\mu(\pi)\\
&=
\bigl( \int_{\IrrG}^{\oplus} \ov{C(\ov\pi)}\I_{\HS(\msf{H}_\pi)}\md\mu(\pi)
\bigr) \bigl(
\int_{\IrrG}^{\oplus} T_\pi\md\mu(\pi)\bigr)
\end{split}
\]
for any $\mu'$-almost everywhere bounded measurable function $C\colon\IrrG\rightarrow\CC$ and arbitrary vector $\int_{\IrrG}^{\oplus} T_\pi\md\mu(\pi)\in\int_{\IrrG}^{\oplus}\HS(\msf{H}_\pi)\md\mu(\pi)$.\\
 Let
\[
\mc{Q}'_L=\mc{V}\circ \mc{Q}_L\colon
\LdG\rightarrow \int_{\IrrG}^{\oplus}\HS(\msf{H}_\pi) \md\mu'(\pi).
\]
It is clear that it maps $\LL^{\infty}(\whG)\cap\LL^{\infty}(\whG)'$ into diagonalisable operators, because $\mc{Q}_L$ does so and $\mc{V}$ maps diagonalisable operators to diagonalisable operators. In order to make use of Lemma \ref{lemat13} (in the Plancherel measure version) we need to check that
\[	
\mc{Q}'_L (\omega\otimes\id)\mrW=
\bigl(\int_{\IrrG}^{\oplus} (\omega\otimes\id)U^{\pi}\otimes\I_{\ov{\msf{H}_{\pi}}}\md\mu'(\pi)\bigr)\mc{Q}'_L.
\]
for $\omega\in\Lj$. We already know that the following holds
\[
\mc{Q}_L (\omega\otimes\id)\mrW=
\bigl(\int_{\IrrG}^{\oplus} (\omega\otimes\id)U^{\pi}\otimes\I_{\ov{\msf{H}_{\pi}}}\md\mu(\pi)\bigr)\mc{Q}_L,
\]
therefore it is enough to check that
\[
\mc{V}\bigl(\int_{\IrrG}^{\oplus} (\omega\otimes\id)U^{\pi}\otimes\I_{\ov{\msf{H}_{\pi}}}\md\mu(\pi)\bigr)=
\bigl(\int_{\IrrG}^{\oplus} (\omega\otimes\id)U^{\pi}\otimes\I_{\ov{\msf{H}_{\pi}}}\md\mu'(\pi)\bigr) \mc{V}.
\]
Assume that we have $\omega,\alpha\in\Ljsharp$ such that $\alpha\circ R,\omega\circ R\in\Ljsharp\cap \ov{\mc{I}_R}$ and $\lambda(\alpha)\in\mf{N}_{\hvp}$. Then $\lambda(\omega\star\alpha)=\lambda(\omega)\lambda(\alpha)\in\mf{N}_{\hvp}$ and
\[
\ov{(\omega\star \alpha)\circ R}=\ov{(\alpha\circ R)\star(\omega\circ R)}=
\ov{\alpha\circ R}\star \ov{\omega\circ R}\in\mc{I}_R
\]
due to Lemma \ref{lemat21}. We know that functionals $\omega$ as above form a dense subspace (Lemma \ref{lemat28}). For those functionals we can calculate
\[\begin{split} 
&\quad\;
\mc{V}\bigl(\int_{\IrrG}^{\oplus} (\omega\otimes\id)U^{\pi}\otimes\I_{\ov{\msf{H}_{\pi}}}\md\mu(\pi)\bigr)
\int_{\IrrG}^{\oplus} (\alpha\otimes\id)U^\pi D_\pi^{-1}\md\mu(\pi)\\
&=
\mc{V} \int_{\IrrG}^{\oplus}
(\omega\star\alpha\otimes\id)U^{\pi} D_\pi^{-1}
\md\mu(\pi)\\
&=
\int_{\IrrG}^{\oplus}
(\omega\star\alpha\otimes\id)U^{\pi} \rho_\pi\circ\jmath_{\msf{H}_{\ov\pi}}(E_{\ov\pi}^{-1})
\md\mu'(\pi)\\
&=
\bigl(\int_{\IrrG}^{\oplus} (\omega\otimes\id)U^{\pi}\otimes\I_{\ov{\msf{H}_{\pi}}}\md\mu'(\pi)\bigr)
\int_{\IrrG}^{\oplus}
(\alpha\otimes\id)U^{\pi} \rho_\pi\circ\jmath_{\msf{H}_{\ov\pi}}(E_{\ov\pi}^{-1})
\md\mu'(\pi)\\
&=
\bigl(\int_{\IrrG}^{\oplus} (\omega\otimes\id)U^{\pi}\otimes\I_{\ov{\msf{H}_{\pi}}}\md\mu'(\pi)\bigr)
\mc{V}\int_{\IrrG}^{\oplus}
(\alpha\otimes\id)U^{\pi} D_\pi^{-1}
\md\mu(\pi).
\end{split}\]
Lemma \ref{lemat28} gives us density of vectors $\int_{\IrrG}^{\oplus}(\alpha\otimes\id)U^\pi D_\pi^{-1}\md\mu(\pi)$, which finishes the claim.
\end{proof}

For $\beta\in\LL^1(\GG)$ such that $\lambda(\beta)\ge 0$ we have
\[\begin{split}
\hat{\vp}(\lambda(\beta))&=\int_{\IrrG} 
\Tr\bigl(\pi((\beta\otimes\id){\WW}^{\GG})D_\pi^{-2}\bigr) \md\mu(\pi),\\
\hat{\psi}(\lambda(\beta))&=\int_{\IrrG} 
\Tr\bigl(\pi((\beta\otimes\id){\WW}^{\GG})E_{\pi}^{-2}\bigr) \md\mu(\pi).
\end{split}\]
Moreover, for such $\beta$ we have
\[
0\le\hat{R}(\lambda(\beta))=
(\beta\otimes\hat{R})\mrW=
(\beta\circ R\otimes\id)\mrW=
\lambda(\beta\circ R),
\]
and consequently, since $\hpsi=\hvp\circ\hat{R}$
\begin{equation}\label{eq17}
\begin{split}
&\quad\;\int_{\IrrG} 
\Tr\bigl(\pi((\beta\circ R\otimes\id){\WW}^{\GG})D_\pi^{-2}\bigr) \md\mu(\pi)=
\hat{\vp}(\lambda(\beta\circ R))=
\hat{\vp}(\hat{R}(\lambda(\beta)))\\
&=
\hat{\psi}(\lambda(\beta))
=
\int_{\IrrG} 
\Tr\bigl(\pi((\beta\otimes\id){\WW}^{\GG})E_{\pi}^{-2}\bigr) \md\mu(\pi)
\end{split}
\end{equation}

We can make use of this observation to derive a relation between $\Tr(E^{-2}_{\pi})$ and $\Tr(D_{\ov{\pi}}^{-2})$:

\begin{proposition}\label{stw13}
Let $\mu'$ be the Plancherel measure $\mu$ composed with conjugation. We have $\Tr(E_\pi^{-2})=\Tr(D_{\ov\pi}^{-2})\tfrac{\md\mu'}{\md\mu}(\pi)$ for almost all $\pi\in \IrrG$.
\end{proposition}

\begin{proof}
Let us fix a meaurable subset of finite measure $\Omega\subseteq\IrrG$ such that $\int_{\Omega}\Tr(E^{-2}_{\pi})\md\mu(\pi)<+\infty$. Let $(\beta_n)_{n\in\NN}$ be a sequence in $\LL^1_{\sharp}(\GG)$ such that $\lambda(\beta_n)\in\mf{N}_{\hpsi}$ and
\[
\int_{\IrrG}^{\oplus} \pi(\lambda^u(\beta_n)) E_\pi^{-1}
\md\mu(\pi)\xrightarrow[n\to\infty]{} \int_{\IrrG}^{\oplus} \chi_{\Omega}(\pi) E_\pi^{-1} \md\mu(\pi).
\]
Such a sequence exists due to Lemma \ref{lemat28} and the fact that $\mc{Q}_R$ is unitary. Taking subsequence, we can assume that $\pi(\lambda^u(\beta_n))E_\pi^{-1}\xrightarrow[n\to\infty]{\HS(\msf{H}_\pi)} \chi_{\Omega}(\pi) E_\pi^{-1}$ almost everywhere. Observe that thanks to the equation \eqref{eq17} we have
\[
\begin{split}
&\quad\;\bigl\|\int_{\IrrG}^{\oplus} \pi(\lambda^u(\beta_n-\beta_m)) E_\pi^{-1}
\md\mu(\pi)\bigr\|^2\\
&=
\int_{\IrrG} \Tr\bigl(
\pi(\lambda^u((\beta_n-\beta_m)^{\sharp}\star
(\beta_n-\beta_m)))
E_\pi^{-2}
\bigr)\md\mu(\pi)\\
&=
\int_{\IrrG} \Tr\bigl(
\pi(\lambda^u(((\beta_n-\beta_m)^{\sharp}\star
(\beta_n-\beta_m))\circ R))
D_\pi^{-2}
\bigr)\md\mu(\pi)\\
&=
\int_{\IrrG} \Tr\bigl(
\pi\circ \hat{R}^{u}(\lambda^u(\beta_n-\beta_m))
\pi\circ \hat{R}^{u}(\lambda^u(\beta_n-\beta_m))^*
D_\pi^{-2}
\bigr)\md\mu(\pi)\\
&=
\int_{\IrrG} \Tr_{\ov{\msf{H}_\pi}}\bigl(
\jmath_{\msf{H}_\pi}\circ\pi\circ \hat{R}^{u}(\lambda^u(\beta_n-\beta_m))^*
\jmath_{\msf{H}_\pi}\circ\pi\circ \hat{R}^{u}(\lambda^u(\beta_n-\beta_m))
\jmath_{\msf{H}_\pi}(D_\pi^{-2})
\bigr)\md\mu(\pi)\\
&=
\int_{\IrrG} \Tr_{\msf{H}_\pi}\bigl(\bigl(
\ov{\pi}(\lambda^u(\beta_n-\beta_m))
\rho_{\ov{\pi}}\circ\jmath_{\msf{H}_\pi}(D_\pi^{-1})\bigr)^*
\bigl(
\ov{\pi}(\lambda^u(\beta_n-\beta_m))
\rho_{\ov{\pi}}\circ\jmath_{\msf{H}_\pi}(D_\pi^{-1})\bigr)
\bigr)\md\mu(\pi)\\
&=
\bigl\| \int_{\IrrG}^{\oplus}
\ov{\pi}(\lambda^{u}(\beta_n-\beta_m)) 
\rho_{\ov{\pi}}\circ\jmath_{\msf{H}_\pi}(D_\pi^{-1}) \md\mu(\pi)\bigr\|^2,
\end{split}
\]
which shows that the sequence $\bigl( \int_{\IrrG}^{\oplus}
\ov{\pi}(\lambda^{u}(\beta_n)) \rho_{\ov{\pi}}\circ\jmath_{\msf{H}_\pi}(D_\pi^{-1})\md\mu(\pi) \bigr)_{n\in\NN}$ is a sequence in $\int_{\IrrG}^{\oplus} \HS(\msf{H}_\pi)\md\mu(\pi)$, convergent to some $\int_{\IrrG}^{\oplus} T_\pi \md\mu(\pi)$.
Above, we have made use of the following relations:
\[
\lambda^u((\beta^{\sharp}\star\beta)\circ R)=
\hat{R}^{u}(\lambda^u(\beta^{\sharp}\star\beta))=
\hat{R}^{u} (\lambda^u(\beta)^*\lambda^u(\beta))=
\hat{R}^{u} (\lambda^u(\beta))
\hat{R}^{u} (\lambda^u(\beta))^*
\]
and the fact that $\jmath_{\msf{H}_\pi}$ preserves the trace. The same calculation gives us
\[
\bigl\|\int_{\IrrG}^{\oplus} \pi(\lambda^u(\beta_n)) E_\pi^{-1}
\md\mu(\pi)\bigr\|=
\bigl\| \int_{\IrrG}^{\oplus}
\ov{\pi}(\lambda^{u}(\beta_n)) 
\rho_{\ov{\pi}}\circ\jmath_{\msf{H}_\pi}(D_\pi^{-1})\md\mu(\pi)\bigr\|\quad(n\in\NN).
\]
Again, taking subsequence $(n_k)_{k\in\NN}$ we can assume that
\[
\begin{split}
T_\pi
&=\lim_{k\to\infty}
\ov{\pi}(\lambda^{u}(\beta_{n_k})) 
\rho_{\ov{\pi}}\circ\jmath_{\msf{H}_\pi}(D_\pi^{-1})=
\lim_{k\to\infty}
\ov{\pi}(\lambda^u(\beta_{n_k})) E_{\ov{\pi}}^{-1}\,
E_{\ov{\pi}}\, \rho_{\ov{\pi}}\circ\jmath_{\msf{H}_\pi}(D_\pi^{-1})
\\
&=
\chi_{\Omega}(\ov\pi)
E_{\ov\pi}^{-1} E_{\ov{\pi}} \rho_{\ov{\pi}}\circ \jmath_{\msf{H}_\pi}(D_{\pi}^{-1})
=
\chi_{\ov\Omega}(\pi) \rho_{\ov{\pi}}\circ \jmath_{\msf{H}_\pi}(D_\pi^{-1})
\end{split}
\]
almost everywhere (the above limit is taken in the norm of $\HS(\msf{H}_\pi)$). In the above calculation we have used Proposition \ref{stw12}: $\mu$ and $\mu'$ are equivalent. We arrive at
\[
\begin{split}
&\quad\;\int_{\ov{\Omega}} \Tr(D_\pi^{-2})\md\mu(\pi)=
\int_{\ov{\Omega}} \Tr (\rho_{\ov{\pi}}\circ \jmath_{\msf{H}_\pi}(D_\pi^{-2}))\md\mu(\pi)=
\bigl\| \int_{\IrrG}^{\oplus} T_\pi\md\mu(\pi)\bigr\|^2\\
&=
\lim_{n\to\infty} \bigl\| 
\int_{\IrrG}^{\oplus}
\ov{\pi}(\lambda^{u}(\beta_{n})) 
\rho_{\ov{\pi}}\circ\jmath_{\msf{H}_\pi}(D_\pi^{-1}) \md\mu(\pi)
\bigr\|^2
=
\lim_{n\to\infty} \bigl\| 
\int_{\IrrG}^{\oplus}
\pi(\lambda^u(\beta_n)) E_\pi^{-1} \md\mu(\pi)\bigr\|^2\\
&=
\bigl\| \int_{\IrrG}^{\oplus} \chi_{\Omega}(\pi) E_\pi^{-1} \md\mu(\pi)
\bigr\|^2
=
\int_{\Omega} \Tr(E_\pi^{-2})\md\mu(\pi).
\end{split}
\]
Next, we get
\[
\begin{split}
\int_{\Omega} \Tr(E_\pi^{-2})\md\mu(\pi)&=
\int_{\IrrG} \chi_{\ov\Omega}(\pi) \Tr(D_\pi^{-2})\md\mu(\pi)=
\int_{\IrrG} \chi_{\Omega}(\pi) \Tr(D_{\ov\pi}^{-2}) \md\mu'(\pi)\\
&=
\int_{\Omega}\Tr(D_{\ov\pi}^{-2}) \tfrac{\md\mu'}{\md\mu}(\pi)\md\mu(\pi).
\end{split}
\]
Because $\Omega$ was arbitrary, we have proven the claim.
\end{proof}

We can get a similar relation between $\Tr(E_\pi^2)$ and $\Tr(D_{\ov\pi}^2)$ this time using the formula $\psi=\vp\circ R$ for the Haar integrals on $\GG$:

\begin{proposition}\label{stw5}
Let $\mu'$ be the Plancherel measure $\mu$ composed with conjugation. We have $\Tr(E_\pi^2)=\tfrac{\md\mu}{\md\mu'}(\pi) \Tr(D_{\ov\pi}^2)$ for almost all $\pi\in\IrrG$.
\end{proposition}

\begin{proof}
For good vector fields, as in Proposition \ref{Casp}, we have
\[
\begin{split}
\psi\bigl(\bigl(
\int_{\IrrG} (\id\otimes\omega_{\xi_\pi,\eta_\pi})U^{\pi} \md\mu(\pi)
\bigr)^*&\bigl(
\int_{\IrrG} (\id\otimes\omega_{\xi'_\pi,\eta'_\pi})U^{\pi} \md\mu(\pi)
\bigr)\bigr)\\
&=
\int_{\IrrG}\is{\xi'_\pi}{\xi_\pi}
\is{E_\pi \eta_\pi}{E_\pi \eta'_\pi}\md\mu(\pi)\\
\vp\bigl(\bigl(
\int_{\IrrG} (\id\otimes\omega_{\xi_\pi,\eta_\pi})(U^{\pi *}) \md\mu(\pi)
\bigr)^*&\bigl(
\int_{\IrrG} (\id\otimes\omega_{\xi'_\pi,\eta'_\pi})(U^{\pi *}) \md\mu(\pi)
\bigr)\bigr)\\
&=
\int_{\IrrG}\is{\xi'_\pi}{\xi_\pi}
\is{D_\pi \eta_\pi}{D_\pi \eta'_\pi}\md\mu(\pi).
\end{split}
\]
We know that isomorphisms $\rho_{\ov{\pi}}$ are given by conjugation with unitaries: denote these unitaries by $\mc{V}_{\ov{\pi}}\colon \ov{\CC^{\dim(\pi)}}\rightarrow \CC^{\dim(\pi)}$. Observe that for any $\xi,\eta\in \msf{H}_\pi$ and $\omega\in \Lj$ the following holds
\[
\begin{split}
&\quad\;
\omega(R((\id\otimes\omega_{\xi,\eta}) U^\pi))=
\ismaa{\xi}{ 
\pi\circ \hat{R}^{u}\circ\lambda^u(\omega)
\eta}
=
\ismaa{\xi}{ \jmath_{\ov{\msf{H}}_\pi}\circ
\rho_{\ov{\pi}}^{-1}
( (\omega\otimes\id)U^{\ov\pi})\eta}\\
&=
\ismaa{\xi}{\ov{ \rho_{\ov\pi}^{-1}((\omega\otimes\id)U^{\ov{\pi}})^* \ov{\eta}}}
=
\ismaa{ \rho_{\ov\pi}^{-1}((\omega\otimes\id)U^{\ov{\pi}})^* \ov{\eta}}{\ov{\xi}}
=
\ismaa{ \ov{\eta}}{\rho_{\ov{\pi}}^{-1}((\omega\otimes\id)U^{\ov{\pi}})\,\ov{\xi}}\\
&=
\ismaa{ \mc{V}_{\ov\pi}\ov{\eta}}{
((\omega\otimes\id)U^{\ov{\pi}})\,\mc{V}_{\ov\pi}\ov{\xi}}=
\omega((\id\otimes\omega_{\mc{V}_{\ov\pi}\ov{\eta},\mc{V}_{\ov\pi}\ov{\xi}})U^{\ov{\pi}}),
\end{split}
\]
hence
\[
R((\id\otimes\omega_{\xi,\eta}) U^\pi)=
(\id\otimes\omega_{\mc{V}_{\ov\pi}\ov{\eta},\mc{V}_{\ov\pi}\ov{\xi}})U^{\ov{\pi}}\quad (\xi,\eta\in\msf{H}_\pi).
\]
Fix a field of orthonormal bases $\{(\xi^k_\pi)_{\pi\in\IrrG}\}_{k=1}^{\infty}$ and a measurable subset of finite measure $\Omega\subseteq\IrrG\rest_p$ for some $p\in \NN$ such that
\[
\sup_{\pi\in\Omega} \Tr(E^2_{\pi}),\quad
\sup_{\pi\in\Omega} \tfrac{\md\mu}{\md\mu'}(\pi)^2 \Tr(D_{\ov\pi}^2)<+\infty.
\]
Observe that for any measurable subset $V\subseteq\IrrG$ we have
\[
\int_{V} \tfrac{\md\mu}{\md\mu'}(\ov{\pi})\md\mu(\pi)=
\int_{\IrrG}\chi_{V}(\ov\pi)\tfrac{\md\mu}{\md\mu'}(\pi)
\md\mu'(\pi)=
\int_{\ov{V}}\md\mu=\int_V\md\mu'=
\int_V\tfrac{\md\mu'}{\md\mu} \md\mu,
\]
hence
\[
\tfrac{\md\mu}{\md\mu'}(\ov\pi)=\tfrac{\md\mu'}{\md\mu}(\pi)
\]
almost everywhere. Using the above formulas, \swot-continuity of $R$ and $\psi=\vp\circ R$ we get for any $k\in \{1,\dotsc,p\}$
\[
\begin{split}
&\quad\;
\int_{\IrrG}\chi_{\Omega}(\pi)\ismaa{\xi^k_\pi}{\xi^k_\pi}
\ismaa{E_\pi \xi^k_\pi}{E_\pi \xi^k_\pi}\md\mu(\pi)\\
&=
\psi\bigl(\bigl(
\int_{\IrrG} (\id\otimes\omega_{\chi_{\Omega}(\pi)\xi^k_\pi})(U^{\pi}) \md\mu(\pi)
\bigr)^*\bigl(
\int_{\IrrG} (\id\otimes\omega_{\chi_{\Omega}(\pi)\xi^k_\pi})(U^{\pi}) \md\mu(\pi)
\bigr)\bigr)\\
&=
\vp\circ R\bigl(\bigl(
\int_{\IrrG} (\id\otimes\omega_{\chi_{\Omega}(\pi)\xi^k_\pi})(U^{\pi}) \md\mu(\pi)
\bigr)^*\bigl(
\int_{\IrrG} (\id\otimes\omega_{\chi_{\Omega}(\pi)\xi^k_\pi})(U^{\pi}) \md\mu(\pi)
\bigr)\bigr)\\
&=
\vp\bigl(\bigl(
\int_{\IrrG} (\id\otimes\omega_{\mc{V}_{\ov\pi}\,\ov{\chi_{\Omega}(\pi)\xi^k_\pi}})(U^{\ov\pi}) \md\mu(\pi)
\bigr)\bigl(
\int_{\IrrG} (\id\otimes\omega_{\mc{V}_{\ov\pi}\,\ov{\chi_{\Omega}(\pi)\xi^k_\pi}})(U^{\ov\pi}) \md\mu(\pi)
\bigr)^*\bigr)\\
&=
\vp\bigl(\bigl(
\int_{\IrrG} 
\tfrac{\md\mu}{\md\mu'}(\pi)(\id\otimes\omega_{\mc{V}_{\ov\pi}\,\ov{\chi_{\Omega}(\pi)\xi^k_\pi}})(U^{\ov\pi}) \md\mu'(\pi)
\bigr)\bigl(
\int_{\IrrG} 
\tfrac{\md\mu}{\md\mu'}(\pi)(\id\otimes\omega_{\mc{V}_{\ov\pi}\,\ov{\chi_{\Omega}(\pi)\xi^k_\pi}})(U^{\ov\pi}) \md\mu'(\pi)
\bigr)^*\bigr)\\
&=
\vp\bigl(\bigl(
\int_{\IrrG} 
\tfrac{\md\mu}{\md\mu'}(\ov\pi)(\id\otimes\omega_{\mc{V}_{\pi}\,\ov{\chi_{\Omega}(\ov\pi)\xi^k_{\ov\pi}}})(U^{\pi}) \md\mu(\pi)
\bigr)\bigl(
\int_{\IrrG} 
\tfrac{\md\mu}{\md\mu'}(\ov\pi)(\id\otimes\omega_{\mc{V}_{\pi}\,\ov{\chi_{\Omega}(\ov\pi)\xi^k_{\ov\pi}}})(U^{\pi}) \md\mu(\pi)
\bigr)^*\bigr)\\
&=
\vp\bigl(\bigl(
\int_{\IrrG} 
\tfrac{\md\mu}{\md\mu'}(\ov\pi)(\id\otimes\omega_{\mc{V}_{\pi}\,\ov{\chi_{\Omega}(\ov\pi)\xi^k_{\ov\pi}}})(U^{\pi *}) \md\mu(\pi)
\bigr)^*\bigl(
\int_{\IrrG} 
\tfrac{\md\mu}{\md\mu'}(\ov\pi)(\id\otimes\omega_{\mc{V}_{\pi}\,\ov{\chi_{\Omega}(\ov\pi)\xi^k_{\ov\pi}}})(U^{\pi *}) \md\mu(\pi)
\bigr)\bigr)\\
&=
\int_{\IrrG}\tfrac{\md\mu}{\md\mu'}(\ov\pi)^2
\chi_{\Omega}(\ov\pi)\ismaa{\mc{V}_{\pi}\, \ov{\xi^k_{\ov\pi}}}{\mc{V}_{\pi}\,\ov{\xi^k_{\ov\pi}}}
\ismaa{ D_\pi \mc{V}_{\pi}\,\ov{\xi^k_{\ov\pi}}}{D_\pi\mc{V}_{\pi}\,\ov{\xi^k_{\ov\pi}}} \md\mu(\pi)\\
&=
\int_{\IrrG}\tfrac{\md\mu}{\md\mu'}(\ov\pi)
\chi_{\Omega}(\ov\pi)
\ismaa{ D_\pi \mc{V}_{\pi}\,\ov{\xi^k_{\ov\pi}}}{D_\pi\mc{V}_{\pi}\,\ov{\xi^k_{\ov\pi}}} \md\mu'(\pi)\\
&=
\int_{\IrrG}\tfrac{\md\mu}{\md\mu'}(\pi)
\chi_{\Omega}(\pi)
\ismaa{ D_{\ov\pi}\mc{V}_{\ov\pi}\, \ov{\xi^k_{\pi}}}{D_{\ov\pi}\mc{V}_{\ov\pi}\,\ov{\xi^k_{\pi}}} \md\mu(\pi)\\
&=
\int_{\IrrG}\tfrac{\md\mu}{\md\mu'}(\pi)
\chi_{\Omega}(\pi)
\ismaa{ \mc{V}_{\ov\pi}\,\ov{\xi^k_{\pi}}}{D_{\ov\pi}^2\mc{V}_{\ov\pi}\,\ov{\xi^k_{\pi}}} \md\mu(\pi).
\end{split}
\]
If we take a sum over $k\in\{1,\dotsc,p\}$ we get
\[
\int_\Omega \Tr(E_\pi^2) \md\mu(\pi)=
\int_\Omega \tfrac{\md\mu}{\md\mu'}(\pi) \Tr(D^2_{\ov\pi})\md\mu(\pi)
\]
Since $\Omega$ was somewhat arbitrary, we arrive at
\[
\Tr(E_\pi^2)=\tfrac{\md\mu}{\md\mu'}(\pi)\Tr(D_{\ov{\pi}}^2)
\]
almost everywhere.
\end{proof}

Observe that there always exists a Plancherel measure which is symmetric in the sense that $\mu'=\mu$: it is enough to define a measure $\tilde{\mu}=\mu+\mu'$ (which is equivalent to $\mu$) and use Proposition \ref{stw14}. \\

Recall that in Definition \ref{defadmissible} we have introduced a notion of admissibility of finite dimensional representations.\\
For any finite dimensional representation $U$ we define its character in the usual way: $\chi(U)=(\id\otimes\Tr)U$. It is well known that unitarily equivalent representations have the same character. The next lemma says that the character of the conjugate representation is adjoint of the character of the original representation, provided that the representation is admissible:

\begin{lemma}\label{lemat32}
Let $U\in \M(\CG\otimes\mc{K}(\msf{H}))$ be a finite dimensional representation of $\GG$ which is admissible and let $U^c$ be a representation conjugate to $U$. Then $\chi(U^c)=\chi(U)^*$.
\end{lemma}

The above result is a consequence of \cite[Remark 3.5]{DasDawsSalmi}. Now we use this result to calculate adjoints of integral characters.

\begin{lemma}\label{lemat10}
Let $\Omega\subseteq \IrrG$ be a measurable subset such that $\int_{\Omega}\dim \md\mu<+\infty$. Assume that $\mu$ is a Plancherel measure invariant under $\IrrG\ni\pi\mapsto\ov{\pi}\in\IrrG$ and $\Omega$ contains only admissible representations. In this situation:
\begin{enumerate}[label=\arabic*)]
\item we have $\chi^{\int} ( \int_{\Omega}^{\oplus}\pi\md\mu_{\Omega}(\pi))^*=
\chi^{\int} ( \int_{\ov{\Omega}}^{\oplus}\pi\md\mu_{\ov{\Omega}}(\pi))$,
\item if $\Omega=\ov{\Omega}$ then the integral character $\int_{\Omega}^{\oplus}\pi\md\mu_{\Omega}(\pi)$ is self-adjoint.
\end{enumerate}
\end{lemma}

\begin{proof}
We have
\[
\begin{split}
&\quad\;\chi^{\int} ( \int_{\Omega}^{\oplus}\pi\md\mu_{\Omega}(\pi))^*=
\int_{\Omega} \chi(U^{\pi})^*\md\mu_{\Omega}(\pi)=
\int_{\Omega} \chi(U^{\ov{\pi}})\md\mu_{\Omega}(\pi)\\
&=
\int_{\ov{\Omega}}\chi(U^{\pi})\md\mu_{\ov{\Omega}}(\pi)
=
\chi^{\int} ( \int_{\ov{\Omega}}^{\oplus}\pi\md\mu_{\ov{\Omega}}(\pi)).
\end{split}
\]
In the above calculation we have used the invariance of the measure $\mu$, equation $\chi(U^{\ov{\pi}})=\chi(U^{\pi})^*$ and the fact that the integral in the definition of an integral character is \swot\, convergent.
\end{proof}

\begin{definition}
A subset $\Omega\subseteq\IrrG$ such that $\ov{\Omega}=\Omega$ will be called \emph{symmetric}.
\end{definition}

Let us recall the definition of a coamenable locally compact quantum group (see e.g.~\cite[Theorem 3.12]{Brannan}): 

\begin{definition}\label{defcoamenable}
A locally compact quantum group $\GG$ is \emph{coamenable} if there exists a net $(\xi_i)_{i\in \mc{I}}$ of unit vectors in $\LdG$ such that
\[
\| \mrW (\xi_i\otimes \eta)-\xi_i\otimes\eta\|\xrightarrow[i\in\mc{I}]{}0
\]
for all $\eta\in \LdG$.
\end{definition}

One easily sees that whenever $\LdG$ is separable (which is the case in our situation), $\GG$ is coamenable if and only if there exists a sequence $(\xi_n)_{n\in\NN}$ of unit vectors in $\LdG$ which satisfies the above condition. Now we are able to prove the main theorem of this section which relates coamenability of $\GG$ to the properties of spectra of integral characters.

\begin{theorem}\label{tw2}
Let $\GG$ be a second countable locally compact quantum group. Assume moreover that $\GG$ is type I and has only finite dimensional irreducible representations. Consider the following conditions:
\begin{enumerate}[label=\arabic*)]
\item $\GG$ is coamenable.
\item For any Plancherel measure $\mu$ and any measurable subset $\Omega\subseteq \IrrG$ such that $\int_{\Omega}\dim\md\mu<+\infty$ we have $\int_{\Omega}\dim\md\mu\in \sigma(\chi^{\int}(\int_{\Omega}^{\oplus}\pi\md\mu_{\Omega}(\pi)))$.
\item
Let $\mu$ be a Plancherel measure which is invariant under $\IrrG\ni\pi\mapsto\ov{\pi}\in\IrrG$. For any symmetric measurable subset $\Omega\subseteq \IrrG$ such that $\int_{\Omega}\dim\md\mu<+\infty$ we have $\int_{\Omega}\dim\md\mu\in \sigma(\chi^{\int}(\int_{\Omega}^{\oplus}\pi\md\mu_{\Omega}(\pi)))$.
\end{enumerate}
We have $1)\Rightarrow 2)\Rightarrow 3)$. If all irreducible representations of $\GG$ are admissible then also $3)\Rightarrow 1)$.
\end{theorem}

\begin{proof}
Implication $2)\Rightarrow 3)$ is trivial. Assume that point $3)$ holds and $\mu$ is a Plancherel measure invariant under conjugation. Let $\Upsilon$ be the family of symmetric measurable subsets $\Omega\subseteq\IrrG$ such that $\int_{\Omega}\dim \md\mu <+\infty$. On $\Upsilon$ we have partial ordering given by inclusion, with which $\Upsilon$ becomes a directed set, moreover $\bigcup \Upsilon=\IrrG$. For any $\Omega\in\Upsilon$ the integral character $\chi^{\int}(\int_{\Omega}^{\oplus}\pi\md\mu_{\Omega}(\pi))$ is self-adjoint and has $\int_\Omega\dim\md\mu$ in its spectrum. Therefore, there exists a sequence of approximate eigenvectors with eigenvalue $\int_{\Omega}\dim \md\mu$. Using them, we can build vector functionals (normal states) $\omega^{\Omega,n}$ on $\Linf$ such that
\[
\omega^{\Omega,n}(\chi^{\int}(\int_{\Omega}^{\oplus}\pi \md\mu_{\Omega}(\pi)))\approx_{\tfrac{1}{n}}
\int_{\Omega}\dim \md\mu\quad(n\in\NN).
\]
For any measurable symmetric $\Omega'\subseteq \Omega$ we have
\[
\begin{split}
\int_{\Omega'}\dim \md\mu+
\int_{\Omega\setminus \Omega'}\dim \md\mu &=
\int_{\Omega}\dim\md\mu \approx_{\tfrac{1}{n}}
\omega^{\Omega,n}( \chi^{\int}(\int_{\Omega}^{\oplus} \pi \md\mu_{\Omega}(\pi)))\\
&=
\omega^{\Omega,n}(\chi^{\int}(\int_{\Omega'}^{\oplus}\pi \md\mu_{\Omega'}(\pi)))+
\omega^{\Omega,n}(\chi^{\int}(\int_{\Omega\setminus\Omega'}^{\oplus}\pi \md\mu_{\Omega\setminus\Omega'}(\pi)))
\end{split}
\]
and
\[
|\omega^{\Omega,n}(\chi^{\int}(\int_{\Omega'}^{\oplus}\pi \md\mu_{\Omega'}(\pi)))|\le
\int_{\Omega'}\dim \md\mu,\quad
|\omega^{\Omega,n}(\chi^{\int}(\int_{\Omega\setminus\Omega'}^{\oplus}\pi \md\mu_{\Omega\setminus\Omega'}(\pi)))|\le
\int_{\Omega\setminus\Omega'}\dim \md\mu.
\]
Consequently 
\begin{equation}\label{eq4}
\omega^{\Omega,n}(\chi^{\int}(\int_{\Omega'}^{\oplus}\pi \md\mu_{\Omega'}(\pi)))\approx_{\tfrac{1}{n}}
\int_{\Omega'}\dim \md\mu\quad(n\in\NN).
\end{equation}
We can find unit vectors $\xi^{\Omega,n}\in \LdG$ such that $\omega^{\Omega,n}|_{\Linf}=\omega_{\xi^{\Omega,n}}$. Let $\mc{Q}_L$ be a unitary operator from Theorem \ref{PlancherelL}. Recall that it satisfies
\[
\mc{Q}_L\, ((\nu\otimes\id)\mrW)=\bigl(\int_{\IrrG}^{\oplus}
(\nu\otimes\id)U^{\pi}\otimes\I_{\ov{\msf{H}_\pi}}\md\mu(\pi)\bigr)\;\mc{Q}_L
\quad(\nu\in\LL^1(\GG)).
\]
In order to prove that $\GG$ is coamenable (with a net of unit vectors $(\xi^{\Omega,n})_{(\Omega,n)\in\Upsilon\times\NN}$) we need to show that
\[
\|\mrW (\xi^{\Omega,n}\otimes\eta)-\xi^{\Omega,n}\otimes\eta\|
\xrightarrow[(\Omega,n)\in\Upsilon\times\NN]{}0
\]
for all $\eta\in\LdG$, equivalently
\[
\begin{split}
&\quad\;\bigl\langle\eta \big| (\omega_{\xi^{\Omega,n}}\otimes\id) (\tfrac{\mrW^*+\mrW}{2}) \eta-\eta\bigr\rangle\\
&=\Re\ismaa{\eta}{(\omega_{\xi^{\Omega,n}}\otimes\id)\mrW\eta-\eta}\\
&=\Re\ismaa{\xi^{\Omega,n}\otimes\eta}{
\mrW (\xi^{\Omega,n}\otimes\eta)-\xi^{\Omega,n}\otimes\eta}
\xrightarrow[(\Omega,n)\in\Upsilon\times\NN]{}0.
\end{split}
\]
Clearly, it is enough to show this convergence for $\eta$ from a linearly dense set. If we move everything to the level of direct integrals and multiply by $-1$, the above condition can be written as
\[
\ismaa{\eta}{\bigl(
\I-\tfrac{1}{2}\int_{\IrrG}^{\oplus} (\omega_{\xi^{\Omega,n}}\otimes\id)((U^{\pi} )^*+U^\pi)\otimes\I_{\ov{\msf{H}_\pi}}\md\mu(\pi)
\bigr)\eta}
\xrightarrow[(\Omega,n)\in\Upsilon\times\NN]{}0
\]
for $\eta\in\int_{\IrrG}^{\oplus}\HS(\msf{H}_\pi)\md\mu(\pi)$ from a lineary dense subset. Observe that
\[
\bigl\|
\tfrac{1}{2}\int_{\IrrG}^{\oplus} (\omega_{\xi^{\Omega,n}}\otimes\id)((U^{\pi} )^*+U^\pi)\otimes\I_{\ov{\msf{H}_\pi}}\md\mu(\pi)\bigr\|\le
1,
\]
consequently, as $\tfrac{1}{2}\int_{\IrrG}^{\oplus} (\omega_{\xi^{\Omega,n}}\otimes\id)((U^{\pi} )^*+U^\pi)\otimes\I_{\ov{\msf{H}_\pi}}\md\mu(\pi)$ is self-adjoint we have
\[
0\le \I-\tfrac{1}{2}\int_{\IrrG}^{\oplus} (\omega_{\xi^{\Omega,n}}\otimes\id)((U^{\pi} )^*+U^\pi)\otimes\I_{\ov{\msf{H}_\pi}}\md\mu(\pi).
\]
Similar reasoning gives us
\[
0\le \I_{\msf{H}_\pi}\otimes\I_{\ov{\msf{H}_\pi}}-\tfrac{1}{2}(\omega_{\xi^{\Omega,n}}\otimes\id)((U^{\pi} )^*+U^\pi)\otimes\I_{\ov{\msf{H}_\pi}}\quad(\pi\in\IrrG).
\]
Let $\eta=\int_{\IrrG}^{\oplus} \xi_{\pi}\otimes\ov{\zeta_\pi}\md\mu(\pi)\in\int_{\IrrG}^{\oplus}\HS(\msf{H}_\pi)\md\mu(\pi)$. We have
\[
\begin{split}
&\quad\;
\ismaa{\eta}{\bigl(
\I-\tfrac{1}{2}\int_{\IrrG}^{\oplus} (\omega_{\xi^{\Omega,n}}\otimes\id)((U^{\pi})^*+U^\pi)\otimes\I_{\ov{\msf{H}_\pi}}\md\mu(\pi)
\bigr)\eta}\\
&=
\int_{\IrrG}
\is{\xi_\pi\otimes\ov{\zeta_\pi}}{
\xi_\pi\otimes\ov{\zeta_\pi}-\tfrac{1}{2}((\omega_{\xi^{\Omega,n}}\otimes\id)((U^{\pi})^*+U^\pi) \xi_\pi)\otimes\ov{\zeta_\pi}
}
\md\mu(\pi)\\
&=
\int_{\IrrG}
\is{\xi_\pi}{
\xi_\pi-\tfrac{1}{2}(\omega_{\xi^{\Omega,n}}\otimes\id)((U^{\pi})^*+U^\pi) \xi_\pi}
\|\zeta_\pi\|^2
\md\mu(\pi).
\end{split}
\]
Let us now fix $\Omega'\subseteq \IrrG$, a symmetric measurable subset of finite measure and a vector $\eta=\int_{\IrrG}^{\oplus} \xi_{\pi}\otimes\ov{\zeta_\pi}\md\mu(\pi)$ such that $\xi_{\pi}=0$ for $\pi\in \IrrG\setminus \Omega'$. Assume moreover that there exists $M\ge 0$ such that $\|\xi_\pi\|=1,\|\zeta_\pi\|\le M$ for $\pi\in\Omega'$. It is clear that subset of such $\eta$ is linearly dense in $\int_{\IrrG}^{\oplus}\HS(\msf{H}_\pi)\md\mu(\pi)$. We have
\[
\begin{split}
0&\le
\ismaa{\eta}{\bigl(
\I-\tfrac{1}{2}\int_{\IrrG}^{\oplus} (\omega_{\xi^{\Omega,n}}\otimes\id)
((U^{\pi})^*+U^\pi)\otimes\I_{\ov{\msf{H}_\pi}}\md\mu(\pi)
\bigr)\eta}\\
&=\int_{\IrrG}
\is{\xi_\pi}{
\xi_\pi-\tfrac{1}{2}(\omega_{\xi^{\Omega,n}}\otimes\id)((U^{\pi})^*+U^\pi) \xi_\pi}
\|\zeta_\pi\|^2
\md\mu(\pi)\\
&\le
M^2\int_{\IrrG}\chi_{\Omega'}(\pi) \Tr_\pi( \I-\tfrac{1}{2}(\omega_{\xi^{\Omega,n}}\otimes\id)((U^{\pi})^*+U^\pi))
\md\mu(\pi)\\
&=
M^2 \bigl(\int_{\Omega'}\dim\md\mu-\tfrac{1}{2}
\omega_{\xi^{\Omega,n}}\bigl(
\chi^{\int}(\int_{\Omega'}^{\oplus}\pi\md\mu_{\Omega'}(\pi))^*\bigr)
-
\tfrac{1}{2}\omega_{\xi^{\Omega,n}}\bigl(
\chi^{\int}(\int_{\Omega'}^{\oplus}\pi\md\mu_{\Omega'}(\pi))\bigr)\bigr)\\
&=
M^2 \bigl(\int_{\Omega'}\dim\md\mu-
\omega_{\xi^{\Omega,n}}\bigl(
\chi^{\int}(\int_{\Omega'}^{\oplus}\pi\md\mu_{\Omega'}(\pi))\bigr)\bigr)
\xrightarrow[(\Omega,n)\in\Upsilon\times\NN]{}0.
\end{split}
\]
The last convergence follows from \eqref{eq4} and the fact that for $\Omega$ large enough we have $\Omega'\subseteq \Omega$. We have also used the fact that $\Omega'$ is symmetric, in order to justify self-adjointness of $\chi^{\int}(\int_{\Omega'}^{\oplus}\pi\md\mu_{\Omega'}(\pi))$. Therefore we have proven the implication $3)\Rightarrow 1)$.\\
Let us now show the implication $1) \Rightarrow 2)$. Assume that $\GG$ is coamenable and $\Omega\subseteq\IrrG$ is a measurable subset such that $\int_\Omega \dim \md\mu<+\infty$. We do not assume anything about the measure $\mu$ or admissibility of irreducible representations. Let $(\xi_n)_{n\in\NN}$ be a sequence of unit vectors in $\LdG$, from the definition of coamenability (Definition \ref{defcoamenable}). Note that we can find a sequence, not a net, since $\LdG$ is separable. We have
\[
\lim_{n\to\infty}(\omega_n\otimes\omega_{\xi}) \mrW=\|\xi\|^2\quad
(\xi\in\LdG)
\]
where $\omega_n=\omega_{\xi_n}$ is a vector state. For $\int_{\IrrG}^{\oplus}\xi_\pi\otimes\ov{\zeta_\pi}\md\mu(\pi)\in\int_{\IrrG}^{\oplus}\HS(\msf{H}_\pi)\md\mu(\pi)$ such that $\xi_\pi=0$ when $\pi\notin \Omega$ we have
\[
\begin{split}
&\quad\;
\ismaa{\xi_n}{ (\int_{\Omega}(\id\otimes\omega_{\xi_\pi})U^{\pi} \|\zeta_\pi\|^2 \md\mu_{\Omega}(\pi))\xi_n}\\
&=\int_{\Omega} (\omega_{n}\otimes\omega_{\xi_\pi})U^\pi \|\zeta_\pi\|^2
\md\mu_{\Omega}(\pi)\\
&=
\int_{\Omega} 
\ismaa{\xi_\pi\otimes\ov{\zeta_\pi}}{ ((\omega_n\otimes\id)U^\pi\otimes
\I_{\ov{\msf{H}_\pi}})\,\xi_\pi\otimes\ov{\zeta_\pi}}
\md\mu_{\Omega}(\pi)\\
&=
\int_{\IrrG} 
\ismaa{\xi_\pi\otimes\ov{\zeta_\pi}}{ ((\omega_n\otimes\id)U^\pi\otimes
\I_{\ov{\msf{H}_\pi}})\,\xi_\pi\otimes\ov{\zeta_\pi}}
\md\mu(\pi)\\
&=
\bigl\langle \int_{\IrrG}^{\oplus} \xi_\pi\otimes\ov{\zeta_\pi}\md\mu(\pi)
\big|
\bigl(\int_{\IrrG}^{\oplus}
((\omega_n\otimes\id)U^\pi\otimes
\I_{\ov{\msf{H}_\pi}})\md\mu(\pi)\bigr)
\int_{\IrrG}^{\oplus} \xi_\pi\otimes\ov{\zeta_\pi}\md\mu(\pi)\bigr\rangle\\
&=
\bigl\langle \int_{\IrrG}^{\oplus} \xi_\pi\otimes\ov{\zeta_\pi}\md\mu(\pi)
\big|\mc{Q}_L
(\omega_n\otimes\id)\mrW
\mc{Q}_L^*\int_{\IrrG}^{\oplus} \xi_\pi\otimes\ov{\zeta_\pi}\md\mu(\pi)\bigr\rangle\\
&\quad\;\xrightarrow[n\to\infty]{}
\bigl\|
\int_{\IrrG}^{\oplus} \xi_\pi\otimes\ov{\zeta_\pi}\md\mu(\pi)
\bigr\|^2.
\end{split}
\]
Pick $\eps>0$. Since $\int_\Omega\dim\md\mu<+\infty$ there exists a $K\in\NN$ such that
\[
\int_{\Omega\cap\, \bigcup_{k=1}^{K}\IrrG\rest_k}
\dim \md\mu\approx_{\tfrac{\eps}{3}}\int_\Omega\dim\md\mu.
\]
Let $\{(\xi_\pi^k)_{\pi\in\IrrG}\,|\,k\in\NN\}$ be a field of orthonormal bases restricted to $\Omega\cap \bigcup_{k'=1}^{K}\IrrG\rest_{k'}$. The following holds:
\[
\begin{split}
&\quad\;\ismaa{\xi_n}{\chi^{\int}(\int_{\Omega}^{\oplus} \pi\md\mu_{\Omega}(\pi)) \xi_n}\\
&=
\ismaa{\xi_n}{\int_{\Omega\setminus\bigcup_{k'=1}^{K}\IrrG\rest_{k'}} \chi(U^\pi)\md\mu(\pi) \xi_n}
+
\ismaa{\xi_n}{\int_{\Omega\cap\bigcup_{k'=1}^{K}\IrrG\rest_{k'}} \chi(U^\pi)\md\mu(\pi) \xi_n}\\
&=
\ismaa{\xi_n}{\int_{\Omega\setminus\bigcup_{k'=1}^{K}\IrrG\rest_{k'}} \chi(U^\pi)\md\mu(\pi) \xi_n}+
\sum_{k=1}^{K}\int_{\Omega} \ismaa{\xi_n}{(\id\otimes\omega_{\xi^{k}_\pi})U^{\pi} \xi_n}\md\mu(\pi)\\
&=
\ismaa{\xi_n}{\int_{\Omega\setminus\bigcup_{k'=1}^{K}\IrrG\rest_{k'}} \chi(U^\pi)\md\mu(\pi) \xi_n}+
\sum_{k=1}^{K}\int_{\Omega} (\omega_{\xi_n}\otimes\omega_{\xi^{k}_\pi})U^{\pi}\md\mu(\pi).
\end{split}
\]
Absolute value of the first term is less or equal to $\tfrac{\eps}{3}$:
\[
\bigl|\ismaa{\xi_n}{\int_{\Omega\setminus\bigcup_{k'=1}^{K}\IrrG\rest_{k'}} \chi(U^\pi)\md\mu(\pi) \xi_n}\bigr|
\le
\int_{\Omega\setminus\bigcup_{k'=1}^{K}\IrrG\rest_{k'}} \dim \md\mu\le
\tfrac{\eps}{3}.
\]
Moreover
\[
\sum_{k=1}^{K}\int_{\Omega} (\omega_{\xi_n}\otimes\omega_{\xi^{k}_\pi})U^{\pi}\md\mu(\pi)
\xrightarrow[n\to\infty]{}
\sum_{k=1}^{K}\bigl\|
\int_{\IrrG}^{\oplus} \xi^k_\pi\otimes\ov{\xi^k_\pi}\md\mu(\pi)
\bigr\|^2=
\int_{\Omega\cap\bigcup_{k'=1}^{K}\IrrG\rest_{k'}}\dim \md\mu,
\]
consequently, there exists $n_0$ such that
\[
\ismaa{\xi_n}{\chi^{\int}(\int_{\Omega}^{\oplus} \pi\md\mu_{\Omega}(\pi)) \xi_n}\approx_{\eps}
\int_{\Omega}\dim\md\mu,
\]
for $n\ge n_0$. Since the norm of $\chi^{\int}(\int_{\Omega}^{\oplus} \pi \md\mu_{\Omega}(\pi))$ is less or equal to $\int_{\Omega} \dim \md\mu$, this implies that $\int_{\Omega} \dim \md\mu$ is an approximate eigenvalue of $\chi^{\int}(\int_{\Omega}^{\oplus} \pi \md\mu_{\Omega}(\pi))$ and proves
\[
\int_{\Omega}\dim\md\mu\in\sigma(\chi^{\int}(\int_{\Omega}^{\oplus}
\pi\md\mu_{\Omega}(\pi))).
\]
\end{proof}

\section{$\mf{A}$ - a \cst-algebra generated by integral characters}\label{cstA}
In this section we introduce a \cst-algebra $\mf{A}$, which is generated by certain integral characters. We prove two theorems which gives conditions equivalent to coamenability: Theorem \ref{tw3} states that $\GG$ is coamenble if and only if there exists a certain character on $\mf{A}$, and Theorem \ref{tw4} relates coamenability to the properties of spectra $\sigma(\mc{L}_\nu)$.\\
Let us fix any Plancherel measure $\mu$ for $\GG$. Recall that whenever we take an integral representation $\pi_X\in \Rep^{\int}(\GG)$, it comes together with a measure space and measurable fields: of Hilbert spaces and of representations:
\[
(X,\mf{M}_X,\mu_X,(\msf{H}_x)_{x\in X},\{(e^i_x)_{x\in X}\}_{i\in\NN},(\pi_x)_{x\in X}).
\]
For an integral representation $\pi_X\in \Rep^{\int}(\GG)$ such that $\pi_X\lec_q\Lambda_{\whG}$ we can find a descending family of sets $\{\F^{n}_{\pi_X}\}_{n\in\NN}$ defined uniquely up to measure zero, such that
\[
\pi_X=\int_X^{\oplus}\pi_x\md\mu_X(x)\simeq
\bigoplus_{n\in\NN} 
\int_{\F^{n}_{\pi_X}}^{\oplus} \kappa\md\mu_{\F^n_{\pi}}(\kappa).
\]
If $\sum_{n=1}^{\infty} \chi_{\F^n_{\pi_X}}<+\infty$ pointwise (in particular $\E^{\aleph_0}_{\pi_X}=\emptyset$) then we can use Proposition \ref{treq}: we get a measurable function $\varpi^{\pi_X}$ on $\F^1_{\pi_X}$ such that the measure $\tilde{\mu}=\varpi^{\pi_X}\mu_{\F^1_{\pi_X}}$ satisfies
\[
\int_X\dim\md\mu_X=\sum_{n=1}^{\infty} \int_{\F^n_{\pi_X}}
\dim \md\tilde{\mu}
\]
and we have equality of the integral weights. If the above integral dimensions are finite, we have equality of the integral characters:
\[
\begin{split}
\chi^{\int}\bigl(\int_{X}^{\oplus}
\pi_x \md\mu_{X}(x)\bigr)=
\sum_{n=1}^{\infty} \chi^{\int}(
\int_{\F^{n}_{\pi_X}}^{\oplus}
\kappa \md\tilde{\mu}_{\F^{n}_{\pi_X}}(\kappa))=
\sum_{n=1}^{\infty}
\int_{\F^{n}_{\pi_X}}
\varpi^{\pi_X}(\kappa)
\chi(U^\kappa) \md\mu_{\F^{n}_{\pi_X}}(\kappa).
\end{split}
\]
Let us now introduce new families of representations:

\begin{definition}$ $
\begin{itemize}
\item
Let $\Rep^{\int}_{q}(\GG)$ be the family of integral representations $\pi_X\in \Rep^{\int}(\GG)$ satisfying $\pi_X\lecq\Lambda_{\whG}$.
\item
Let $\Rep^{\int}_{f}(\GG)$ be the family of integral representations $\pi_X\in \Rep^{\int}(\GG)$ such that $\pi_X\lecq\Lambda_{\whG}$ and $\sum_{n=1}^{\infty} \chi_{\F^n_{\pi_X}}<+\infty$ pointwise.
\end{itemize}
\end{definition}

We define also
\[
\Rep^{\int}_{q,<+\infty}(\GG)=\Rep^{\int}_q(\GG)\cap \Rep^{\int}_{<+\infty}(\GG)\textnormal{ and }
\Rep^{\int}_{f,<+\infty}(\GG)=\Rep^{\int}_f(\GG)\cap \Rep^{\int}_{<+\infty}(\GG).
\]
Take integral representations $\pi_X,\pi_Y\in \Rep^{\int}_{<+\infty}(\GG)$. We have justified in Section \ref{secintrep} that their direct sum, tensor product and conjugate representation end up in $\Rep^{\int}_{<+\infty}(\GG)$. Properties of the regular representation and quasi-containment shows that if $\pi_X,\pi_Y\lec_q \Lambda_{\whG}$ then the same is true for representations constructed via basic operations. Indeed, Proposition \ref{stw10} shows that this is the case for a tensor product, it is also clear for a direct sum. In the Subsubsection \ref{conjrep} we have proven that $\pi_{\ov X}\simeq \jmath_{\msf{H}_{\pi_X}}\circ \pi_X \circ \hat{R}^{u}$, therefore $\pi_{\ov X}$ is quasi-contained in the representation conjugate to $\Lambda_{\whG}$. However, $\Lambda_{\whG}$ is self-conjugate up to quasi equivalence: by \cite[Theorem 2.2]{closedqsub} we need to show the existence of a normal $\star$-isomorphism:
\[
\gamma\colon
\{(\eta\otimes\id)\mrW \,|\, \eta\in \Lj\}''\rightarrow
\{(\eta\otimes\id)(R\otimes \jmath_{\LdG})\mrW \,|\, \eta\in \Lj\}''.
\]
satisfying $(\id\otimes\gamma)\mrW=(R\otimes\jmath_{\LdG})\mrW$. However, we have
\[\begin{split}
\Linfd&=\{(\eta\otimes\id)\mrW \,|\, \eta\in \Lj\}'',\\
\jmath_{\LdG}(\Linfd)&=
\{(\eta\otimes\id)(R\otimes \jmath_{\LdG})\mrW \,|\, \eta\in \Lj\}''\subseteq \B(\ov{\LdG}),
 \end{split}\]
therefore we can take $\gamma=\jmath_{\LdG}\circ \hat{R}$. Indeed, it is a normal linear $\star$-multiplicative map, which is clearly an isomorphism between the von Neumann algebras $\Linfd$ and $\jmath_{\LdG}(\Linfd)$. Moreover, it satisfies
\[
(\id\otimes\gamma)\mrW=(\id\otimes\jmath_{\LdG}\circ\hat{R})\mrW=(R\otimes\jmath_{\LdG})\mrW.
\]
It follows that $\pi_{\ov{X}}\lec_q \Lambda_{\whG}$ whenever $\pi_X\lec_q \Lambda_{\whG}$. let us introduce the following subsets:
\[
\mscr{A}=\bigl\{
\sum_{k=0}^{3} i^k \chi^{\int}(\int_{X_k}^{\oplus} \pi_x\md\mu_{X_k}(x))
\,\big|\, \forall_{k\in\ov{0,3}}\;
\pi_{X_k}\in \Rep^{\int}_{q,<+\infty}(\GG)\bigr\}\subseteq \Linf
\]
and $\mf{A}=\ov{\mscr{A}}^{\|\cdot\|}$. Both $\mscr{A}$ and $\mf{A}$ are subalgebras of $\Linfd$. It follows from the fact that we can express a sum and a product of integral characters as an integral character. If irreducible representations of $\GG$ are admissible then $\mscr{A}$ is a $\star$-algebra, and $\mf{A}$ is a \cst-algebra in $\Linf$ -- it is the case since then we can express adjoint of integral character as an integral character. It can happen that $\I\notin\ov{\mf{A}}^{\wot}$ (see example $\RR\rtimes \ZZ_2$), so it is not necessarily the case that $\mf{A}$ is a nondegenerate \cst-subalgebra of $\B(\LdG)$. For any subset $\Omega\subseteq\IrrG$ such that $\int_\Omega\dim\md\mu<+\infty$ we have $\chi^{\int}(\int_\Omega^{\oplus}\pi\md\mu(\pi))\in\mscr{A}\subseteq\mf{A}$.\\
Recall that in Lemma \ref{lemat34} we have defined operator $T$, which is an isometry $\LL^2(\IrrG)\rightarrow\LdG$. Moreover, we have introduced a closed subspace $\mc{H}$ in $\LdG$ as the image of $T$.\\
The next proposition is in part a generalization of propositions \ref{stw9} and \ref{stw11} to the case of more general integral representations.

\begin{proposition}\label{stw1}
Let $\pi_X\in \Rep^{\int}_{f,<+\infty}(\GG)$. We have
\[
\Tr(E^2_{\bullet})^{\frac{1}{2}}\varpi^{\pi_X}\sum_{n=1}^{\infty} \chi_{\F^{n}_{\pi_X}}\in \LL^2(\IrrG)
\Leftrightarrow 
\chi^{\int}(\int_X^{\oplus}\pi_x\md\mu_X(x))\in\mf{N}_\psi.
\]
If the above conditions hold, then also
\[
\Lambda_{\psi}\bigl((\chi^{\int}(\int_X^{\oplus}\pi_x\md\mu_X(x))\bigr)
=T(\Tr(E^2_\bullet)^{\frac{1}{2}}\varpi^{\pi_X}\sum_{n=1}^{\infty} \chi_{\F^{n}_{\pi_X}})\in\mc{H}.
\]
Consequently
\[
\begin{split}
\mc{H}=\ov{\lin}\bigl\{
\Lambda_{\psi}\bigl(
&\chi^{\int}(\int_X^{\oplus}\pi_x\md\mu_X(x))\bigr)\,\big|\pi_X\in \Rep^{\int}_{f,<+\infty}(\GG)\,:\,
\Tr(E^2_\bullet)^{\frac{1}{2}}\varpi^{\pi_X}\sum_{n=1}^{\infty} \chi_{\F^{n}_{\pi_X}}\in \LL^2(\IrrG)\bigr\}.
\end{split}
\]
\end{proposition}

\begin{proof}
Let $\pi_X$ be an integral representation in $\Rep_{f,<+\infty}^{\int}(\GG)$. We have
\[
\dim\varpi^{\pi_X}\sum_{n=1}^{\infty} \chi_{\F^n_{\pi_X}}\in\LL^1(\IrrG)
\]
and
\[
\begin{split}
\chi^{\int}(\int_X^{\oplus}\pi_x\md\mu_X(x))&=
\sum_{n=1}^{\infty} \int_{\IrrG} 
(\varpi^{\pi_X}\chi_{\F^n_{\pi_X}})(\pi)\chi(U^\pi)
\md\mu(\pi)\\
&=
\int_{\IrrG} 
\bigl(\varpi^{\pi_X}\sum_{n=1}^{\infty} \chi_{\F^n_{\pi_X}}\bigr)(\pi)\chi(U^\pi)
\md\mu(\pi),
\end{split}
\]
We can express the integral character using equivalent Plancherel measure: define $\tilde{\mu}=(\chi_{\IrrG\setminus \F^1_{\pi_X}}+\varpi^{\pi_X} (\sum_{n=1}^{\infty} \chi_{\F^n_{\pi_X}}))\mu$, then we have
\[
\chi^{\int}(\int_X^{\oplus}\pi_x\md\mu_X(x))=
\int_{\F^1_{\pi_X}} \chi(U^{\pi}) \md\tilde{\mu}(\pi)=
\chi^{\int}(\int_{\F^1_{\pi_X}}^{\oplus}\pi \md\tilde{\mu}_{\F^1_{\pi_X}}(\pi)).
\]
Operators $\tilde{E}_\pi$ corresponding to this Plancherel measure are given by
\[
\tilde{E}_\pi^2=\varpi^{\pi_X}(\pi)(\sum_{n=1}^{\infty} \chi_{\F^n_{\pi_X}}(\pi)) E^2_\pi\quad(\pi\in\F^1_{\pi_X}).
\]
Therefore, if the integral character is in $\mf{N}_\psi$, Proposition \ref{stw9} gives us
\[
+\infty>\int_{\F^1_{\pi_X}} \Tr(\tilde{E}^2_\pi)\md\tilde{\mu}(\pi)=
\int_{\IrrG}\varpi^{\pi_X}(\pi)^2
(\sum_{n=1}^{\infty} \chi_{\F^n_{\pi_X}}(\pi))^2 \Tr(E^2_\pi)\md\mu(\pi),
\]
which proves the implication $\Leftarrow$.\\
Assume now that $\Tr(E^2_\bullet)^{\frac{1}{2}}\varpi^{\pi_X}\sum_{n=1}^{\infty} \chi_{\F^n_{\pi_X}}\in \LL^2(\IrrG)$. Let $\{V_p\}_{p\in\NN}$ be an increasing family of measurable subsets of  $\IrrG$ which have finite measure and moreover
\[
\sup_{\pi\in V_p} \bigl(
\dim(\pi)+
\Tr(E^2_\pi)^{\frac{1}{2}}\varpi^{\pi_X}(\pi)\sum_{n=1}^{\infty} \chi_{\F^n_{\pi_X}}(\pi)+
\Tr(E^2_\pi)^{-1}\bigr)\le p\quad(p\in \NN).
\]
For any $p\in\NN$ and a measurable subset $V\subseteq V_p$, the function $\chi_{V} \Tr(E^2_\bullet)^{\frac{1}{2}}\varpi^{\pi_X}\sum_{n=1}^{\infty} \chi_{\F^n_{\pi_X}}$ belongs to the original domain of $T$, $\mc{D}(T)$. Let $\{(\xi^n_\kappa)_{\kappa\in\IrrG}\}_{n\in\NN}$ be a field of orthonormal bases. For any $p'>p$ we have
\begin{equation}\label{eq13}
\begin{split}
&\quad\;\int_{\IrrG}\chi_{V_{p'}\setminus V_p} \bigl|\Tr(E^2_\bullet)^{\frac{1}{2}}
\varpi^{\pi_X}\sum_{n=1}^{\infty} \chi_{\F^n_{\pi_X}}\bigr|^2\md\mu\\
&=
\bigl\|T\bigl(
\chi_{V_{p'}\setminus V_p}
\Tr(E^2_\bullet)^{\frac{1}{2}}\varpi^{\pi_X}\sum_{n=1}^{\infty} \chi_{\F^n_{\pi_X}}
\bigr)\bigr\|^2\\
&=\bigl\|
\Lambda_{\psi}\bigl(
\int_{\IrrG} \chi_{V_{p'}\setminus V_p}(\pi)
(\varpi^{\pi_X}\sum_{n=1}^{\infty} \chi_{\F^n_{\pi_X}})(\pi)
\chi(U^\pi)\md\mu(\pi)\bigr)
\bigr\|^2.
\end{split}
\end{equation}
Equation \eqref{eq13} shows that
\[
\bigl(\Lambda_{\psi}\bigl(
\int_{\IrrG}(\chi_{V_p}
\varpi^{\pi_X}\sum_{n=1}^{\infty} \chi_{\F^n_{\pi_X}})(\pi)
\chi(U^\pi)\md\mu(\pi)\bigr)\bigr)_{p\in\NN}
\]
is a Cauchy sequence in $\LdG$, therefore since $\Lambda_{\psi}$ is a closed map we have
\[
\chi^{\int}(\int_X^{\oplus}\pi_x\md\mu_X(x))=\int_{\IrrG}
\bigl(\varpi^{\pi_X}\sum_{n=1}^{\infty} \chi_{\F^n_{\pi_X}}\bigr)(\pi)\chi(U^\pi)\md\mu(\pi)\in\mf{N}_\psi
\]
and
\[
\begin{split}
&\quad\;\Lambda_{\psi}\bigl(\chi^{\int}(\int_X^{\oplus} \pi_x \md\mu_X(x))\bigr)=
\Lambda_{\psi}\bigl(\int_{\IrrG}
\bigl(\varpi^{\pi_X}\sum_{n=1}^{\infty} \chi_{\F^n_{\pi_X}}\bigr)(\pi)\chi(U^\pi)\md\mu(\pi)\bigr)\\
&=
\lim_{p\to\infty}
\Lambda_{\psi}\bigl(\int_{\IrrG}
\bigl(\chi_{V_p}\varpi^{\pi_X}\sum_{n=1}^{\infty} \chi_{\F^n_{\pi_X}}\bigr)(\pi)\chi(U^\pi)\md\mu(\pi)\bigr)\\
&=\lim_{p\to\infty}
T(\chi_{V_p}\Tr(E^2_\bullet)^{\frac{1}{2}}
\varpi^{\pi_X}\sum_{n=1}^{\infty} \chi_{\F^n_{\pi_X}})=
T(\Tr(E^2_\bullet)^{\frac{1}{2}}\varpi^{\pi_X}\sum_{n=1}^{\infty} \chi_{\F^n_{\pi_X}}).
\end{split}
\]
This also shows the inclusion
\[
\mc{H}\supseteq\bigl\{
\Lambda_{\psi}\bigl(
\chi^{\int}(\int_X^{\oplus}\pi_x\md\mu_X(x))\bigr)\,\big|\,
\pi_X\in \Rep^{\int}_{f,<+\infty}(\GG)\,:\,
\Tr(E^2_\bullet)^{\frac{1}{2}}
\varpi^{\pi_X}\sum_{n=1}^{\infty} \chi_{\F^{n}_{\pi_X}}\in \LL^2(\IrrG)\bigr\}.
\]
By the definition of $\mc{H}$ we have $\mc{H}=T(\LL^2(\IrrG))$. The Hilbert space $\LL^2(\IrrG)$ is generated by functions $\Tr(E^2_\bullet)^{\frac{1}{2}}\chi_{\Omega}$, where $\Omega$ is a measurable, finite measure subset which satisfies $\sup_{\pi\in\Omega}\Tr(E^2_\pi)<+\infty$ and contains representations of dimension $\le p$ for some $p\in \NN$. Then of course $\int_\Omega \dim\md\mu<+\infty$. It is also clear that
\[
\Lambda_{\whG}\gecq\sigma_\Omega=\int_\Omega^{\oplus}\kappa\md\mu_\Omega(\kappa)
\]
and $\F^n_{\sigma_\Omega}=\emptyset$ for $n\ge 2$, hence $\sigma_{\Omega}\in\Rep^{\int}_{f,<+\infty}(\GG)$. We also have
\[
\Tr(E^2_\bullet)^{\frac{1}{2}}\varpi^{\sigma_{\Omega}}\; \sum_{n=1}^{\infty} \chi_{\F^n_{\sigma_\Omega}}=\Tr(E^2_\bullet)^{\frac{1}{2}}\chi_{\Omega}\in\LL^2(\IrrG).
\]
Since the function $\Tr(E^2_\bullet)^{\frac{1}{2}}\chi_\Omega$ is in the original domain of $T$, we have
\[
T(\Tr(E^2_\bullet)^{\frac{1}{2}}\chi_\Omega)=\Lambda_{\psi}\bigl(\int_{\IrrG}
\chi_{\Omega}(\kappa) \chi(U^\kappa)\md\mu(\kappa)\bigr)=
\Lambda_{\psi}\bigl(
\chi^{\int}(\int_\Omega^{\oplus}\kappa\md\mu_\Omega(\kappa))\bigr),
\]
which ends the proof.
\end{proof} 

The above proposition says that the subspace $\mc{H}\subseteq \LdG$ is roughly speaking the Hilbert space of square integrable integral characters. The next lemma says that the \cst-algebra $\mf{A}$ preserves this subspace.

\begin{lemma}\label{lemat23}
We have $\mf{A}\mc{H}\subseteq\mc{H}$.
\end{lemma}

\begin{proof}
Let $\chi^{\int}(\int_X^{\oplus} \pi_x \md\mu_X(x))=\int_{X}\chi(U^{\pi_x})\md\mu_X(x)\in \mf{A}$. Since $\pi_X\lecq \Lambda_{\whG}$, Lemma \ref{lemat22} gives us $\pi_x\lec \Lambda_{\whG}$ for almost all $x\in X$, therefore Theorem \ref{tw1} implies that $\chi(U^{\pi_x})T=T\mc{L}_{\pi_x}$ for almost all $x\in X$. Consequently, for $f\in \LL^2(\IrrG,\mu)$ we have
\[
\chi^{\int}(\int_X^{\oplus} \pi_x \md\mu_X(x)) T(f)=
\int_{X} \chi(U^{\pi_x})T(f)\md\mu_X(x)=
\int_X T \mc{L}_{\pi_x} (f) \md\mu_X(x)\in\mc{H}
\]
(integral converges in the weak topology on $\LdG$, image of the integrand lies in $\mc{H}$ therefore the integral belongs to $\mc{H}$ due to Riesz theorem). Taking limits and linear combinations gives us the claim.
\end{proof}

The above lemma tells us that the integral character of any integral representation in $\Rep^{\int}_{q,<+\infty}(\GG)$ preserves the subspace $\mc{H}$, hence we can consider the restricted operator. The next result provides us with information on how this operation changes its spectrum. Its proof is elementary hence will be ommited.

\begin{lemma}\label{lemat12}
Assume that all irreducible representations of $\GG$ are admissible. For any element $a\in \mf{A}$ we have an inclusion of spectra: $\sigma_{\Linf}(a)\supseteq\sigma_{\B(\mc{H})}(a|_{\mc{H}})$.
\end{lemma}

Define a $\star$-homomorphism $\phi$ as the restriction map:
\[
\phi\colon \mf{A}\in a \mapsto a|_{\mc{H}}\in \B(\mc{H}).
\]

\begin{lemma}
Assume that all irreducible representations of $\GG$ are admissible. Then $\phi$ is a faithful $\star$-homomorphism.
\end{lemma}

\begin{proof}
Assume that there exists an operator $a\in \mf{A}$ in the kernel of $\phi$ with $\|a\|=1$. In particular this means that for any measurable subset $\Omega\subseteq\IrrG$ such that
\[
\mu(\Omega)+\sup_{\pi\in\Omega} (\Tr(E^2_\pi)+\dim(\pi))<+\infty,
\]
and any bounded measurable function $f$ with support in $\Omega$ we have
\[
aT(\Tr(E^2_\bullet)^{\frac{1}{2}}f)=a\Lambda_{\psi}(\int_\Omega f(\kappa)\chi(U^\kappa)\md\mu(\kappa))=
\Lambda_{\psi}(\int_\Omega f(\kappa)a\chi(U^\kappa)\md\mu(\kappa))=0.
\]
As the weight $\psi$ is faithful, the above equation gives us
\[
\int_{\Omega} f(\kappa) a \chi(U^\kappa)\md\mu(\kappa)=0.
\]
Choose vectors $\xi,\zeta\in\LdG$ and function $f(\kappa)=\ov{\is{\xi}{a\chi(U^\kappa)\zeta}}\chi_{\Omega}(\kappa)\,(\kappa\in\IrrG)$. We arrive at
\[
\bigl\langle\xi \big|
\bigl(\int_{\Omega} f(\kappa) a \chi(U^\kappa)\md\mu(\kappa)\bigr)
\zeta\bigr\rangle=
\int_{\Omega} |\is{\xi}{a\chi(U^{\kappa})\zeta}|^2\md\mu(\kappa)=0.
\]
Since $\xi,\zeta$, $\Omega$ were (more or less) arbitrary, we get $a\chi(U^\kappa)=0$ for almost all $\kappa\in \IrrG$. Consequently also $\chi(U^{\kappa})a^*=0$ for almost all $\kappa\in\IrrG$ (here we use the fact that $\kappa$ is admissible and that Plancherel measure is equivalent to the Plancherel measure composed with conjugation). Pick integral representations $\pi_{X_0},\dotsc,\pi_{X_3}\in 
\Rep^{\int}_{q,<+\infty}(\GG)$ such that
\[
\bigl\|a-\sum_{k=0}^{3} i^k
\chi^{\int}(\int_{X_k}^{\oplus}\pi_x \md\mu_{X_k}(x))\bigr\|\le \tfrac{1}{2}.
\]
We arrive at a contradiction:
\[
1-\tfrac{1}{2}\le
\bigl\|\bigl(\sum_{k=0}^{3} i^k
\chi^{\int}(\int_{X_k}^{\oplus}\pi_x \md\mu_{X_k}(x))\bigr)a^*\bigr\|=
\bigl\|
\sum_{k=0}^{3} i^k\sum_{n=1}^{\infty}
\int_{\F^n_{\pi_{X_k}}}  \varpi^{\pi_{X_k}}(\kappa)
\chi(U^{\kappa}) a^*
\md\mu_{\F^n_{\pi_{X_k}}}(\kappa)\bigr\|=0.
\]
\end{proof}

\begin{theorem}\label{tw3}
Let $\GG$ be a second countable locally compact quantum group. Assume moreover that $\GG$ is type I and has only finite dimensional irreducible representations all of which are admissible. Then $\GG$ is coamenable if and only if there exists a character $\delta\in \mf{A}^*$ such that
\[
\delta(\chi^{\int}(\int_{X}^{\oplus}\pi_x \md\mu_{X}(x)))=
\int_X \dim \md\mu_X
\]
for every integral representation $\pi_X\in\Rep^{\int}_{q,<+\infty}(\GG)$.
\end{theorem}

Character $\delta$ is closely related to the counit -- formally the counit satisfies the above equation. However, the counit (if exists) is defined only on $\mathrm{C}_0(\GG)$ not on the whole $\Linf$. The second difficulty stems from the fact that the above integral coverges in \swot, and the counit is only norm continuous.\\
Character $\delta$ will be defined (similarly as one can find the counit) by taking appropriate limit of vector functionals, given by vectors appearing in the definition of coamenability.

\begin{proof}
Assume that $\GG$ is coamenable. Choose any integral representation $\pi_X\in \Rep_{q,<+\infty}^{\int}(\GG)$. There exists a unitary intertwiner
\[
\mc{O}\colon \int_X^{\oplus}\msf{H}_x\md\mu_X(x)\rightarrow
\bigoplus_{n=1}^{\infty} \int_{\F^n_{\pi_X}}^{\oplus} \msf{H}_\pi \md\mu_{\F^n_{\pi_X}}(\pi).
\]
Let $\{(\xi^{k}_{x})_{x\in X}\}_{k=1}^{\infty}$ be a measurable field of orthonormal bases. We have
\[
+\infty>\int_X\dim\md\mu_X=\sum_{k=1}^{\infty} \int_X \|\xi^k_x\|^2 \md\mu_X(x),
\]
which means that for each $k\in\NN$ we can consider $\int_X^{\oplus}\xi^k_x \md\mu_X(x)\in\int_X^{\oplus}\msf{H}_x\md\mu_X(x)$ and
\[
\mc{O}\int_X^{\oplus}\xi^k_x \md\mu_X(x)=
\bigoplus_{n=1}^{\infty} \int_{\F^n_{\pi_X}}^{\oplus} \zeta^k_{n,\pi}\md\mu_{\F^n_{\pi_X}}(\pi)\quad(k\in\NN)
\]
for some vectors $\zeta^k_{n,\pi}$. By the definition of integral character the following holds:
\[
\chi^{\int}(\int_X^{\oplus} \pi_x\md\mu_X(x))=
\sum_{k=1}^{\infty}\int_X (\id\otimes\omega_{\xi^k_x})U^{\pi_x}\md\mu_X(x).
\]
Let us choose an arbitrary functional $\omega\in\Lj$. We have
\[\begin{split} 
&\quad\;\omega(\chi^{\int}(\int_X^{\oplus} \pi_x\md\mu_X(x)))=
\sum_{k=1}^{\infty}\int_X \omega_{\xi^k_x}((\omega\otimes\id)U^{\pi_x})\md\mu_X(x)\\
&=
\sum_{k=1}^{\infty}
\bigl\langle \int_X^{\oplus} \xi^{k}_{x}\md\mu_X(x)\big|
\bigl(\int_X^{\oplus} \pi_x(\lambda^u(\omega))\md\mu_X(x)\bigr)
\int_X^{\oplus} \xi^{k}_{x}\md\mu_X(x)\big\rangle\\
&=
\sum_{k=1}^{\infty}
\bigl\langle 
\mc{O}^*\bigoplus_{n=1}^{\infty}\int_{\F^n_{\pi_X}}^{\oplus} \zeta^{k}_{n,\pi}\md\mu_{\F^n_{\pi_X}}(x)\big|
\bigl(\int_X^{\oplus} \pi_x(\lambda^u(\omega))\md\mu_X(x)\bigr)\mc{O}^*
\bigoplus_{n=1}^{\infty}\int_{\F^n_{\pi_X}}^{\oplus} \zeta^{k}_{n,\pi}\md\mu_{\F^n_{\pi_X}}(x)\big\rangle\\
&=
\sum_{k=1}^{\infty}
\bigl\langle 
\bigoplus_{n=1}^{\infty}\int_{\F^n_{\pi_X}}^{\oplus} \zeta^{k}_{n,\pi}\md\mu_{\F^n_{\pi_X}}(x)\big|
\bigl(\bigoplus_{n=1}^{\infty}\int_{\F^n_{\pi_X}}^{\oplus} 
\pi(\lambda^u(\omega))\md\mu_{\F^n_{\pi_X}}(\pi)\bigr)
\bigoplus_{n=1}^{\infty}\int_{\F^n_{\pi_X}}^{\oplus} \zeta^{k}_{n,\pi}\md\mu_{\F^n_{\pi_X}}(x)\big\rangle\\
&=
\sum_{k=1}^{\infty}\sum_{n=1}^{\infty}
\int_{\F^n_{\pi_X}} \ismaa{\zeta^{k}_{n,\pi}}{
\pi(\lambda^u(\omega))\zeta^k_{n,\pi}}\md\mu_{\F^n_{\pi_X}}(\pi).
\end{split}\]
Let us extend fields $\{(\zeta^k_{n,\pi})_{\pi\in \F^n_{\pi_X}}\}_{k,n=1}^{\infty}$ by $0$ to the whole $\IrrG$. Since
\[
+\infty>\sum_{k=1}^{\infty}
\bigl\| \int_X^{\oplus} \xi^k_x\md\mu_X(x)\bigr\|^2=
\sum_{k=1}^{\infty}\bigl\|\bigoplus_{n=1}^{\infty} \int_{\F^n_{\pi_X}}^{\oplus}
\zeta^k_{n,\pi} \md\mu_{\F^n_{\pi_X}}(\pi)\bigr\|^2=
\sum_{k,n=1}^{\infty} \int_{\F^n_{\pi_X}} \|\zeta^k_{n,\pi}\|^2 \md\mu_{\F^n_{\pi_X}}(\pi),
\]
for any $n,k\in\NN$ the vector field $(\zeta^k_{n,\pi})_{\pi\in\IrrG}$ is square integrable. We can further write
\[\begin{split} 
&\quad\;\omega(\chi^{\int}(\int_X^{\oplus} \pi_x\md\mu_X(x)))=
\sum_{k=1}^{\infty} \sum_{n=1}^{\infty}
\int_{\IrrG}
\bigl\langle \zeta^k_{n,\pi}\otimes \tfrac{\ov{\zeta^k_{n,\pi}}}{\|\zeta^k_{n,\pi}\|} \big|
(\pi(\lambda^u(\omega))\otimes\I_{\ov{\msf{H}_\pi}})\;
\zeta^k_{n,\pi}\otimes \tfrac{\ov{\zeta^k_{n,\pi}}}{\|\zeta^k_{n,\pi}\|} \bigr\rangle
\md\mu(\pi)\\
&=
\sum_{k=1}^{\infty} \sum_{n=1}^{\infty}
\bigl\langle\int_{\IrrG}^{\oplus} \zeta^k_{n,\pi}\otimes \tfrac{\ov{\zeta^k_{n,\pi}}}{\|\zeta^k_{n,\pi}\|} \md\mu(\pi)\big|
\bigl(
\int_{\IrrG}^{\oplus} \pi(\lambda^u(\omega))\otimes\I_{\ov{\msf{H}_\pi}})\md\mu(\pi)\bigr)
\bigl(
\int_{\IrrG}^{\oplus} \zeta^k_{n,\pi}\otimes \tfrac{\ov{\zeta^k_{n,\pi}}}{\|\zeta^k_{n,\pi}\|} \md\mu(\pi)\bigr\rangle\\
&=
\sum_{k=1}^{\infty} \sum_{n=1}^{\infty}
\bigl\langle\int_{\IrrG}^{\oplus} \zeta^k_{n,\pi}\otimes \tfrac{\ov{\zeta^k_{n,\pi}}}{\|\zeta^k_{n,\pi}\|} \md\mu(\pi)\big|
\mc{Q}_L
(\omega\otimes\id)\mrW\mc{Q}_L^*
\bigl(
\int_{\IrrG}^{\oplus} \zeta^k_{n,\pi}\otimes \tfrac{\ov{\zeta^k_{n,\pi}}}{\|\zeta^k_{n,\pi}\|} \md\mu(\pi)\bigr\rangle.
\end{split}\]
Let us introduce vectors
\[
\eta_{n,k}=\mc{Q}_L^* \bigl(\int_{\IrrG}^{\oplus} 
\zeta^k_{n,\pi}\otimes \tfrac{\ov{\zeta^k_{n,\pi}}}{\|\zeta^k_{n,\pi}\|} \md\mu(\pi)\bigr)\in \LdG\quad(n,k\in\NN)
\]
These vector satisfy the following
\[
\sum_{n,k=1}^{\infty} \|\eta_{n,k}\|^2=
\sum_{n,k=1}^{\infty} \int_{\IrrG} \|\zeta^k_{n,\pi}\|^2\md\mu(\pi)=\sum_{k=1}^{\infty}
\bigl\| \int_X^{\oplus} \xi^k_x\md\mu_X(x)\bigr\|^2=
\int_{X}\dim\md\mu_X<+\infty.
\]
Above, we have proven the following equality
\[
\omega(\chi^{\int}(\int_X^{\oplus} \pi_x\md\mu_X(x)))=
\sum_{k=1}^{\infty}\sum_{n=1}^{\infty}
(\omega\otimes\omega_{\eta_{n,k}})\mrW
\]
for any $\omega\in\Lj$. In particular, we can substitute $\omega_m=\omega_{\xi_m}$, where $(\xi_m)_{m\in\NN}$ is a sequence of unit vectors from the definition of coamenability (Definition \ref{defcoamenable}). For any $n,k\in\NN$ we have
\[
(\omega_m\otimes\omega_{\eta_{n,k}})\mrW=
\ismaa{\xi_m\otimes\eta_{n,k}}{\mrW\;\xi_m\otimes\eta_{n,k}}
\xrightarrow[m\to\infty]{} \|\eta_{n,k}\|^2,
\]
therefore
\[
\omega_m(\chi^{\int}(\int_X^{\oplus} \pi_x\md\mu_X(x)))=
\sum_{k=1}^{\infty}\sum_{n=1}^{\infty}
(\omega_m\otimes\omega_{\eta_{n,k}})\mrW
\xrightarrow[m\to\infty]{}
\sum_{n,k=1}^{\infty} \|\eta_{n,k}\|^2=\int_X\dim\md\mu_X.
\]
We can interchange the signs of a limit and series. Indeed, for any $m,n,k\in\NN$ we have inequalities: $|(\omega_m\otimes\omega_{\eta_{n,k}})\mrW|\le \|\eta_{n,k}\|^2$ and $\sum_{n,k=1}^{\infty} \|\eta_{n,k}\|^2<+\infty$, therefore we can use the dominated convergence theorem.\\
Due to the $w^*$-compactness of the closed unit ball in $\mf{A}^*$ we can find a subnet $(\xi_{m_i})_{i\in I}$ such that the formula
\[
\delta(a)=\lim_{i\in I}\is{\xi_{m_i}}{a \xi_{m_i}}\quad(a\in \mf{A}).
\]
defines a bounded functional (with norm $\le 1$) on $\mf{A}$. We have proven that it satisfies
\[
\delta(\chi^{\int}(\int_X^{\oplus} \pi_x\md\mu_X(x)))=
\lim_{i\in I} \ismaa{\xi_{m_i}}{
\chi^{\int}(\int_X^{\oplus} \pi_x\md\mu_X(x)) \xi_{m_i}}
=\int_X\dim\md\mu_X
\]
for any $\pi_X\in \Rep_{q,<+\infty}^{\int}(\GG)$. Functional $\delta$ is a character (a nonzero $\star$-homomorphism $\mf{A}\rightarrow\CC$). Indeed, we know that it is a well defined nonzero bounded linear map. It is therefore enough to prove that it is $\star$-multiplicative on $\mscr{A}$. Let $\chi(\int_{X}^{\oplus}\pi_x\md\mu_{X}(x))\in\mscr{A}$. First, let us check that it preserves the star:
\[
\begin{split}
&\quad\;\delta\bigl(\bigl(
\chi^{\int}(\int_{X}^{\oplus}\pi_x\md\mu_{X}(x))\bigr)^*\bigr)=
\delta\bigl(\chi^{\int}(\int_{\ov{X}}^{\oplus}\pi_{\ov{x}}\md\mu_{\ov{X}}(x))\bigr)\\
&=
\int_{\ov{X}} \dim (\pi_{\ov x})
\md\mu_{\ov{X}}(\ov x)=
\int_{X} \dim (\pi_{ x})
\md\mu_{X}(x)\\
&=
\delta\bigl(\chi^{\int}(\int_{X}^{\oplus}\pi_{x}\md\mu_{X}(x))\bigr)=
\ov{\delta\bigl(
\chi^{\int}(\int_{X}^{\oplus}\pi_x\md\mu_{X}(x))\bigr)}.
\end{split}
\]
Now we check that $\delta$ is multiplicative: we have
\[
\begin{split}
&\quad\;\delta\bigl(
\chi^{\int}(\int_{X}^{\oplus}\pi_x\md\mu_{X}(x))\chi^{\int}(\int_{Y}^{\oplus}\sigma_y\md\mu_{Y}(y))=
\delta\bigl(
\chi^{\int}(\int_{X\times Y}^{\oplus}\pi_x\tp\sigma_y \md\mu_{X\times Y}((x,y))\bigr)\\
&=
\int_{X\times Y} \dim \md\mu_{X \times Y}=
\bigl(\int_X \dim \md\mu_X\bigr)
\bigl(\int_Y \dim \md\mu_Y\bigr)\\
&=
\delta\bigl(
\chi^{\int}(\int_{X}^{\oplus}\pi_x\md\mu_{X}(x))\bigr)
\delta\bigl(
\chi^{\int}(\int_{Y}^{\oplus}\sigma_y\md\mu_{Y}(y))\bigr)
\end{split}
\]
for any $\chi^{\int}(\int_{X}^{\oplus}\pi_x\md\mu_{X}(x)),\chi^{\int}(\int_{Y}^{\oplus}\sigma_y\md\mu_{Y}(y))\in\mscr{A}$, which proves the first implication.\\
Assume now that there exists a character $\delta\in \mf{A}^*$ as in the statement of the theorem. It is clear that $\|\delta\|=1$. Since $\mf{A}\subseteq\Linf$, we can extend $\delta$ to a state on $\Linf$ and then find a net of unit vectors $(\xi_i)_{i\in \mc{I}}$ in $\LdG$ such that
\[
\delta(a)=\lim_{i\in \mc{I}} \ismaa{\xi_i}{a\xi_i} \quad
(a\in \mf{A}).
\]
Take any Plancherel measure $\mu$ and a measurable subset $\Omega\subseteq \IrrG$ such that $\int_\Omega \dim\md\mu<+\infty$. Since
\[
\int_{\Omega}\dim \md\mu=
\delta\bigl(
\chi^{\int}(\int_{\Omega}\pi\md\mu_{\Omega}(\pi))\bigr)=
\lim_{i\in \mc{I}}
\ismaa{\xi_i}{\chi^{\int}(\int_{\Omega}\pi\md\mu_{\Omega}(\pi)) \xi_i},
\]
we have
\[
\int_\Omega\dim\md\mu\in \sigma\bigl(
\chi^{\int}(\int_{\Omega}\pi\md\mu_{\Omega}(\pi))\bigr),
\]
and the second point of Theorem \ref{tw2} is satisfied.
\end{proof}

Let us introduce a \cst-algebra with a unit
\[
\mf{B}=\phi(\mf{A})+\CC\I_{\mc{H}}\subseteq\B(\mc{H}).
\]
We know that $\mf{B}$ is a \cst-algebra because image of a \cst-algebra under a $\star$-homomorphism is a \cst-algebra, and a sum of a closed subspace and finite dimensional one is closed.

\begin{lemma}\label{lemat33}
Assume that all irreducible representations of $\GG$ are admissible. If $\GG$ is coamenable then
\[
\int_\Omega\dim \md\mu\in 
\sigma_{\mf{B}}( \chi^{\int}(\int_{\Omega}^{\oplus}
\pi \md\mu_{\Omega}(\pi))\big|_{\mc{H}} )=
\sigma_{\B(\mc{H})}( \chi^{\int}(\int_{\Omega}^{\oplus}
\pi \md\mu_{\Omega}(\pi))\big|_{\mc{H}} )
\]
for any Plancherel measure $\mu$ and any measurable subset $\Omega\subseteq\IrrG$ such that $\int_\Omega\dim\md\mu<+\infty$.
\end{lemma}

\begin{proof}
Define a functional $\delta'$ on $\mf{B}$:
\[
\delta'\colon \mf{B}\ni \phi(a)\mapsto
\delta(a)\in \CC
\]
if $\I_{\mc{H}}\in \phi(\mf{A})$ and
\[
\delta'\colon \mf{B}\ni \phi(a) + z \I_{\mc{H}}\mapsto
\delta(a)+z\in \CC
\]
if\footnote{In this case $\mf{B}$ is the minimal unitization of $\phi(\mf{A})$.} $\I_{\mc{H}}\notin \phi(\mf{A})$, where $\delta$ is the character the existence of which was proven in Theorem \ref{tw3}. It is well defined since $\phi$ is an isometry. Functional $\delta'$ is a character. Indeed, in the first case it is clear, since $\delta$ is a character. In the second case $\delta'$ is a bounded linear functional. The fact that $\delta'$ is $\star$-multiplicative requires only simple calculation. We have defined a character $\delta'$ on $\mf{B}$ such that
\[
\delta'\bigl( \phi\bigl(\chi^{\int}(\int_{\Omega}^{\oplus}
\pi \md\mu_{\Omega}(\pi))\bigr)\bigr)=
\delta'( \chi^{\int}(\int_{\Omega}^{\oplus}
\pi \md\mu_{\Omega}(\pi))\big|_{\mc{H}})=
\int_{\Omega}\dim\md\mu.
\]
Consequently
\[
\int_\Omega\dim \md\mu\in 
\sigma_{\mf{B}}( \chi^{\int}(\int_{\Omega}^{\oplus}
\pi \md\mu_{\Omega}(\pi))\big|_{\mc{H}} )=
\sigma_{\B(\mc{H})}( \chi^{\int}(\int_{\Omega}^{\oplus}
\pi \md\mu_{\Omega}(\pi))\big|_{\mc{H}} ).
\]
The last equality follows from the fact that $\mf{B}\subseteq \B(\mc{H})$ is a unital \cst-subalgebra with the same unit.
\end{proof}

Recall that in equation \eqref{eq46} we have introduced operator $\mc{L}_\nu=\int_{\IrrG} \tfrac{\nu(\kappa)}{\dim(\kappa)} \mc{L}_{\kappa} \md\mu(\kappa)$. Having the above result, we can derive an analog of Theorem \ref{tw2} for operators $\mc{L}_\nu$.

\begin{theorem}\label{tw4}
Let $\GG$ be a second countable locally compact quantum group. Assume moreover that $\GG$ is type I and has only finite dimensional irreducible representations.
Consider the following conditions:
\begin{enumerate}[label=\arabic*)]
\item $\GG$ is coamenable.
\item Let $\mu$ be any Plancherel measure and $\Omega\subseteq\IrrG$ a measurable subset such that $\int_\Omega\dim\md\mu<+\infty$. Define $\nu=\dim \chi_\Omega$. Then $\int_\Omega\dim\md\mu\in \sigma(\mc{L}_\nu)$.
\item Let $\mu$ be a Plancherel measure which is invariant under conjugation. Let $\Omega\subseteq \IrrG$ be a symmetric subset such that $\nu=\dim\chi_\Omega\in\LL^1(\IrrG,\mu)$. Then $\int_{\Omega}\dim\md\mu\in \sigma(\mc{L}_\nu)$.
\end{enumerate}
We have $1)\Rightarrow 2)\Rightarrow 3)$. If all irreducible representations of $\GG$ are admissible then also $3)\Rightarrow 1)$.
\end{theorem}

\begin{proof}
Assume that $\GG$ is coamenable, we have chosen an arbitrary Plancherel measure $\mu$ and a subset $\Omega\subseteq\IrrG$ such that $\nu=\dim\chi_\Omega\in\LL^1(\IrrG)$. Thanks to Lemma \ref{lemat33} we know that
\[
\int_\Omega\dim\md\mu\in \sigma_{\B(\mc{H})}(\chi^{\int}(\int_\Omega^{\oplus} \pi\md\mu_\Omega(\pi))\big|_{\mc{H}}).
\]
Corollary \ref{wniosek2} gives us a unitary equivalence
\[
T'\mc{L}_\nu =\chi^{\int}(\int_{\Omega}^{\oplus}\pi\md\mu_{\Omega}(\pi))\big|_{\mc{H}}T',
\]
($T'\colon\LL^2(\IrrG)\rightarrow T(\LL^2(\IrrG))=\mc{H}$ is an operator $T$ with restricted codomain, so that it is unitary not just isometric). As unitary equivalence does not change the spectrum, we have $1)\Rightarrow 2)$.\\
Implication $2)\Rightarrow 3)$ is trivial.\\
Assume that all irreducible representations of $\GG$ are admissible and point $3)$ holds. Let $\mu$ be a Plancherel measure which is invariant under conjugation, and let $\Omega$ be a symmetric measurable subset such that $\int_\Omega \dim\md\mu<+\infty$. Let $\nu=\dim \chi_\Omega$. Point $3)$ gives us $\int_\Omega \dim \md\mu\in \sigma(\mc{L}_\nu)$, and as before Corollary \ref{wniosek2} implies 
\[
\int_\Omega\dim\md\mu\in 
\sigma_{\B(\mc{H})}(\chi^{\int}(\int_\Omega^{\oplus} \pi\md\mu_\Omega(\pi))\big|_{\mc{H}}).
\]
Lemma \ref{lemat12} shows that
\[
\int_\Omega\dim\md\mu\in \sigma(\chi^{\int}(\int_\Omega^{\oplus} \pi\md\mu_\Omega(\pi))).
\]
Implication $3)\Rightarrow 1)$ from Theorem \ref{tw2} ends the proof.
\end{proof}

\section{Examples}
\subsection{$\GG$ compact}
Assume now that $\GG$ is a compact quantum group with at most countably many classes of irreducible representations. For the theory of compact quantum groups let us refer to \cite{NeshTu}. We have
\[
\CGDu=\CGD=\mathrm{c}_0(\whG)=\bigoplus_{\alpha\in\IrrG}^{c_0} \B(\msf{H}_\alpha),\quad
\Linfd=\ell^{\infty}(\whG)=\bigoplus_{\alpha\in\IrrG}^{\ell^{\infty}}
\B(\msf{H}_\alpha)
\]
and $\IrrG$ is a discrete measurable space. We declare all vector fields on $\IrrG$ to be measurable. Define operators $\mc{Q}_L,\mc{Q}_R$ to be
\[
\begin{split}
&\mc{Q}_L\colon \LL^2(\GG)\ni
\Lhvp\big((T_\alpha)_{\alpha\in\IrrG}\bigr)\mapsto
\int_{\IrrG}^{\oplus}T_\alpha\,{\uprho_\alpha}^{-\frac{1}{2}}\md\mu(\alpha) 
\in\int_{\IrrG}^{\oplus}\HS(\msf{H}_\alpha)\md\mu(\alpha),\\
&
\mc{Q}_R\colon \LL^2(\GG)\ni
\hat{J}J\Lambda_{\widehat{\psi}}\big((T_\alpha)_{\alpha\in\IrrG}\bigr)\mapsto
\int_{\IrrG}^{\oplus}T_\alpha\,{\uprho_\alpha}^{\frac{1}{2}}\md\mu(\alpha) 
\in\int_{\IrrG}^{\oplus}\HS(\msf{H}_\alpha)\md\mu(\alpha),
\end{split}
\]
where $(T_\alpha)_{\alpha\in\IrrG}$ belongs respectively: to $\mf{N}_{\hvp}$ in the case of $\mc{Q}_L$ and $\mf{N}_{\widehat{\psi}}$ in the case of $\mc{Q}_R$. Define positive invertible operators $D_\alpha,E_\alpha \in \B(\msf{H}_\alpha)$ and a measure $\mu$ on $\IrrG$ via
\[
D_\alpha={\uprho_\alpha}^{\frac{1}{2}},\;
E_\alpha={\uprho_\alpha}^{-\frac{1}{2}},\;
\mu(\{\alpha\})=d_\alpha\quad(\alpha\in\IrrG).
\]
where $d_\alpha$ is the quantum dimension of $\alpha$ (see \cite{NeshTu}).

\begin{proposition}\label{stw18}
The objects
\[
\mc{Q}_L,\;\mc{Q}_R,\;\mu,\; (D_\alpha)_{\alpha\in\IrrG},\;(E_\alpha)_{\alpha\in\IrrG}
\]
 satisfy all the conditions of theorems \ref{PlancherelL}, \ref{PlancherelR}.
\end{proposition}

In order to prove this proposition, we will use point $7)$ of theorems \ref{PlancherelL}, \ref{PlancherelR}. First, let us check that $\mc{Q}_L$ is a well defined isometry:
\[
\begin{split}
&\bigl\|\int_{\IrrG}^{\oplus}
T_\alpha {\uprho_\alpha}^{-\frac{1}{2}}\md\mu(\alpha)\bigr\|^2=
\int_{\IrrG} \bigl\| T_\alpha {\uprho_\alpha}^{-\frac{1}{2}} \bigr\|^2_{\HS} \md\mu(\alpha)\\
&=
\sum_{\alpha\in\IrrG} d_\alpha \Tr({\uprho_\alpha}^{-1} {T_\alpha}^* T_\alpha)=
\| \Lhvp ((T_\alpha)_{\alpha\in\IrrG} )\|^2
\end{split}
\]
It is clear that the image of $\mc{Q}_L$ is dense, hence $\mc{Q}_L$ is a unitary operator. Analogous argument shows that $\mc{Q}_R$ also is unitary. \\
For $\omega\in\Lj$ such that $\lambda(\omega)\in\mf{N}_{\hvp}$ we have
\[
\mc{Q}_L \Lhvp(\lambda(\omega))=
\int_{\IrrG}^{\oplus} \alpha(\lambda(\omega))
{\uprho_\alpha}^{-\frac{1}{2}}\md\mu(\alpha)=
\int_{\IrrG}^{\oplus} (\omega\otimes\id)(U^{\alpha})
{\uprho_\alpha}^{-\frac{1}{2}}\md\mu(\alpha).
\]
Similarly, for $\omega\in\Lj$ such that $\lambda(\omega)\in \mf{N}_{\hpsi}$ we have
\[
\mc{Q}_R \hat{J}J \Lambda_{\hpsi}(\lambda(\omega))=
\int_{\IrrG}^{\oplus} (\omega\otimes\id)(U^{\alpha})\,\uprho_{\alpha}^{\frac{1}{2}}\md\mu(\alpha),
\]
which proves point $7.2)$.\\
Take $x\in\mf{N}_{\hvp}$ and $\omega\in\Lj$. We have
\[
\begin{split}
\mc{Q}_L ((\omega\otimes\id)\mrW) \Lhvp(x)
&=
\mc{Q}_L \Lhvp(((\omega\otimes\id)\mrW) x)=
\int_{\IrrG}^{\oplus} \alpha(((\omega\otimes\id)\mrW)x)
{\uprho_\alpha}^{-\frac{1}{2}} \md\mu(\alpha)\\
&=
\int_{\IrrG}^{\oplus} (\omega\otimes\id)(U^{\alpha})\alpha(x)
{\uprho_\alpha}^{-\frac{1}{2}} \md\mu(\alpha),
\end{split}
\]
on the other hand
\[
\begin{split}
&\quad\;\bigl(\int_{\IrrG}^{\oplus} (\omega\otimes\id)(U^{\alpha})\otimes
\I_{\ov{\msf{H}_\alpha}}\md\mu(\alpha)\bigr)\mc{Q}_L \Lhvp(x)\\
&=
\int_{\IrrG}^{\oplus}
((\omega\otimes\id)(U^{\alpha})\otimes
\I_{\ov{\msf{H}_\alpha}})
\alpha(x){\uprho_\alpha}^{-\frac{1}{2}}\md\mu(\alpha)\\
&=
\int_{\IrrG}^{\oplus}
(\omega\otimes\id)(U^{\alpha}) \alpha(x) {\uprho_\alpha}^{-\frac{1}{2}}
\md\mu(\alpha).
\end{split}
\]
The last equality follows from the isomorphism $\HS(\msf{H}_\alpha)=\msf{H}_\alpha\otimes\ov{\msf{H}_\alpha}$. The above calculation proves the commutation rule 
\[
\mc{Q}_L (\omega\otimes\id)\mrW=
\bigl(\int_{\IrrG}^{\oplus} (\omega\otimes\id)(U^{\alpha})\otimes
\I_{\ov{\msf{H}_\alpha}}\md\mu(\alpha)\bigr)\mc{Q}_L\quad(\omega\in\Lj).
\]
Let us introduce a dense *-subalgebra in $c_0(\whG)$:
\[
c_{00}(\whG)=\bigoplus_{\alpha\in\IrrG}^{alg} \B(\msf{H}_\alpha).
\]
In order to show the second commutation rule, we need the following lemma:
\begin{lemma}
The subspace $c_{00}(\whG)$ is a $\ssot\times \|\cdot\|$ core for $\Lhvp$.
\end{lemma}

Above (and everywhere else) we treat $\ell^{\infty}(\whG)$ as a subalgebra of $\B(\LdG)$, not $\B(\bigoplus_{\alpha\in\IrrG} \msf{H}_\alpha)$.
\begin{proof}
Let $T=(T_\alpha)_{\alpha\in\IrrG}\in \mf{N}_{\hvp}$, that is
\[
\hvp(T^*T)=\sum_{\alpha\in\IrrG}d_\alpha \Tr(T_\alpha^*T_\alpha \uprho^{-1}_\alpha)<+\infty.
\] 
Let $\{X_n\,|\,n\in\NN\}$ be any increasing family of finite subsets of $\IrrG$. Let $(T^n)_{n\in\NN}$ be a sequence of elements of $c_{00}(\whG)$ given by $T^n_\alpha=\chi_{X_n}(\alpha) T_\alpha$. It is clear that $T^n\in \mf{N}_{\hvp}$ for each $n\in\NN$. We have $\Lhvp(T^n)\xrightarrow[n\to\infty]{}\Lhvp(T)$. Indeed
\[\begin{split} 
&\quad\;\|\Lhvp(T)-\Lhvp(T^n)\|^2=
\sum_{\alpha\in\IrrG}d_\alpha \Tr((T_\alpha-T^n_\alpha)^*
(T_\alpha-T^n_\alpha) \uprho^{-1}_\alpha)\\
&=
\sum_{\alpha\in X_n}
d_\alpha \Tr((T_\alpha-T^n_\alpha)^*
(T_\alpha-T^n_\alpha) \uprho^{-1}_\alpha)+
\sum_{\alpha\in \IrrG\setminus X_n}
d_\alpha \Tr((T_\alpha-T^n_\alpha)^*
(T_\alpha-T^n_\alpha) \uprho^{-1}_\alpha)\\
&=
\sum_{\alpha\in\IrrG\setminus X_n}d_\alpha \Tr(T_\alpha^*
T_\alpha\uprho^{-1}_\alpha)\xrightarrow[n\to\infty]{}0.
\end{split}\]
Furthermore, we have $T^n\xrightarrow[n\to\infty]{\ssot}T$. Indeed: as the sequence $(T^n)_{n\in\NN}$ is bounded, it is enough to check convergence in \sot\, and for vectors from a dense subspace $\{\Lhvp(S)\,|\,S\in \mf{N}_{\hvp}\cap\Dom(\sigma^{\hvp}_{i/2})\}$. For any $S\in \mf{N}_{\hvp}\cap\Dom(\sigma^{\hvp}_{i/2})$ we have
\[\begin{split}
&\quad\;\|T\Lhvp(S)-T^n\Lvp(S)\|=
\|\Lhvp((T-T^n)S)\|=
\|\hat{J}\sigma^{\hvp}_{-i/2}(S^*)\hat{J} \Lhvp(T-T^n)\|\\
&\le
\|\sigma^{\hvp}_{-i/2}(S^*)\| \|\Lhvp(T-T^n)\|\xrightarrow[n\to\infty]{}0,
 \end{split}\]
which proves the claim.
\end{proof}

Let us now check the second commutation rule. Take any $T=(T_\alpha)_{\alpha\in\IrrG}\in c_{00}(\whG)$ and $\omega\in\Lj$ such that $\lambda(\omega)\in c_{00}(\whG)$. Let us note that the unbounded operators $\uprho=\bigoplus_{\alpha\in \IrrG}\uprho_\alpha$ and $\nabla_{\hvp}$ have the subspace $\Lhvp(c_{00}(\whG))$ in their domain, and moreover this subspace is preserved by them. Indeed, it is clear for $\uprho$, and we know that $\nabla_{\hvp}\Lhvp(e^\alpha_{i,j})=\tfrac{(\uprho_\alpha)_j}{(\uprho_\alpha)_{i}} \Lhvp(e^\alpha_{i,j})$. By the definition of $\mrV$ we have
\[\begin{split}
\mc{Q}_L (\omega\otimes\id)\chi(\mrV)\Lhvp(T)=
\mc{Q}_L\hat{J} \hat{R} ((\omega\otimes\id)\mrW) ^*\hat{J}
\Lhvp(T).
 \end{split}\]
On the other hand
\[\begin{split} 
&\quad\;
\bigl(\int_{\IrrG}^{\oplus} \I_{\msf{H}_\alpha}\otimes \alpha^{c}((\omega\otimes\id)\mrW) \md\mu(\alpha)\bigr)\mc{Q}_L \Lhvp(T)\\
&=
\int_{\IrrG}^{\oplus} \bigl(\I_{\msf{H}_\alpha}\otimes \alpha^{c}((\omega\otimes\id)\mrW) \bigr)
\bigl( \sum_{j=1}^{\dim(\alpha)} |\zeta^\alpha_j\rangle\langle
(T_\alpha \uprho_\alpha^{-\frac{1}{2}} )^* \zeta^\alpha_j| \md\mu(\alpha)\\
&=
\int_{\IrrG}^{\oplus} 
\sum_{j=1}^{\dim(\alpha)} \zeta^\alpha_j\otimes 
\ov{
\alpha\circ \hat{R}((\omega\otimes\id)\mrW)^*
\;(T_\alpha \uprho_\alpha^{-\frac{1}{2}} )^* \zeta^\alpha_j} \md\mu(\alpha)\\
&=
\int_{\IrrG}^{\oplus} 
T_\alpha \uprho_\alpha^{-\frac{1}{2}}
\alpha\circ \hat{R}((\omega\otimes\id)\mrW)\md\mu(\alpha)\\
&=
\mc{Q}_L \Lhvp(( T_\alpha \uprho_\alpha^{-\frac{1}{2}}
\alpha\circ \hat{R}((\omega\otimes\id)\mrW) \uprho_\alpha^{\frac{1}{2}})_{\alpha\in \IrrG})\\
&=
\mc{Q}_L \hat{J}\nabla_{\hvp}^{\frac{1}{2}}
(\uprho_\alpha^{-\frac{1}{2}}
\alpha\circ \hat{R}((\omega\otimes\id)\mrW) \uprho_\alpha^{\frac{1}{2}})_{\alpha\in\IrrG}^* \hat{J}\nabla_{\hvp}^{\frac{1}{2}} \Lhvp(T)\\
&=
\mc{Q}_L \hat{J}\nabla_{\hvp}^{\frac{1}{2}}
(\uprho_\alpha^{\frac{1}{2}}
\alpha\circ \hat{R}((\omega\otimes\id)\mrW)^* \uprho_\alpha^{-\frac{1}{2}})_{\alpha\in\IrrG} \nabla_{\hvp}^{-\frac{1}{2}}\hat{J} \Lhvp(T),
\end{split}\]
where $\{\zeta^{\alpha}_j\,|\,j\in\{1,\dotsc,\dim(\alpha)\}\}$ is any orthonormal basis in $\msf{H}_\alpha$. Since
\[\begin{split}
\nabla_{\hvp}^{\frac{1}{2}}\rho^{\frac{1}{2}}\Lhvp(e^\beta_{k,l})=
(\uprho_\beta)_k^{\frac{1}{2}}\nabla^{\frac{1}{2}}_{\hvp}\Lhvp(e^{\beta}_{k,l})=
(\uprho_\beta)_k^{\frac{1}{2}}
\bigl(\tfrac{(\uprho_\beta)_{l}}{(\uprho_\beta)_k}\bigr)^{\frac{1}{2}}
\Lhvp(e^\beta_{k,l})=
(\uprho_\beta)_{l}^{\frac{1}{2}}\Lhvp(e^\beta_{k,l})
 \end{split}\]
then
\[
e^{\alpha}_{i,j}\nabla_{\hvp}^{\frac{1}{2}}\rho^{\frac{1}{2}}\Lhvp(e^\beta_{k,l})=
\delta_{\alpha,\beta}\delta_{j,k} (\uprho_\beta)_{l}^{\frac{1}{2}}
\Lhvp(e^\beta_{i,l})=
\nabla_{\hvp}^{\frac{1}{2}}\rho^{\frac{1}{2}}
e^{\alpha}_{i,j}\Lhvp(e^\beta_{k,l})
\]
and the operator $\nabla_{\hvp}^{\frac{1}{2}}\rho^{\frac{1}{2}}$ commutes with operators from $\LL^{\infty}(\whG)$ (on $\Lhvp(c_{00}(\whG))$). Consequently
\[\begin{split}
&\quad\;
\mc{Q}_L \hat{J}\nabla_{\hvp}^{\frac{1}{2}}
(\uprho_\alpha^{\frac{1}{2}}
\alpha\circ \hat{R}((\omega\otimes\id)\mrW)^* \uprho_\alpha^{-\frac{1}{2}})_{\alpha\in\IrrG} \nabla_{\hvp}^{-\frac{1}{2}}\hat{J} \Lhvp(T)\\
&=
\mc{Q}_L \hat{J}(
\alpha\circ \hat{R}((\omega\otimes\id)\mrW)^*)_{\alpha\in\IrrG}
\nabla_{\hvp}^{\frac{1}{2}}\uprho^{\frac{1}{2}} 
 \uprho^{-\frac{1}{2}} \nabla_{\hvp}^{-\frac{1}{2}}\hat{J} \Lhvp(T)\\
&=
\mc{Q}_L\hat{J}
(\alpha\circ \hat{R}((\omega\otimes\id)\mrW)^*)_{\alpha\in\IrrG}
\hat{J}\Lhvp(T),
 \end{split}\]
and the second commutation relation holds.\\
Assume that $x=(x_\alpha)_{\alpha\in\IrrG}$ is an element of $\LL^{\infty}(\whG)\cap\LL^{\infty}(\whG)'$. Triviality of the center of $\B(\msf{H}_\alpha)$ impliess that $x_\alpha\in \CC\I_\alpha$ for each $\alpha\in\IrrG$. Operator $x$ is mapped via $\mc{Q}_L$ to $\int_{\IrrG}^{\oplus} x_\alpha\md\mu(\alpha)$, which is by the definition a diagonalisable operator. On the other hand, any diagonalisable operator $\int_{\IrrG}^{\oplus} y_\alpha \md\mu(\alpha)\;(y_\alpha\in\CC \I_{\HS(\msf{H}_\alpha)})$ is an image of $(y_\alpha)_{\alpha\in\IrrG}\in\LL^{\infty}(\whG)\cap\LL^{\infty}(\whG)'$. This proves that we have identified objects that are given by the left version of Theorem \ref{PlancherelL}.\\
Let us now check that $\mc{Q}_R$ and $E_\pi={\uprho_\pi}^{-\frac{1}{2}}$ satisfy conditions from point $7)$ of Theorem \ref{PlancherelR}: we only need to check the commutation rules, since the rest is clear. Now we need to use the formula $\nabla_{\hpsi}\Lambda_{\hpsi}(e^\alpha_{i,j})=\tfrac{(\uprho_\alpha)_i}{(\uprho_\alpha)_j}\Lambda_{\hpsi}(e^\alpha_{i,j})$. Take $\omega$ and $T$ as before. We have
\[\begin{split}
&\quad\;
\mc{Q}_R \hat{J}J (\omega\otimes\id)\mrW \Lambda_{\hpsi}(T)=
\mc{Q}_R \hat{J}J \Lambda_{\hpsi}(((\omega\otimes\id)\mrW)\, T)\\
&=
\int_{\IrrG}^{\oplus} \alpha( ((\omega\otimes\id)\mrW) T) \uprho_{\alpha}^{\frac{1}{2}} \md\mu(\alpha)
=
\int_{\IrrG}^{\oplus} 
(\omega\otimes\id)U^\alpha T_\alpha
\uprho_{\alpha}^{\frac{1}{2}} \md\mu(\alpha)\\
&=
\bigl(\int_{\IrrG}^{\oplus} (\omega\otimes\id)U^{\alpha} \otimes\I_{\ov{\msf{H}_\alpha}}\md\mu(\alpha)\bigr)
\mc{Q}_R \hat{J}J \Lambda_{\hpsi}(T),
\end{split}\]
which shows the first commutation rule. Let us now prove the second one:
\[\begin{split}
&\quad\;
\bigl(\int_{\IrrG}^{\oplus} 
\I_{\msf{H}_\alpha}\otimes \alpha^{c}((\omega\otimes\id){\WW})
\md\mu(\alpha)\bigr)
\mc{Q}_R \hat{J}J \Lambda_{\hpsi}(T)\\
&=
\int_{\IrrG}^{\oplus} T_\alpha \uprho_{\alpha}^{\frac{1}{2}}
\,\alpha\circ \hat{R} ((\omega\otimes\id)\mrW)\md\mu(\alpha)\\
&=
\mc{Q}_R \hat{J}J \Lambda_{\hpsi}\bigl((T_\alpha \uprho_{\alpha}^{\frac{1}{2}}
\,\alpha\circ \hat{R} ((\omega\otimes\id)\mrW) \uprho_{\alpha}^{-\frac{1}{2}}
)_{\alpha\in\IrrG}\bigr)\\
&=
\mc{Q}_R \hat{J}J
J^{\hpsi} \nabla_{\hpsi}^{\frac{1}{2}}
\bigl( \uprho_{\alpha}^{\frac{1}{2}}
\,\alpha\circ \hat{R} ((\omega\otimes\id)\mrW) \uprho_{\alpha}^{-\frac{1}{2}}
\bigr)_{\alpha\in\IrrG}^*
J^{\hpsi} \nabla_{\hpsi}^{\frac{1}{2}}\Lambda_{\hpsi}(T)\\
&=
\mc{Q}_R \hat{J}J
\hat{J} 
\hat{R} ((\omega\otimes\id)\mrW) ^*
\hat{J}\Lambda_{\hpsi}(T)=
\mc{Q}_R \hat{J}J
(\omega\otimes\id)\chi(\mrV)
\Lambda_{\hpsi}(T),
\end{split}\]
which concludes the proof of Proposition \ref{stw18}.

\begin{remark}
Note that one gets a general Plancherel measure by taking any positive measure on $\IrrG$ with full support. Indeed, let $c\colon \IrrG\rightarrow \RR_{>0}$ be an arbitrary function. Define measure $\mu^c\colon \{\alpha\}\mapsto c(\alpha)$. It is equivalent to the above measure $\mu=\mu^{d_\bullet}$ and we have
\[
\frac{\md\mu^c}{\md\mu}=\frac{c}{d_\bullet}.
\]
With this choice of a Plancherel measure we can relate the following Duflo-Moore operators:
\[
\begin{split}
D_\alpha&=\sqrt{\frac{\md\mu^c}{\md\mu}(\alpha)} {\uprho_\alpha}^{\frac{1}{2}}=c(\alpha)^{\frac{1}{2}} {d_\alpha}^{-\frac{1}{2}} {\uprho_\alpha}^{\frac{1}{2}},\\
E_\alpha&=\sqrt{\frac{\md\mu^c}{\md\mu}(\alpha)} {\uprho_\alpha}^{-\frac{1}{2}}=c(\alpha)^{\frac{1}{2}} {d_\alpha}^{-\frac{1}{2}} {\uprho_\alpha}^{-\frac{1}{2}}.
\end{split}
\]
\end{remark}

\subsubsection{Functions $\varpi$}
Choose the Plancherel measure $\mu$ with $c=1$, that is $\mu(\{\alpha\})=\mu^{c}(\{\alpha\})=1$ for all $\alpha\in \IrrG$. Take $\Omega\subseteq\IrrG$, and $\kappa\in \IrrG$. Define $\sigma_\Omega=\bigoplus_{\alpha\in \Omega}\alpha$. We have the following decomposition:
\[
\kappa\tp\sigma_\Omega\simeq\bigoplus_{k=1}^{\infty} \bigoplus_{\alpha\in \F^k_{\kappa\stp\sigma_\Omega}}\alpha.
\]
Define representations $\pi$ and $\sigma$ to be
\[
\pi=
\bigoplus_{k=1}^{\infty} \bigoplus_{\alpha\in \F^k_{\kappa\stp\sigma_{\Omega}}}\alpha,\quad
\sigma=
\bigoplus_{\alpha\in \Omega}\kappa\tp\alpha.
\]
Let $\mc{O}$ be a unitary intertwiner between $\pi$ and $\sigma$. We wish to find a function $\varpi^{\kappa,\Omega,\mu}\colon\F^1_{\kappa\stp\sigma_\Omega}\rightarrow\RR_{>0}$ satisfying
\[
\sum_{m=1}^{\infty} (\int_{\Omega\restriction_{m}}\otimes \Tr_m)
\bigl( \mc{O} \pi(a)\mc{O}^*\bigr)=
\sum_{n=1}^{\infty} \int_{\F^n_{\kappa\stp\sigma_\Omega}} \varpi(\alpha) \Tr(\alpha(a)) \md\mu_{\F^n_{\kappa\stp\sigma_\Omega}}(\alpha)\quad(a\in c_{00}(\whG)).
\]
Choose any function with finite support $g\colon \IrrG\rightarrow \CC$ and define $a\in c_{00}(\whG)$ to be $a=(g(\alpha)
\I_{\msf{H}_\alpha})_{\alpha\in\IrrG}$. We have
\[
\mc{O} \pi(a)\mc{O}^*=
\sigma(a)=
\bigoplus_{\alpha\in\Omega}\kappa\tp\alpha( a).
\]
For any $\alpha\in\Omega$ we have
\[
\kappa\tp\alpha(a)\simeq
\bigoplus_{k=1}^{\infty}\bigoplus_{\beta\in\F^{k}_{\kappa\stp\alpha}}
\beta(a),
\]
therefore
\[
\Tr(\kappa\tp\alpha(a))=\sum_{k=1}^{\infty} \sum_{\beta\in\F^k_{\kappa\stp\alpha}}\Tr(\beta(a))=\sum_{k=1}^{\infty}\sum_{\beta\in\F^k_{\kappa\stp\alpha}} \dim(\beta)g(\beta)
\]
The above information gives us
\begin{equation}\begin{split} \label{eq21}
&\quad\;
\sum_{\alpha\in\Omega}\sum_{k=1}^{\infty} \sum_{\beta\in
\F^{k}_{\kappa\stp\alpha}}\dim(\beta)g(\beta)=
\int_{\Omega} \Tr(\kappa\tp\alpha(a))\md\mu_{\Omega}(\alpha)\\
&=
\sum_{n=1}^{\infty} \int_{\F^n_{\kappa\stp\sigma_\Omega}} \varpi^{\kappa,\mu,\Omega}(\alpha) \Tr(\alpha(a)) \md\mu_{\F^n_{\kappa\stp\sigma_\Omega}}(\alpha)=
\sum_{n=1}^{\infty} \sum_{\alpha\in\F^n_{\kappa\stp\sigma_\Omega}} \varpi^{\kappa,\mu,\Omega}(\alpha) \dim(\alpha)g(\alpha)
\end{split}\end{equation}
Observe that the following holds:
\begin{equation}\label{eq29}
\bigoplus_{\alpha\in \Omega} \bigoplus_{k'=1}^{\infty} \bigoplus_{\beta\in \F^{k'}_{\kappa\stp\alpha}}\beta\simeq
\bigoplus_{\alpha\in \Omega} (\kappa\tp \alpha)\simeq
\kappa\tp \sigma_\Omega\simeq
\bigoplus_{k'=1}^{\infty} \bigoplus_{\beta\in \F^{k'}_{\kappa\stp \sigma_\Omega}}\beta,
\end{equation}
therefore for any $\beta\in\IrrG$ we have
\[
\sum_{\alpha\in\Omega} \sum_{k=1}^{\infty} \chi_{ \F^k_{\kappa\stp\alpha}}(\beta)
=\sum_{k=1}^{\infty} \chi_{\F^{k}_{\kappa\stp\sigma_{\Omega}}}(\beta)
\]
(it is the multiplicity of $\beta$ in the unitarily equivalent representations \eqref{eq29}). Let us muliply this formula by $\dim(\beta)g(\beta)$ and take a sum over $\beta$:
\[\begin{split}
&\quad\;
\sum_{\alpha\in\Omega}\sum_{k=1}^{\infty}
\sum_{\beta\in \F^k_{\kappa\stp\alpha}} \dim(\beta) g(\beta) =
\sum_{\beta\in \IrrG} \dim(\beta) g(\beta)\sum_{\alpha\in\Omega} \sum_{k=1}^{\infty} \chi_{\F^k_{\kappa\stp\alpha}}(\beta)\\
&=
\sum_{\beta\in\IrrG}\dim(\beta)g(\beta)
\sum_{k=1}^{\infty} \chi_{\F^{k}_{\kappa\stp\sigma_{\Omega}}}(\beta)=
\sum_{k=1}^{\infty}\sum_{\beta\in\F^k_{\kappa\stp\sigma_\Omega}} \dim(\beta)g(\beta).
\end{split}
\]
Once we substitute this result to the equation \eqref{eq21} we get
\[\begin{split}
&\quad\;
\sum_{k=1}^{\infty} \sum_{\beta\in\F^k_{\kappa\stp\sigma_\Omega}} \dim(\beta) g(\beta)=
\sum_{\alpha\in\Omega} \sum_{k=1}^{\infty} \sum_{\beta\in
\F^{k}_{\kappa\stp\alpha}}\dim(\beta)g(\beta)=
\sum_{k=1}^{\infty} \sum_{\alpha\in\F^k_{\kappa\stp\sigma_\Omega}} \varpi^{\kappa,\Omega,\mu}(\alpha) \dim(\alpha)g(\alpha).
\end{split}
\]
Since $g$ was arbitrary, we arrive at:
\[
\varpi^{\kappa,\Omega,\mu}(\alpha)=1\quad(\alpha\in\F^1_{\kappa\stp\sigma_{\Omega}}).
\]
\begin{proposition}
Let $\mu$ be a Plancherel measure given by $\mu(\{\alpha\})=1$ for all $\alpha\in \IrrG$ and let $\Omega\subseteq\IrrG$ be any subset. Then
\[
\varpi^{\kappa,\Omega,\mu}(\alpha)=1\quad(\alpha\in\F^1_{\kappa\stp\sigma_{\Omega}}).
\]
\end{proposition}

\subsubsection{Operators $\mc{L}_\kappa,\mc{L}_\nu$ and integral characters}
We stick to the choice $c=1$, that is the Plancherel measure $\mu=\mu^c$ is given by $\mu(\{\alpha\})=1\, (\alpha\in\IrrG)$. We have $E_\alpha= {d_\alpha}^{-\frac{1}{2}}{\uprho_\alpha}^{-\frac{1}{2}}$, therefore $\Tr(E^2_\bullet)=1$. Note also that once we choose this Plancherel measure we can identify the Hilbert spaces $\LL^2(\IrrG)$ and $\ell^2(\IrrG)$. Let us fix $\kappa\in\IrrG$. According to the definition of $\mc{L}_\kappa$ (equation \eqref{eq24}) we have
\[
\mc{L}_\kappa\colon \Tr(E^2_\bullet)^{\frac{1}{2}}\chi_\Omega=\chi_\Omega
\mapsto
\Tr(E^2_\bullet)^{\frac{1}{2}}\varpi^{\kappa,\Omega,\mu}\sum_{n=1}^{\infty}\chi_{\F^n_{\kappa\stp\sigma_\Omega}}=
\sum_{n=1}^{\infty}\chi_{\F^n_{\kappa\stp\sigma_\Omega}}
\]
for finite subsets $\Omega\subseteq\IrrG$, consequently we have an equality of $\mc{L}_\kappa$ and the operator $\lambda_\kappa$ of \cite{NeshTu}. Furthermore, for $\nu\in\LL^1(\IrrG)=\ell^1(\IrrG)$ we have:
\[
\mc{L}_\nu=\int_{\IrrG}\tfrac{\nu(\kappa)}{\dim(\kappa)}\mc{L}_\kappa\md\mu(\kappa)=
\sum_{\alpha\in\IrrG} \tfrac{\nu(\alpha)}{\dim(\alpha)}\mc{L}_\alpha.
\]

Let us recall that a part of \cite[Theorem 2.7.10]{NeshTu} states that a compact quantum group $\GG$ is coamenable if and only if the the fusion ring of $\GG$ is amenable. By \cite[Definition 2.7.6]{NeshTu} this condition means that $1\in \sigma(\mc{L}_\nu)$ for every probability measure $\nu\in \ell^1(\IrrG)$. It follows that our Theorem \ref{tw4} in the case of compact quantum group $\GG$ is closely related to this theorem (the difference is that we consider only measures which have full support, on the other hand our implication $3)\Rightarrow 1)$ in Theorem \ref{tw3} is stronger). We also wish to point out that in \cite[Theorem 4.1]{HiaiIzumi} a similar result appears, though in slightly different context of fusion algebras.\\

In the case of compact quantum groups (and the Plancherel measure with $c=1$) the integral in definition of the integral character (equation \eqref{eq25}) reduces to a sum, therefore for finite dimensional integral representations, the notion of an integral character and of a character coincides. A result relating coamenability of a matrix compact quantum group $\GG$ to the real part of a spectrum of a character of the fundamental representation was derived by Skandalis (see \cite[Theorem 6.1]{Banica}). Our Theorem \ref{tw2} is however more closely related to a version which does not assume the existence of the fundamental representation. Equivalence $(ii)\Leftrightarrow (iii)$ in \cite[Theorem 2.7.10]{NeshTu} states that $\GG$ is coamenable if and only if $\dim U\in \sigma(\chi(U))$ for every finite dimensional representation $U$. Theorem \ref{tw2} is similar to this result, however we also consider infinite dimensional representations.

\subsection{$\whG$ classical}\label{secdclass}
Assume now that $\whG$ is a classical locally compact group which is second countable. Let us denote by $\hat{\delta}$ the modular element of $\whG$, which is defined as the Radon-Nikodym derivative $\hat{\delta}=\tfrac{\md\mu_R}{\md\mu_L}$, where $\mu_L,\mu_R$ are respectively the left and the right Haar measure on $\whG$ (note that in the classical theory of topological groups one usually considers $\tfrac{\md\mu_L}{\md \mu_R}$). We have $\IrrG=\whG$ as a topological space, $\CGDu=\CGD$ and every point $\zeta\in\whG$ corresponds to the one dimensional representation of $\CGD$ given by evaluation at $\zeta$. We will abuse the notation and identify (as sets) $\msf{H}_\zeta$ and $\B(\msf{H}_\zeta)$ with $\CC$ for each $\zeta\in\whG$.\\
Take any $p\in \RR$. Define a measure $\mu_p=\hat{\delta}^{p}\mu_L=\hat{\delta}^{p-1}\mu_R$, the structure of measurable field of Hilbert spaces $(\CC)_{\zeta\in\whG}$ given by measurable functions on $\whG$, positive operators $D_\zeta=\hat{\delta}(\zeta)^{\frac{p}{2}},E_\zeta=\hat{\delta}(\zeta)^{\frac{p-1}{2}}\,(\zeta\in\whG)$ and operators $\mc{Q}_L,\mc{Q}_R$ given by
\[
\begin{split}
\mc{Q}_L\colon\LL^2(\GG)\ni \Lambda_{\hvp}(f) &\mapsto 
\int_{\whG}^{\oplus} f(\zeta)\hat{\delta}(\zeta)^{-\frac{p}{2}}\md\mu_p(\zeta)\in
\int_{\whG}^{\oplus} \HS(\msf{H}_\zeta)\md\mu_p(\zeta),\\
\mc{Q}_R\colon\LL^2(\GG)\ni \hat{J}J\Lambda_{\widehat{\psi}}(f)
 &\mapsto 
\int_{\whG}^{\oplus} f(\zeta)\hat{\delta}(\zeta)^{-\frac{p-1}{2}}\md\mu_p(\zeta)\in
\int_{\whG}^{\oplus} \HS(\msf{H}_\zeta)\md\mu_p(\zeta).
\end{split}
\]

 Operators $\mc{Q}_L,\mc{Q}_R$ are at first only densely defined: $f$ belongs respectively to $\mf{N}_{\hvp}$ and $\mf{N}_{\wh{\psi}}$.

\begin{proposition}\label{stw18}
For each $p\in\RR$ the objects
\[
\mc{Q}_L,\;\mc{Q}_R,\;\mu_p,\; (D_\zeta)_{\zeta\in\whG},\;(E_{\zeta})_{\zeta\in\whG}
\]
 satisfy all the conditions of theorems \ref{PlancherelL}, \ref{PlancherelR}.
\end{proposition}

From this proposition follows that a general Plancherel measure is given by $g\mu_L$ for a strictly positive function $g$. We restrict our attention to the case $g=\hat{\delta}^{p}$ because this choice simplifies our calculations. On the other hand, we could as well chose $p=\tfrac{1}{2}$, because it will turn out that the measure $\mu_{\frac{1}{2}}$ is invariant under the conjugation. However, we prefer to describe a more general situation; it will give us a more general result (Corollary \ref{wniosek3}) and includes natural choices of measures $\mu_L,\mu_R$. Yet another reason for this choice is Proposition \ref{stw20} -- in this result we calculate functions $\varpi^{\zeta,\Omega,\mu_p}$ and see how it depends on $p$.\\

First, let us check that densely defined operators $\mc{Q}_L,\mc{Q}_R$ are isometric:
\[
\begin{split}
\bigl\| \int_{\whG}^{\oplus} f(\zeta) \hat{\delta}(\zeta)^{-\frac{p}{2}}
\md\mu_p(\zeta) \bigr\|^2
&=
\int_{\whG} |f(\zeta)|^2 \hat{\delta}(\zeta)^{-p} \hat{\delta}(\zeta)^{p}\md\mu_L(\zeta)=
\|\Lambda_{\hvp}(f)\|^2,\\
\bigl\| \int_{\whG}^{\oplus} f(\zeta) \hat{\delta}(\zeta)^{-\frac{p-1}{2}}
\md\mu_p(\zeta) \bigr\|^2
&=
\int_{\whG} |f(\zeta)|^2 \hat{\delta}(\zeta)^{-p+1} \hat{\delta}(\zeta)^{p-1}\md\mu_R(\zeta)=
\|\Lambda_{\wh{\psi}}(f)\|^2.
\end{split}
\]
It follows that they extend to the whole $\LdG$. It is clear that they have dense image, therefore are unitary.\\
We will use point $7)$ of theorems \ref{PlancherelL}, \ref{PlancherelR}. We have
\[
\begin{split}
\mc{Q}_L (\Lhvp(\lambda(\alpha)))&=
\int_{\whG}^{\oplus} ((\alpha\otimes\id)\mrW) (\zeta)
\hat{\delta}(\zeta)^{-\frac{p}{2}}\md\mu_p(\zeta)\\
&=
\int_{\whG}^{\oplus} (\alpha\otimes\id)(U^{\zeta}) D_\zeta^{-1}
\md\mu_p(\zeta)
\end{split}
\]
for $\alpha\in\Lj$ such that $\lambda(\alpha)\in\mf{N}_{\hvp}$. Similarly,
\[
\begin{split}
\mc{Q}_R \hat{J}J(\Lambda_{\widehat{\psi}}(\lambda(\alpha)))
&=
\int_{\whG}^{\oplus} ((\alpha\otimes\id)\mrW) (\zeta)
\hat{\delta}(\zeta)^{-\frac{p-1}{2}}\md\mu_p(\zeta)\\
&=
\int_{\whG}^{\oplus} (\alpha\otimes\id)(U^{\zeta}) E_\zeta^{-1}
\md\mu_p(\zeta)
\end{split}
\]
for $\alpha\in\LL^1(\GG)$ such that $\lambda(\alpha)\in\mf{N}_{\hpsi}$. Consequently, point $7.2)$ holds. Now, for $f\in\mf{N}_{\hvp}$ and $\omega\in \Lj$ we have
\[
\begin{split}
&\mc{Q}_L (\omega\otimes\id)\mrW \Lambda_{\hvp}(f)
\\
=&\mc{Q}_{L} \Lambda_{\hvp} ((\omega\otimes\id)(\mrW) f)\\
=&\int_{\whG}^{\oplus} ((\omega\otimes\id)\mrW) (\zeta) f(\zeta
)
\hat{\delta}(\zeta)^{-\frac{p}{2}}\md\mu_p(\zeta)\\
=&
\bigl(\int_{\whG}^{\oplus} ((\omega\otimes\id)\mrW) (\zeta)\otimes\I_{\ov{\msf{H}_\zeta}} \md\mu_p(\zeta)\bigr)
\mc{Q}_L \Lambda_{\hvp}(f)\\
=&
\bigl( \int_{\whG}^{\oplus} (\omega\otimes\id)(U^{\zeta})\otimes
\I_{\ov{\msf{H}_\zeta}}\md\mu_p(\zeta)\bigr)\mc{Q}_L \Lhvp(f),
\end{split}
\]
which gives us the first commutation relation. We have $\chi(\mrV)=(\hat{J}\otimes\hat{J}){\mrW}^* (\hat{J}\otimes\hat{J})$, so that $(\omega\otimes\id)\chi(\mrV)=\hat{J} \hat{R}((\omega\otimes\id)\mrW)^* \hat{J}$ for $\omega\in\Lj$. The Haar integrals on $\whG$ are tracial, hence the operator $\hat{J}$ acts as follows: $\hat{J}\Lhvp(f)=\Lhvp(f^*)\,(f\in\mf{N}_{\hvp})$. Consequently for each $x\in \LL^\infty(\whG),f,g\in\mf{N}_{\hvp}$ the following holds
\[
\ismaa{\Lhvp(g)}{\hat{J} x^* \hat{J} \Lhvp(f)}=
\ismaa{\Lhvp(g)}{\Lhvp(fx)}=
\hvp(g^* fx)=\hvp(g^*xf)=
\ismaa{\Lhvp(g)}{x\Lhvp(f)}.
\]
It follows that $\hat{J}x^*\hat{J}=x\,(x\in \LL^{\infty}(\whG))$ and
\[
(\omega\otimes\id)\chi(\mrV)=\hat{R}((\omega\otimes\id)\mrW)\quad(\omega\in\Lj).
\]
We obviously have $\hat{R}(x)(\zeta)=x(\zeta^{-1})\,(x\in \LL^{\infty}(\whG),\zeta\in\whG)$, therefore
\[\begin{split} 
&\quad\; \mc{Q}_L (\omega\otimes\id)\chi(\mrV) \Lhvp(f)\\
&=\int_{\whG}^{\oplus} ((\omega\otimes\id)\chi(\mrV) f)(\zeta) \hat{\delta}(\zeta)^{-\frac{p}{2}} \md\mu_p(\zeta)\\
&=
\int_{\whG}^{\oplus}  (\omega\otimes\id)\mrW (\zeta^{-1})f(\zeta) \hat{\delta}(\zeta)^{-\frac{p}{2}} \md\mu_p(\zeta)
\end{split}\]
and on the other hand
\[\begin{split} 
&\quad\;\bigl(\int_{\IrrG}^{\oplus} \I_{\msf{H}_\pi}\otimes \pi^{c}((\omega\otimes\id)\mrW)\md\mu(\pi)\bigr)\mc{Q}_L\Lhvp(f)\\
&=
\bigl(\int_{\whG}^{\oplus} \I_{\msf{H}_\pi}\otimes \jmath_{\msf{H}_\zeta}(((\omega\otimes\id)\mrW)(\zeta^{-1}))\md\mu_p(\zeta)\bigr)\mc{Q}_L\Lhvp(f)\\
&=
\int_{\whG}^{\oplus} 
\bigl(
\I_{\msf{H}_\pi}\otimes \jmath_{\msf{H}_\zeta}(((\omega\otimes\id)\mrW)(\zeta^{-1}))
\bigr)
(f(\zeta)\hat{\delta}(\zeta)^{-\frac{p}{2}})
\md\mu_p(\zeta)\\
&=
\int_{\whG}^{\oplus} 
((\omega\otimes\id)\mrW)(\zeta^{-1})
f(\zeta)\hat{\delta}(\zeta)^{-\frac{p}{2}}
\md\mu_p(\zeta)
\end{split}\]
for $\omega\in\Lj,f\in\mf{N}_{\hvp}$, which ends the proof of commutation relations for $\mc{Q}_L$.\\
We have $\LL^{\infty}(\whG)\cap \LL^{\infty}(\whG)'=\LL^{\infty}(\whG)$ and it is clear that operator $\mc{Q}_L$ maps a function $x\in\LL^{\infty}(\whG)$ to the operator $\int_{\whG}^{\oplus} x(\zeta)\md\mu_p(\zeta)$. Note that for each $x\in \Linfd$ and $f\in\mf{N}_{\hpsi}$ we have
\[
\begin{split}
&\quad\;
\mc{Q}_R x \mc{Q}_R^*  \int_{\whG}^{\oplus} f(\zeta)
\hat{\delta}(\zeta)^{-\frac{p-1}{2}} \md\mu_p(\zeta)=
\mc{Q}_R x \hat{J}J \Lambda_{\hpsi}(f)=
\mc{Q}_R \hat{J}JJ \hat{J}x \hat{J}J \Lambda_{\hpsi}(f)\\
&=
\mc{Q}_R \hat{J}JJ x^*J \Lambda_{\hpsi}(f)=
\mc{Q}_R \hat{J}J\hat{R}(x) \Lambda_{\hpsi}(f)=
\mc{Q}_R \hat{J}J\Lambda_{\hpsi}(\hat{R}(x) f)\\
&=
\int_{\IrrG}^{\oplus} \hat{R}(x)(\zeta) f(\zeta) 
\hat{\delta}(\zeta)^{-\frac{p-1}{2}} \md\mu_p(\zeta),
\end{split}
\]
therefore $\mc{Q}_R x \mc{Q}_R^*=\int_{\whG}^{\oplus} \hat{R}(x) (\zeta)\md\mu_p(\zeta)$. Consequently, point $7.3)$ of theorems \ref{PlancherelL}, \ref{PlancherelR} holds. We are left to show the commutation relations for $\mc{Q}_R$. Take any $\omega\in \Lj$ and $f\in \mf{N}_{\hpsi}$. We have
\[\begin{split}
&\quad\;
\mc{Q}_R \hat{J}J(\omega\otimes\id)\mrW \Lambda_{\hpsi}(f)=
\int_{\whG}^{\oplus} ((\omega\otimes\id)\mrW)(\zeta) f(\zeta)
\hat{\delta}(\zeta)^{-\frac{p-1}{2}}\md\mu_p(\zeta)\\
&=
\int_{\whG}^{\oplus} (\omega\otimes\id)(U^{\zeta}) f(\zeta)
\hat{\delta}(\zeta)^{-\frac{p-1}{2}}\md\mu_p(\zeta)=
\bigl(\int_{\whG}^{\oplus}
(\omega\otimes\id)(U^{\zeta})\otimes\I_{\msf{H}_\zeta}\md\mu_p(\zeta)\bigr)
\mc{Q}_R \hat{J}J \Lambda_{\hpsi}(f)
\end{split}\]
and
\[\begin{split}
&\quad\;
\mc{Q}_R \hat{J}J (\omega\otimes\id)\chi(\mrV)
\Lambda_{\hpsi}(f)=
\mc{Q}_R \hat{J}J \hat{R}((\omega\otimes\id)\mrW)
\Lambda_{\hpsi}(f)\\
&=
\int_{\whG}^{\oplus}
\hat{R}((\omega\otimes\id)\mrW)(\zeta) f(\zeta)
\hat{\delta}(\zeta)^{-\frac{p-1}{2}}\md\mu_p(\zeta)\\
&=
\bigl( \int_{\whG}^{\oplus} \I_{\msf{H}_\zeta}\otimes
\ov{\zeta}((\omega\otimes\id)\mrW) \md\mu_p(\zeta)\bigr)
\mc{Q}_R \hat{J}J \Lambda_{\hpsi}(f).
\end{split}\]
This concludes the proof of Proposition \ref{stw18}.

\subsubsection{Functions $\varpi$}
Let us fix $\zeta\in\whG$ and take a measurable subset $\Omega\subseteq\whG$. Let $\sigma_{\Omega}=\int_{\whG}^{\oplus} \zeta' \,\md\chi_{\Omega}\mu_p(\zeta')$ be a representation of $\CGD$. Then 
\[
\begin{split}
\sigma_{\Omega}(f)=
\int_{\whG}^{\oplus} \zeta'(f)\,\md\chi_\Omega\mu_p(\zeta')=
\int_{\whG}^{\oplus} f(\zeta')\,\md\chi_\Omega\mu_p(\zeta')=\operatorname{M}_{f}\in\B(\LL^2(\whG,\chi_\Omega\mu_p))
\quad (f\in \CGD),
\end{split}
\]
where $\M_f$ is a multiplication operator and we identify the direct integral Hilbert space $\int_{\whG}^{\oplus} \CC \md\chi_{\Omega}\mu_p(\zeta)$ with $\LL^2(\whG,\chi_\Omega \mu_p)$.
Let us define an operator
\[
U\colon \CC\otimes\LL^2(\whG,\chi_\Omega\mu_p) \ni
1\otimes \xi\mapsto \hat{\delta}(\zeta)^{-\frac{p-1}{2}}R_{\zeta^{-1}}\xi\in \LL^2(\whG,\chi_{\Omega\zeta}\mu_p),
\]
where $R^p_{\zeta^{-1}}\in \B(\LL^2(\whG,\chi_\Omega\mu_p))$ is the usual shift by $\zeta^{-1}$ from the right. Operator $U$ is a well defined isometry. Indeed, for any $\xi\in\LL^2(\whG,\chi_\Omega\mu_p)$ we have
\[
\begin{split}
\|U(1\otimes \xi)\|^2=
\|\hat{\delta}(\zeta)^{-\frac{p-1}{2}} R_{\zeta^{-1}}\xi\|^2
&=
\hat{\delta}(\zeta)^{-p+1}\int_{\whG}|\xi(\zeta'\zeta^{-1})|^2\md \chi_{\Omega\zeta}\mu_p\\
&=
\hat{\delta}(\zeta)^{-p+1}\int_{\whG}|\xi(\zeta'\zeta^{-1})|^2
\chi_{\Omega\zeta}(\zeta')\hat{\delta}(\zeta')^{p-1}\md\mu_R(\zeta')\\
&=
\hat{\delta}(\zeta)^{-p+1}\int_{\whG}|\xi(\zeta')|^2
\chi_{\Omega}(\zeta')\hat{\delta}(\zeta'\zeta)^{p-1}\md\mu_R(\zeta')\\
&=
\int_{\whG}|\xi(\zeta')|^2
\chi_{\Omega}(\zeta')\hat{\delta}(\zeta')^{p-1}\md\mu_R(\zeta')
=
\|1\otimes\xi\|^2.
\end{split}
\]
Surjectivity of $U$ is clear, hence $U$ is unitary. Moreover, it is an intertwiner between $\zeta\tp\sigma_\Omega=(\zeta\otimes\sigma_\Omega) \Delta_{\whG}^{u,\operatorname{op}}$ and $\sigma_{\Omega\zeta}$: for $f\in \mathrm{C
}_0(\whG),\xi,\eta\in \LL^2(\whG,\chi_{\Omega\zeta}\mu_p)$ we have
\[
\begin{split}
\is{\xi}{U \zeta\tp\sigma_\Omega(f) U^* \eta}&=
\is{\hat{\delta}(\zeta)^{\frac{p-1}{2}}1\otimes R_{\zeta}\xi}{\zeta\tp \sigma_\Omega(f) (\hat{\delta}(\zeta)^{\frac{p-1}{2}}1\otimes R_{\zeta}\eta)}\\
&=
\hat{\delta}(\zeta)^{p-1}\is{R_{\zeta}\xi}{f( \cdot\,\zeta )R_{\zeta}\eta}\\
&=
\hat{\delta}(\zeta)^{p-1}\int_{\whG} \ov{\xi(\zeta' \zeta)} f(\zeta'\zeta) \eta(\zeta'\zeta)\md\chi_\Omega \mu_p(\zeta')\\
&=
\hat{\delta}(\zeta)^{p-1}\int_{\whG} \chi_{\Omega}(\zeta'\zeta^{-1})
\ov{\xi(\zeta')} f(\zeta') \eta(\zeta')\hat{\delta}(\zeta'\zeta^{-1})^{p-1}\md\mu_R(\zeta')\\
&=
\int_{\whG} \ov{\xi(\zeta')} f(\zeta') \eta(\zeta')\md\chi_{\Omega\zeta}\mu_p(\zeta')=
\is{\xi}{M_f \eta}.
\end{split}
\]
Consequently $\E^{1}_{\zeta\stp\sigma_\Omega}=\E^{1}_{\sigma_{\Omega\zeta}}=\Omega\zeta$ and $\E^{n}_{\zeta\stp\sigma_\Omega}=\emptyset$ for $n\ge 2$. \\
We would like to find the function $\varpi^{\zeta,\Omega,\mu_p}$. Let us introduce two integral representations:
\[
\pi=\bigoplus_{n=1}^{\infty}\int_{\F^n_{\zeta\stp\sigma_\Omega}}^{\oplus}\zeta'\md\,(\mu_p)_{\F^n_{\zeta\stp\sigma_\Omega}}(\zeta')=
\int_{\Omega\zeta}^{\oplus}\zeta'\md\,(\mu_p)_{\Omega\zeta}(\zeta')
\]
and
\[
\gamma=\int_{\Omega}^{\oplus} \zeta\tp \zeta'\md\,(\mu_p)_\Omega(\zeta').
\]
Between these representations we have a unitary intertwiner $\mc{O}$.
For $f\in \CGD_+$  by the definition of the function $\varpi^{\zeta,\Omega,\mu_p}$ we have
\[
\int_{\Omega} \Tr(\mc{O} \pi(f)\mc{O}^*(\zeta'))\md\mu_p(\zeta')
=
\sum_{n=1}^{\infty} \int_{\F^n_{\kappa\stp\sigma_\Omega}} 
\varpi^{\zeta,\Omega,\mu_p}(\zeta') \Tr(\zeta'(f))\md\mu_p(\zeta')
\]
which can be translated to
\[
\int_{\Omega} f(\zeta'\zeta)\md\mu_p(\zeta')=
\int_{\Omega} \Tr(\zeta\tp\zeta'(f))\md\mu_p(\zeta')=
\int_{\Omega\zeta} \varpi^{\zeta,\Omega,\mu_p}(\zeta')f(\zeta')\md\mu_p(\zeta').
\]
Once we expand the definition of $\mu_p$ and use the right invariance of $\mu_R$ we arrive at
\[\begin{split}
&\quad\;\int_{\Omega\zeta} f(\zeta') \hat{\delta}(\zeta'\zeta^{-1})^{p-1}\md\mu_R(\zeta')=
\int_\Omega f(\zeta'\zeta) \hat{\delta}(\zeta')^{p-1}\md\mu_R(\zeta')\\
&=
\int_{\Omega\zeta} \varpi^{\zeta,\Omega,\mu_p}(\zeta')f(\zeta')\hat{\delta}(\zeta')^{p-1}\md\mu_R(\zeta').
\end{split}
\]
Since $f$ is arbitrary we have proved the following:

\begin{proposition}\label{stw20}
For any $\zeta\in\whG$ and any measurable subset $\Omega\subseteq\IrrG$ we have
\[
\zeta \tp \sigma_\Omega\simeq \sigma_{\Omega \zeta}\quad \textnormal{ and } \quad
\varpi^{\zeta,\Omega,\mu_p}(\zeta')=\hat{\delta}(\zeta)^{-p+1}\quad(\zeta'\in\Omega\zeta).
\]
\end{proposition}

\subsubsection{Operators $\mc{L}_\zeta$ and integral charactes}
Let us now find out how the operator $\mc{L}_\zeta$ works. We have $E_\zeta=\hat{\delta}(\zeta)^{\frac{p-1}{2}}$, hence $\Tr(E^2_\zeta)^{\frac{1}{2}}=\hat{\delta}(\zeta)^{\frac{p-1}{2}}$. By the definition of $\mc{L}_\zeta$ we have
\[\begin{split}
\mc{L}_\zeta\colon \LL^2(\whG,\mu_p)\ni&
\Tr(E^2_\bullet)^{\frac{1}{2}}\chi_\Omega=
\hat{\delta}^{\frac{p-1}{2}}\chi_\Omega\mapsto\\
\mapsto&
\Tr(E^2_\bullet)^{\frac{1}{2}} \varpi^{\zeta,\Omega,\mu_p}\sum_{n=1}^{\infty}\chi_{\F^n_{\zeta\stp\sigma_\Omega}}=
\hat{\delta}^{\frac{p-1}{2}} \hat{\delta}(\zeta)^{-p+1} \chi_{\Omega\zeta}\in\LL^2(\whG,\mu_p).
\end{split}
\]
For a finite measure subset $\Omega\subseteq\whG$ such that $\sup_{\Omega}\hat{\delta}<+\infty$ we can write:
\[\begin{split}
&\quad\;\mc{L}_\zeta(\hat{\delta}^{\frac{p-1}{2}}\chi_\Omega)=
\hat{\delta}(\zeta)^{-p+1}\hat{\delta}^{\frac{p-1}{2}}\chi_{\Omega\zeta}=
\hat{\delta}(\zeta)^{-p+1}(\hat{\delta}^{\frac{p-1}{2}}\chi_{\Omega})(\cdot\,\zeta^{-1}) \hat{\delta}(\zeta)^{\frac{p-1}{2}}\\
&=
\hat{\delta}(\zeta)^{-\frac{p-1}{2}}(\hat{\delta}^{\frac{p-1}{2}}\chi_{\Omega})(\cdot\,\zeta^{-1})=
\hat{\delta}(\zeta)^{-\frac{p-1}{2}} R^p_{\zeta^{-1}}(\hat{\delta}^{\frac{p-1}{2}} \chi_{\Omega}),
\end{split}
\]
where as before $R^p_{\zeta^{-1}}\colon \LL^2(\whG,\mu_p)\ni f \mapsto f(\cdot \,\zeta^{-1})\in\LL^2(\whG,\mu_p)$ is a bounded operator. By the density of the subspace spanned by functions $\hat{\delta}^{\frac{p-1}{2}}\chi_\Omega$, we arrive at the following result:

\begin{proposition}
The operator $\mc{L}_{\zeta}$ is given by
\[
\mc{L}_\zeta=\hat{\delta}(\zeta)^{-\frac{p-1}{2}} R^p_{\zeta^{-1}}\quad(\zeta\in\whG).
\]
\end{proposition}

Let $\nu\in \LL^1(\whG,\mu_p)$ be a positive function. We define as usual
\[
\mc{L}_\nu=\int_{\whG}\tfrac{\nu(\zeta)}{\dim(\zeta)} \mc{L}_\zeta \md\mu_p(\zeta)=\int_{\whG} \nu(\zeta)\mc{L}_\zeta\md\mu_p(\zeta),
\]
and the above result gives us
\[
\mc{L}_\nu(f)
=
\int_{\whG}\nu(\zeta)\hat{\delta}(\zeta)^{-\frac{p-1}{2}}
f(\cdot \zeta^{-1})
\md\mu_p(\zeta)\quad(f\in\LL^2(\whG,\mu_p)).
\]
Consider the operator
\[
V_p\colon \LL^2(\whG,\mu_L)\ni f \mapsto \hat{\delta}^{-\frac{p-1}{2}}\,f(\cdot^{-1})\in\LL^2(\whG,\mu_p).
\]
It is unitary: indeed, it is isometric:
\[
\begin{split}
\|V_p(f)\|_{\LL^2(\whG,\mu_p)}^2&=
\int_{\whG}|f(\zeta^{-1})|^2 \hat{\delta}(\zeta)^{-p+1} \md\mu_p(\zeta)
=
\int_{\whG}|f(\zeta^{-1})|^2 \hat{\delta}(\zeta) \md\mu_L(\zeta)\\
&=
\int_{\whG} |f|^2\md\mu_L=\|f\|_{\LL^2(\whG,\mu_L)}^2
\quad(f\in \LL^2(\whG,\mu_L))
\end{split}
\]
and it is clear that $V_p$ is surjective. The inverse is given by $V_p^*\colon \LL^2(\whG,\mu_p)\ni f\mapsto\hat{\delta}^{-\frac{p-1}{2}}f(\cdot^{-1})\in\LL^2(\whG,\mu_L)$:\\
\[
V_pV_p^*(f)=\hat{\delta}^{-\frac{p-1}{2}} V_p^*(f)(\cdot^{-1})=
\hat{\delta}^{-\frac{p-1}{2}} (\hat{\delta}^{-\frac{p-1}{2}} f(\cdot^{-1}))(\cdot^{-1})=
\hat{\delta}^{-\frac{p-1}{2}}\hat{\delta}^{\frac{p-1}{2}}f=f.
\]
We need to translate our function $\nu\in \LL^1(\whG,\mu_p)$ to $\LL^1(\whG,\mu_L)$: define $\nu'=\hat{\delta}^{-p+1} \nu(\cdot^{-1})$. Then
\[
\|\nu'\|_{\LL^1(\whG,\mu_L)}=
\int_{\whG} |\nu(\zeta^{-1})| \hat{\delta}(\zeta)^{-p+1} \md\mu_{L}(\zeta)=
\int_{\whG} |\nu(\zeta)|\hat{\delta}(\zeta) \hat{\delta}(\zeta)^{p-1}\md\mu_L(\zeta)=
\|\nu\|_{\LL^1(\whG,\mu_p)}.
\]
Denote by $\lambda$ the regular representation of $\whG$ on $\LL^2(\whG,\mu_L)$: $\lambda_\zeta(f)(\zeta')=f(\zeta^{-1}\zeta')$. The operator $V_p$ transports the regular representation to $\LL^2(\whG,\mu_p)$: let $\zeta,\zeta'\in\whG,f\in \LL^2(\whG,\mu_p)$. We have
\[
\begin{split}
V_p\lambda_\zeta V_p^*(f)(\zeta')
&=
\hat{\delta}(\zeta')^{-\frac{p-1}{2}}\,\lambda_\zeta V_p^*(f)(\zeta'^{-1})\\
&=
\hat{\delta}(\zeta')^{-\frac{p-1}{2}}\,V_p^*(f)(\zeta^{-1}\zeta'^{-1})\\
&=
\hat{\delta}(\zeta')^{-\frac{p-1}{2}}\,
\hat{\delta}^{-\frac{p-1}{2}}(\zeta^{-1}\zeta'^{-1})\,f(\zeta'\zeta)\\
&=
\hat{\delta}(\zeta)^{\frac{p-1}{2}}f(\zeta'\zeta).
\end{split}
\]
Observe that we can write our operator $\mc{L}_\nu$ as an integral of operators $V_p\lambda_\zeta V_p^*$:
\[
\begin{split}
\mc{L}_{\nu}(f)&=
\int_{\whG}\nu(\zeta) \hat{\delta}(\zeta)^{-\frac{p-1}{2}} 
f(\cdot \zeta^{-1})\,\md\mu_p(\zeta)\\
&=
\int_{\whG}\nu(\zeta) \hat{\delta}(\zeta)^{-\frac{p-1}{2}}
f(\cdot \zeta^{-1}) \hat{\delta}(\zeta)^p\,\md\mu_L(\zeta)\\
&=
\int_{\whG} \nu(\zeta^{-1})\hat{\delta}(\zeta)^{\frac{p-1}{2}} 
f(\cdot \zeta) \hat{\delta}(\zeta)^{-p}\hat{\delta}(\zeta)\,\md\mu_L(\zeta)\\
&=
\int_{\whG} \nu(\zeta^{-1})\hat{\delta}(\zeta)^{-\frac{p-1}{2}} 
f(\cdot \zeta) \,\md\mu_L(\zeta)\\
&=
\int_{\whG}\nu(\zeta^{-1}) \hat{\delta}(\zeta)^{-p+1} 
(V_p \lambda_\zeta V_p^* )f\,\md\mu_L(\zeta)=
\int_{\whG} \nu'(\zeta) (V_p \lambda_\zeta V_p^* ) f \md\mu_L(\zeta).
\end{split}
\]
We have proved the following:

\begin{proposition}
The operator $\mc{L}_\nu$ is unitarily equivalent to the operator $\int_{\whG} \nu'(\zeta) \lambda_\zeta \md\mu_L(\zeta)$.
\end{proposition}

It follows that our Theorem \ref{tw4} in the case of classical $\whG$ is closely related to the Kesten criterion of amenability (see for example \cite[Theorem G.4.4]{KP(T)}).\\
Now we would like to establish what the action of the integral characters is. Let us start with the following well known lemma:

\begin{lemma}
Let $\HH$ be a classical locally compact group. For each $\zeta\in\HH$ we have
\[
(\zeta\otimes\id)((\mrW^{\HH})^*)=\lambda_{\zeta^{-1}},
\]
where $\lambda_{\zeta}\in\M(\cst_{r}(\HH))=\M(\mathrm{C}_0(\wh{\HH}))$ is the left regular representation.
\end{lemma}

Note that in this lemma $\HH$, not $\widehat{\HH}$ is classical.\\

Now, let us go back to the situation where $\whG$ is a classical group. Take $\zeta\in\whG$. By the above lemma we have
\[
\chi(U^\zeta)=(\id\otimes\zeta)\mrW^{\GG}=(\zeta\otimes\id)((\mrW^{\whG})^*)=\lambda_{\zeta^{-1}}\in \M(\CG),
\]
where $\lambda_{\zeta}\in\M(\CG)\subseteq\B(\LdG)=\B(\LL^2(\whG))$ is the left regular representation of the classical group $\whG$.\\
Let $\Omega\subseteq\IrrG$ be a measurable subset with $\mu_p(\Omega)<+\infty$. Then of course $\int_\Omega\dim\md\mu_p<+\infty$, an according to the previous lemma, the integral character of the representation $\int_\Omega^{\oplus} \zeta \md\mu_p(\zeta)$ is given by
\[
\chi^{\int}(\int_\Omega^{\oplus} \zeta \md\mu_p(\zeta))=
\int_\Omega \lambda_{\zeta^{-1}}\md\mu_p(\zeta)\in \LL^{\infty}(\GG)=\LL(\whG).
\]
(Up to equivalence) the representation conjugate to $\zeta$ is $\zeta^{-1}$ -- it follows from the fact that $\hat{R}^u=\hat{R}$ is given by composition with the inverse. Let us now establish for which $p$ the measure $\mu_p$ is invariant under taking inverse: let $V\subseteq\whG$ be a compact subset. We have
\[
\mu_p(V^{-1})=
\int_{\whG} \chi_{V^{-1}} \hat{\delta}^p \md\mu_L=
\int_{\whG} \chi_V \hat{\delta}^{-p} \hat{\delta}\md\mu_L=
\mu_{-p+1}(V),
\]
which forces $p=-p+1$ and consequently $p=\tfrac{1}{2}$ ($\mu_{\frac{1}{2}}$ is the measure "between" the left and the right Haar measure on $\whG$). We know that a quantum group $\GG$ is coamenable if and only if the classical group $\whG$ is amenable (see e.g.~\cite{Brannan}). Note also that all finite dimensional representations of $\GG$ are admissible (\cite[Remark 3.3]{DasDawsSalmi}). We would like now to establish a corollary of Theorem \ref{tw2} in the case of classical $\whG$. Since the conditions in Theorem \ref{tw2} are concerned only with the spectrum intersected with the real line, we can consider the adjoint of the integral character and we arrive at the following corollary.

\begin{corollary}\label{wniosek3}
Let $G$ be a second countable locally compact group with the left Haar measure $\mu_L$, the right Haar measure $\mu_R$ and the modular element $\hat{\delta}=\tfrac{\md\mu_R}{\md \mu_L}$. The following conditions are equivalent:
\begin{enumerate}[label=\arabic*)]
\item $G$ is amenable.
\item For any $p\in \RR$ and any measurable subset $\Omega\subseteq G$ such that $\int_\Omega \hat{\delta}^p \md\mu_L<+\infty$ we have
\[
\int_\Omega \hat{\delta}^p\md\mu_L \in \sigma( \int_\Omega \hat{\delta}(\zeta)^p\lambda_{\zeta}  \md\mu_L(\zeta)).
\]
\item For any measurable subset $\Omega\subseteq G$ such that $\Omega=\Omega^{-1}$ and $\int_\Omega \hat{\delta}^{\frac{1}{2}} \md\mu_L<+\infty$ we have
\[
\int_\Omega \hat{\delta}^{\frac{1}{2}}\md\mu_L \in \sigma( \int_\Omega \hat{\delta}(\zeta)^{\frac{1}{2}}\lambda_{\zeta}  \md\mu_L(\zeta)).
\]
\end{enumerate}
\end{corollary}

Let us remark that in \cite[Theorem 1]{BergChristensen} Christian Berg and Jens Peter Reus Christens obtain a theorem in a similar spirit. However, they consider $\LL^p$ spaces and norm condition (rather than a condition concerning spectrum). Moreover, they consider more general measures (see also \cite[Theorem 3.2.2]{Greenleaf}).

\subsection{$\GG$ classical}
Let $\GG$ be a classical locally compact group and let $\mu_{\GG}$ be its left Haar measure. The symbol $\delta$ will denote the modular element of the classical group $\GG$, hence the Radon-Nikodym derivative $\delta=\tfrac{\md\mu_R}{\md \mu_\GG}$, where $\mu_R$ is the right invariant Haar measure given by $\mu_\GG$ composed with the inverse (again, note the difference in conventions). It is a well known result that the algebras associated with the dual locally compact quantum group $\whG$ are given by
\[
\CGD=\mathrm{C}_r^*(\GG),\quad
\CGDu=\mathrm{C}^*(\GG),\quad
\Linfd=\operatorname{L}(\GG).
\]
Moreover we have $\lambda(\omega)=\int_{\GG} \omega(x) \lambda_x\md\mu_{\GG}(x)$ for any $\omega\in\Lj$, where $\lambda_x\in \B(\LdG)$ is the unitary operator given by the left regular representation.
Let us now identify the subspaces $\mc{I},\mc{I}_R,\Ljsharp$:

\begin{proposition}$ $
\begin{enumerate}
\item We have $\mc{I}=\LL^1(\GG)\cap\LL^2(\GG)=\{\omega\in\LL^1(\GG)\,|\, \lambda(\omega)\in\mf{N}_{\hvp}\}$ and $\Lhvp(\lambda(\omega))=\omega$ for all $\omega\in \mc{I}$.
\item We have $\mc{I}_R=\LL^1(\GG)\cap \delta^{\frac{1}{2}}\LL^2(\GG)$ and $\xi_R(\omega)=\omega\delta^{-\frac{1}{2}}$ for all $\omega\in \mc{I}_R$.
\item We have $\LL^1_{\sharp}(\GG)=\Lj$ and $\alpha^{\sharp}=\ov{\alpha(\cdot^{-1})}\,\delta$ for all $\alpha\in \LL_{\sharp}^{1}(\GG)$.
\end{enumerate}
\end{proposition}

\begin{proof}
 Take any $\omega\in \Lj\cap\LdG$ and $f\in \mf{N}_{\vp}$. We have
\[
\omega(f^*)=\int_{\GG} \omega(x) \ov{f(x)}\md\mu_{\GG}(x)=
\is{\Lambda_\vp (f)}{\omega},
\]
from which it follows that
\[
\omega\in\mc{I},\quad\Lhvp(\lambda(\omega))=\omega\quad(\omega\in\LL^1(\GG)\cap \LL^2 (\GG)).
\]
On the other hand, if $\omega\in\LL^1(\GG)$ is a function such that the operator
\[
\LdG\supseteq\Lambda_{\vp}(\mf{N}_{\vp})\ni \Lambda_{\vp}(f) \mapsto \int_{\GG}\omega(x) \ov{f(x)}\md\mu_{\GG}(x)\in\CC
\]
is bounded, then there exists a function $\nu\in \LL^2(\GG)$ satisfying
\[
\int_{\GG}\omega(x) \ov{f(x)} \md\mu_{\GG}(x)=
\int_{\GG}\nu(x) \ov{f(x)} \md\mu_{\GG}(x)\quad
(f\in \mf{N}_\vp).
\]
It follows that $\omega=\nu$ $\mu_{\GG}$-a.e., hence we have proved $\mc{I}=\LL^1(\GG)\cap \LL^2(\GG)$. Now we will prove that we also have
\[
\mc{I}=\{\omega\in \Lj\,|\, \lambda(\omega)\in \mf{N}_{\hvp}\}.
\]
Indeed, the inclusion $\subseteq$ holds in general. Take now $\omega\in\Lj$ such that $\lambda(\omega)\in \mf{N}_{\hvp}$. As $\lambda(\mc{I})$ is a core for $\Lhvp$ (\cite[Lemma 2.5]{KustermansVaes}) we can find a net $(\omega_j)_{j\in \mc{J}}$ in $\mc{I}$ such that $\lambda(\omega_j)\xrightarrow[j\in \mc{J}]{\ssots}\lambda(\omega)$ and $\Lhvp(\lambda(\omega_j))=\omega_j\xrightarrow[j\in\mc{J}]{}\Lhvp(\lambda(\omega))$. From the first convergence it follows that we have
\[
\omega_j ((\id\otimes\theta)\mrW)=\theta(\lambda(\omega_j))
\xrightarrow[j\in\mc{J}]{} \theta(\lambda(\omega))=
\omega((\id\otimes\theta)\mrW)
\]
for all $\theta\in \Ljd$. Take $f\in \mf{N}_{\vp}$ such that $f=(\id\otimes\theta)\mrW$ for some $\theta\in\LL^1(\whG)$. Notice that the space of such $f$'s forms a $\ssots\times\|\cdot\|$ core for $\Lambda_{\vp}$ (use this time \cite[Proposition 2.6]{KustermansVaes} for $\whG$). We get
\[
\ismaa{\Lvp(f^*)}{\Lhvp(\lambda(\omega))}=
\lim_{j\in\mc{J}}\ismaa{\Lvp(f^*)}{\Lhvp(\lambda(\omega_j))}=
\lim_{j\in\mc{J}} \omega_j(f)=\omega(f).
\]
Now, for arbitrary $f\in \mf{N}_{\vp}$ it follows that $f^* \in \mf{N}_{\vp}$ and we can find a net $(g_k)_{k\in\mc{K}}$ in \\$\{(\id\otimes\theta)\mrW\,|\,\theta\in\Ljd\}$ such that
\[
g_k\xrightarrow[k\in\mc{K}]{\ssots} f^*,\quad
\Lvp(g_k)\xrightarrow[k\in\mc{K}]{} \Lvp(f^*).
\]
It easily follows that we have $\Lvp(g_k^*)\xrightarrow[k\in\mc{K}]{} \Lvp(f)$ (recall that the weight $\vp$ is tracial). We arrive at
\[
\ismaa{\Lvp(f)}{\Lhvp(\lambda(\omega))}=
\lim_{k\in\mc{K}}\ismaa{\Lvp(g_k^*)}{\Lhvp(\lambda(\omega))}=
\lim_{k\in\mc{K}} \omega(g_k)=\omega(f^*),
\]
hence $\omega\in \mc{I}$.\\
Now, let $f\in \mf{N}_\psi$ be a function such that $\delta^{\frac{1}{2}}f\in \mf{N}_\vp$. The GNS construction of $\psi$ implies that $\Lambda_\psi(f)=\Lvp(f \delta^{\frac{1}{2}})=f \delta^{\frac{1}{2}}$. Consequently, for $\omega\in \LL^1(\GG)\cap \delta^{\frac{1}{2}} \LL^2(\GG)$ and $f$ as above we have
\[
\omega(f^*)=
\int_{\GG} \omega \ov{f}\md\mu_{\GG}=
\int_{\GG} \omega\delta^{-\frac{1}{2}} \ov{f \delta^{\frac{1}{2}}} \md\mu_{\GG}=
\ismaa{\Lambda_\psi(f)}{\omega\delta^{-\frac{1}{2}}}
\]
By density of the $\star$-algebra of the continuous functions with compact support we have $\omega\in \mc{I}_R$ and $\xi_R(\omega)=\omega\delta^{-\frac{1}{2}}\in\LL^2(\GG)$. Argument as above implies
\[
\mc{I}_R=\LL^1(\GG)\cap \delta^{\frac{1}{2}}\LL^2(\GG).
\]
Take $\alpha\in\LL^1(\GG)$. Then
\[
\lambda(\alpha)^*=
\int_{\GG} \ov{\alpha(x)} \lambda_{x^{-1}} \md\mu_{\GG}(x)=
\int_{\GG} \ov{\alpha(x^{-1})}\delta(x) \lambda_{x} \md\mu_{\GG}(x),
\]
hence $\LL^1_{\sharp}(\GG)=\LL^1(\GG)$ and $\alpha^{\sharp}=\ov{\alpha(\cdot^{-1})}\delta$.
\end{proof}

Since $\GG$ is classical, groups $(\tau_t)_{t\in\RR}, (\sigma^{\psi}_t)_{t\in\RR}$ are trivial and we have $\hat{\delta}^{it}=\I$ for all $t\in \RR$ (\cite[Proposition 5.15]{Daele}). We can therefore define $E_\pi=D_\pi$, and $\mc{Q}_R=\mc{Q}_L\circ J\hat{J}$.\\
Let us now describe a class of locally compact groups given by the semidirect product, which satisfy all the assumptions made in this paper:

\subsubsection{Semidirect product $N\rtimes H$}
Let $N$ be a locally compact, Hausdorff, second countable abelian group and $H$ a finite discrete group acting on $N$ via action $\alpha\colon H\rightarrow \operatorname{Aut}(N)$. Define $G=\GG$ to be the semidirect product $G=\GG=N\rtimes H$. This means that we have $(n,h)(n',h')=(n+\alpha_h(n'),hh')$ for all $(g,h),(g',h')\in G$. One easily checks that $G$ is unimodular and the Haar measure is given by the product measure of the Haar measure on $N$ and the counting measure on $H$. Since $N$ can be identified with a normal subgroup in $G$ (via $n\simeq (n,e_H))$ we have an action of $G$ on $\wh{N}$ given by
\[
g\cdot \nu(n)=\nu(g^{-1}ng)\quad(g\in G,n\in N,\,\nu\in \wh{N})
\]
Denote by $G_\nu$ the stabilizer of $\nu$. We wish to show that $G$ is a type I group with finite dimensional irreducible representations. Our tool is Theorem 6.43 from the Folland's book \cite{Folland}. In order to use it, we have to show that the action of $G$ on $\wh{N}$ is regular, which means (in this case) that for each $\nu\in\wh{N}$, the bijective map
\[
G/G_\nu\ni xG_\nu\mapsto x\cdot \nu\in \mc{O}_\nu=\{g\cdot\nu\,|\,g\in G\}
\]
is a homeomorphism (page 196 of \cite{Folland}).

\begin{lemma}
The action of $G$ on $\wh{N}$ is regular.
\end{lemma}

\begin{proof}
For $g=(n,h)\in G$, $n'\in N$, $\nu\in \wh{N}$ we have
\[
g^{-1}n'g=
(-\alpha_{h^{-1}}(n),h^{-1})(n',e_H)(n,h)=
(\alpha_{h^{-1}}(-n+n'),h^{-1})(n,h)
=
(\alpha_h(n'),e_H),
\]
hence
\[
g\cdot \nu(n')=\nu(\alpha_{h}(n')) \quad\textnormal{ and }\quad
\mc{O}_\nu=\{g\cdot\nu\,|\, g\in G\}=
\{\nu\circ\alpha_h\,|\,h\in H\},
\]
in particular, the orbit of $\nu$ is finite. The space $\wh{N}$ is Hausdorff, hence $\mc{O}_\nu$ is a discrete subspace. It is clear that $N\subseteq G_\nu$, therefore the space $G/G_\nu$ is also discrete. Indeed, any point in $G/G_\nu$ can be written as $h G_\nu$ for certain $h\in H$. Next, the set $\{h G_\nu\}\subseteq G/G_\nu$ is an image of an open set $\{(h,n)\,|\,n\in N\}$ via an open map $G\rightarrow G/G_\nu$. Consequently $G/G_\nu \rightarrow \mc{O}_\nu$ is a bijection between discrete spaces.
\end{proof}

Now we can describe all irreducible representations of $G$. Take $\nu\in \wh{N}$ and $\rho$, an irreducible representation of $H_\nu=G_\nu\cap H$ on the Hilbert space $\msf{H}_\rho$. Define a representation $\nu\rho$ of $G_\nu$ on $\msf{H}_\rho$ by
\[
(\nu\rho)(n,h)=\nu(n)\rho(h).
\]
Let $q\colon G\rightarrow G/G_{\nu}$ be the quotient map. Now we define a representation of $G$ as the induced representation $\Ind^G_{G_\nu}(\nu\rho)$: its Hilbert space $\msf{H}_{\Ind^G_{G_\nu}(\nu\rho)}$ is
\[
\msf{H}_{\Ind^G_{G_\nu}(\nu\rho)}=\{f\in C(G,\msf{H}_{\rho\nu})\,|\,
\forall_{g\in G,\,h\in G_{\nu}}\, f(gh)=(\nu\rho)(h^{-1})f(g)\}.
\]
Since $G/G_\nu$ is finite, there is no need for completion and $\msf{H}_{\Ind^G_{G_\nu}(\nu\rho)}$ has finite dimension. Scalar product on $\msf{H}_{\Ind^G_{G_\nu}(\nu\rho)}$ is given by:
\[
\is{f}{h}=\sum_{gG_\nu\in G/G_\nu} \is{f(g)}{h(g)}
\quad(f,h\in \msf{H}_{\Ind^G_{G_\nu}(\nu\rho)} )
\]
and the representation $\Ind^G_{G_\nu}(\nu\rho)$ is given by
\[
\Ind^G_{G_\nu}(\nu\rho)(g)f(g')=
f(g^{-1}g')\quad(g,g'\in G,f\in \msf{H}_{\Ind^G_{G_\nu}(\nu\rho)}).
\]
Theorem 6.43 of \cite{Folland} gives us a description of all irreducible representations of $G$:

\begin{proposition}
As a set, the space $\Irr(G)$ is given by
\[
\operatorname{Irr}(G)=\{[\Ind^{G}_{G_\nu}(\nu\rho)]\,|\, \nu\in\hat{N},\,\rho \textnormal{ - irreducible representation of }H_\nu\}.
\]
Moreover, representations $\Ind^{G}_{G_\nu}(\nu\rho)$ and $\Ind^{G}_{G_{\nu'}}(\nu'\rho')$ are equivalent if and only if there exists $x\in H$ such that $\nu'=x\cdot \nu$ and the representations $\rho,\;\rho'(x \cdot x^{-1})$ are equivalent.
\end{proposition}

As a corollary we get the following result.

\begin{proposition}
$G=N\rtimes H$ is a type I locally compact group and its irreducible representations have dimension less or equal to 
\[
|H| \max\{\dim(\rho)\,|\, \rho \textnormal{ - irreducible representation of }K\le H\}<+\infty.
\]
\end{proposition}

\begin{proof}
The second claim follows from the observation that
\[
\dim(\msf{H}_{\Ind^G_{G_\nu}(\nu\rho)})\le |G/G_\nu| \dim(\msf{H}_\rho)
\]
and the following facts: $N\subseteq G_\nu$ and $|G/N|=|H|$. We have the following criterion: a second countable locally compact group $G$ is type I if and only if for every irreducible representation $\pi$ of $G$ we have $\mc{K}(\msf{H}_\pi)\subseteq \pi(\mathrm{C}^{*}(G))$ (\cite[Theorem 7.6]{Folland}). In our case irreducible representations are finite dimensional, hence
\[
\pi(\mathrm{C}^*(G))=
\pi(\mathrm{C}^*(G))''\supseteq
\pi(G)''=\B(\msf{H}_\pi)=\mc{K}(\msf{H}_\pi).
\]
\end{proof}

\subsubsection{Special case: $\RR\rtimes \ZZ_2$}

In this subsection we will describe in detail the special case of the above construction, given by a semidirect product of $\RR$ and $\ZZ_2$, with the action of $\ZZ_2=\{1,-1\}$ on $\RR$ given by $\alpha_{-1}(t)=-t\,(t\in \RR)$\footnote{This construction can be generalized to $G\rtimes \ZZ_2$ with any abelian locally compact second countable Hausdorff group $G$.}. With $r\in \RR$ let us associate a one dimensional representation of $\RR$:
\[
\pi_r\colon \RR\ni t \mapsto \pi_r(t)=e^{itr}\in U(\CC^{1}).
\]
It is well known that we have $\wh{\RR}=\{\pi_r\,|\, r\in \RR\}$. Let us calculate the action of $G$ on $\wh{\RR}$: we have
\[
((r,a) \cdot \pi_s) (t)=
\pi_s((r,a)^{-1} (t,1) (r,a))=
\pi_s( (-ar,a) (t+r,a))=
\pi_s(at,1)=\pi_{as}(t)
\]
for all $(r,a)\in G,\,s,t\in\RR$, hence $(r,a)\cdot \pi_s=\pi_{as}$. Consequently, the orbits look as follows:
\[
G\cdot  \pi_0=\{\pi_0\},\quad
G\cdot \pi_s=\{\pi_{-s},\pi_s\}\quad(s\in\RR\setminus\{0\})
\]
and the stabilizer subgroups are given by
\[
G_{\pi_0}=G,\quad
G_{\pi_s}=\RR,\quad
H_{\pi_0}=H=\ZZ_2,\quad
H_{\pi_s}=\{1\}
\quad(s\in\RR\setminus\{0\}).
\]
(we have introduced notation $H=\ZZ_2$). Consequently, the group $H_{\pi_s}$ has only the trivial irreducible representation when $s\neq 0$ and two irreducible representations when $s=0$: trivial and the identity representation.\\
Let us first consider the case $s=0$. According to the general procedure we take $\rho$, an irreducible representation of $H_{\pi_0}=\ZZ_2$ and consider the representation of $G_{\pi_0}=G$ given by
\[
\pi_0 \rho\colon G_{\pi_0}\ni (t,a)\mapsto \pi_{0}(t)\rho(a)=\rho(a)\in U(\CC^1).
\]
Since $G_{\pi_0}=G$, the induction producedure is trivial and we have
\[
\Ind_{G_{\pi_0}}^{G}(\pi_0 \rho)\colon G\ni (t,a)\mapsto \pi_{0}(t)\rho(a)=\rho(a)\in U(\CC^1).
\]
It is clear that in this way we get two nonequivalent representations of $G$: trivial and 
\[
\sigma\colon G\ni (t,a)\mapsto a \in U(\CC^1).
\]
Let us now look at the case $s\in \RR\setminus\{0\}$. Since the little group $H_{\pi_s}$ is trivial, there is no representation $\rho$ involved. We have
\[
\msf{H}_{\operatorname{Ind}^{G}_{\RR}(\pi_s)}=
\{f\in C(G)\,|\, f(x\,(h,1))=\pi_r(-h)f(x)\,\forall x\in G,h\in \RR\}.
\]
Take $f\in \msf{H}_{\operatorname{Ind}^{G}_{\RR}(\pi_s)}$. Since
\[
f(t,a)=f((0,a)(at,1))=\pi_r(-at)f(0,a)=
e^{-iart} f(0,a)\quad((t,a)\in G)
\]
the dimension of $ \msf{H}_{\operatorname{Ind}^{G}_{\RR}(\pi_s)}$ equals $2$ and as a basis of this space we can take functions $f^s_1,f^s_{-1}$ determined by
\[
f^s_1(0,1)=1,\quad f^s_1(0,-1)=0,\quad
f^s_{-1}(0,1)=0,\quad f^s_{-1}(0,-1)=1.
\]
One easily checks that the induced representation $\operatorname{Ind}^{G}_{\RR}(\pi_s)$ is given by
\begin{equation}\label{eq26}
\operatorname{Ind}^{G}_{\RR}(\pi_s)(t,1) =
\begin{bmatrix}
e^{ist} & 0 \\
0 & e^{-ist}
\end{bmatrix},\quad
\operatorname{Ind}^{G}_{\RR}(\pi_s)(t,-1)=
\begin{bmatrix}
0 & e^{ist} \\
e^{-ist} & 0
\end{bmatrix}\quad(t,s\in\RR)
\end{equation}
in basis $(f^{s}_1,f^{s}_{-1})$. It follows from the general theory that $\Ind_{\RR}^{\GG}(\pi_s)$ is an irreducible representation which is equivalent to $\Ind_{\RR}^{\GG}(\pi_{-s})$ $(s\in\RR_{>0})$. Moreover, when $s>r>0$ then representations $\Ind_{\RR}^{\GG}(\pi_s),\Ind_{\RR}^{\GG}(\pi_r)$ are not equivalent. In order to ease the notation, let us write $\sigma_r=\Ind^{G}_{G_{\pi_r}}(\pi_r)\,(r\in \RR_{>0})$. Consequently, as a set the dual space of $G$ is given by
\[
\operatorname{Irr}(G)=\{ \sigma_r\,|\, r\in \RR_{>0}\}\cup
\{1\}\cup \{\sigma\}
\]
(we will abuse the notation and identify a class of representations with its representative).\\
One easily checks that the Mackey-Borel structure on $\operatorname{Irr}(G)$ (which is the Borel $\sigma$-algebra since $G$ is type I) is the most obvious one: sets $\{1\},\{\sigma\}$ are measurable, and a subset $\{\sigma_r\,|\, r\in E\}$ is measurable if and only if the corresponding subset $E\subseteq \RR_{>0}$ is measurable (for the relevant definitions, see \cite{DixmierC, Folland}). To sum up, we have proved the following result:

\begin{proposition}
The space $\Irr(G)$ is given by 
\[
\operatorname{Irr}(G)=\{ \sigma_r\,|\, r\in \RR_{>0}\}\cup
\{1\}\cup \{\sigma\},
\]
where $\sigma_r\,(r\in\RR_{>0})$ are two-dimensional representations given by \eqref{eq26}, $1$ is the trivial representation and $\sigma$ is the one-dimensional representation given by $\sigma(t,s)=s \,((t,s)\in G)$. In the above decomposition of $\Irr(G)$, sets $\{1\},\{\sigma\}$ are measurable, and the measurable structure on $\{\sigma_r\,|\, r\in \RR_{>0}\}$ agrees with the standard one on $\RR_{>0}$.
\end{proposition}

Denote by $\mu_G$ the Haar measure on $G$ and define a measure $\mu$ on $\Irr(G)$ by $\mu(\{1,\sigma\})=0$ and by $\tfrac{\md r}{4\pi}$ on $\{\sigma_r\,|\,r\in \RR_{>0}\}$. Next, define positive operators $D_\pi=\tfrac{\I_{\msf{H}_\pi}}{\sqrt{2}} (\pi\in \Irr(G))$ and a unitary operator
\begin{equation}\label{eq22}
\mc{Q}_L\colon \LL^2(G) \supseteq \LL^2(G)\cap \LL^1(G) \ni \alpha\mapsto 
\sqrt{2}\int_{\Irr(G)}^{\oplus} (\alpha\otimes\id)U^{\pi} \md\mu(\pi)\in
\int_{\Irr(G)}^{\oplus} \HS(\msf{H}_\pi) \md\mu(\pi).
\end{equation}
We will show that these objects satisfy assumptions of Theorem \ref{PlancherelL}. Because the set $\{1,\sigma\}$ is of $\mu$-measure $0$, we will identify $\Irr(G)$ with $\RR_{>0}$ as a measure space.\\ First, let us justify that $\mc{Q}_L$ is a well defined unitary operator. Take any $\alpha,\beta\in \LL^2(G)\cap\LL^1(G)$ and $r\in \RR\setminus\{0\}$. We have
\[
(\alpha\otimes\id)U^{\sigma_r}=\int_G \alpha(x) \sigma_r(x)\md\mu_G(x),
\]
and similarly for $\beta$, hence
\[
\Tr (((\alpha\otimes\id)U^{\sigma_r} )^*
 ((\beta\otimes\id)U^{\sigma_r}))=
\int_{G}\int_{G} \ov{\alpha(x)} \beta(x') \Tr( \sigma_r(x)^* 
\sigma_r(x'))
\md\mu(x)\md\mu(x')
\]
Let us calculate these traces: take any $t,t'\in \RR$. Using the equation \eqref{eq26} we have
\[
\Tr(\sigma_r(t,1)^* \sigma_r(t',1))=e^{ir (-t+t')}+e^{-ir (-t+t')}=
\Tr(\sigma_r(t,-1)^* \sigma_r(t',-1))
\]
and $\Tr(\sigma_r(t,1)^* \sigma_r(t',-1)=0$. It follows that
\[
\begin{split}
&\Tr (((\alpha\otimes\id)U^{\sigma_r} )^*
 ((\beta\otimes\id)U^{\sigma_r}))\\
&=
\int_{\RR}\int_{\RR} \ov{\alpha(t,1)} \beta(t',1)
 (e^{ir(t-t')} + e^{-ir(t-t')}) \md t \md t'\\
&\quad \,+
\int_{\RR}\int_{\RR} \ov{\alpha(t,-1)} \beta(t',-1)
 (e^{ir(t-t')} + e^{-ir(t-t')}) \md t \md t'.
\end{split}
\]
Let us introduce new functions of real variable $t$: $\alpha_1(t)=\alpha(t,1)$, $\alpha_{-1}(t)=\alpha(t,-1)$ and similarly for $\beta$. We will use the Fourier transform on $\RR$: $\mc{F}(f)(x)=\int_{\RR} e^{-2 \pi i x y} f(y) \,\md y$. We can rewrite the above expression:
\[
\begin{split}
&\quad\;\Tr(((\alpha\otimes\id)U^{\sigma_r} )^* ((\beta\otimes\id)U^{\sigma_r} ))\\
&=
\ov{\mc{F}(\alpha_1)}(\tfrac{r}{2\pi}) \mc{F}(\beta_1)(\tfrac{r}{2\pi})+
\ov{\mc{F}(\alpha_1)}(\tfrac{-r}{2\pi}) \mc{F}(\beta_1)(\tfrac{-r}{2\pi})\\
&\quad+
\ov{\mc{F}(\alpha_{-1})}(\tfrac{r}{2\pi}) \mc{F}(\beta_{-1})(\tfrac{r}{2\pi})+
\ov{\mc{F}(\alpha_{-1})}(\tfrac{-r}{2\pi}) \mc{F}(\beta_{-1})(\tfrac{-r}{2\pi}).
\end{split}
\]
Observe the following identity:
\[
\begin{split}
\Tr (((\alpha\otimes\id)U^{\sigma_{-r}} )^*
 ((\beta\otimes\id)U^{\sigma_{-r}}))=
\Tr (((\alpha\otimes\id)U^{\sigma_{r}} )^*
 ((\beta\otimes\id)U^{\sigma_{r}}))\quad(r\in\RR\setminus\{0\}).
\end{split}
\]
Using it, we arrive at the following equations:
\[
\begin{split}
&\quad\;\int_{\RR>0}\Tr(((\alpha\otimes\id)U^{\sigma_r} D_{\sigma_r}^{-1})^* ((\beta\otimes\id)U^{\sigma_r} D_{\sigma_r}^{-1})) \frac{\md r}{4\pi}\\
&=
\int_{\RR>0}\Tr(((\alpha\otimes\id)U^{\sigma_r} )^* ((\beta\otimes\id)U^{\sigma_r} )) \frac{\md r}{2\pi}\\
&=\int_{\RR}\Tr(((\alpha\otimes\id)U^{\sigma_r} )^* ((\beta\otimes\id)U^{\sigma_r} )) \frac{\md r}{4\pi}\\
&=
\int_{\RR}\bigl(
\ov{\mc{F}(\alpha_1)}(\tfrac{r}{2\pi}) \mc{F}(\beta_1)(\tfrac{r}{2\pi})+
\ov{\mc{F}(\alpha_1)}(\tfrac{-r}{2\pi}) \mc{F}(\beta_1)(\tfrac{-r}{2\pi})\\
&\quad+
\ov{\mc{F}(\alpha_{-1})}(\tfrac{r}{2\pi}) \mc{F}(\beta_{-1})(\tfrac{r}{2\pi})+
\ov{\mc{F}(\alpha_{-1})}(\tfrac{-r}{2\pi}) \mc{F}(\beta_{-1})(\tfrac{-r}{2\pi})\bigr)
\frac{\md r}{4 \pi}\\
&=
\int_{\RR} 2(\ov{\alpha_1} \beta_1+ \ov{\alpha_{-1}} \beta_{-1} ) \frac{\md r}{2}=
\int_{\RR} \bigl(\ov{\alpha(r,1)} \beta(r,1) + \ov{\alpha(r,-1)} \beta(r,-1) \bigr) \md r\\
&=
\int_{G} \ov{\alpha}\beta \md \mu_G.
\end{split}
\]
We have used the well known fact that the Fourier transform is a unitary operator. The above calculation shows that $\mc{Q}_L$ is a well defined isometry, hence it extends to the whole $\LdG$. Let us now justify that it is surjective. Take any function $f\in \LL^1(\RR)\cap \LL^2(\RR)$ and define $\alpha\colon G \ni (t,k)\mapsto \delta_{k,1} f(t)\in \CC$. We have
\[
(\alpha\otimes\id) U^{\sigma_r}=
\begin{bmatrix}
\int_{\RR} f(t)e^{irt}\md t & 0 \\
0 & \int_{\RR} f(t)e^{-irt}\md t
\end{bmatrix}=
\begin{bmatrix}
\mc{F}(f)(\tfrac{-r}{2\pi}) & 0\\
0 & \mc{F}(f)(\tfrac{r}{2\pi})
\end{bmatrix}
\]
and hence
\[
\mc{Q}_L (\alpha)=
\sqrt{2}\int_{\RR_{>0}}^{\oplus}
\begin{bmatrix}
\mc{F}(f)(\tfrac{-r}{2\pi}) & 0\\
0 & \mc{F}(f)(\tfrac{r}{2\pi})
\end{bmatrix}\frac{\md r}{4\pi}.
\]
For any $g\in \LL^2(\RR)$ we can find a sequence $(f_n)_{n\in\NN}$ in $\LL^1(\RR)\cap \LL^2(\RR)$ such that $(\mc{F}(f_n))_{n\in\NN}$ converges to $g$ in $\LL^2(\RR)$. It follows that
\[\begin{split}
&\quad\;\bigl\|\sqrt{2}
\int_{\RR_{>0}}^{\oplus}
\begin{bmatrix}
\mc{F}(f_n)(\tfrac{-r}{2\pi}) & 0\\
0 & \mc{F}(f_n)(\tfrac{r}{2\pi})
\end{bmatrix}\frac{\md r}{4\pi}-
\sqrt{2}\int_{\RR_{>0}}^{\oplus}
\begin{bmatrix}
g(\tfrac{-r}{2\pi})& 0\\
0 & g(\tfrac{r}{2\pi})
\end{bmatrix}\frac{\md r}{4\pi}
\bigr\|^2\\
&=
2\int_{\RR_{>0}} \bigl\|
\begin{bmatrix}
\mc{F}(f_n)(\tfrac{-r}{2\pi})-g(\tfrac{-r}{2\pi}) & 0\\
0 & \mc{F}(f_n)(\tfrac{r}{2\pi})-g(\tfrac{r}{2\pi})
\end{bmatrix}\bigr\|^2
\frac{\md r}{4\pi}\\
&=
\|\chi_{\RR_{<0}}(\mc{F}(f_n) - g)\|^2 + 
\|\chi_{\RR_{>0}}(\mc{F}(f_n)-g)\|^2=
\|\mc{F}(f_n)-g\|^2\xrightarrow[n\to\infty]{}0
\end{split}\]
and $\sqrt{2}\int_{\RR_{>0}}^{\oplus}
\begin{bmatrix} g(\tfrac{-r}{2\pi}) & 0 \\ 0 & g(\tfrac{r}{2\pi}) \end{bmatrix} \tfrac{\md r }{4\pi}$ belongs to the image of $\mc{Q}_L$. We can choose $g$ on $\RR_{\ge 0}$ and $\RR_{\le 0}$ independently, consequently for any $g,g'\in \LL^2(\RR_{>0})$ we have
\[
\int_{\RR_{>0}}^{\oplus}
\begin{bmatrix} g(r) & 0 \\ 0 & g'(r) \end{bmatrix} 
\tfrac{\md r }{4\pi}
\in \mc{Q}_L(\LL^2(G)).
\]
Analogous argument, this time using $\alpha(t,k)=\delta_{k,-1}f(t)$ shows that we have
\[
\int_{\RR_{>0}}^{\oplus}
\begin{bmatrix} g(r) & f(r) \\ f'(r) & g'(r) \end{bmatrix} 
\tfrac{\md r }{4\pi}
\in \mc{Q}_L(\LL^2(G))\quad(f,g,f',g'\in \LL^2(\RR_{>0}))
\]
and $\mc{Q}_L$ is a surjective operator. Moreover, since $\Lhvp(\lambda(\omega))=\omega$, we have checked point $7.2)$ of Theorem \ref{PlancherelL}.\\
Let us check the commutation relation $\mc{Q}_L (\alpha\otimes\id)\mrW^G=(\int_{\IrrG}^{\oplus} (\alpha\otimes\id)U^{\pi}\otimes \I_{\ov{\msf{H}_\pi}} \md\mu(\pi))\mc{Q}_L$ for $\alpha\in \LL^1(G)$. For any $\omega\in \mc{I}$ we have
\[\begin{split}
&\quad\;
\mc{Q}_L (\alpha\otimes\id)\mrW^G \Lhvp(\lambda(\omega))=
\mc{Q}_L \int_G \alpha(x) \lambda_x(\omega) \md\mu_G(x)=
\mc{Q}_L (\alpha\star\omega)\\
&=
\sqrt{2}\int_{\IrrG}^{\oplus}
(\alpha\star\omega\otimes\id) U^\pi\md\mu(\pi)=
\sqrt{2}\int_{\IrrG}^{\oplus} (\alpha\otimes\id)U^{\pi} \; (\omega\otimes\id)U^{\pi}
\md\mu(\pi)\\
&=
\sqrt{2}(\int_{\IrrG}^{\oplus} (\alpha\otimes\id)U^{\pi}\otimes \I_{\ov{\msf{H}_\pi}} \md\mu(\pi))
\int_{\IrrG}^{\oplus} 
(\omega\otimes\id)U^{\pi}
\md\mu(\pi)\\
&=
(\int_{\IrrG}^{\oplus} (\alpha\otimes\id)U^{\pi}\otimes \I_{\ov{\msf{H}_\pi}} \md\mu(\pi))\mc{Q}_L \Lhvp(\lambda(\omega))
\end{split}\]
from which the above equality follows.\\
Now the second relation. Since $G$ is unimodular, the dual group $\whG$ has tracial Haar integrals and as in the case of compact quantum groups we have
\[
(\alpha\otimes\id)\chi(\mrV^G)=\hat{J} \hat{R}((\alpha\otimes\id)\mrW^G)^* \hat{J}=
\hat{J} \lambda(\alpha^{\sharp}\circ R) \hat{J} \quad(\alpha\in \Lj),
\]
hence
\[\begin{split}
&\quad\;
\mc{Q}_L (\alpha\otimes\id)\chi(\mrV^G) \Lhvp(\lambda(\omega))=
\mc{Q}_L \hat{J} \Lhvp(\lambda(\alpha^{\sharp}\circ R)\lambda(\omega^{\sharp}))=
\mc{Q}_L \Lhvp(\lambda(\omega\star( \alpha\circ R)))\\
&=
\int_{\IrrG}^{\oplus} (\omega\star(\alpha\circ R)\otimes\id)U^{\pi}
\md\mu(\pi)\\
&=
\int_{\IrrG}^{\oplus} (\omega\otimes\id)U^{\pi}\;
(\alpha\circ R\otimes\id)U^{\pi}
\md\mu(\pi)\\
&=
\bigl(\int_{\IrrG}^{\oplus}
\I_{\msf{H}_\pi}\otimes ((\alpha\circ R\otimes\id)U^\pi)^T
\md\mu(\pi)\bigr)
\int_{\IrrG}^{\oplus} (\omega\otimes\id)U^\pi\md\mu(\pi)\\
&=
\bigl(\int_{\IrrG}^{\oplus}
\I_{\msf{H}_\pi}\otimes 
\pi^{c}((\alpha\otimes\id){\WW}^G)
\md\mu(\pi)\bigr)
\Lhvp(\lambda(\omega))
\end{split}\]
for any $\omega\in \mc{I}$. Let us check the equality $\mc{Q}_L( \Linfd\cap\Linfd')\mc{Q}_L^*=\Diag(\int_{\RR_{>0}}^{\oplus} \HS(\msf{H}_{\sigma_r}) \tfrac{\md r}{4\pi} )$. Equivalently we need to show $\mc{Q}_L (\Linfd\vee \Linfd') \mc{Q}_L^*=\Dec(\int_{\RR_{>0}}^{\oplus} \HS(\msf{H}_{\sigma_r}) \tfrac{\md r}{4\pi} )$. Since the right leg of $\mrW^G$ generates $\Linfd$, and the left leg of $\mrV^G$ generates $\Linfd'$ we have shown above the inclusion $\subseteq$. Equality follows from the reasoning similar to the one which showed surjectivity of $\mc{Q}_L$ -- the only difference is that one needs to pass to a subsequence in order to get convergence for almost all $r\in \RR_{>0}$.
This way we have proved the following:

\begin{proposition}
Operator $\mc{Q}_L$ introduced in the equation \eqref{eq22} and operators $D_\pi=\tfrac{1}{\sqrt{2}} \I_{\msf{H}_\pi}\,(\pi\in\Irr(G))$ are the operators given by Theorem \ref{PlancherelL}. Moreover, operators $\mc{Q}_R=\mc{Q}_L\circ J\hat{J},E_\pi=D_\pi,(\pi\in\Irr(G))$ are the operators given by Theorem \ref{PlancherelR}.
\end{proposition}

\begin{remark}
Group $G=\RR\rtimes\ZZ_2$ is amenable: a right invariant mean for $G$ is given by 
\[
m\colon \LL^{\infty}(\RR\rtimes\ZZ_2)\ni f \mapsto
\tfrac{1}{2}( m_{\RR}(f(\cdot,1))+m_{\RR}(f(\cdot,-1)))\in\CC,
\]
where $m_{\RR}\in {\LL^{\infty}(\RR)}^*$ is an invariant mean for $\RR$. Since $G$ is classical, it follows that the dual locally compact quantum group is coamenable (\cite{Brannan}) and consequently, \cite[Corollary 3.4.9]{Desmedt} implies that the support of $\mu$ is the whole $\Irr(G)$.
\end{remark}

In order to find out how the operator $\mc{L}_{\sigma_r}$ looks like, we need to establish a decomposition of the tensor product of representations. Take $r\in\RR_{>0}$ and a measurable subset $\Omega\subseteq \Irr(G)$. Since the set $\{1,\sigma\}$ is of measure zero, we can assume that $1,\sigma\notin \Omega$. With $\Omega$ we associate an integral representation $\sigma_\Omega=\int_{\Omega}^{\oplus} \sigma_s \md\mu(s)$. An arbitrary vector $\xi\in \int_{\Omega}^{\oplus} \msf{H}_{\sigma_{s}}\,\md\mu(s)$ can be written as
\[
\xi=\int_{\Omega}^{\oplus} (\xi_1(s)f^{s}_1+\xi_{-1}(s) f^{s}_{-1} )\md\mu(s)
\]
for certain measurable, square integrable functions $\xi_1,\xi_{-1}\colon \Omega\rightarrow \CC$. For $t\in \RR,a\in \ZZ_2$ we have
\[
\sigma_r \tp \sigma_\Omega (1,t) (f^r_a \otimes \xi)=
e^{iart} f^{r}_{a} \otimes
\int_{\Omega}^{\oplus} ( e^{ist} \xi_1(s) f^{s}_1 +
e^{-ist} \xi_{-1}(s) f^{s}_{-1} ) \md \mu(s)
\]
and
\[
\sigma_r \tp \sigma_\Omega (-1,t) (f^r_a \otimes \xi)=
e^{-iart} f^{r}_{-a} \otimes
\int_{\Omega}^{\oplus} ( e^{-ist} \xi_1(s) f^{s}_{-1} +
e^{ist} \xi_{-1}(s) f^{s}_{1} ) \md \mu(s).
\]
Let us introduce the following operator
\[
\begin{split}
U &\colon \msf{H}_{\sigma_r}\otimes \int_{\Omega}^{\oplus} \msf{H}_{\sigma_{s}} \,\md \mu(s)\rightarrow
\int_{\Omega+r}^{\oplus} \msf{H}_{\sigma_{s}} \,\md \mu(s)
\oplus
\int_{(\Omega-r)\cap \RR_{>0}}^{\oplus} \msf{H}_{\sigma_{s}} \,\md \mu(s)
\oplus
\int_{-((\Omega-r)\cap \RR_{<0})}^{\oplus} \msf{H}_{\sigma_{s}} \,\md \mu(s)\\
&\colon
f^r_1 \otimes \int_{\Omega}^{\oplus}
(\xi_1(s) f^{s}_{1} + \xi_{-1}(s) f^{s}_{-1})
\md\mu(s)\\
&\mapsto
\int_{\Omega+r}^{\oplus}
\xi_1(s-r) f^{s}_{1}\md\mu(s)
\oplus
\int_{(\Omega-r)\cap \RR_{>0}}^{\oplus}
\xi_{-1}(s+r) f^{s}_{-1}\md\mu(s)
\oplus
\int_{-((\Omega-r)\cap \RR_{<0})}^{\oplus}
\xi_{-1}(-s+r) f^{s}_{1}\md\mu(s)\\
&\colon
f^r_{-1} \otimes \int_{\Omega}^{\oplus}
(\xi_1(s) f^{s}_{1} + \xi_{-1}(s) f^{s}_{-1})
\md\mu(s)\\
&\mapsto
\int_{\Omega+r}^{\oplus}
\xi_{-1}(s-r) f^{s}_{-1}\md\mu(s)
\oplus
\int_{(\Omega-r)\cap \RR_{>0}}^{\oplus}
\xi_{1}(s+r) f^{s}_{1}\md\mu(s)
\oplus
\int_{-((\Omega-r)\cap \RR_{<0})}^{\oplus}
\xi_{1}(-s+r) f^{s}_{-1}\md\mu(s).
\end{split}
\]
It is clear that $U$ is a unitary operator. A straightforward but lengthy calculations which we skip, show that $U$ is an intertwiner between $\sigma_r\tp\sigma_\Omega$ and $\sigma_{\Omega+r}\oplus \sigma_{(\Omega-r)\cap \RR_{>0}}\oplus \sigma_{-((\Omega-r)\cap \RR_{<0})}$, hence these representations are equivalent. Let us introduce the following notation:
\[
(\Omega-r)\cap \RR_{>0}=(\Omega-r)^+,\quad
-((\Omega-r)\cap \RR_{<0})=(\Omega-r)^-.
\]
We always have $(\Omega+r)\cap (\Omega-r)^-=\emptyset$, consequently
\[
\begin{split}
&\quad\;\sigma_{\Omega+r}\oplus \sigma_{(\Omega-r)\cap \RR_{>0}}\oplus \sigma_{-((\Omega-r)\cap \RR_{<0})}\\
&\simeq
\sigma_{(\Omega+r)\setminus (\Omega-r)^+}
\oplus
2\cdot \sigma_{(\Omega+r)\cap (\Omega-r)^+}\oplus
\sigma_{(\Omega+r)^+\setminus ((\Omega+r)\cup (\Omega-r)^-)}\oplus
2\cdot \sigma_{(\Omega-r)^+\cap (\Omega-r)^-}\oplus
\sigma_{(\Omega-r)^-\setminus (\Omega-r)^+}
\end{split}
\]
and we have
\[
\begin{split}
\E^{1}_{\sigma_r \stp \sigma_\Omega}&=
((\Omega+r)\setminus (\Omega-r)^+)\cup
((\Omega+r)^+\setminus ((\Omega+r)\cup (\Omega-r)^-))\cup
((\Omega-r)^-\setminus (\Omega-r)^+)\\
&=
(\Omega+r)\triangle (\Omega-r)^+\triangle (\Omega-r)^-
\\
\E^{2}_{\sigma_r \stp \sigma_\Omega}&=
((\Omega+r)\cap (\Omega-r)^+)\cup ((\Omega-r)^+\cap (\Omega-r)^-)
,\quad
\E^{n}_{\sigma_r \stp \sigma_\Omega}=
\emptyset\;(n\ge 3).
\end{split}
\]
Let us now prove that $\varpi^{\sigma_r,\Omega,\mu}=1$ on $\F^1_{\sigma_r\stp \sigma_{\Omega}}$. Since $\CGDu=\mathrm{C}^*(G)$, we have in particular $\mathrm{C}_c(G)\subseteq \mathrm{C}^*(G)$. Take any $g\in \mathrm{C}_c(G)$ such that $g(t,-1)=0$ for all $t\in \RR$. By the definition of function $\varpi^{\sigma_r,\Omega,\mu}$ we have
\[
\int_\Omega \Tr(\sigma_r\stp \sigma_s(g))\tfrac{\md s}{4\pi}=
\int_{\RR_{>0}}\varpi^{\sigma_r,\Omega,\mu}(s)
 \bigl(\sum_{n=1}^{2} n \chi_{\E^n_{\sigma_r\stp \sigma_\Omega}}\bigr)(s) \Tr(\sigma_s(g))\tfrac{\md s}{4\pi}
\]
(note that for any $g\in \mathrm{C}_c(G)$, treated as an element of $\mathrm{C}^*(G)$ we have $\int_{\Omega} \Tr(\sigma_r\tp\sigma_s (|g|)) \md s <+\infty$ hence we indeed can use this element). Left hand side of the above equality (multiplied by $4\pi$) looks as follows:
\[\begin{split}
&\quad\;
\int_\Omega \Tr(\sigma_r\stp \sigma_s(g))\md s=
\int_\Omega \int_{\RR} g(t,1) (e^{irt}+e^{-irt})(e^{ist}+e^{-ist})
\md t \md s\\
&=
\int_{\RR_{>0}} \chi_\Omega(s) (
\mc{F}(g_1)(\tfrac{-(r+s)}{2\pi})+ 
\mc{F}(g_1)(\tfrac{-(r-s)}{2\pi})+ 
\mc{F}(g_1)(\tfrac{r-s}{2\pi})+ 
\mc{F}(g_1)(\tfrac{r+s}{2\pi}))\md s
\end{split}\]
where as before, $g_1=g(\cdot,1)$, wheras the right hand side (again, multiplied by $4\pi$) is
\[\begin{split}
&\quad\;
\int_{\RR_{>0}}
\varpi^{\sigma_r,\Omega,\mu}(s) \bigl(\sum_{n=1}^{2} n \chi_{\E^n_{\sigma_r\stp \sigma_\Omega}}\bigr)(s) \Tr(\sigma_s(g))\md s\\
&=
\int_{\RR_{>0}}
\varpi^{\sigma_r,\Omega,\mu}(s) (\chi_{\Omega+r}+ \chi_{(\Omega-r)^+}+
\chi_{(\Omega-r)^-})(s) \int_{\RR} g(t,1) (e^{ist}+e^{-ist})\md t \md s\\
&=
\int_{\RR_{>0}} 
\varpi^{\sigma_r,\Omega,\mu}(s)(\chi_{\Omega+r}+ \chi_{(\Omega-r)^+}+
\chi_{(\Omega-r)^-})(s) 
(\mc{F}(g_1)(\tfrac{-s}{2\pi})+ \mc{F}(g_1)(\tfrac{s}{2\pi}))
\md s
\end{split}\]
It follows that we have
\begin{equation}\begin{split}\label{eq27}
&\quad\;\int_{\RR_{>0}} \chi_\Omega(s) (
\mc{F}(g_1)(\tfrac{-(r+s)}{2\pi})+ 
\mc{F}(g_1)(\tfrac{-(r-s)}{2\pi})+ 
\mc{F}(g_1)(\tfrac{r-s}{2\pi})+ 
\mc{F}(g_1)(\tfrac{r+s}{2\pi}))\md s\\
&=
\int_{\RR_{>0}} 
\varpi^{\sigma_r,\Omega,\mu}(s)(\chi_{\Omega+r}+ \chi_{(\Omega-r)^+}+
\chi_{(\Omega-r)^-})(s) 
(\mc{F}(g_1)(\tfrac{-s}{2\pi})+ \mc{F}(g_1)(\tfrac{s}{2\pi}))
\md s
\end{split}\end{equation}
for all $g\in \mathrm{C}_c(G)$ supported on $\RR\times \{1\}$. Since $\Omega$ has finite measure, $\mathrm{C}_c(\RR)$ is dense in $\LL^2(\RR)$, Fourier transform is unitary and $\varpi^{\sigma_r,\Omega,\mu}$ is bounded (see Proposition \ref{stw4}), by continuity argument we can plug in equation \eqref{eq27} any square integrable function on $\RR$ instead of $\mc{F}(g_1)$. Thus we have
\[
\begin{split}
&\quad\;\int_{\RR_{>0}} \chi_\Omega(s) (
h(-r-s)+ 
h(-r+s)+ 
h(r-s)+ 
h(r+s))\md s\\
&=
\int_{\RR_{>0}} 
\varpi^{\sigma_r,\Omega,\mu}(s)(\chi_{\Omega+r}+ \chi_{(\Omega-r)^+}+
\chi_{(\Omega-r)^-})(s) 
(h(-s)+ h(s))
\md s
\end{split}
\]
for all $h\in \LL^2(\RR)$. For $h\in \LL^2(\RR)$ such that $h(-t)=0\,(t\in \RR_{>0})$ the above expression simplifies:
\[
\begin{split}
&\quad\;
\int_{\RR_{>0}} (\chi_\Omega (s+r)+\chi_\Omega (-s+r) + \chi_{\Omega}(s-r))h(s)\md s\\
&=\int_{\RR_{>0}} \chi_\Omega(s) (
h(-r+s)+ 
h(r-s)+ 
h(r+s))\md s\\
&=
\int_{\RR_{>0}} 
\varpi^{\sigma_r,\Omega,\mu}(s)(\chi_{\Omega+r}+ \chi_{(\Omega-r)^+}+
\chi_{(\Omega-r)^-})(s) 
h(s)
\md s
\end{split}
\]
and it follows that we have
\[
\chi_\Omega (s+r)+\chi_\Omega (-s+r) + \chi_{\Omega}(s-r)=
\varpi^{\sigma_r,\Omega,\mu}(s)(\chi_{\Omega+r}+ \chi_{(\Omega-r)^+}+
\chi_{(\Omega-r)^-})(s) 
\]
for almost all $s\in \RR_{>0}$. Since
\[
\chi_{\Omega}(s+r)=\chi_{(\Omega-r)^+}(s),\quad
\chi_{\Omega}(-s+r)=\chi_{(\Omega-r)^-}(s),\quad
\chi_{\Omega}(s-r)=\chi_{\Omega+r}(s),
\]
for all $s\in \RR_{>0}$ we arrive at the conclusion:
\begin{proposition}
For any $r\in\RR_{>0}$ and a measurable subset $\Omega\subseteq\Irr(G)$ we have 
\[
\sigma_r\tp \sigma_{\Omega}\simeq
\sigma_{\Omega+r}\oplus\sigma_{(\Omega-r)^+}\oplus
\sigma_{(\Omega-r)^-},
\]
and
\[
\varpi^{\sigma_r,\Omega,\mu}(s)=1\quad(s\in\F^1_{\sigma_r\stp \sigma_\Omega}).
\]
\end{proposition}

Armed with this result, we can proceed with calculating what is the action of the operator $\mc{L}_{\sigma_r}$. By the definition we have
\[
\mc{L}_{\sigma_r} \chi_{\Omega}=\sum_{i=1}^{\infty} \chi_{\F^i_{\sigma_r\stp\sigma_{\Omega}}}=
\chi_{\Omega+r}+\chi_{(\Omega-r)^-}+\chi_{(\Omega-r)^+}
\]
for all measurable subsets $\Omega\subseteq\Irr(G)$ of finite measure (note that we identify $\LL^2(\Irr(G))$ with $\LL^2(\RR_{>0},\tfrac{dr}{4\pi})$). Thanks to Theorem \ref{tw1} we know that $\mc{L}_{\sigma_r}$ is a well defined bounded operator on $\LL^2(\Irr(G))$. Let us introduce three shift operators:
\[
L_{r},\;L^+_{-r},\;L_{-r}^{-}\colon \LL^2(\IrrG)\rightarrow \LL^2(\IrrG)\quad(r\in\RR_{>0})
\]
defined by
\[
L_r(f)(s)=\begin{cases} f(s-r) & s-r>0 \\ 0 & s-r \le 0 \end{cases}
,\quad
L_{-r}^+(f)(s)=f(r+s),\quad
L_{-r}^-(f)(s)=\begin{cases} f(r-s) & r-s >0 \\ 0 & r-s \le 0 \end{cases}
\]
for any $f\in \LL^2(\IrrG)$ and $r,s\in\RR_{>0}$. It is clear that these are well defined contractions. For finite measure measurable subset $\Omega\subseteq\IrrG$ and $r,s\in\RR_{>0}$ we have
\[
\begin{split}
L_{r}(\chi_{\Omega})(s)&=
\chi_{\Omega}(s-r)=\chi_{\Omega+r}(s),\quad
L_{-r}^{+}(\chi_\Omega)(s)=
\chi_\Omega(r+s)=\chi_{(\Omega-r)^+}(s)\\
L_{-r}^{-}(\chi_\Omega)(s)&=
\chi_\Omega(r-s)=\chi_{-(\Omega-r)}(s)=
\chi_{(\Omega-r)^-}(s).
\end{split}
\]
It follows that 
\[
\mc{L}_{\sigma_r}(\chi_\Omega)=(L_{r}+L^+_{-r}+L_{-r}^{-})(\chi_\Omega)\quad(r\in\RR_{>0})
\]
and consequently by linearity and continuity we have the following result:
\begin{proposition}
The operator $\mc{L}_{\sigma_r}$ is given by $\mc{L}_{\sigma_r}=L_{r}+L^+_{-r}+L_{-r}^{-}$ for all $r\in \RR_{>0}$.
\end{proposition}

Theorem \ref{tw1} implies that the norm of $\mc{L}_{\sigma_r}$ is less or equal to $2$. For $\nu\in \LL^1(\IrrG)$ we can define an operator
\[
\mc{L}_\nu=\int_{\RR_{>0}} \tfrac{\nu(\sigma_r)}{\dim(\sigma_r)} \mc{L}_{\sigma_r}\tfrac{\md r}{4\pi}\in \B(\LL^2(\Irr(G))).
\]
Locally compact group $G$ is classical, hence coamenable -- by Theorem \ref{tw4} we know that for any measurable subset of finite measure $\Omega\subseteq \IrrG$ and $\nu=\dim \chi_\Omega$ the number $\|\nu\|_1=2\mu(\Omega)$ belongs to the spectrum of $\mc{L}_{\nu}$. One could check this also directly: the sequence of unit vectors $(f_m)_{m\in\NN}$ with $f_m=\sqrt{\tfrac{4\pi}{m}} \chi_{[m,2m]}\,(m\in\NN)$ forms an approximate eigenvector with eigenvalue $\|\nu\|_1$ (this holds for a general positive function $\nu\in\LL^1(\Irr(G))$).\\

At the end of this example, let us take a look at the integral characters of $G=\RR\rtimes \ZZ_2$ and calculate the \swot\, closure of the \cst-algebra $\mf{A}$ defined in Section \ref{cstA}. Our aim is to show the following result:

\begin{proposition}
The \swot-closure of the \cst-algebra $\mf{A}$ is given by
\begin{equation}\label{eq28}
\ov{\mf{A}}^{\swot}=\{f\in \LL^{\infty}(G)\,|\, \supp(f)\subseteq \RR\times \{1\},\;
\forall_{t\in \RR} \,f(t,1)=f(-t,1)\},
\end{equation}
i.e. $\ov{\mf{A}}^{\swot}$ consist of measurable, bounded functions supported on $\RR\times\{1\}$ which are even. 
\end{proposition}
Note that the above result implies that $\mf{A}$ is a degenerate \cst-algebra and $\I\notin \ov{\mf{A}}^{\swot}$.\\

Take any integral representation $\pi_X\in \Rep_{q,<+\infty}^{\int}(G)$. We start with the following lemma

\begin{lemma}
For $\mu_X$-almost all $x$, the representations $1$, $\sigma$ are not contained in $\pi_x$ (hence almost every $\pi_x$ can be written as a finite direct sum of $\sigma_r$'s).
\end{lemma}

\begin{proof}
Assume that this is not true and we have a measurable subset $X_0\subseteq X$ of positive measure such that $\sigma\subseteq \pi_x$ for all $x\in X_0$. Then 
\begin{equation}\label{eq23}
\sigma\lec_q\int_{X_0}^{\oplus} \pi_x \md\mu_{X_0}(x)=\pi_{X_0},
\end{equation}
where $\mu_{X_0}$ is the restriction of $\mu_X$ to $X_0$. Let us justify why this relation holds.\\
For any $(t,k)\in G$ we have
\[
\pi_{X_0}(t,k)=\bigl(\int_{X_0}^{\oplus} \pi_x \md\mu_{X_0}(x)\bigr)(t,k)=
\int_{X_0}^{\oplus} \pi_x (t,k)\md\mu_{X_0}(x)
\]
and for all $x\in X_0$ in $\msf{H}_{\pi_x}$  there is a direct summand corresponding to $\sigma\subseteq \pi_x$, on which for all $(t,k)\in G$ the operator $\pi_x(t,k)$ acts as $\sigma(t,k)$ (recall that $\sigma$ is the one dimensional representation with $\sigma(t,k)=k\I_{\msf{H}_\sigma}$). It follows that we have a unitary equivalence $\pi_x\simeq m_x\cdot \sigma\oplus \pi'_x$ for some number $m_x\in \NN$ and representation $\pi'_x$. Assume that $m_x$ is maximal such a number, i.e. $\sigma$ is not a subrepresentation of $\pi'_x$.\\
Let $\{t_n\,|\,n\in\NN\}$ be a dense subset in $\RR$. For $k\in \{-1,1\}$ define projections $E_{n,k},E^{x}_{n,k}$ via
\[
E_{n,k}=\chi_{\{k\}}(\pi_{X_0}(t_n,k))=\int_{X_0}^{\oplus}
\chi_{\{k\}}(\pi_x(t_n,k)) \md\mu_{X_0}(x)=
\int_{X_0}^{\oplus} E^x_{n,k} \md\mu_{X_0}(x).
\]
Let $E,E^x$ be projections onto
\[
\bigcap_{n\in\NN,k\in\{-1,1\}} E_{n,k}\msf{H}_{\pi_{X_0}}\subseteq \msf{H}_{\pi_{X_0}},\quad
\bigcap_{n\in\NN,k\in\{-1,1\}} E^{x}_{n,k}\msf{H}_{\pi_{x}}
\subseteq \msf{H}_{\pi_x}.
\]
Lemma \ref{lemat14} together with \cite[Proposition 4, page 173]{DixmiervNA} implies that $E$ can be written as a limit of products of various $E_{n,k}$'s and 
\[
E=\int_{X_0}^{\oplus} E^x\md\mu_{X_0}(x).
\]
In particular, the field of operators $(E^x)_{x\in X_0}$ is measurable. Observe, that for almost all $x\in X_0$ the projection $E^x$ corresponds to the subrepresentation $m_x\cdot \sigma\subseteq\pi_x$. It follows that the field of subspaces $(E^x\msf{H}_{\pi_x})_{x\in X_0}$ is measurable (\cite[Proposition 9, page 173]{DixmiervNA}) and we have a decomposition of $\msf{H}_{\pi_{X_0}}$ into an ortogonal direct sum
\[
\int_{X_0}^{\oplus}\msf{H}_{\pi_x}\md\mu_{X_0}(x)=
\int_{X_0}^{\oplus}E^x\msf{H}_{\pi_x}\md\mu_{X_0}(x)
\oplus
\int_{X_0}^{\oplus}(\I_{\msf{H}_{\pi_x}}-E^x)\msf{H}_{\pi_x}\md\mu_{X_0}(x).
\]
This decomposition is preserved by $\pi_{X_0}$. On the first summand, representation $\pi_{X_0}$ acts as $\int_{X_0}^{\oplus} m_x\cdot\sigma\md\mu_X(x)$, therefore, as the quasi-containment does not see multiplicities, we arrive at $\sigma\lec_q \pi_{X_0}$ and the equation \eqref{eq23} holds. Using this containment we are able to derive a contradiction:
\[
\sigma\lec_q \int_{X_0}^{\oplus} \pi_x \md\mu_{X_0}
\subseteq \int_{X}^{\oplus} \pi_x \md\mu_{X}=\pi_X\lec_q \Lambda_{\whG}
\]
which is false since the singleton $\{\sigma\}\subseteq \Irr(G)$ has Plancherel measure $0$ (see Proposition \ref{stw26}). An analogous argument shows that for $\mu_X$-almost all $x$ the trivial representation is not contained in $\pi_x$.
\end{proof}

By the definition of the integral character we have
\[
\chi^{\int}(\pi_X)=\int_{X} \chi(U^{\pi_x}) \md \mu_X(x)\in \LL^{\infty}(G).
\]
Since almost all $\pi_x$ are direct sums of various $\sigma_r$'s, we have (see equation \eqref{eq26})
\[\begin{split}
&\quad\;
\chi^{\int}(\pi_X)(t,k)=
\int_X \chi(U^{\pi_x})(t,k)\md\mu_X(x)\\
&=
\delta_{k,1}\int_X \chi(U^{\pi_x})(-t,k)\md\mu_X(x)=
\delta_{k,1}\chi^{\int}(\pi_X)(-t,k)
\end{split}\]
for all $(t,k)\in G$ and it follows that we have inclusion $\subseteq$ in equation \eqref{eq28} (right hand side is $\swot$ closed). Now, take any even function $f\in \mathrm{C}_c(\RR)\subseteq\LL^{\infty}(\RR)$ and consider $\Omega_n=]0,n](n\in\NN)$ treated as a measurable subset of $\Irr(G)$. For any $(t,k)\in G$ and $n\in\NN$ we have
\[
\chi^{\int}( \int^{\oplus}_{\Omega_n} \sigma_r \md f \mu(r))(t,k)=
\delta_{k,1}\int_0^n f(r) (e^{irt}+e^{-irt})\tfrac{\md r}{4\pi}=
\tfrac{\delta_{k,1}}{4\pi}\int_{-n}^{n}
 f(r) e^{-irt}\md r.
\]
We claim that
\[
\chi^{\int}(\int_{\Omega_n}^{\oplus} \sigma_r \md f \mu(r))
\xrightarrow[n\to\infty]{\swot} F,
\]
where $F\in\LL^{\infty}(G)$ is a function given by $F(t,k)=\tfrac{\delta_{k,1}}{4\pi} \mc{F}(f)(\tfrac{t}{2\pi})$. Indeed, take any $\omega\in \LL^1(G)$. Calculation
\[\begin{split}
&\quad\;
\bigl|\int_G \bigl(
\chi^{\int}(\int_{\Omega_n}^{\oplus} \sigma_r \md f \mu(r))-F\bigr)\omega \md\mu_G\bigr|=
\tfrac{1}{4\pi}\bigl|\int_{\RR} \bigl(\int_{-n}^{n} f(r) e^{-irt} \md r-
\mc{F}(\tfrac{t}{2\pi}) \bigr)\omega(t,1)\md t\bigr|\\
&=
\tfrac{1}{4\pi}\bigl|\int_{\RR} \bigl(\int_{\RR\setminus [-n,n]} f(r) e^{-irt} \md r\bigr)\omega(t,1)\md t\bigr|\xrightarrow[n\to\infty]{}0
\end{split}\]
which uses the fact that $f$ has compact support, proves the claim. Since the Fourier transform is unitary and maps even functions to even functions, we get the desired equality
\[
\ov{\mf{A}}^{\swot}=\{f\in \LL^{\infty}(G)\,|\, \supp(f)\subseteq \RR\times \{1\},\;
\forall_{t\in \RR} \,f(t,1)=f(-t,1)\}.
\]

\subsection{Bicrossed product $G^{\times N}\bowtie\ZZ_N$}
An important way of constructing new interesting examples of locally compact quantum groups is the bicrossed product construction. We refer the reader to \cite{VaesVainerman} for the introduction to the theory. We will focus on a class of examples given by a bicrossed product of $G^{\times N}$ and $\ZZ_N$ for some $N\in\NN$. Let us be more precise: let $G$ be a (nontrivial) second countable locally compact Hausdorff group, $N$ a natural number greater that $1$, $G^{\times N}$ the $N$-th cartesian power of $G$ and $\ZZ_N=\{0,\dotsc,N-1\}$ the additive group modulo $N$. Form a semidirect product $G^{\times N}\rtimes \ZZ_N$ with group operation as follows:
\[
((g_a)_{a=1}^{N},k)\,((g'_a)_{a=1}^{N},k')=((g_a g'_{a+k})_{a=1}^{N},k+k')
\]
for $((g_a)_{a=1}^{N},k),((g'_{a})_{a=1}^{N},k')\in G^{\times N}\rtimes \ZZ_N$ (this is an instance of the wreath product). Together with the canonical inclusions
\[
i\colon G^{\times N}\ni (g_a)_{a=1}^{N} \mapsto ((g_a)_{a=1}^{N},0)\in G^{\times N}\rtimes \ZZ_N,\quad
j\colon \ZZ_N\ni k \mapsto (e,k)\in G^{\times N}\rtimes \ZZ_N
\]
$(G^{\times N},\ZZ_N)$ forms a matched pair of locally compact groups (\cite[Definition 4.7.]{VaesVainerman}). Note that we have $\Omega=\{i((g_a)_{a=1}^{N})j(k)\,|\,(g_a)_{a=1}^{N}\in G^{\times N},k\in\ZZ_N\}=G^{\times N}\rtimes \ZZ_N$. One easily checks that the action $\alpha$ is trivial and $\beta$ is given by
\[
\beta_k((g_a)_{a=1}^{N})=(g_{a-k})_{a=1}^{N}\quad(k\in\ZZ_N,(g_a)_{a=1}^{N}\in G^{\times N}).
\]
Now, let $\GG$ be the locally compact quantum group given by the bicrossed construction with trivial cocycles. We will denote it by $
\GG=G^{\times N}\bowtie \ZZ_N$.\\

We would like to show that $\GG$ is a type I locally compact quantum group with finite dimensional irreducible representations. In order to do that, we first need to go through some elementary topological considerations: consider the topological space $\ZZ_N\setminus (G^{\times N}\rtimes \ZZ_N)$ with the quotient topology. Since we have 
\[
(e,l)((g_a)_{a=1}^{N},k)=((g_{a+l})_{a=1}^{N},l+k)
\]
for all $((g_a)_{a=1}^{N},k)\in G^{\times N}\rtimes\ZZ_N,l\in \ZZ_N$, it follows that
\[
[((g_a)_{a=1}^{N},k)]=\{((g_{a+l})_{a=1}^{N},k+l)\,|\,l\in\ZZ_N\}\in \ZZ_N\setminus (G^{\times N}\rtimes \ZZ_N)
\]
for each $((g_a)_{a=1}^{N},k)\in G^{\times N}\rtimes\ZZ_N$. It is therefore clear that we have a bijection
\[
\Psi\colon G^{\times N}\ni (g_a)_{a=1}^{N} \mapsto [((g_a)_{a=1}^{N},0)] \in \ZZ_N\setminus (G^{\times N}\rtimes \ZZ_N).
\]
It is a homeomorphism. Indeed, assume that $\{[((g_a)_{a=1}^{N},0)]\,|\,(g_a)_{a=1}^{N}\in A\}$ is an open subset of $\ZZ_N\setminus(G^{\times N}\rtimes \ZZ_N)$. Then by the definition, its preimage under the quotient map,
\[
\{((g_{a+l})_{a=1}^{N},l)\,|\,(g_a)_{a=1}^{N}\in A,l\in\ZZ_N\}\subseteq G^{\times N}\rtimes \ZZ_N
\]
is open. Topologically we have $G^{\times N}\rtimes\ZZ_N=G^{\times N}\times \ZZ_N$ hence it follows that the projection onto the first coordinate, $A\subseteq G^{\times N}$ is open. Consequently $\Psi$ is continuous. Now, to show that $[((g_a)_{a=1}^{N},0)]\mapsto (g_a)_{a=1}^{N}$ is also continuous take any open subset $A\subseteq G^{\times N}$. Its preimage under the composition of the quotient map and $\Psi^{-1}$ is $\{((g_{a+l})_{a=1}^{N},l)\,|l\in\ZZ_N,\,a\in A\}$ which is clearly open in $G^{\times N}\rtimes\ZZ_N$.\\
By \cite[Proposition 3.7]{BaajSkandalisVaes} (note different conventions, see also \cite{VaesVainermanlow}) we have 
\[
\mathrm{C}_0^{u}(\whG)=\mathrm{C}_0(\ZZ_N \setminus (G^{\times N}\rtimes\ZZ_N))\rtimes\ZZ_N,
\]
 where the right action of $\ZZ_N$ on $\ZZ_N\setminus (G^{\times N}\rtimes \ZZ_N)$ is given by
\[
([((g_a)_{a=1}^{N},k)],l)\mapsto [((g_a)_{a=1}^{N},k)(e,l)]=[((g_a)_{a=1}^{N},k+l)]
\]
and here $\rtimes$ denotes the full crossed product. Under the identification $\ZZ_N\setminus (G^{\times N}\rtimes \ZZ_N)\simeq G^{\times N}$, the action of $\ZZ_N$ looks as follows: $G^{\times N}\times \ZZ_N\ni ((g_a)_{a=1}^{N},l)\mapsto (g_{a-l})_{a=1}^{N}\in G^{\times N}$. It follows that the \cst-algebra $\CGDu$ is isomorphic to $\mathrm{C}_0(G^{\times N})\rtimes\ZZ_N$. Clearly it is separable. Proposition 7.30 of \cite{Williams} tells us that the above \cst-algebra is of type I provided we can show that the quotient space $G^{\times N}/\ZZ_N$ is $T_0$. Indeed this is the case, as the following more general lemma shows:

\begin{lemma}
Let $H=\{h_1,\dotsc,h_n\}$ be a finite group acting continuously on a Hausdorff space $X$. Then the quotient space $H\setminus X$ is Hausdorff.
\end{lemma}

\begin{proof}
Take any two distinct orbits $[x],[y]\in H\setminus X$. Since $X$ is Hausdorff, there are open sets $\mc{U}_{(k,l)},\mc{V}_{(k,l)}\,(k,l\in\{1,\dotsc,n\})$ in $X$ such that $h_k x\in \mc{U}_{(k,l)},h_l y \in \mc{V}_{(k,l)}$ and $\mc{U}_{(k,l)}\cap\mc{V}_{(k,l)}=\emptyset$ for all $k,l\in\{1,\dotsc,n\}$. Define $\mc{U}=\cap_{k,l=1}^{n} h_k^{-1} \mc{U}_{(k,l)},\mc{V}=\cap_{k,l=1}^{n} h_l^{-1} \mc{V}_{(k,l)}$. These are open sets, and we have $x\in \mc{U},y\in \mc{V}$, consequently $[x]\in [\mc{U}],[y]\in[\mc{V}]$. Since the quotient map is open, sets $[\mc{U}],[\mc{V}]$ are open in $H\setminus X$. Let us justify that they are also disjoint. Assume that $[z]\in [\mc{U}]\cap [\mc{V}]$. It follows that there exists $k,l\in\{1,\dotsc,n\}$ such that $h_k^{-1} z\in \mc{U}$ and $h_l^{-1} z\in \mc{V}$. However this implies that $z\in \mc{U}_{(k,l)}\cap\mc{V}_{(k,l)}$ gives us a contradiction.
\end{proof}

Irreducible representations of $\GG$ can be identified with irreducible representations of $\CGDu\simeq \mathrm{C}_0(G^{\times N})\rtimes \ZZ_N$, and those have dimension less or equal to $N$ (\cite[Theorem 8.39, Proposition 5.4]{Williams}).
To sum up, we have proved the following:

\begin{proposition}
The \cst-algebra $\mathrm{C}_0^{u}(\whG)$ is isomorphic to the crossed product \cst-algebra $\mathrm{C}_0(G^{\times N})\rtimes \ZZ_N$ with an action of $\ZZ_N$ on $G^{\times N}$ given by
\[
G^{\times N}\times \ZZ_N\ni ((g_a)_{a=1}^{N},l)\mapsto (g_{a-l})_{a=1}^{N}\in G^{\times N}.
\]
Moreover, this \cst-algebra is separable, of type I and its irreducible representations have dimension less or equal to $N$.
\end{proposition}

Let us now consider further properties of our group $\GG=G^{\times N}\bowtie \ZZ_N$. Since $G$ is a nontrivial group, the action of $\ZZ_N$ on $\mathrm{C}_0(G^{\times N})$ is also nontrivial hence $\mathrm{C}_0(G^{\times N})\rtimes \ZZ_N$ is noncommutative and locally compact quantum group $\whG$ is not classical. Furthermore, $\GG$ is classical if and only if $G$ is abelian. Indeed, since $\alpha$ is trivial we have
\[
\LL^{\infty}(\GG)=G^{\times N}\ltimes \LL^{\infty}(\ZZ_N)=
\mathrm{L}(G^{\times N})\bar{\otimes} \LL^{\infty}(\ZZ_N)
\]
which is commutative precisely when $G$ is abelian.
\\
By \cite[Proposition 2.7]{VaesVainerman}, $\GG$ is compact if and only if $G$ is discrete, and $\GG$ is discrete if and only if $G$ is compact.\\
Let $\mu_G,\delta_G$ be a left Haar measure and the modular element\footnote{Recall that the modular element is given by the Radon-Nikodym derivative of the right Haar measure divided by the left Haar measure.} for $G$. One easily checks that $\mu_G^{\times N}$ is a left Haar measure for $G^{\times N}$ and $G^{\times N}\ni(g_a)_{a=1}^{N}\mapsto \prod_{a=1}^{N}\delta_G(g_a)\in\RR_{>0}$ is the modular element for $G^{\times N}$. Next, without much effort one checks that the product of $\mu^{\times N}_G$ on $G^{\times N}$ and the counting measure on $\ZZ_N$ gives a left Haar measure on $G^{\times N}\rtimes \ZZ_N$ and the modular element is given by
\[
G^{\times N}\rtimes \ZZ_N\ni ((g_a)_{a=1}^{N},k)\mapsto
\prod_{a=1}^{N} \delta_G(g_a)\in \RR_{>0}.
\]
Propositions 4.14 and 4.15 of \cite{VaesVainerman} allow us to identify $\LL^2(\GG)$ with $\LL^2(G^{\times N}\times \ZZ_N)$ and describe operators associated with $\GG$ and $\whG$. Let us summarize the above observation in a proposition:

\begin{proposition}\label{stw16}
Locally compact quantum group $\whG$ is never classical. We have
\[
\Linf=\LL(G^{\times N})\bar{\otimes} \LL^{\infty}(\ZZ_N),
\]
hence $\GG$ is classical if and only if $G$ is abelian.\\
Furthermore, $\GG$ is compact if and only if $G$ is discrete, and $\GG$ is discrete if and only if $G$ is compact.\\
Under the identification $\LdG=\LL^2(G^{\times N} \times \ZZ_N)$ the modular operator $\nabla_\vp$ and the modular element for $\whG$, $\delta_{\whG}$ are given by multiplication with strictly positive functions (denoted with the same letter):
\[
\nabla_{\vp}((g_a)_{a=1}^{N},k)=\prod_{a=1}^{N}\delta_G(g_a)^{-1},\quad
\delta_{\whG}((g_a)_{a=1}^{N},k)=\prod_{a=1}^{N} \delta_G(g_a),
\]
where $((g_a)_{a=1}^{N},k)\in G^{\times N}\times \ZZ_N.$. Modular conjugations $J,\hat{J}$ for the left invariant Haar integrals $\vp,\hvp$ are given by
\[
(J\xi)((g_a)_{a=1}^{N},k)=
\bigl(\prod_{a=1}^{N} \delta_{G}(g_a)^{\frac{1}{2}}\bigr)
\,\ov{\xi ((g_a^{-1})_{a=1}^{N},k)},\quad
(\hat{J}\xi)((g_a)_{a=1}^{N},k)=
\ov{\xi ((g_{a-k})_{a=1}^{N},-k)},
\]
where $\xi\in \LL^2(G^{\times N}\times \ZZ_N)$ and $((g_a)_{a=1}^{N},k)\in G^{\times N}\times \ZZ_N)$.  Moreover we always have
\[
P=\delta_{\GG}=\nabla_{\hvp}=\I
\]
and the scaling constant of $\GG$ is $1$.
\end{proposition}

At the end of this subsection let us establish when $\GG,\whG$ are (co)amenable. We will do this using \cite[Theorem 2.2.21, Theorem 2.2.23]{Desmedt}. Moreover, we need the observation that since $\GG$ is a bicrossed product of $G^{\times N}$ and $\ZZ_N$, the dual quantum group $\whG$ is a bicrossed product of $\ZZ_N$ and $G^{\times N}$ (\cite[Proposition 2.9]{VaesVainerman}).

\begin{proposition}\label{stw19}$ $
\begin{enumerate}
\item Locally compact quantum group $\whG$ is always coamenable, hence $\GG$ is always amenable,
\item $\GG$ is coamenable if and only if $\whG$ is amenable if and only if classical group $G$ is amenable.
\end{enumerate}
\end{proposition}

\subsubsection{Plancherel measure for $G^{\times 2}\bowtie \ZZ_2$}
Assume from now on that $G$ is an uncountable group (in particular $G$ is not discrete). Fix $\mu_{G^{\times 2}}$, a left Haar measure for $G^{\times 2}$. Note that the fact that $G$ is $\sigma$-compact and uncountable implies that the diagonal in $G^{\times 2}$ is of measure $0$. From now on, we will identify $\ZZ_2$ with $\{-1,1\}\subseteq \RR\setminus \{0\}$. As the Haar measure for $\ZZ_2$ we choose the counting measure. In this section we will describe the Plancherel measure of $\GG$ and the spectrum of $\CGDu\simeq\mathrm{C}_0(G^{\times 2})\rtimes \ZZ_2$ as a measure space. Let us start with the following simple lemma:

\begin{lemma}\label{lemat35}
Let $H$ be a finite group and $X$ a Hausdorff, locally compact, $\sigma$-compact topological space with a continuous action of $H$, $\alpha\colon H\curvearrowright X$. Let $\pi\colon X\ni x \mapsto [x]\in X/H$ be the canonical quotient map. Then there exists a Borel map $f\colon X/H\rightarrow X$ such that $\pi\circ f([x])=[x]$ for all $x\in X$ and the set $f(X/H)$ is Borel in $X$.
\end{lemma}

\begin{proof}
Take any $x\in X$. One easily sees (by taking appropriate intersections) that there exists a compact neighbourhood $x\in U_x\subseteq X$ such that $\alpha_{h}(U_x)\cap \alpha_{h'}(U_x)=\emptyset$ for distinct $h,h'\in H$. Since the quotient map $\pi$ is open and continuous, the set $\pi(U_x)$ is a compact neighbourhood of $[x]$. Note also that since $\pi$ is open and $\pi|_{U_x}$ is injective, subspace $\pi(U_x)\subseteq X/H$ is Hausdorff. Define map $f_x\colon \pi(U_x)\ni [y] \mapsto y \in U_x\subseteq X$. This map is well defined. It is also continuous: indeed, it is an inverse to $\pi|_{U_x}\colon U_x\rightarrow \pi(U_x)$ which is continuous. As continuous bijections between compact Hausdorff spaces are homeomorphisms, $f_x$ itself is continuous. \\
 Notice that the quotient space $X/H$ is $\sigma$-compact. Consequently, there exists a countable set $\{x_n\,|\, n\in\NN\}$ such that $X/H=\bigcup_{n\in\NN} \pi(U_{x_n})$. Define $V_1=\pi(U_{x_1})$ and $V_k= \pi(U_{x_k})\setminus (V_1\cup\cdots \cup V_{k-1})$ for $k \ge 2$. These are disjoint Borel sets and their union is the whole $X/H$. Define $f\colon X/H\rightarrow X$ to be such a function that $f(y)=f_{x_k}(y)$ whenever $y\in V_k\,(k\in\NN)$. It is clear that we have $\pi\circ f([y])=[y]$ for all $y\in X$. We are left to check that $f$ is Borel. For any open subset $U\subseteq X$ we have
\[
f^{-1}(U)=\bigcup_{n\in\NN} V_n \cap f^{-1}(U)=
\bigcup_{n\in\NN} V_n \cap f_{x_n}^{-1}(U\cap U_{x_n})
\]
which is a Borel set. We have $f(X/H)=\bigcup_{n\in\NN} f(V_n)$. Since for each $n\in\NN$, the set $f(V_n)$ is Borel in $U_{x_n}$, it is Borel $X$ and consequently $f(X/H)$ is Borel.
\end{proof}

Let $\wh{\ZZ}_2=\{\hat{1},e\}$ be the group dual to $\ZZ_2$, where $\hat{1}(k)=1,e(k)=k\,(k\in\ZZ_2=\{-1,1\})$. We can describe the spectrum of $\CGDu$ using \cite[Theorem 8.39, Proposition 5.4]{Williams}. Indeed, these results tell us that it is homeomorphic to the following space (with the quotient topology)
\[
(G^{\times 2} \times \wh{\ZZ}_2)/_{\sim},
\]
where the relation $\sim$ is given by
\[
((g,g'),\tau)\sim ((h,h'),\pi)\textnormal{ if and only if }
\bigl( \ZZ_2 (g,g')=\ZZ_2 (h,h')\textnormal{ and }
\forall_{k\in (\ZZ_2)_{(g,g')}}\; \tau(k)=\pi(k)\bigr).
\]
In other words $((g,g'),\tau)\sim ((h,h'),\pi)$ if and only if
\[
\{(g,g'),(g',g)\}=\{(h,h'),(h',h)\}\quad\textnormal{ and }
\quad
\forall_{k\in \ZZ_2: k\cdot (g,g')=(g,g')} \; \tau(k)=\pi(k).
\]
Let us now determine equivalence classes of $\sim$. Assume first that $g=g'$. Then $((g,g),\tau)\sim ((h,h'),\pi)$ gives us
\[
\{(g,g)\}=\{(h,h'),(h',h)\}\quad\textnormal{ and }\quad
\forall_{k\in \ZZ_2}\; \tau(k)=\pi(k).
\]
Consequently
\[
[((g,g),\tau)]_{\sim} = \{((g,g),\tau)\}.
\]
Consider now case $g\neq g'$. Relation $((g,g'),\tau)\sim ((h,h'),\pi)$ implies
\[
\{(g,g'),(g',g)\}=\{(h,h'),(h',h)\}\quad\textnormal{ and }\quad
\tau(1)=\pi(1),
\]
 hence
\[
[((g,g'),\tau)]_{\sim}=\{
((g,g'),\pi),((g',g),\pi)\,|\,\pi\in \widehat{\ZZ}_2\},
\]
is a set with $4$ elements. It follows that our space decomposes (as a set) into
\[
(G^{\times 2} \times \widehat{\ZZ}_2)/_{\sim}=
\{ [((g,g),\tau)]_{\sim}\,|\, g\in G,\tau\in \wh{\ZZ}_2\}\cup
\{ [((g,g'),\hat{1})]_{\sim}\,|\, g,g'\in G: g\neq g'\}.
\]
We will now proceed to give a simpler description of this (measurable) space. Let $\Delta\subseteq G^{\times 2}$ be the diagonal, $\Delta=\{(g,g)\,|\, g\in G\}$. Define relation $\sim_{\Delta}$ on $G^{\times 2}\setminus \Delta$ via $[(g,g')]_{\sim_\Delta}=\{(g,g'),(g',g)\}$ for all $(g,g')\in G^{\times 2}\setminus \Delta$. Consider the map
\[
\Phi\colon (G^{\times 2}\times \wh{\ZZ}_2)/_{\sim} \colon 
\begin{cases} [((g,g),\tau)]_{\sim} \mapsto (g,\tau) \\ 
[((g,g'),\hat{1})]_{\sim} \mapsto [(g,g')]_{\sim_\Delta}
\end{cases}\in
(G\times \wh{\ZZ}_2)\sqcup ((G^{\times 2}\setminus \Delta)/_{\sim_{\Delta}}).
\]

Put on $(G^{\times 2}\setminus \Delta)/\sim_{\Delta}$ the quotient topology and the corresponding measurable structure given by Borel sets.

\begin{lemma}
$\Phi$ is a Borel isomorphism, hence $\IrrG$ can be identified as a measurable space with
\[
(G\times \hat{\ZZ}_2)\sqcup ((G^{\times 2}\setminus \Delta)/_{\sim_\Delta}).
\]
\end{lemma}

\begin{proof}
Clearly $\Phi$ is a bijection. First we show that $\Phi$ is measurable. Take any closed subset $E=\{(g,\hat{1}),(h,e)\,|\, g\in E_{\hat{1}},h\in E_{e}\}\subseteq G\times \wh{\ZZ}_2$. It follows that $E_{\hat{1}},E_e$ are closed in $G$. Preimage of $E$ under $\Phi$ is
\[
\{[((g,g),\hat{1})]_{\sim},[((h,h),e)]_{\sim}\,|\,
g\in E_{\hat{1}},h\in E_e\}\subseteq (G^{\times 2}\times \wh{\ZZ}_2)/_{\sim}.
\]
Next, preimage of this set under the quotient map is
\[
\{((g,g),\hat{1}),((h,h),e)\,|\,
g\in E_{\hat{1}},h\in E_e\}
\]
which is closed in $G^{\times 2}\times \wh{\ZZ}_2$ since $E_{\hat{1}},E_e$ are closed. It follows that $\Phi^{-1}(E)$ is closed.\\
Take now closed set $F=\{[(g,g')]_{\sim_\Delta}\,|\,
(g,g')\in \tilde{F}\}$ in $(G^{\times 2}\setminus \Delta)/_{\sim_{\Delta}}$. We can assume that $(g,g')\in \tilde{F}$ if and only if $(g',g)\in \tilde{F}$. As the quotient map is continuous, $\tilde{F}$ is closed in $G^{\times 2}\setminus \Delta$. Preimage of $F$ under $\Phi$ is
\[
\{[((g,g'),\hat{1})]_{\sim}\,|\, (g,g')\in \tilde{F}\}
\subseteq (G^{\times 2}\times \wh{\ZZ}_2)/_{\sim}.
\]
This set is Borel. Indeed, it can be written as a difference
\[\begin{split}
&\{[((g,g'),\hat{1})]_{\sim}\,|\, (g,g')\in \tilde{F}\}
\\=&
\{[((g,g'),\hat{1})]_{\sim},[((h,h),\tau)]_{\sim}
\,|\, (g,g')\in \tilde{F},h\in G,\tau\in \wh{\ZZ}_2\}
\setminus\{[((h,h),\tau)]_{\sim}\,|\,
h\in G,\tau\in\wh{\ZZ}_2\}.
\end{split}\]
The first set is closed, since its preimage under the quotient map is
\[
\{((g,g'),\tau),
((h,h),\tau)\,|\, (g,g')\in\tilde{F},h\in G,\tau\in\wh{\ZZ}_2\}
\]
which is closed in $G^{\times 2}\times \wh{\ZZ}_2$. Second set is also closed: its preimage under the quotient map is $\Delta\times \wh{\ZZ}_2$. It follows that $\Phi$ is measurable.\\
Let us now justify that both $(G^{\times 2}\times\wh{\ZZ}_2)/_{\sim}$ and $(G\times \wh{\ZZ}_2)\sqcup ((G^{\times 2}\setminus \Delta)/_{\sim_{\Delta}})$ are standard Borel spaces. The first one is homeomorphic to a spectrum of a separable \cst-algebra of type I, hence it is a standard Borel space (\cite[Proposition 4.6.1]{DixmierC}). Next, \cite[Example A.9]{KerrLi} tells us that any second countable locally compact Hausdorff space is Polish (hence the corresponding measurable space is standard). It follows that $G\times \wh{\ZZ}_2$ is standard. Furthermore, $(G^{\times 2}\setminus \Delta)/\sim_{\Delta}$ is also second countable, locally compact and Hausdorff. Consequently, $\Phi$ is a measurable bijection between two standard Borel spaces. Any such mapping is a Borel isomorphism, i.e. the inverse map is also measurable (see B 22. in \cite{DixmierC}).
\end{proof}

Choose $p\in\RR$ and define a measure $\mu_p$ on $\IrrG= (G\times \wh{\ZZ}_2)\sqcup ((G^{\times 2}\setminus \Delta)/_{\sim_\Delta})$ via
\[
\mu_p(G\times \wh{\ZZ}_2)=0,\quad
\mu_p(E)=\tfrac{1}{2}(\delta_{G^{\times 2}}^{p}\,\mu_{G^{\times 2}})(p_{\Delta}^{-1}(E))\quad(E\subseteq (G^{\times 2}\setminus \Delta)/_{\sim_\Delta}),
\]
where $p_\Delta\colon G^{\times 2}\setminus \Delta \rightarrow (G^{\times 2}\setminus \Delta)/_{\sim_\Delta} $ is the canonical projection. We will further show that $\mu_p$ is the Plancherel measure of $\GG$. Let us comment on why we choose to work with this family of measures. Our reasoning in this case is similar to the case of quantum groups dual to classical (Subsection \ref{secdclass}); properties of the functions $\delta^p_{G^{\times 2}}$ simplifies our calculations and on the other hand we think it is worthwile to consider a family of Plancherel measures -- we will calculate how functions $\varpi$ depends on $p$ and in Proposition \ref{stw21} we will get a more general result concerning amenability of $G$. \\
Since the set $G\times \wh{\ZZ}_2$ is of $\mu_p$-measure zero, from now on we will identify $\IrrG$ with $(G^{\times 2}\setminus \Delta)/_{\sim_\Delta}$. The following result is an immediate consequence.

\begin{proposition}
The measure space $(\IrrG,\mu_p)$ can be identified with $((G^{\times 2}\setminus \Delta)/_{\sim_\Delta},\mu_p)$.
\end{proposition}

Let us now describe representations of $\GG$ corresponding to the points in $(G^{\times 2}\setminus \Delta)/_{\sim_\Delta}$. Thanks to Lemma \ref{lemat35} we can choose representatives $H([(g,g')]_{\sim_\Delta})\in[(g,g')]_{\sim_\Delta}\in (G^{\times 2}\setminus \Delta)/_{\sim_\Delta}$ in such a way that the map
\[
H\colon (G^{\times 2}\setminus \Delta)/_{\sim_\Delta}\ni 
[(g,g')]_{\sim_\Delta} \mapsto H([(g,g')]_{\sim_\Delta})\in G^{\times 2} \setminus \Delta
\]
is measurable. Let $H^{c}$ be the map $H$ composed with the flip $(g,g')\mapsto (g',g)$. To ease the notation, we will write
\[
H[g,g']=H([(g,g')]_{\sim_\Delta}),\quad
H^{c}[g,g']=H^{c}([(g,g')]_{\sim_\Delta})\quad
(
[(g,g')]_{\sim_\Delta}\in (G^{\times 2}\setminus \Delta)/_{\sim_\Delta}).
\]
Take $[(g,g')]_{\sim_\Delta}\in (G^{\times 2}\setminus \Delta)/_{\sim_\Delta}$. This point corresponds to $[((g,g'),\hat{1})]_{\sim}$ in $(G^{\times 2}\times \wh{\ZZ}_2)/_{\sim}$ and the corresponding representation of $\mathrm{C}_0(G^{\times 2})\rtimes \ZZ_2$ is $\Pi_{(g,g')}\rtimes U$, where $\Pi_{(g,g')}, U$ are representations of $\mathrm{C}_0(G^{\times 2}),\ZZ_2$ on $\mc{V}^{(g,g')}=\CC^2$ given by
\begin{equation}\label{eq39}
\Pi_{(g,g')}(f)=
\begin{bmatrix}
f(H[g,g']) & 0 \\
0 & f(H^{c}[g,g'])
\end{bmatrix},\quad
U(1)=
\begin{bmatrix}
1 & 0 \\
0 & 1
\end{bmatrix},\quad
U(-1)=
\begin{bmatrix}
0 & 1 \\
1 & 0
\end{bmatrix}
\end{equation}
for all $f\in \mathrm{C}_0(G^{\times 2})$ (see \cite[Proposition 5.4]{Williams} and the discussion before).\\
Choose measurable field of Hilbert spaces on $\IrrG$ to be $\msf{H}_{[(g,g')]_{\sim_\Delta}}=\mc{V}^{(g,g')}=\CC^2$, where arbitrary vector field $(\xi_{[(g,g')]_{\sim_\Delta}})_{
[(g,g')]_{\sim_\Delta}\in\IrrG}$ is said to be measurable if and only if the function $\IrrG\ni [(g,g')]_{\sim_\Delta} \mapsto
\xi_{[(g,g')]_{\sim_\Delta}}\in \CC^2$ is measurable. As a measurable field of representations on $\IrrG$ take the field $(\Pi_{(g,g')}\rtimes U )_{[(g,g')]_{\sim_\Delta}\in \IrrG}$. It is measurable since we have chosen representatives in such a way that the function $[(g,g')]_{\sim_\Delta}\mapsto H[g,g']$ is measurable. Define $D_{[(g,g')]_{\sim_\Delta}}=\delta_{G^{\times 2}}(g,g')^{\frac{p}{2}}\I_{\mc{V}^{(g,g')}}$ for all $[(g,g')]_{\sim_\Delta}$. We will now show that the above objects satisfy conditions of Theorem \ref{PlancherelL}.\\
Consider a linear map
\begin{equation}\label{eq43}\begin{split}
\mc{Q}_L &\colon \LL^2(G^{\times 2})\otimes \LL^2(\ZZ_2)\rightarrow \int_{\IrrG}^{\oplus} \HS(\mc{V}^{(g,g')}) \md\mu_p([(g,g')]_{\sim_\Delta})
\\
&\colon f\otimes \nu \mapsto
\int_{\IrrG}^{\oplus}
\delta_{G^{\times 2}}(g,g')^{-\frac{p}{2}}
\begin{bmatrix}
\nu(1) & \nu(-1) \\
\nu(-1) & \nu(1)
\end{bmatrix}
\,
\Pi_{(g,g')}(f) \md \mu_p([(g,g')]_{\sim_\Delta}),
\end{split}\end{equation}
first defined for $f\in \mathrm{C}_0(G^{\times 2})\cap \LL^2(G^{\times 2})$. We need to check that this map is well defined. Take $f\in \mathrm{C}_0(G^{\times 2})\cap \LL^2(G^{\times 2})$ and $\nu\in \LL^2(\ZZ_2)$. We have
\[\begin{split}
&\quad\;\int_{\IrrG} \bigl\|
\delta_{G^{\times 2}}(g,g')^{-\frac{p}{2}}\begin{bmatrix}
\nu(1) & \nu(-1) \\
\nu(-1) & \nu(1)
\end{bmatrix}
\,
\Pi_{(g,g')}(f) \bigr\|^{2}_{
\HS(\mc{V}^{(g,g')})} \md\mu_p([(g,g')]_{\sim_\Delta})\\
&=
\int_{\IrrG}
\delta_{G^{\times 2}}(g,g')^{-p} \bigl\|
\begin{bmatrix}
f(H[g,g'])\nu(1) & f(H^{c}[g,g'])\nu(-1) \\
f(H[g,g'])\nu(-1) & f(H^{c}[g,g'])\nu(1)
\end{bmatrix}\bigr\|^{2}_{
\HS(\mc{V}^{(g,g')})} \md\mu_p([(g,g')]_{\sim_\Delta})\\
&=
\int_{\IrrG}
\delta_{G^{\times 2}}(g,g')^{-p}
(|f(H[g,g'])|^2 + |f(H^{c}[g,g'])|^2)(|\nu(1)|^2 +
|\nu(-1)|^2)\md\mu_p([(g,g')]_{\sim_\Delta})\\
&=
\tfrac{1}{2}\|\nu\|^2 \int_{G^{\times 2}\setminus \Delta}
(|f(g,g')|^2 + |f(g',g)|^2)
\md\mu_{G^{\times 2}}(g,g')\\
&=
\|\nu\|^2 \int_{G^{\times 2}}
|f(g,g')|^2 
\md\mu_{G^{\times 2}}(g,g')=
\|\nu\|^2 \|f \|^2<+\infty,
\end{split}\]
hence $\mc{Q}_L$ is well defined (on a dense subspace). A similar calculation shows that $\mc{Q}_L$ is isometric, hence it extends to an isometry on the whole $\LL^2(G^{\times 2})\otimes \LL^2(\ZZ_2)$. In the above calculations we have used the observation that $\mu_{G^{\times 2}}(\Delta)=0$. Let us now justify that the isometry $\mc{Q}_L$ is also surjective. Denote by $\nu_{\pm}$ the Dirac delta functions in $1,-1\in\ZZ_2$. Recall that we have chosen function $H$ is such a way that the sets
\begin{equation}\label{eq33}
H((G^{\times 2}\setminus \Delta)/_{\sim_\Delta}),\quad
H^{c}((G^{\times 2}\setminus \Delta)/_{\sim_\Delta})
\end{equation}
are Borel in $G^{\times 2}\setminus \Delta$. It follows that a product of a continuous function in $\mathrm{C}_c(G^{\times 2}\setminus \Delta)$ and $\chi_{H((G^{\times 2}\setminus\Delta)_{\sim_\Delta})}$ is Borel. Take $f,f'\in \mathrm{C}_c(G^{\times 2}\setminus \Delta)$. One easily checks (by approximating with continuous functions with compact support) that we have
\[\begin{split}
&\quad\;
\mc{Q}_L (f\chi_{H((G^{\times 2}\setminus \Delta)/_{\sim_\Delta})}\otimes \nu_{+}+ f'\chi_{H^{c}((G^{\times 2}\setminus \Delta)/_{\sim_\Delta})}\otimes\nu_+)\\
&=
\int_{\IrrG}^{\oplus}
\delta_{G^{\times 2}}(g,g')^{-\frac{p}{2}}
\begin{bmatrix}
f(H[g,g']) & 0 \\
0 & f'(H^{c}[g,g'])
\end{bmatrix} \md\mu_p([(g,g')]_{\sim_\Delta}).
\end{split}\]
 In a similar fashion one checks that
\[\begin{split}
&\quad\;
\mc{Q}_L (f\chi_{H((G^{\times 2}\setminus \Delta)/_{\sim_\Delta})}\otimes \nu_{-}+ f'\chi_{H^{c}((G^{\times 2}\setminus \Delta)/_{\sim_\Delta})}\otimes\nu_-)\\
&=
\int_{\IrrG}^{\oplus}
\delta_{G^{\times 2}}(g,g')^{-\frac{p}{2}}\begin{bmatrix}
0 &f'(H^{c}[g,g'])  \\
 f(H[g,g']) & 0
\end{bmatrix} \md\mu_p([(g,g')]_{\sim_\Delta}).
\end{split}\]
As $f,f'$ are arbitrary, we arrive at the conclusion that $\mc{Q}_L$ is a unitary operator. We have
\[
\mrW^{\whG}=((\beta\otimes \id)\mrW^{G^{\times 2}}\otimes \I)\;
(\I\otimes (\id\otimes\alpha)\mrW^{\wh{\ZZ}_2})=
((\beta\otimes \id)\mrW^{G^{\times 2}}\otimes \I)\;
(\mrW^{\wh{\ZZ}_2})_{24}
\]
and as usual $\mrW^{\GG}=\Sigma_{(12) (34)} (\mrW^{\whG})^* \Sigma_{(12) (34)}$. Direct calculation gives us
\[
\mrW^{\GG}= (\mrW^{\ZZ_2})_{24}\; ((\id\otimes\beta) \mrW^{\wh{G^{\times 2}}})_{134},
\]
consequently for any $\omega_1\in \LL(G^{\times 2})_*$ and $\omega_2\in \LL^1(\ZZ_2)$
\[
(\omega_1\otimes\omega_2\otimes\id\otimes\id)\mrW^{\GG}= \bigl(\I\otimes (\omega_2\otimes\id)\mrW^{\ZZ_2}\bigr)\; \beta((\omega_1\otimes\id) \mrW^{\wh{G^{\times 2}}})\in \Linfd.
\]
Due to \cite[Proposition 2.9]{VaesVainerman} we know that 
\[
\lin\bigl\{\bigl(\I\otimes (\omega_2\otimes \id)\mrW^{\ZZ_2}\bigr)\,\beta(f)\,|\, \omega_2\in \LL^1(\ZZ_2),\, f\in\mf{N}_{\vp_{G^{\times 2}}}
\bigr\}
\]
forms a $\ssots\times\|\cdot\|$ core for $\Lambda_{\hvp}$ and
\[
\Lambda_{\hvp}\bigl(
\bigl(\I\otimes (\omega_2\otimes \id)\mrW^{\ZZ_2}\bigr)\,\beta(f)\bigr)=
\Lambda_{\vp_{G^{\times 2}}}(f)\otimes \omega_2
\]
for all $f\in\mf{N}_{\vp_{G^{\times 2}}},\omega_2\in \LL^1(\ZZ_2)$. One easily sees that if we take only $f$ of the form $f=(\omega_1\otimes\id)\mrW^{\wh{G^{\times 2}}}$, where $\omega_1\in \mc{I}_{\wh{G^{\times 2}}}\subseteq \LL(G^{\times 2})_*$, then we will still get a core. Take $\omega_1\in \mc{I}_{\widehat{G^{\times 2}}},\omega_2\in\LL^1(\ZZ_2)$. By the definition of $\mc{Q}_L$ we have
\[\begin{split}
&\quad\; \mc{Q}_L\Lambda_{\hvp}\bigl(
(\omega_1\otimes\omega_2\otimes\id\otimes\id)\mrW^{\GG}\bigr)= \mc{Q}_L\Lambda_{\hvp}\bigl((\I\otimes (\omega_2\otimes\id)\mrW^{\ZZ_2})\; \beta((\omega_1\otimes\id) \mrW^{\wh{G^{\times 2}}})
\bigr)\\
&=\mc{Q}_L ( \Lambda_{\vp_{G^{\times 2}}}((\omega_1\otimes\id)\mrW^{\wh{G^{\times 2}}})\otimes
\omega_2)
\\
&=
\int_{\IrrG}^{\oplus}
\delta_{G^{\times 2}}(g,g')^{-\frac{p}{2}}
\begin{bmatrix}
\omega_2(1) & \omega_2(-1) \\
\omega_2(-1) & \omega_2(1)
\end{bmatrix}
\,
\Pi_{(g,g')}((\omega_1\otimes\id)\mrW^{\wh{G^{\times 2}}}) \md \mu_p([(g,g')]_{\sim_\Delta})\\
&=
\int_{\IrrG}^{\oplus} (\Pi_{(g,g')}\rtimes U) 
\bigl( (\omega_1\otimes\omega_2\otimes\id\otimes\id)
\mrW^{\GG}\bigr)D_{[(g,g')]_{\sim_\Delta}}^{-1}\md\mu_p([(g,g')]_{\sim_\Delta}).
\end{split}\]
It is enough to check the equality in point $7.2)$ of Theorem \ref{PlancherelL} for elements as above (see Lemma \ref{lemat11}).\\
Next thing we need to check, are the commutation relations for $\mc{Q}_L$ (point $3)$ of Theorem \ref{PlancherelL}). Take any $\omega_1\in \LL(G^{\times 2})_*,\omega_2\in\LL^1(\ZZ_2)$ and $f\in\mathrm{C}_0(G^{\times 2})\cap\LL^2(G^{\times 2}),\nu\in\LL^2(\ZZ_2)$. Define $\wh{\omega}_1=(\omega_1\otimes\id)\mrW^{\wh{G^{\times 2}}}\in \mathrm{C}_0(G^{\times 2})$ and denote by $\lambda^{\ZZ_2}$ the left regular representation of $\ZZ_2$. Recall that $\nu_{\pm}$ are the Dirac delta functions in $\pm 1\in\ZZ_2$. We have
\[\begin{split}
&\quad\;
\mc{Q}_L (\omega_1\otimes\omega_2\otimes\id\otimes\id)\mrW^{\GG}\;
(\Lambda_{\vp_{G^{\times 2}}}(f)\otimes \nu)=
\mc{Q}_L\,
(\I\otimes (\omega_2\otimes\id)\mrW^{\ZZ_2})
\beta(\hat{\omega}_1)\,(\Lambda_{\vp_{G^{\times 2}}}(f)\otimes\nu)\\
&=
\mc{Q}_L\,
(\I\otimes (\omega_2(1) \lambda^{\ZZ_2}_1+\omega_2(-1)
\lambda^{\ZZ_2}_{-1}))
(\wh{\omega}_1\otimes \nu_+ + \beta_{-1}(\wh{\omega}_1)\otimes\nu_-)
\,(\Lambda_{\vp_{G^{\times 2}}}(f)\otimes\nu)\\
&=
\mc{Q}_L\bigl(\nu(1)
\Lambda_{\vp_{G^{\times 2}}}(\wh{\omega}_1 f)\otimes
\omega_2+\nu(-1) \Lambda_{\vp_{G^{\times 2}}}(\beta_{-1}(\wh{\omega}_1)f)\otimes \lambda^{\ZZ_2}_{-1}(\omega_2)\bigr)\\
&=
\int_{\IrrG}^{\oplus}\delta_{G^{\times 2}}(g,g')^{-\frac{p}{2}}
\bigl(
\nu(1) \begin{bmatrix}
\omega_2(1) & \omega_2(-1)\\
\omega_2(-1) & \omega_2(1)
\end{bmatrix}
\Pi_{(g,g')}(\wh{\omega}_1 f)+ \\
&\quad\quad\quad\quad\quad\quad
\quad+
\nu(-1) \begin{bmatrix}
\omega_2(-1) & \omega_2(1)\\
\omega_2(1) & \omega_2(-1)
\end{bmatrix}
\Pi_{(g,g')}(\beta_{-1}(\wh{\omega}_1) f) \bigr)
\md\mu_p([(g,g')]_{\sim_\Delta})\\
&=
\int_{\IrrG}^{\oplus}
\delta_{G^{\times 2}}(g,g')^{-\frac{p}{2}}\bigl(
\nu(1) \begin{bmatrix}
\omega_2(1) \wh{\omega}_1(H[g,g'])& \omega_2(-1)
\wh{\omega}_1(H^{c}[g,g'])\\
\omega_2(-1) \wh{\omega}_1(H[g,g'])& \omega_2(1)
\wh{\omega}_1(H^{c}[g,g'])
\end{bmatrix}
\Pi_{(g,g')}( f)+ \\
&\quad\quad\quad\quad\quad\quad
\quad+
\nu(-1) \begin{bmatrix}
\omega_2(-1) \wh{\omega}_1(H^{c}[g,g'])& \omega_2(1)
\wh{\omega}_1(H[g,g'])\\
\omega_2(1) \wh{\omega}_1(H^{c}[g,g'])& \omega_2(-1)
\wh{\omega}_1(H[g,g'])
\end{bmatrix}
\Pi_{(g,g')}( f) \bigr)
\md\mu_p([(g,g')]_{\sim_\Delta}).
\end{split}\]
On the other hand we have
\[\begin{split}
&\quad\;
\bigl(\int_{\IrrG}^{\oplus} (\omega_1\otimes\omega_2\otimes\id)(\Pi_{(g,g')}\rtimes U) \otimes\I_{\ov{\mc{V}^{(g,g')}}}\md\mu_p([(g,g')]_{\sim_\Delta})\bigr)
\mc{Q}_L (\Lambda_{\vp_{G^{\times 2}}}(f)\otimes \nu)\\
&=
\int_{\IrrG}^{\oplus}
\delta_{G^{\times 2}}(g,g')^{-\frac{p}{2}} U ((\omega_2\otimes\id)W^{\ZZ_2})\;
\Pi_{(g,g')} (\wh{\omega}_1)\;
\begin{bmatrix}
\nu(1) & \nu(-1) \\ 
\nu(-1) & \nu(1)
\end{bmatrix}
\Pi_{(g,g')} (f)
\md\mu_p([(g,g')]_{\sim_{\Delta}})\\
&=
\int_{\IrrG}^{\oplus} 
\delta_{G^{\times 2}}(g,g')^{-\frac{p}{2}}
\begin{bmatrix}
\omega_2(1) & \omega_2(-1) \\
\omega_2(-1) & \omega_2(1)
\end{bmatrix}
\begin{bmatrix}
\nu(1) 
\hat{\omega}_1(H[g,g'])& 
\nu(-1) \hat{\omega}_1(H[g,g'])\\ 
\nu(-1) \hat{\omega}_1(H^{c}[g,g'])\;&
 \nu(1)
\hat{\omega}_1(H^{c}[g,g'])\;
\end{bmatrix}\\
&\quad\quad\quad\quad\quad\quad
\quad\quad\quad\quad\quad\quad\quad\quad\quad
\quad\quad\quad\quad\quad\quad\quad\quad\quad\quad\quad\quad
\Pi_{(g,g')} (f)
\md\mu_p([(g,g')]_{\sim_{\Delta}})
\end{split}\]
which gives us the desired equality
\[
\mc{Q}_L (\omega_1\otimes\omega_2\otimes\id\otimes\id)\mrW^{\GG}=
\bigl(\int_{\IrrG}^{\oplus} (\omega_1\otimes\omega_2\otimes\id)(\Pi_{(g,g')}\rtimes U) \otimes\I_{\ov{\mc{V}^{(g,g')}}}\md\mu_p([(g,g')]_{\sim_\Delta})\bigr)
\mc{Q}_L.
\]
Let us now check that the second commutation rule also holds. We have \\$\chi(\mrV^{\GG})=(\hat{J}\otimes\hat{J}) {\mrW^{\GG}}^* (\hat{J}\otimes\hat{J})$ and consequently
\[
(\omega\otimes\id) \chi(\mrV^{\GG})=\hat{J} \hat{R}( (\omega\otimes\id) \mrW^{\GG})^* \hat{J}=\hat{J} (((\omega\circ R)\otimes\id)\mrW^{\GG})^* \hat{J}=
\hat{J} \lambda(\omega\circ R)^* \hat{J}
\]
for any $\omega\in \LL^1(\GG)$. Since the Haar integrals on $\whG$ are tracial, we have $\hat{J} x^* \hat{J} \Lhvp(y)=\Lhvp(yx)$ for all $x\in \Linfd,y\in \mf{N}_{\hvp}$. Take $\omega_1,\nu_1\in \mc{I}_{\wh{G^{\times 2}}},\omega_2,\nu_2\in \LL^1(\ZZ_2)$ and set $\omega=\omega_1\otimes\omega_2$. Using the above remarks, we arrive at
\[
\begin{split}
&\quad\;
\mc{Q}_L \;((\omega\circ R)\otimes\id)\chi(\mrV^{\GG})\;
\Lhvp( (\nu_1\otimes\nu_2\otimes\id\otimes\id)\mrW^{\GG})\\
&=
\mc{Q}_L
\Lhvp\bigl(
(\nu_1\otimes\nu_2\otimes\id\otimes\id)\mrW^{\GG}
(\omega\otimes\id)W^{\GG}
\bigr)\\
&=
\mc{Q}_L (\nu_1\otimes\nu_2\otimes\id\otimes\id)\mrW^{\GG}
\Lhvp((\omega\otimes\id)\mrW^{\GG})\\
&=
\mc{Q}_L (\nu_1\otimes\nu_2\otimes\id\otimes\id)\mrW^{\GG}
(
\Lambda_{\vp_{G^{\times 2}}}( \hat{\omega}_1) \otimes \omega_2
)=\star,
\end{split}
\]
and using the already derived first commutation rule we can further write
\[
\begin{split}
\star&=
\bigl(\int_{\IrrG}^{\oplus} (\nu_1\otimes\nu_2\otimes\id\otimes\id)
(\Pi_{(g,g')}\rtimes U) \otimes\I_{\ov{\mc{V}^{(g,g')}}} \md\mu_p([(g,g')]_{\sim_\Delta})\bigr)\mc{Q}_L
(
\Lambda_{\vp_{G^{\times 2}}}(\hat{\omega}_1) \otimes \omega_2
)\\
&=
\int_{\IrrG}^{\oplus}
\delta_{G^{\times 2}}(g,g')^{-\frac{p}{2}}
U ((\nu_2\otimes\id)\mrW^{\ZZ_2})\,
\Pi_{(g,g')} (\hat{\nu}_1)
\begin{bmatrix}
 \omega_2(1) & \omega_2(-1) \\
 \omega_2(-1) & \omega_2(1)
\end{bmatrix} 
\Pi_{(g,g')} (\hat{\omega}_1) \md\mu_p([(g,g')]_{\sim_{\Delta}})\\
&=
\bigl(\int_{\IrrG}^{\oplus} \I_{\mc{V}^{(g,g')}}\otimes \jmath_{\mc{V}^{(g,g')}}\bigl(
\begin{bmatrix}
 \omega_2(1) & \omega_2(-1) \\
 \omega_2(-1) & \omega_2(1)
\end{bmatrix} 
\Pi_{(g,g')} (\hat{\omega}_1)\bigr) \md\mu_p([(g,g')]_{\sim_{\Delta}})\bigr)\\
&\quad\quad\quad\quad\quad\quad
\quad\quad\quad\quad\quad\quad\quad\quad\quad
\quad\quad\quad\quad\quad\quad\quad
\quad\quad\quad\quad\quad\quad\mc{Q}_L (
\Lambda_{\vp_{G^{\times 2}}}( \hat{\nu}_1) \otimes \nu_2
)\\
&=\bigl(\int_{\IrrG}^{\oplus} \I_{\mc{V}^{(g,g')}}\otimes 
(\Pi_{(g,g')}\rtimes U)^{c}
(\lambda^u(\omega\circ R)) \md\mu_p([(g,g')]_{\sim_{\Delta}})\bigr)
\mc{Q}_L \Lhvp((\nu_1\otimes\nu_2\otimes\id\otimes\id)\mrW^{\GG}).
\end{split}
\]
This gives us
\[
\bigl(\int_{\IrrG}^{\oplus} \I_{\mc{V}^{(g,g')}}\otimes 
(\Pi_{(g,g')}\rtimes U)^{c}
(\lambda^u(\omega\circ R)) \md\mu_p([(g,g')]_{\sim_{\Delta}})\bigr)\mc{Q}_L
=
\mc{Q}_L \;((\omega\circ R)\otimes\id)\chi(\mrV^{\GG})
\]
and density argument ends the proof of the commutation relations. Let us now justify that $\mc{Q}_L$ transforms $\Linfd$ into $\int_{\IrrG}^{\oplus} \B(\mc{V}^{(g,g')})\otimes\I_{\ov{\mc{V}^{(g,g')}}} \md\mu_p([(g,g')]_{\sim_\Delta})$. The inclusion
\[
\mc{Q}_L \Linfd \mc{Q}_L\subseteq
\int_{\IrrG}^{\oplus} \B(\mc{V}^{(g,g')})\otimes\I_{\ov{\mc{V}^{(g,g')}}} \md\mu_p([(g,g')]_{\sim_\Delta})
\]
is clear thanks to the first commutation relation. Now, take any operator of the form
\[
\int_{\IrrG}^{\oplus}
\begin{bmatrix}
f(H[g,g']) & 0 \\
0 & 0 
\end{bmatrix}
\otimes \I_{\ov{\mc{V}^{(g,g')}}}\md\mu_p([(g,g')]_{\sim_\Delta}),
\]
where $f\in \LL^{\infty}(G^{\times 2})$ is a function satisfying $f=f\chi_{H((G^{\times 2}\setminus \Delta)/_{\sim_\Delta})}$. Take a vector
\[
\zeta=\int_{\IrrG}^{\oplus} 
\begin{bmatrix}
\xi^1_{H[g,g']} \\
\xi^2_{H[g,g']}
\end{bmatrix}
\otimes\eta_{H[g,g']}\md\mu_p([(g,g')]_{\sim_\Delta})\in
\int_{\IrrG}^{\oplus} \HS(\mc{V}^{(g,g')})\md\mu_p([(g,g')]_{\sim_\Delta})
\]
such that the function $G^{\times 2}\ni (g,g')\mapsto \chi_{H((G^{\times 2}\setminus\Delta)/_{\sim_\Delta})}(g,g')\|\eta_{(g,g')}\|\in \RR$ is bounded with compact support and a net $(f_i)_{i\in I}$ in $\mathrm{C}_c(G^{\times 2})\subseteq \LL^{\infty}(G^{\times 2})$ which converges to $f$ in $\swot$. Thanks to the Kaplansky density theorem, we can assume that the net $(f_i)_{i\in I}$ is bounded in norm. We have
\[\begin{split}
&\quad\;|\bigl\langle
\zeta
\big|
\bigl(
\int_{\IrrG}^{\oplus} 
\begin{bmatrix}
f(H[g,g']) & 0 \\
0 & 0
\end{bmatrix}
\otimes \I_{\ov{\mc{V}^{(g,g')}}}
\md\mu_p([(g,g')]_{\sim_\Delta})-
\mc{Q}_L (f_i\otimes\nu_+\otimes\id\otimes\id)\mrW^{\GG}\mc{Q}_L^*
\bigr)\zeta\bigr\rangle|\\
&=
|\bigl\langle\zeta
\big|
\int_{\IrrG}^{\oplus} 
\begin{bmatrix}
f(H[g,g']) -f_i(H[g,g'])& 0 \\
0 & -f_i(H^{c}[g,g'])
\end{bmatrix}
\otimes \I_{\ov{\mc{V}^{(g,g')}}}
\md\mu_p([(g,g')]_{\sim_\Delta})
\zeta\bigr\rangle|\\
&\le
\int_{\IrrG} 
|\bigl\langle 
\begin{bmatrix}
\xi^1_{H[g,g']}\\
\xi^2_{H[g,g']}
\end{bmatrix}
 \big|
\begin{bmatrix}
f(H[g,g']) -f_i(H[g,g'])& 0 \\
0 & -f_i(H^{c}[g,g'])
\end{bmatrix}
\begin{bmatrix}
\xi^1_{H[g,g']}\\
\xi^2_{H[g,g']}
\end{bmatrix}\bigr\rangle|
\|\eta_{H[g,g']}\|^2\md\mu_p([(g,g')]_{\sim_\Delta})\\
&=
\int_{H((G^{\times 2}\setminus \Delta)/_{\sim_\Delta})} 
\delta^{p}_{G^{\times 2}}(g,g')\bigl(|\xi^1_{(g,g')}|^2 |f(g,g')-f_i(g,g')|+
 |\xi^2_{(g,g')}|^2 |f_i(g',g)|
\bigr)\|\eta_{(g,g')}\|^2\md \mu_{G^{\times 2}}(g,g')\\
&\quad\;\xrightarrow[i\in I]{}0,
\end{split}\]
where the convergence follows from the observation that both functions 
\[
G^{\times 2}
\ni(g,g')\mapsto 
\chi_{H((G^{\times 2}\setminus\Delta)/_{\sim_\Delta})}(g,g')
\delta_{G^{\times 2}}^{p}(g,g')|\xi^k_{(g,g')}|^2 \|\eta_{(g,g')}\|^2\in \RR \quad(k\in\{1,2\})
\]
are in $\LL^1(G^{\times 2})$. Since $(f_i)_{i\in I}$ is bounded in norm, a standard approximation argument implies that
\[
\mc{Q}_L (f_i\otimes\nu_+\otimes\id\otimes\id)\mrW^{\GG}\mc{Q}_L^*
\xrightarrow[i\in I]{\swot}
\int_{\IrrG}^{\oplus} \begin{bmatrix}
f(g,g') & 0 \\
0 & 0 
\end{bmatrix}
\md\mu_p([(g,g')]_{\sim_\Delta}),
\]
therefore
\[
\int_{\IrrG}^{\oplus} \begin{bmatrix}
f(g,g') & 0 \\
0 & 0 
\end{bmatrix}
\md\mu_p([(g,g')]_{\sim_\Delta})\in
\mc{Q}_L \Linfd \mc{Q}_L^*.
\]
Analogous arguments show that we have $\mc{Q}_L \Linfd \mc{Q}_L^*=\int_{\IrrG}^{\oplus}\B(\mc{V}^{(g,g')})\otimes
\I_{\ov{\mc{V}^{(g,g')}}}\md\mu_p([(g,g')]_{\sim_\Delta})$. Consequently, after taking the commutant, we arrive at
\[\begin{split}
\mc{Q}_L \Linfd' \mc{Q}_L^*&=
\bigl(
\int_{\IrrG}^{\oplus}\B(\mc{V}^{(g,g')})\otimes
\I_{\ov{\mc{V}^{(g,g')}}}\md\mu_p([(g,g')]_{\sim_\Delta})\bigr)'\\
&=
\int_{\IrrG}^{\oplus}\I_{\mc{V}^{(g,g')}}\otimes
\B(\ov{\mc{V}^{(g,g')}})
\md\mu_p([(g,g')]_{\sim_\Delta})
\end{split}\]
and
\[\begin{split}
&\quad\;
\mc{Q}_L (\Linfd\cap\Linfd')\mc{Q}_L^*=
(\mc{Q}_L (\Linfd\vee\Linfd')\mc{Q}_L^*)'\\
&=
\bigl(
\int_{\IrrG}^{\oplus}\B(\mc{V}^{(g,g')})\otimes
\I_{\ov{\mc{V}^{(g,g')}}}\md\mu_p([(g,g')]_{\sim_\Delta})\vee
\int_{\IrrG}^{\oplus}\I_{\mc{V}^{(g,g')}}\otimes
\B(\ov{\mc{V}^{(g,g')}})
\md\mu_p([(g,g')]_{\sim_\Delta})\bigr)'\\
&=
\Dec( \int_{\IrrG}^{\oplus}
\HS(\mc{V}^{(g,g')})\md\mu_p([(g,g')]_{\sim_\Delta}))'=
\Diag( \int_{\IrrG}^{\oplus}
\HS(\mc{V}^{(g,g')})\md\mu_p([(g,g')]_{\sim_\Delta}))
\end{split}\]
(see \cite[Theorem 4]{DixmiervNA}). We have proved the following:

\begin{proposition}
For each $p\in\RR$, measure $\mu_p$ described above is a Plancherel measure of $\GG$ and the unitary operator $\mc{Q}_L$ defined in \eqref{eq43} is the unitary operator given by Theorem \ref{PlancherelL}. Moreover, we have $D_{[(g,g')]_{\sim_\Delta}}=\delta_{G^{\times 2}}(g,g')^{\frac{p}{2}}\I_{\mc{V}^{(g,g')}}$ for almost all $[(g,g')]_{\sim_\Delta}\in\IrrG$.
\end{proposition}

We will now describe the right version of the above result. Define a unitary operator $\mc{Q}_R$ to be\footnote{We suspect that the relation $\mc{Q}_R=\mc{Q}_L\circ J\hat{J}$ holds also in a more general situation. We will adress the question whether this is indeed the case in our next paper.}
\[
\mc{Q}_R=\mc{Q}_L\circ  J \hat{J}\colon \LdG=\LL^2(G^{\times 2}\times \ZZ_2)
\rightarrow
\int_{\IrrG}^{\oplus} \HS( \mc{V}^{(g,g')}) \md\mu_p([(g,g')]_{\sim_\Delta} )
\]
and set $E_{[(g,g')]_{\sim_\Delta}}= \delta_{G^{\times 2}}(g,g')^{\frac{p-1}{2}} \I_{\mc{V}^{(g,g')}}$ for all $[(g,g')]_{\sim_\Delta}\in \IrrG$. Our aim is to show that these objects satisfy the right version of the Desmedt theorem.\\
Since we have
\[
\hat{J} \Lambda_{\vp}(x)=\Lambda_{\psi}(R(x)^*)\quad(x\in\mf{N}_{\vp}),
\]
it follows by duality that
\[
\Lambda_{\hpsi}(x)=J\Lambda_{\hvp}(\hat{R}(x)^*)\quad(x\in \mf{N}_{\hpsi}).
\]
In order to unravel the above equation, we need to describe the element $\hat{R}(x)^*=JxJ$ for nice $x$. Take $\omega_2\in \LL^1(\ZZ_2),f\in \LL^{\infty}(G^{\times 2})$ and denote $\hat{\omega}_2=(\omega_2\otimes\id)\mrW^{\ZZ_2}$. For any function $\xi\in\LL^2(G^{\times 2}\times \ZZ_2)$ and $((g,g'),k)\in G^{\times 2}\times \ZZ_2$ we have
\[\begin{split}
&\quad\;
(\hat{R} \bigl( (\I\otimes\hat{\omega}_2)\, \beta(f)\bigr)^* \xi)((g,g'),k)=
(J (\I\otimes\hat{\omega}_2) \,\beta(f) J\xi) ((g,g'),k)\\
&=
\delta_{G^{\times 2}}(g,g')^{\frac{1}{2}}
\;
\ov{
((\I\otimes\hat{\omega}_2) \,\beta(f) J \xi) ((g^{-1},{g'}^{-1}),k)
}\\
&=
\delta_{G^{\times 2}}(g,g')^{\frac{1}{2}}
\;
\ov{
\omega_2(1) \,(\beta(f) J \xi) ((g^{-1},{g'}^{-1}),k)+
\omega_2(-1) \,(\beta(f) J \xi) ((g^{-1},{g'}^{-1}),-k)
}\\
&=
\ov{
\omega_2(1) \,f(\beta_{k}(g^{-1},{g'}^{-1})) \ov{\xi ((g,{g'}),k)}+
\omega_2(-1) \,f(\beta_{-k}(g^{-1},{g'}^{-1})) 
\ov{\xi ((g,{g'}),-k)}
}\\
&=
((\I\otimes \hat{\ov{\omega}_2})\; \beta(
\ov{R_{G^{\times 2}}(f)}) \xi)((g,g'),k),
\end{split}\]
hence
\begin{equation}\label{eq40}
\hat{R} \bigl( (\I\otimes(\omega_2\otimes\id)\mrW^{\ZZ_2})\beta(f)\bigr)^*
=
(\I\otimes \hat{\ov{\omega}_2})\; \beta(
\ov{R_{G^{\times 2}}(f)})
\end{equation}
(in the above calculation we have used Proposition \ref{stw16}). Consequently, for any $\omega_1\in \LL(G^{\times 2})_*,\omega_2\in\LL^1(\ZZ_2)$ such that $\hat{\omega}_1\circ R_{G^{\times 2}}\in \mf{N}_{\vp_{G^{\times 2}}}$ we have
\[\begin{split}
\Lambda_{\hpsi}(\lambda(\omega_1\otimes\omega_2))&=
J\Lambda_{\hvp}( \hat{R}(\lambda(\omega_1\otimes\omega_2))^*)=
J \Lambda_{\hvp}\bigl( (\I\otimes \hat{\ov{\omega}}_2)
\,\beta( \ov{ \hat{\omega}_1 \circ R_{G^{\times 2}}})\bigr)\\
&=
J\bigl( \Lambda_{\vp_{G^{\times 2}}}( 
\ov{ \hat{\omega}_1 \circ R_{G^{\times 2}}} ) \otimes
\ov{\omega}_2\bigr)=
({\delta_{G^{\times 2}}}^{\frac{1}{2}} \hat{\omega}_1)\otimes\omega_2.
\end{split}\]
Furthermore
\[\begin{split}
&\quad\;
\mc{Q}_R \hat{J} J \Lambda_{\hpsi}(\lambda(\omega_1\otimes\omega_2))=
\mc{Q}_L ({\delta_{G^{\times 2}}}^{\frac{1}{2}} \hat{\omega}_1\otimes
\omega_2)\\
&=
\int_{\IrrG}^{\oplus} 
\delta_{G^{\times 2}}(g,g')^{-\frac{p}{2}}
\begin{bmatrix}
\omega_2(1) & \omega_2(-1) \\
\omega_2(-1) & \omega_2(1) 
\end{bmatrix}\;
\Pi_{(g,g')} ( {\delta_{G^{\times 2}}}^{\frac{1}{2}}
\hat{\omega}_1 )
\md\mu_p([(g,g')]_{\sim_\Delta})\\
&=
\int_{\IrrG}^{\oplus} 
\delta_{G^{\times 2}}(g,g')^{-\frac{p}{2}}
\begin{bmatrix}
\omega_2(1) & \omega_2(-1) \\
\omega_2(-1) & \omega_2(1) 
\end{bmatrix}\;
\Pi_{(g,g')} (
\hat{\omega}_1 )
\delta_{G^{\times 2}}(g,g')^{\frac{1}{2}}
\md\mu_p([(g,g')]_{\sim_\Delta})\\
&=
\int_{\IrrG}^{\oplus} 
(\Pi_{(g,g')}\rtimes U)\bigl(
(\omega_1\otimes\omega_2\otimes\id\otimes\id)W^{\GG}\bigr)
E_{[(g,g')]_{\sim_\Delta}}^{-1}
\md\mu_p([(g,g')]_{\sim_\Delta})
\end{split}\]
and points $7.1)$, $7.2)$ of Theorem \ref{PlancherelR} hold.
Since we have defined $\mc{Q}_R$ in such a way that $\mc{Q}_R \circ\hat{J}J=\mc{Q}_L$, it is clear that the commutation relations also are satisfied. We arrive at the following proposition:

\begin{proposition}\label{stw17}
For each $p\in\RR$, the measure $\mu_p$, the unitary operator $\mc{Q}_R=\mc{Q}_L \circ J\hat{J}$ and operators
\[
E_{[(g,g')]_{\sim_\Delta}}=\delta_{G^{\times 2}}(g,g')^{\frac{p-1}{2}}\I_{\mc{V}^{(g,g')}}\quad([(g,g')]_{\sim_\Delta}\in\IrrG)
\]
are the operators given by the Theorem \ref{PlancherelR}.
\end{proposition}

\subsubsection{Operators $\mc{L}_{[(r,r')]_{\sim_\Delta}}$}
In this section we will describe the action of operators $\mc{L}_{[(r,r')]_{\sim_\Delta}}$. In order to do that, we first need to establish decomposition of the tensor products. It will turn out that it resembles the case of $\RR\rtimes \ZZ_2$.\\
Take a measurable subset $\Omega\subseteq \IrrG=(G^{\times 2}\setminus \Delta)/_{\sim _\Delta}$ and fix an irreducible representation $[(r,r')]_{\sim_\Delta}\in\IrrG$ with $r\neq r'$. To ease the notation, we will write $\mu_p$ instead of $(\mu_p)_{\Omega}$ etc. A typical element of $\CGD\subseteq\Linfd$ looks like 
\[
(\omega_1\otimes\omega_2\otimes\id\otimes\id)\mrW^{\GG}= \bigl(\I\otimes (\omega_2\otimes\id)\mrW^{\ZZ_2}\bigr)\; \beta(\hat{\omega}_1)\in \Linfd,
\]
where $\omega_1\in \LL(G^{\times 2})_*,\omega_2\in\LL^1(\ZZ_2)$ and if we apply $\Delta_{\whG}$ to it, we get
\[\begin{split}
\Delta_{\whG}\bigl(
(\omega_1\otimes\omega_2\otimes\id\otimes\id)\mrW^{\GG}\bigr)=
\Delta_{\wh{\ZZ}_2}\bigl((\omega_2\otimes\id)\mrW^{\ZZ_2}\bigr)_{24}\,
(\beta\otimes\beta)(\Delta_{G^{\times 2}}(\hat{\omega}_1)).
\end{split}\]
Consequently, when we apply the representation $((\Pi_{(r,r')}\rtimes U)\otimes \sigma_\Omega)\Sigma$ to this element we arrive at
\begin{equation}\begin{split}\label{eq34}
&\quad\;
(\Pi_{(r,r')}\rtimes U)\tp \sigma_\Omega
\bigl(
(\omega_1\otimes\omega_2\otimes\id\otimes\id)\mrW^{\GG}
\bigr)\\
&=
\bigl(\omega_2(1)
\begin{bmatrix} 1 & 0 \\ 0 & 1 \end{bmatrix}\otimes
\int_{\Omega}^{\oplus}
\begin{bmatrix} 1 & 0 \\ 0 & 1 \end{bmatrix}
\md\mu_p([(g,g')]_{\sim_\Delta})+
\omega_2(-1)
\begin{bmatrix}  0 & 1  \\ 1 & 0 \end{bmatrix}\otimes
\int_{\Omega}^{\oplus}
\begin{bmatrix} 0 & 1 \\ 1 & 0 \end{bmatrix}
\md\mu_p([(g,g')]_{\sim_\Delta})\bigr)\\
&\quad\quad\quad \bigl(\Pi_{(r,r')}\otimes \int_{\Omega}^{\oplus}
\Pi_{(g,g')}\md\mu_p([(g,g')]_{\sim_\Delta})\bigr)
( \Delta_{(G^{\times 2})^{op}}(\hat{\omega}_1))
\end{split}\end{equation}
(see equation \eqref{eq39}). Recall that we have chosen representatives for each class $[(g,g')]_{\sim_\Delta}$; let as before
\[
H\colon 
(G^{\times 2}\setminus \Delta)/_{\sim_\Delta}\rightarrow
G^{\times 2}\setminus \Delta
\]
denote the corresponding measurable map. Assume that $(r,r')=H([(r,r')]_{\sim_\Delta})$. Let $H^c$ be the map $H$ composed with a flip $(g,g')\mapsto (g',g)$ and let $p_\Delta$ be the canonical quotient map
\[
p_{\Delta}\colon 
G^{\times 2}\setminus \Delta
\rightarrow
(G^{\times 2}\setminus \Delta)/_{\sim_\Delta}.
\]
 Next, define a linear operator $\mc{O}$ via\footnote{We give a formula for the action of $\delta_{G^{\times 2}}(r,r')^{-1}\mc{O}$ rather than $\mc{O}$ to keep the equations shorter.}
\[\begin{split}
\delta_{G^{\times 2}}^{-1}(r,r')&\mc{O}\colon \mc{V}^{(r,r')}\otimes
\int_{\Omega}^{\oplus} \mc{V}^{(g,g')} \md\mu_{p}([(g,g')]_{\sim_\Delta})\rightarrow\\
&\rightarrow
\int_{H^{-1}(p_{\Delta}^{-1}(\Omega)(r,r'))}^{\oplus} \mc{V}^{(g,g')} \md\mu_p([(g,g')]_{\sim_\Delta})
\oplus
\int_{{H}^{c\,-1}(p_{\Delta}^{-1}(\Omega)(r,r'))}^{\oplus} \mc{V}^{(g,g')} \md\mu_p([(g,g')]_{\sim_\Delta})
\end{split}\]
given by
\[\begin{split}
\vect{1}{0}\otimes
\int_{\Omega}^{\oplus} \vect{\xi^{H[g,g']}}{0}\md\mu_{p}([(g,g')]_{\sim_\Delta})\mapsto &
\int_{H^{-1}(H(\Omega)(r,r'))}^{\oplus} \vect{
\delta_{G^{\times 2}}^{-\frac{p}{2}}(g,g')
\xi^{H[g,g'](r,r')^{-1}}
}{0}\md\mu_p([(g,g')]_{\sim_\Delta}) \oplus\\
\oplus
\int_{H^{c\,-1}(H(\Omega)(r,r'))}^{\oplus} &\vect{0}{
\delta_{G^{\times 2}}^{-\frac{p}{2}}(g,g')
\xi^{H^c [g,g'](r,r')^{-1}}
}\md\mu_p([(g,g')]_{\sim_\Delta}),\\
\vect{0}{1}\otimes
\int_{\Omega}^{\oplus} \vect{0}{
\eta^{H[g,g']}}\md\mu_{p}([(g,g')]_{\sim_\Delta})\mapsto &
\int_{H^{-1}(H(\Omega)(r,r'))}^{\oplus}
 \vect{0}{
\delta_{G^{\times 2}}^{-\frac{p}{2}}(g,g')
\eta^{H[g,g'](r,r')^{-1}}}\md\mu_p([(g,g')]_{\sim_\Delta})\oplus\\
\oplus 
\int_{H^{c\,-1}(H(\Omega)(r,r'))}^{\oplus} &\vect{
\delta_{G^{\times 2}}^{-\frac{p}{2}}(g,g')
\eta^{H^{c} [g,g'](r,r')^{-1}}
}{0}\md\mu_p([(g,g')]_{\sim_\Delta}),
\\
\vect{1}{0}\otimes
\int_{\Omega}^{\oplus} \vect{0}{
\eta^{H[g,g']}}\md\mu_{p}([(g,g')]_{\sim_\Delta})\mapsto&\,
\int_{H^{-1}(H^{c}(\Omega)(r,r'))}^{\oplus} \vect{
\delta_{G^{\times 2}}^{-\frac{p}{2}}(g,g')
\eta^{H^{c}[g,g'](r',r)^{-1}}}{0}\md\mu_p([(g,g')]_{\sim_\Delta})\oplus \\
\oplus
\int_{H^{c\,-1}(H^{c}(\Omega)(r,r'))}^{\oplus} &\vect{0}{
\delta_{G^{\times 2}}^{-\frac{p}{2}}(g,g')
\eta^{H[g,g'](r',r)^{-1}}}\md\mu_p([(g,g')]_{\sim_\Delta}),\\
\vect{0}{1}\otimes
\int_{\Omega}^{\oplus} \vect{
\xi^{H[g,g']}}{0}\md\mu_{p}([(g,g')]_{\sim_\Delta})\mapsto&\,
\int_{H^{-1}(H^{c}(\Omega)(r,r'))}^{\oplus} \vect{0}{
\delta_{G^{\times 2}}^{-\frac{p}{2}}(g,g')
\xi^{H^{c}[g,g'](r',r)^{-1}}}\md\mu_p([(g,g')]_{\sim_\Delta})\oplus\\
\oplus\int_{H^{c\,-1}(H^{c}(\Omega)(r,r'))}^{\oplus} &\vect{
\delta_{G^{\times 2}}^{-\frac{p}{2}}(g,g')
\xi^{H[g,g'](r',r)^{-1}}}{0}\md\mu_p([(g,g')]_{\sim_\Delta})
\end{split}\]
(we identify 
\[
H^{-1}( p_{\Delta}^{-1} (\Omega)(r,r'))
\quad\textnormal{ and }\quad
H^{-1}( H (\Omega)(r,r'))\cup
H^{-1}( H^{c} (\Omega)(r,r')),
\]
same for $H^{c\,-1}( p_{\Delta}^{-1} (\Omega)(r,r'))$). One easily sees that $\mc{O}$ is a well defined surjective operator, let us check that it is isometric: we have
\[\begin{split}
&\quad\;
\bigl\|
(\delta_{G^{\times 2}}^{-1}(r,r')\mc{O})\bigl(\vect{1}{0}\otimes
\int_{\Omega}^{\oplus} \vect{\xi^{H[g,g']}}{0}\md\mu_{p}([(g,g')]_{\sim_\Delta})\bigr)\bigr\|^2\\
&=
\int_{H^{-1}(H(\Omega)(r,r'))} 
\delta_{G^{\times 2}}^{-p}(g,g')
|\xi^{H[g,g'](r,r')^{-1}}
|^2\md\mu_p([(g,g')]_{\sim_\Delta})\\
&\quad\quad\quad+
\int_{H^{c\,-1}(H(\Omega)(r,r'))}
\delta_{G^{\times 2}}^{-p}(g,g')
|\xi^{H^c [g,g'](r,r')^{-1}}|^2\md\mu_p([(g,g')]_{\sim_\Delta})\\
&=
\tfrac{1}{2}\int_{G^{\times 2}\setminus \Delta}
\bigl(\chi_{H^{-1}(H(\Omega)(r,r'))}([(g,g')]_{\sim_\Delta})
|\xi^{H[g,g'] (r,r')^{-1}}|^2 \\
&\quad\quad\quad\quad\quad\quad\quad\quad+
\chi_{H^{c\,-1}(H(\Omega)(r,r'))}([(g,g')]_{\sim_\Delta})
|\xi^{H^{c}[g,g'] (r,r')^{-1}}|^2 \bigr)
\md\mu_{G^{\times 2}}(g,g')\\
&=
\tfrac{1}{2}\bigl(
\int_{H((G^{\times 2}\setminus \Delta)/_{\sim_\Delta})}
(\chi_{H(\Omega)(r,r')}(g,g') |\xi^{(g,g')(r,r')^{-1}}|^2+
\chi_{H(\Omega)(r,r')}(g',g) |\xi^{(g',g)(r,r')^{-1}}|^2)
\\
&\quad\quad\quad\quad\quad\quad\quad\quad
\quad\quad\quad\quad\quad\quad\quad\quad
\quad\quad\quad\quad\quad\quad\quad\quad
\quad\quad\quad\quad\quad\quad\quad\quad
\quad\quad\md\mu_{G^{\times 2}}(g,g')+\\
&\quad\;+
\int_{H^{c}((G^{\times 2}\setminus \Delta)/_{\sim_\Delta})}
(\chi_{H(\Omega)(r,r')}(g',g) |\xi^{(g',g)(r,r')^{-1}}|^2+
\chi_{H(\Omega)(r,r')}(g,g') |\xi^{(g,g')(r,r')^{-1}}|^2)\\
&\quad\quad\quad\quad\quad\quad\quad\quad
\quad\quad\quad\quad\quad\quad\quad\quad
\quad\quad\quad\quad\quad\quad\quad\quad
\quad\quad\quad\quad\quad\quad\quad\quad
\quad\quad\md\mu_{G^{\times 2}}(g,g')
\bigr)\\
&=
\tfrac{1}{2}
\int_{G^{\times 2}\setminus \Delta}
( \chi_{H(\Omega)(r,r')}(g,g') |\xi^{(g,g')(r,r')^{-1}}|^2+
\chi_{H(\Omega)(r,r')}(g',g) |\xi^{(g',g)(r,r')^{-1}}|^2)
\md\mu_{G^{\times 2}}(g,g')\\
&=
\delta_{G^{\times 2}}^{-1}(r,r')\,\tfrac{1}{2}
\int_{G^{\times 2}\setminus \Delta}
(\chi_{H(\Omega)} (g,g')|\xi^{(g,g')}|^2 + \chi_{H(\Omega)} (g',g)| \xi^{(g',g)}|^2)
\md\mu_{G^{\times 2}}(g,g')\\
&=
\delta_{G^{\times 2}}^{-1}(r,r')
 \int_{H(\Omega)}
 |\xi^{(g,g')}|^2
\md\mu_{G^{\times 2}}(g,g')
=
\delta_{G^{\times 2}}^{-1}(r,r')
\tfrac{1}{2}\int_{H(\Omega)\cup H^{c}(\Omega)}
|\xi^{H[g,g']}|^2 \md\mu_{G^{\times 2}}(g,g')\\
&=
\delta_{G^{\times 2}}^{-1}(r,r')\int_{\Omega} |\xi^{H[g,g']}|^2
\md\mu_{p}([(g,g')]_{\sim_\Delta})\\
&=
\delta_{G^{\times 2}}^{-1}(r,r') \bigr\|
\vect{1}{0}\otimes
\int_{\Omega}^{\oplus} \vect{\xi^{H[g,g']}}{0}\md\mu_{p}([(g,g')]_{\sim_\Delta})\bigr\|^2,
\end{split}\]
hence
\[
\bigl\|
\mc{O}\bigl(\vect{1}{0}\otimes
\int_{\Omega}^{\oplus} \vect{\xi^{H[g,g']}}{0}\md\mu_{p}([(g,g')]_{\sim_\Delta})\bigr)\bigr\|=
\bigl\|
\vect{1}{0}\otimes
\int_{\Omega}^{\oplus} \vect{\xi^{H[g,g']}}{0}\md\mu_{p}([(g,g')]_{\sim_\Delta})\bigr\|.
\]
A similar calculation applied to the remaining three classes of vectors shows that $\mc{O}$ is unitary. Let us now check that $\mc{O}$ (or rather equivalently $\delta_{G^{\times 2}}^{-1}(r,r')\mc{O}$) is an intertwiner: fix $\omega_1\in \LL(G^{\times 2})_*,\omega_2\in \LL^1(\ZZ_2)$ and a vector
\[
\int_{\Omega}^{\oplus}\vect{\xi^{H[g,g']}}{\eta^{H[g,g']}}\md\mu_p([(g,g')]_{\sim_\Delta})\in
\int_{\Omega}^{\oplus}\mc{V}^{(g,g')} \md\mu_{p}([(g,g')]_{\sim_\Delta})
\]
such that the measurable functions $\xi^{\bullet},\eta^{\bullet}$ are bounded and their supports are compact in $G\times G$. Clearly such vectors span a dense subspace in $\int_{\Omega}^{\oplus} \mc{V}^{(g,g')}\md\mu_{p}([(g,g')]_{\sim_\Delta})$. We can find a compact set $V$ such that supports of $\xi^\bullet,\eta^\bullet$ are contained in the set $V\times V$. Assume moreover that $r,r'\in V$ and that the supports of $\xi^{\bullet},\eta^{\bullet}$ translated by $(r,r')$ or $(r',r)$ are still contained in $V\times V$. Let $e$ be a positive, norm one function in $\mathrm{C}_c(G)$ such that $e(g)=1$ whenever $g\in V$.\\
The function $\Delta_{(G^{\times 2})^{op}}(\hat{\omega}_1)$ belongs to $\mathrm{C}_b(G^{\times 2})=\M(\mathrm{C}_0(G)\otimes\mathrm{C}_0(G))$, hence $(e\otimes e)\Delta_{(G^{\times 2})^{op}}(\hat{\omega}_1)$ can be approximated in the norm topology by linear combinations of simple tensors:
\begin{equation}\label{eq37}
\sum_{k=1}^{N_n} f_{n,k}\otimes f'_{n,k}
\xrightarrow[n\to\infty]{}
(e\otimes e)\Delta_{(G^{\times 2})^{op}}(\hat{\omega}_1).
\end{equation}
Observe that we have
\begin{equation}\begin{split}\label{eq35}
&\quad\;
\bigl(
\Pi_{(r,r')}\otimes\int_\Omega^{\oplus}
\Pi_{(g,g')}\md\mu_p([(g,g')]_{\sim_\Delta})\bigr)
( \Delta_{(G^{\times 2})^{op}}(\hat{\omega}_1)) \bigl(
\vect{1}{0} \otimes \int_{\Omega}^{\oplus} 
\vect{\xi^{H[g,g']}}{\eta^{H[g,g']}} \md\mu_{p}([(g,g')]_{\sim_\Delta})\\
&=
\bigl(
\Pi_{(r,r')}\otimes\int_\Omega^{\oplus}
\Pi_{(g,g')}\md\mu_p([(g,g')]_{\sim_\Delta})\bigr)
( \Delta_{(G^{\times 2})^{op}}(\hat{\omega}_1)) \\
&\quad\quad\quad\quad\quad
\bigl(\Pi_{(r,r')}\otimes\int_\Omega^{\oplus}
\Pi_{(g,g')}\md\mu_p([(g,g')]_{\sim_\Delta})\bigr)
(e\otimes e) 
\bigl(
\vect{1}{0} \otimes \int_{\Omega}^{\oplus} 
\vect{\xi^{H[g,g']}}{\eta^{H[g,g']}} \md\mu_{p}([(g,g')]_{\sim_\Delta})\\
&=
\bigl(
\Pi_{(r,r')}\otimes\int_\Omega^{\oplus}
\Pi_{(g,g')}\md\mu_p([(g,g')]_{\sim_\Delta})\bigr)
((e\otimes e) \Delta_{(G^{\times 2})^{op}}(\hat{\omega}_1)) \bigl(
\vect{1}{0} \otimes \int_{\Omega}^{\oplus} 
\vect{\xi^{H[g,g']}}{\eta^{H[g,g']}} \md\mu_{p}([(g,g')]_{\sim_\Delta}),
\end{split}\end{equation}
and similarly
\begin{equation}\begin{split}\label{eq36}
&\quad\;
\bigl(
\Pi_{(r,r')}\otimes\int_\Omega^{\oplus}
\Pi_{(g,g')}\md\mu_p([(g,g')]_{\sim_\Delta})\bigr)
( \Delta_{(G^{\times 2})^{op}}(\hat{\omega}_1)) \bigl(
\vect{0}{1} \otimes \int_{\Omega}^{\oplus} 
\vect{\xi^{H[g,g']}}{\eta^{H[g,g']}} \md\mu_{p}([(g,g')]_{\sim_\Delta})\\
&=
\bigl(
\Pi_{(r,r')}\otimes\int_\Omega^{\oplus}
\Pi_{(g,g')}\md\mu_p([(g,g')]_{\sim_\Delta})\bigr)
((e\otimes e) \Delta_{(G^{\times 2})^{op}}(\hat{\omega}_1)) \bigl(
\vect{0}{1} \otimes \int_{\Omega}^{\oplus} 
\vect{\xi^{H[g,g']}}{\eta^{H[g,g']}} \md\mu_{p}([(g,g')]_{\sim_\Delta}).
\end{split}\end{equation}
Now, using the definition of $\mc{O}$, equations \eqref{eq34}, \eqref{eq35} and convergence \eqref{eq37} we can check that $\mc{O}$ is an intertwiner. On the one hand we have
\[\begin{split}
&\quad\;(\delta_{G^{\times 2}}^{-1}(r,r')\mc{O})
(\Pi_{(r,r')}\rtimes U)\tp \sigma_\Omega ((\omega_1\otimes\omega_2\otimes\id\otimes\id)\mrW^{\GG})
\bigl(\vect{1}{0}\otimes
\int_{\Omega}^{\oplus} \vect{\xi^{H[g,g']}}{\eta^{H[g,g']}}\md\mu_p([(g,g')]_{\sim_\Delta})\bigr)\\
&=
\lim_{n\to\infty} \sum_{k=1}^{N_n}
(\delta_{G^{\times 2}}^{-1}(r,r')\mc{O})\bigl(
\omega_2(1)\;\vect{f_{n,k}(r,r')}{0}\otimes
\int_{\Omega}^{\oplus} \vect{f'_{n,k}(H[g,g'])\xi^{H[g,g']}}{f'_{n,k}(H^{c}[g,g'])\eta^{H[g,g']}}\md\mu_p([(g,g')]_{\sim_\Delta})\\
&\;\;+
\omega_2(-1)\;\vect{0}{f_{n,k}(r,r')}\otimes
\int_{\Omega}^{\oplus} \vect{f'_{n,k}(H^{c}[g,g'])\eta^{H[g,g']}}{f'_{n,k}(H[g,g'])\xi^{H[g,g']}}\md\mu_p([(g,g')]_{\sim_\Delta})\bigr)\\
&=
(\delta_{G^{\times 2}}^{-1}(r,r')\mc{O})\bigl(
\omega_2(1)\;\vect{1}{0}\otimes
\int_{\Omega}^{\oplus} \vect{
\hat{\omega}_1(H[g,g'](r,r'))\xi^{H[g,g']}}{
\hat{\omega}_1(H^{c}[g,g'](r,r'))\eta^{H[g,g']}}\md\mu_p([(g,g')]_{\sim_\Delta})\\
&\;\;+
\omega_2(-1)\;\vect{0}{1}\otimes
\int_{\Omega}^{\oplus} \vect{
\hat{\omega}_1(H^{c}[g,g'](r,r'))\eta^{H[g,g']}}{
\hat{\omega}_1(H[g,g'](r,r'))\xi^{H[g,g']}}\md\mu_p([(g,g')]_{\sim_\Delta})\bigr)\\
&=
\int_{H^{-1}(H(\Omega)(r,r'))}^{\oplus}
\delta_{G^{\times 2}}^{-p}(g,g')
\vect{ \omega_2(1) \hat{\omega}_1 (H[g,g']) \xi^{H[g,g'](r,r')^{-1}}
 }{
 \omega_2(-1) \hat{\omega}_1 (H[g,g']) \xi^{H[g,g'](r,r')^{-1} }}
\md\mu_p([(g,g')]_{\sim_\Delta})\oplus\\
&\,\oplus\int_{H^{-1}(H^{c}(\Omega)(r,r'))}^{\oplus}
\delta_{G^{\times 2}}^{-p}(g,g')
\vect{ \omega_2(1) \hat{\omega}_1 (H[g,g']) \eta^{H^{c}[g,g'](r',r)^{-1}}
 }{
 \omega_2(-1) \hat{\omega}_1 (H[g,g']) \eta^{H^{c}[g,g'](r',r)^{-1}} }
\md\mu_p([(g,g')]_{\sim_\Delta})\\
&\,\oplus\int_{H^{c\,-1}(H(\Omega)(r,r'))}^{\oplus}
\delta_{G^{\times 2}}^{-p}(g,g')
\vect{ 
\omega_2(-1)\hat{\omega}_1(H^{c}[g,g']) \xi^{H^{c}[g,g'] (r,r')^{-1}}
 }{
\omega_2(1)\hat{\omega}_1(H^{c}[g,g']) \xi^{H^{c}[g,g'] (r,r')^{-1}} }
\md\mu_p([(g,g')]_{\sim_\Delta})\\
&\,\oplus\int_{H^{c\,-1}(H^{c}(\Omega)(r,r'))}^{\oplus}
\delta_{G^{\times 2}}^{-p}(g,g')
\vect{
\omega_2(-1)\hat{\omega}_1(H^{c}[g,g']) \eta^{H[g,g'] (r',r)^{-1}}
 }{
\omega_2(1)\hat{\omega}_1(H^{c}[g,g']) \eta^{H[g,g'] (r',r)^{-1}} }
\md\mu_p([(g,g')]_{\sim_\Delta}),
\end{split}\]
where in the third equality we have used the dominated convergence theorem. On the other hand we have
\[\begin{split}
&\quad\;
(\sigma_{H^{-1}(p_{\Delta}^{-1}(\Omega)(r,r'))}\oplus 
\sigma_{H^{c\,-1}(p_{\Delta}^{-1}(\Omega)(r,r'))})((\omega_1\otimes\omega_2\otimes\id\otimes\id)\mrW^{\GG})
\\
&\quad\quad\quad\quad\quad\quad
\quad\quad\quad\quad\quad\quad
\quad\quad\quad\quad
(\delta_{G^{\times 2}}^{-1}(r,r')\mc{O})\bigl(\vect{1}{0}\otimes
\int_{\Omega}^{\oplus} \vect{\xi^{H[g,g']}}{\eta^{H[g,g']}}\md\mu_p([(g,g')]_{\sim_\Delta})\bigr)\\
&=
\int_{H^{-1}(H(\Omega)(r,r'))}^{\oplus}
\delta_{G^{\times 2}}^{-p}(g,g')
\vect{\omega_2(1)\hat{\omega}_1(H[g,g'])
\xi^{H[g,g'](r,r')^{-1} }}{
\omega_2(-1)\hat{\omega}_1(H[g,g'])\xi^{H[g,g'](r,r')^{-1} }}
\md\mu_p
([(g,g')]_{\sim_\Delta})\oplus\\
&\oplus 
\int_{H^{-1}(H^{c}(\Omega)(r,r'))}^{\oplus}
\delta_{G^{\times 2}}^{-p}(g,g')
\vect{\omega_2(1)\hat{\omega}_1(H[g,g'])\eta^{H^{c}[g,g'](r',r)^{-1} }}{
\omega_2(-1)\hat{\omega}_1(H[g,g'])\eta^{H^{c}[g,g'](r',r)^{-1} }}
\md\mu_p
([(g,g')]_{\sim_\Delta})\oplus\\
&\oplus\int_{H^{c\,-1}(H(\Omega)(r,r'))}^{\oplus} 
\delta_{G^{\times 2}}^{-p}(g,g')
\vect{\omega_2(-1)\hat{\omega}_1(H^{c}[g,g'])\xi^{ H^{c}[g,g'](r,r')^{-1}}}{
\omega_2(1)\hat{\omega}_1(H^{c}[g,g'])\xi^{H^{c}[g,g'](r,r')^{-1} }}
\md\mu_p
([(g,g')]_{\sim_\Delta})\oplus\\
&\oplus
\int_{H^{c\,-1}(H^{c}(\Omega)(r,r'))}^{\oplus}
\delta_{G^{\times 2}}^{-p}(g,g')
\vect{
\omega_2(-1) \hat{\omega}_1(H^{c}[g,g']) \eta^{H[g,g'](r',r)^{-1}}}
{\omega_2(1) \hat{\omega}_1(H^{c}[g,g']) \eta^{H[g,g'](r',r)^{-1}}}
\md\mu_p
([(g,g')]_{\sim_\Delta}),
\end{split}\]
hence we get the desired equality on the vectors of the form $\vect{1}{0}\otimes\int_{\Omega}^{\oplus} \vect{\xi^{H[g,g']}}{\eta^{H[g,g']}} \md\mu_p ([(g,g')]_{\sim_\Delta})$. Let us now check the second case, this time using equation \eqref{eq36}:
\[\begin{split}
&\quad\;
(\delta_{G^{\times 2}}^{-1}(r,r')\mc{O})
(\Pi_{(r,r')}\rtimes U)\tp \sigma_\Omega ((\omega_1\otimes\omega_2\otimes\id\otimes\id)\mrW^{\GG})
\bigl(\vect{0}{1}\otimes
\int_{\Omega}^{\oplus} \vect{\xi^{H[g,g']}}{\eta^{H[g,g']}}\md\mu_p([(g,g')]_{\sim_\Delta})\bigr)\\
&=
\lim_{n\to\infty}
\sum_{k=1}^{N_n}
(\delta_{G^{\times 2}}^{-1}(r,r')\mc{O})\bigl(
\omega_2(1)\;\vect{0}{f_{n,k}(r',r)}\otimes
\int_{\Omega}^{\oplus} \vect{f_{n,k}'(H[g,g'])\xi^{H[g,g']}}{f_{n,k}'(H^{c}[g,g'])\eta^{H[g,g']}}\md\mu_p([(g,g')]_{\sim_\Delta})\\
&\quad\;+
\omega_2(-1)\;\vect{f_{n,k}(r',r)}{0}\otimes
\int_{\Omega}^{\oplus} \vect{f_{n,k}'(H^{c}[g,g'])\eta^{H[g,g']}}{f_{n,k}'(H[g,g'])\xi^{H[g,g']}}\md\mu_p([(g,g')]_{\sim_\Delta})\bigr)\\
&=
(\delta_{G^{\times 2}}^{-1}(r,r')\mc{O})\bigl(
\omega_2(1)\;\vect{0}{1}\otimes
\int_{\Omega}^{\oplus} \vect{\hat{\omega}_1(H[g,g'](r',r))\xi^{H[g,g']}}{\hat{\omega}_1(H^{c}[g,g'](r',r))\eta^{H[g,g']}}\md\mu_p([(g,g')]_{\sim_\Delta})\\
&\quad\;+
\omega_2(-1)\;\vect{1}{0}\otimes
\int_{\Omega}^{\oplus} \vect{\hat{\omega}_1(H^{c}[g,g'](r',r))\eta^{H[g,g']}}{\hat{\omega}_1(H[g,g'](r',r))\xi^{H[g,g']}}\md\mu_p([(g,g')]_{\sim_\Delta})\bigr)\\
&=
\int_{H^{-1}(H(\Omega)(r,r'))}^{\oplus}
\delta_{G^{\times 2}}^{-p}(g,g')
\vect{\omega_2(-1)\hat{\omega}_1(H^{c}[g,g']) \eta^{H[g,g'](r,r')^{-1}} }{
\omega_2(1)\hat{\omega}_1(H^{c}[g,g']) \eta^{H[g,g'](r,r')^{-1}} }
\md\mu_p([(g,g')]_{\sim_\Delta})\oplus\\
&\;\oplus\int_{H^{-1}(H^{c}(\Omega)(r,r'))}^{\oplus}
\delta_{G^{\times 2}}^{-p}(g,g')
\vect{
\omega_2(-1)\hat{\omega}_1(H^{c}[g,g']) \xi^{H^{c}[g,g'](r',r)^{-1}}
 }{
\omega_2(1)\hat{\omega}_1(H^{c}[g,g']) \xi^{H^{c}[g,g'](r',r)^{-1}}
}
\md\mu_p([(g,g')]_{\sim_\Delta})\bigr)\oplus\\
&\;\oplus
\int_{H^{c\,-1}(H(\Omega)(r,r'))}^{\oplus}
\delta_{G^{\times 2}}^{-p}(g,g')
\vect{
\omega_2(1)\hat{\omega}_1(H[g,g']) \eta^{H^{c}[g,g'](r,r')^{-1}}
}{
\omega_2(-1)\hat{\omega}_1(H[g,g']) \eta^{H^{c}[g,g'](r,r')^{-1}}
}
\md\mu_p([(g,g')]_{\sim_\Delta})\oplus\\
&\;\oplus\int_{H^{c\,-1}(H^{c}(\Omega)(r,r'))}^{\oplus}
\delta_{G^{\times 2}}^{-p}(g,g')
\vect{
\omega_2(1)\hat{\omega}_1(H[g,g']) \xi^{H[g,g'](r',r)^{-1}}
}{
\omega_2(-1)\hat{\omega}_1(H[g,g']) \xi^{H[g,g'](r',r)^{-1}}
}
\md\mu_p([(g,g')]_{\sim_\Delta})
\end{split}\]
and on the other hand
\[\begin{split}
&\quad\;
(\sigma_{H^{-1}(p_{\Delta}^{-1}(\Omega)(r,r'))}\oplus 
\sigma_{H^{c\,-1}(p_{\Delta}^{-1}(\Omega)(r,r'))})((\omega_1\otimes\omega_2\otimes\id\otimes\id)\mrW^{\GG})\\
&\quad\quad\quad\quad\quad\quad
\quad\quad\quad\quad\quad\quad
\quad
(\delta_{G^{\times 2}}^{-1}(r,r')\mc{O})\bigl(\vect{0}{1}\otimes
\int_{\Omega}^{\oplus} \vect{\xi^{H[g,g']}}{\eta^{H[g,g']}}\md\mu_p([(g,g')]_{\sim_\Delta})\bigr)\\
&=
\int_{H^{-1}(H(\Omega)(r,r'))}^{\oplus}
\delta_{G^{\times 2}}^{-p}(g,g')
\vect{\omega_2(-1) \hat{\omega}_1(H^{c}[g,g'])\eta^{H[g,g'](r,r')^{-1}}}{\omega_2(1) \hat{\omega}_1(H^{c}[g,g'])\eta^{H[g,g'](r,r')^{-1}}}
\md\mu_p
([(g,g')]_{\sim_\Delta})\oplus\\
&\oplus
\int_{H^{-1}(H^{c}(\Omega)(r,r'))}^{\oplus}
\delta_{G^{\times 2}}^{-p}(g,g')
\vect{
\omega_2(-1) \hat{\omega}_1(H^{c}[g,g'])\xi^{H^{c}[g,g'] (r',r)^{-1} }
}{
\omega_2(1) \hat{\omega}_1(H^{c}[g,g'])\xi^{H^{c}[g,g'] (r',r)^{-1} }}
\md\mu_p
([(g,g')]_{\sim_\Delta})\oplus\\
&\oplus
\int_{H^{c\,-1}(H(\Omega)(r,r'))}^{\oplus}
\delta_{G^{\times 2}}^{-p}(g,g')
\vect{
\omega_2(1) \hat{\omega}_1(H[g,g'])\eta^{H^{c}[g,g'](r,r')^{-1}}}{
\omega_2(-1) \hat{\omega}_1(H[g,g'])\eta^{H^{c}[g,g'](r,r')^{-1}}
}
\md\mu_p
([(g,g')]_{\sim_\Delta})\oplus\\
&\oplus
\int_{H^{c\,-1}(H^{c}(\Omega)(r,r'))}^{\oplus}
\delta_{G^{\times 2}}^{-p}(g,g')
\vect{\omega_2(1) \hat{\omega}_1(H[g,g'])\xi^{H[g,g'] (r',r)^{-1}}}{\omega_2(-1) \hat{\omega}_1(H[g,g'])\xi^{H[g,g'] (r',r)^{-1}}}
\md\mu_p
([(g,g')]_{\sim_\Delta}).
\end{split}\]
Consequently, we have proved the following:
\begin{proposition}
For any $[(r,r')]_{\sim_\Delta}\in\IrrG$ and a measurable subset $\Omega\subseteq\IrrG$ we have
\begin{equation}\label{eq38}
(\Pi_{(r,r')}\rtimes U)\tp \sigma_\Omega\simeq
\sigma_{H^{-1}(p_{\Delta}^{-1}(\Omega)(r,r'))}\oplus 
\sigma_{H^{c\,-1}(p_{\Delta}^{-1}(\Omega)(r,r'))}.
\end{equation}
\end{proposition}

Note that even though the sets $H^{-1}(p_{\Delta}^{-1}(\Omega)(r,r')),H^{c\,-1}(p_{\Delta}^{-1}(\Omega)(r,r'))$ depend on the choice of representatives, the function $\chi_{H^{-1}(p_{\Delta}^{-1}(\Omega)(r,r'))}+
\chi_{H^{c\,-1}(p_{\Delta}^{-1}(\Omega)(r,r'))}$ does not. Indeed, we have
\[\begin{split}
&\quad\;
H^{-1}(p_{\Delta}^{-1}(\Omega)(r,r'))\cup H^{c\,-1}(p_{\Delta}^{-1}(\Omega)(r,r'))\\
&=
\{
[(g,g')]_{\sim_\Delta}\,|\,
(g,g')\in p_{\Delta}^{-1}(\Omega)(r,r')\,\textnormal{ or }
(g',g)\in p_{\Delta}^{-1}(\Omega)(r,r')
\}
\end{split}\]
and 
\[\begin{split}
&\quad\;
H^{-1}(p_{\Delta}^{-1}(\Omega)(r,r'))\cap 
H^{c\,-1}(p_{\Delta}^{-1}(\Omega)(r,r'))\\
&=
\{
[(g,g')]_{\sim_\Delta}\,|\,
(g,g')\in p_{\Delta}^{-1}(\Omega)(r,r')\,\textnormal{ and }
(g',g)\in p_{\Delta}^{-1}(\Omega)(r,r')
\}.
\end{split}\]
Furthermore, one easily sees that these sets remain unaffected by the change $(r,r')\mapsto (r',r)$ -- a property one should expect, since the representation on the left hand side of the equation \eqref{eq38} is independent of the choice of a representative of $[(r,r')]_{\sim_\Delta}$ (up to a unitary equivalence).\\
In order to describe the operator $\mc{L}_{[(r,r')]_{\sim_\Delta}}$ we need to find the function $\varpi^{[(r,r')]_{\sim_\Delta},\Omega,\mu_p}$. By definition, for all $a\in \CGDu^+$ we have
\begin{equation}\label{eq44}\begin{split}
&\quad\;\int_\Omega \Tr( (\Pi_{(r,r')}\rtimes U)\tp (\Pi_{(g,g')}\rtimes U) (a))
\md\mu_{p}([(g,g')]_{\sim_\Delta})\\
&=
\int_{H^{-1}( p_{\Delta}^{-1}(\Omega)(r,r'))}
\Tr ((\Pi_{(g,g')}\rtimes U)(a))
\varpi^{[(r,r')]_{\sim_\Delta},\Omega,\mu_p}([(g,g')]_{\sim_\Delta})
\md\mu_p([(g,g')]_{\sim_\Delta})\\
&+
\int_{H^{c\, -1}( p_{\Delta}^{-1}(\Omega)(r,r'))}
\Tr ((\Pi_{(g,g')}\rtimes U)(a))
\varpi^{[(r,r')]_{\sim_\Delta},\Omega,\mu_p}([(g,g')]_{\sim_\Delta})
\md\mu_p([(g,g')]_{\sim_\Delta}).
\end{split}\end{equation}
Take $a=\beta(f)$ for a positive function $f\in \mathrm{C}_0(G^{\times 2})$, then
\[
\Tr\bigl((\Pi_{(r,r')}\rtimes U)\tp (\Pi_{(g,g')}\rtimes U) (a)\bigr)=
f(gr,g'r')+f(g'r,gr')+f(gr',g'r)+f(g'r',gr)
\]
and the left hand side of the equation \eqref{eq44} is
\[\begin{split}
&\quad\;
\int_\Omega ( f(gr,g'r')+f(g'r,gr')+f(gr',g'r)+f(g'r',gr))\md\mu_p
([(g,g')]_{\sim_\Delta})\\
&=
\tfrac{1}{2}\int_{p_{\Delta}^{-1}\Omega} 
\delta_{G^{\times 2}}^{p}(g,g')
( f(gr,g'r')+f(g'r,gr')+f(gr',g'r)+f(g'r',gr))\md\mu_{G^{\times 2}}(g,g')\\
&=
\int_{p_{\Delta}^{-1}\Omega} 
\delta_{G^{\times 2}}^{p}(g,g')( f(gr,g'r')+f(gr',g'r))\md\mu_{G^{\times 2}}(g,g')\\
&=
\delta_{G^{\times 2}}(r,r')^{1-p} \int_{G^{\times 2}}
\delta_{G^{\times 2}}^{p}\;(\chi_{p_{\Delta}(\Omega)(r,r')}+
\chi_{p_{\Delta}(\Omega)(r,r')}) f \md\mu_{G^{\times 2}}
\end{split}\]
Note that
\[
p_{\Delta}^{-1} H^{-1} (p_{\Delta}^{-1}(\Omega)(r,r'))\cup
p_{\Delta}^{-1} H^{c\,-1} (p_{\Delta}^{-1}(\Omega)(r,r'))=
p_{\Delta}^{-1}(\Omega)(r,r')\cup
p_{\Delta}^{-1}(\Omega)(r',r)
\]
and
\[
p_{\Delta}^{-1} H^{-1} (p_{\Delta}^{-1}(\Omega)(r,r'))\cap
p_{\Delta}^{-1} H^{c\,-1} (p_{\Delta}^{-1}(\Omega)(r,r'))=
p_{\Delta}^{-1}(\Omega)(r,r')\cap
p_{\Delta}^{-1}(\Omega)(r',r),
\]
consequently
\[
\chi_{p_{\Delta}^{-1} H^{-1} (p_{\Delta}^{-1}(\Omega)(r,r'))}+
\chi_{p_{\Delta}^{-1} H^{c\,-1} (p_{\Delta}^{-1}(\Omega)(r,r'))}=
\chi_{p_{\Delta}^{-1}(\Omega)(r,r')}+
\chi_{p_{\Delta}^{-1}(\Omega)(r',r)}
\]
and the right hand side of the equation \eqref{eq44} equals
\[\begin{split}
&\quad\;
\int_{H^{-1}(p_{\Delta}^{-1}(\Omega)(r,r'))}
\Tr ((\Pi_{(g,g')}\rtimes U)(a))
\varpi^{[(r,r')]_{\sim_\Delta},\Omega,\mu_p}([(g,g')]_{\sim_\Delta})
\md\mu_p([(g,g')]_{\sim_\Delta})\\
&+
\int_{H^{c\,-1}(p_{\Delta}^{-1}(\Omega)(r,r'))}
\Tr ((\Pi_{(g,g')}\rtimes U)(a))
\varpi^{[(r,r')]_{\sim_\Delta},\Omega,\mu_p}([(g,g')]_{\sim_\Delta})
\md\mu_p([(g,g')]_{\sim_\Delta})\\
&=
\tfrac{1}{2}\int_{p_\Delta^{-1} H^{-1}(p_{\Delta}^{-1}(\Omega)(r,r'))}
\delta_{G^{\times 2}}^{p}(g,g')
(f(g,g')+f(g',g))
\varpi^{[(r,r')]_{\sim_\Delta},\Omega,\mu_p}(p_\Delta(g,g'))
\md\mu_{G^{\times 2}}(g,g')\\
&+
\tfrac{1}{2}\int_{p_\Delta^{-1} H^{c\,-1}(p_{\Delta}^{-1}(\Omega)(r,r'))}
\delta_{G^{\times 2}}^{p}(g,g')
(f(g,g')+f(g',g))
\varpi^{[(r,r')]_{\sim_\Delta},\Omega,\mu_p}(p_{\Delta}(g,g'))
\md\mu_{G^{\times 2}}(g,g')\\
&=
\tfrac{1}{2}\int_{G^{\times 2}}
\delta_{G^{\times 2}}^{p}(g,g')
(\chi_{p_\Delta^{-1}(\Omega)(r,r')} +
 \chi_{p_\Delta^{-1} (\Omega)(r',r)})(g,g')
(f(g,g')+f(g',g))\\
&\quad\quad\quad\quad\quad\quad\quad\quad\quad\quad
\quad\quad\quad\quad\quad\quad\quad\quad
\quad\quad\quad\quad
\varpi^{[(r,r')]_{\sim_\Delta},\Omega,\mu_p}(p_{\Delta}(g,g'))
\md\mu_{G^{\times 2}}(g,g')\\
&=
\int_{G^{\times 2}}
\delta_{G^{\times 2}}^{p}\;
(\chi_{p_\Delta^{-1}(\Omega)(r,r')} + \chi_{p_\Delta^{-1} (\Omega)(r',r)})\,
f\,
\varpi^{[(r,r')]_{\sim_\Delta},\Omega,\mu_p}\circ p_{\Delta}
\md\mu_{G^{\times 2}}.
\end{split}\]

The above calculation gives us an equality
\[\begin{split}
&\quad\;\delta_{G^{\times 2}}(r,r')^{1-p}\int_{G^{\times 2}}
\delta_{G^{\times 2}}^{p}\;
 (\chi_{p_{\Delta}^{-1}(\Omega)(r,r')}+
\chi_{p_{\Delta}^{-1}(\Omega)(r',r)}) f \md\mu_{G^{\times 2}}\\
&=
\int_{G^{\times 2}}
\delta_{G^{\times 2}}^{p}\;
(\chi_{p_{\Delta}^{-1}(\Omega)(r,r')}+
\chi_{p_{\Delta}^{-1}(\Omega)(r',r)}) f 
\varpi^{[(r,r')]_{\sim_\Delta},\Omega,\mu_p}\circ p_{\Delta}
\md\mu_{G^{\times 2}}
\end{split}\]
for all positive functions $f\in \mathrm{C}_0(G^{\times 2})$. It follows that
\[
\varpi^{[(r,r')]_{\sim_\Delta},\Omega,\mu_p}=\delta_{G^{\times 2}}(r,r')^{1-p}
\]
on 
\[
\F^1_{(\Pi_{(r,r')}\rtimes U)\stp \sigma_\Omega}=p_{\Delta}( p_{\Delta}^{-1}(\Omega)(r,r')\cup
p_{\Delta}^{-1}(\Omega)(r',r)).
\]
 
Having derived this result, we can calculate the action of $\mc{L}_{[(r,r')]_{\sim_\Delta}}$: for $\Omega\subseteq \IrrG$ such that $\Tr(E^{2}_{\bullet})^{\frac{1}{2}}\chi_\Omega\in\LL^2(\IrrG)$ and $[(g,g')]_{\sim_\Delta}\in\IrrG$ by definition of $\mc{L}_{[(r,r')]_{\sim_\Delta}}$ we have
\[\begin{split}
&\quad\;
\mc{L}_{[(r,r')]_{\sim_\Delta}}\bigl(
\Tr(E^{2}_{\bullet})^{\frac{1}{2}}\chi_\Omega\bigr) ([(g,g')]_{\sim_\Delta})\\
&=
\delta_{G^{\times 2}}(r,r')^{1-p}
\Tr(E^{2}_{\bullet})([(g,g')]_{\sim_\Delta})^{\frac{1}{2}}
\bigl(
\chi_{H^{-1} (p_\Delta^{-1}(\Omega) (r,r'))} +
\chi_{H^{c\,-1} (p_\Delta^{-1}(\Omega)(r,r') )} 
\bigr)([(g,g')]_{\sim_\Delta})\\
&=
\delta_{G^{\times 2}}(r,r')^{1-p}
\Tr(E^{2}_{\bullet})([(g,g')]_{\sim_\Delta})^{\frac{1}{2}}
\bigl(
\chi_{p_\Delta^{-1}(\Omega)(r,r') } +
\chi_{p_\Delta^{-1}(\Omega)(r',r)} 
\bigr)(g,g')\\
&=
\delta_{G^{\times 2}}(r,r')^{1-p}
\Tr(E^{2}_{\bullet})([(g,g')]_{\sim_\Delta})^{\frac{1}{2}}
\bigl(
\chi_{\Omega } ([ (g,g')(r,r')^{-1}]_{\sim_\Delta})+
\chi_{\Omega } ([ (g,g')(r',r)^{-1}]_{\sim_\Delta})
\bigr).
\end{split}\]
By linearity and continuity we can extend the above formula to arbitrary function $f$ such that $\Tr(E^{2}_{\bullet})^{\frac{1}{2}}f\in \LL^2(\IrrG)$:
\[\begin{split}
&\quad\;
\mc{L}_{[(r,r')]_{\sim_\Delta}}\bigl(
\Tr(E^{2}_{\bullet})^{\frac{1}{2}}f\bigr) ([(g,g')]_{\sim_\Delta})\\
&=
\delta_{G^{\times 2}}(r,r')^{1-p}
\Tr(E^{2}_{\bullet})([(g,g')]_{\sim_\Delta})^{\frac{1}{2}}
\bigl(
f ([(g,g')(r,r')^{-1} ]_{\sim_\Delta})+
f ([ (g,g')(r',r)^{-1}]_{\sim_\Delta})
\bigr)\\
&=
\delta_{G^{\times 2}}(r,r')^{1-p}
\Tr(E^{2}_{\bullet})([(g,g')]_{\sim_\Delta})^{\frac{1}{2}}
\bigl(
\Tr(E^{2}_{\bullet})([(g,g')(r,r')^{-1}]_{\sim_\Delta})^{-\frac{1}{2}}
(\Tr(E^{2}_{\bullet})^{\frac{1}{2}}f) ([ (g,g')(r,r')^{-1}]_{\sim_\Delta})\\
&\quad\quad\quad\quad\quad\quad+
\Tr(E^{2}_{\bullet})([(g,g')(r',r)^{-1}]_{\sim_\Delta})^{-\frac{1}{2}}
(\Tr(E^{2}_{\bullet})^{\frac{1}{2}}f) ([ (g,g')(r',r)^{-1}]_{\sim_\Delta})
\bigr)
\end{split}\]
Proposition \ref{stw17} implies that $\Tr(E^2_{[(g,g')]})=
2 \,\delta_{G^{\times 2}}(g,g')^{p-1}$ for almost all $[(g,g')]_{\sim_\Delta}\in\IrrG$, hence
\[\begin{split}
&\quad\;
\mc{L}_{[(r,r')]_{\sim_\Delta}}\bigl(
\Tr(E^{2}_{\bullet})^{\frac{1}{2}}f\bigr) ([(g,g')]_{\sim_\Delta})\\
&=
\delta_{G^{\times 2}}(r,r')^{1-p}
\delta_{G^{\times 2}}(g,g')^{\frac{p-1}{2}}
\bigl(
\delta_{G^{\times 2}}((g,g')(r,r')^{-1})^{-\frac{p-1}{2}}
(\Tr(E^{2}_{\bullet})^{\frac{1}{2}}f) ([ (g,g')(r,r')^{-1}]_{\sim_\Delta})\\
&\quad\quad\quad\quad\quad\quad+
\delta_{G^{\times 2}}((g,g')(r',r)^{-1})^{-\frac{p-1}{2}}
(\Tr(E^{2}_{\bullet})^{\frac{1}{2}}f) ([ (g,g')(r',r)^{-1}]_{\sim_\Delta})
\bigr)\\
&=
\delta_{G^{\times 2}}(r,r')^{\frac{1-p}{2}}
\bigl(
(\Tr(E^{2}_{\bullet})^{\frac{1}{2}}f) ([ (g,g')(r,r')^{-1}]_{\sim_\Delta})+
(\Tr(E^{2}_{\bullet})^{\frac{1}{2}}f) ([ (g,g')(r',r)^{-1}]_{\sim_\Delta})
\bigr).
\end{split}\]
We have proved the following:

\begin{proposition}
For each $p\in \RR$ we have
\[
\varpi^{[(r,r')]_{\sim_\Delta},\Omega,\mu_p}([(g,g')]_{\sim_\Delta})=
\delta_{G^{\times 2}} (r,r')^{1-p}
\]
for all $[(g,g')]_{\sim_\Delta}\in \F^1_{(\Pi_{(r,r')}\rtimes U)\stp \sigma_\Omega}$. The operator $\mc{L}_{[(r,r')]_{\sim_\Delta}}\in\B(\LL^2(\IrrG))$ is given by
\[\begin{split}
\mc{L}_{[(r,r')]_{\sim_\Delta}}(
 f) ([(g,g')]_{\sim_\Delta})=
\delta_{G^{\times 2}}(r,r')^{\frac{1-p}{2}}
\bigl(
f ([ (g,g')(r,r')^{-1}]_{\sim_\Delta})+
f ([(g,g')(r',r)^{-1} ]_{\sim_\Delta})
\bigr).
\end{split}\]
for all $f\in \LL^2(\IrrG)$ and $[(g,g')]_{\sim_\Delta}\in\IrrG$.
\end{proposition}

Let us check that for $p=\tfrac{1}{2}$, the measure $\mu_p$ is invariant under conjugation. In equation \eqref{eq40} we have shown that for $\omega_2\in\LL^1(\ZZ_2),f\in \LL^{\infty}(G^{\times 2})$ we have
\[
\hat{R}\bigl((\I\otimes (\omega_2\otimes\id)\mrW^{\ZZ_2}) \beta(f)\bigr)=
\beta(R_{G^{\times 2}}(f))\, (\I\otimes \hat{\ov{\omega}_2})^*,
\]
where $\hat{\ov{\omega}_2}=(\id\otimes \ov{\omega}_2)\mrW^{\ZZ_2}$. It follows that for every $(g,g')\in G^{\times 2}\setminus \Delta$ we have
\[\begin{split}
&\quad\;
\jmath_{\mc{V}^{(g,g')}} \circ
(\Pi_{(g,g')}\rtimes U) \circ \hat{R} \bigl(
(\I\otimes (\omega_2\otimes\id)\mrW^{\ZZ_2}) \beta(f)\bigr)\\
&=
\jmath_{\mc{V}^{(g,g')}}\bigl(
\begin{bmatrix}
f( H[g,g']^{-1}) & 0 \\
0 & f(H^{c}[g,g']^{-1})
\end{bmatrix}
(\ov{\omega_2(1)}
\begin{bmatrix}
1 & 0\\
0 & 1
\end{bmatrix}+
\ov{\omega_2(-1)}
\begin{bmatrix}
0 & 1\\
1 & 0
\end{bmatrix}
)^*
\bigr)\\
&=
\jmath_{\mc{V}^{(g,g')}}\bigl(
\begin{bmatrix}
\omega_2(1)f( H[g,g']^{-1}) & 
\omega_2(-1) f(H[g,g']^{-1}) \\
\omega_2(-1) f(H^{c}[g,g']^{-1}) & \omega_2(1)f(H^{c}[g,g']^{-1})
\end{bmatrix}
\bigr).
\end{split}\]
The element $H[g,g']^{-1}$ is equal to $H[g^{-1},{g'}^{-1}]$ or $H^{c}[g^{-1},{g'}^{-1}]$. As the representations $\Pi_{H[g^{-1},{g'}^{-1}]}\rtimes U$ and $\Pi_{H^{c}[g^{-1},{g'}^{-1}]}\rtimes U$ are unitarily equivalent, we have proved
\[
\jmath_{\mc{V}^{(g,g')}} \circ (\Pi_{(g,g')}\rtimes U)
\circ \hat{R}\simeq \Pi_{H[g^{-1},{g'}^{-1}]}\rtimes U.
\]
Now it follows that the measure $\mu_{\frac{1}{2}}$ is invariant under conjugation: it is a consequence of the fact that the measure $\delta_{G^{\times 2}}^{\frac{1}{2}} \mu_{G^{\times 2}}$ on $G^{\times 2}$ is invariant under the inverse.\\
Since the scaling group of $\GG$ is trivial, admissibility of irreducible representations does not cause any problems \cite[Remark 3.3]{DasDawsSalmi} and Theorem \ref{tw4} together with Proposition \ref{stw19} give us the following result.

\begin{proposition}\label{stw21}$ $
\begin{enumerate}
\item If $G$ is amenable then for every $p\in\RR$ and any subset $\Omega\subseteq\IrrG$ of finite $\mu_p$ measure we have
\[
\mu_p(\Omega)\in \sigma\bigl(
\int_{\Omega} \mc{L}_{[(g,g')]_{\sim_\Delta}} \md\mu_p([(g,g')]_{\sim_\Delta})\bigr).
\]
\item
If for every symmetric subset $\Omega\subseteq\IrrG$ of finite $\mu_{\frac{1}{2}}$ measure we have
\[
\mu_{\frac{1}{2}}(\Omega)\in \sigma\bigl(
\int_{\Omega} \mc{L}_{[(g,g')]_{\sim_\Delta}} \md\mu_{\frac{1}{2}}([(g,g')]_{\sim_\Delta})\bigr),
\]
then $G$ is amenable.
\end{enumerate}
\end{proposition}

\subsection{Products}
Let $\GG_1,\GG_2$ be two arbitrary locally compact quantum groups. In this section we will describe their product $\GG=\GG_1\times \GG_2$. Let us denote objects associated with $\GG_1,\GG_2,\GG$ with appropriate subscripts. Product $\GG_1,\GG_2$ can be considered as a special case of the bicrossed product construction. Indeed, consider the matched pair of locally compact quantum groups $(\widehat{\GG}_1,\GG_2)$ with trivial cocycles and map $\tau=\id\in \B(\LL^{\infty}(\whG_1)\bar{\otimes}\LL^{\infty}(\GG_2))$
(we follow the notation of \cite{VaesVainerman}). Then the corresponding maps $\alpha,\beta$ are given by
\[\begin{split}
&\alpha\colon \LL^{\infty}(\GG_2)\ni y \mapsto \I\otimes y\in
\LL^{\infty}(\whG_1)\oxx \LL^{\infty}(\GG_2),\\
&\beta\colon \LL^{\infty}(\whG_1)\ni x \mapsto x\otimes \I\in
\LL^{\infty}(\whG_1)\oxx \LL^{\infty}(\GG_2).
\end{split}\]
Basic objects associated with $\GG$ are the one we could expect: the von Neumann algebra of $\GG$ is given by
\[
\LL^{\infty}(\GG)=\{\alpha(y),\,(\omega\otimes\id)\mrW^{\whG_1}\otimes \I\,|\,
y\in \LL^{\infty}(\GG_2),\omega\in \LL^{1}(\GG_1)\}''=
\LL^{\infty}(\GG_1)\oxx \LL^{\infty}(\GG_2),
\]
and the comultiplication is given by
\[
\Delta(x\otimes y)=(\id\otimes\sigma\otimes\id)(\Delta_1(x)\otimes \Delta_2(y))\quad(x\in \LL^{\infty}(\GG_1),y\in \LL^{\infty}(\GG_2)),
\]
which follows from \cite[Proposition 2.4, Proposition 2.5]{VaesVainerman}. The left Haar integral of $\GG$ is given by the tensor product of n.s.f. weights $\vp=\vp_1\otimes\vp_2$ and it follows that $\nabla_{\vp}=\nabla_{\vp_1}\otimes\nabla_{\vp_2}$ (\cite[Proposition 8.1]{Stratila}). The same is true for right Haar integrals, hence e.g. property $\psi\circ \sigma^{\vp}_t=\nu^{-t}\psi\,(t\in\RR)$ implies $\nu=\nu_1\nu_2$. Since the modular operator associated with $\vp$ implements the scaling group we have $\tau_t=\tau^1_t\otimes\tau_t^2\,(t\in\RR)$ (see \cite{KustermansVaes}). The dual of $\GG$ is $\whG_1\times \whG_2$, this is a content of \cite[Proposition 2.9]{VaesVainerman}. We have $P=P_1\otimes P_2$. Indeed, take $x\in\mf{N}_{\vp_1},y\in\mf{N}_{\vp_2},t\in\RR$. We have
\[\begin{split}
&\quad\;
P^{it} \Lvp(x\otimes y)=\nu^{\frac{t}{2}} \Lvp(\tau^1_t(x)\otimes\tau_t^2(y))\\
&=
(\nu_1^{\frac{t}{2}} \Lambda_{\vp_1}(\tau^1_t(x)))\otimes
(\nu_2^{\frac{t}{2}} \Lambda_{\vp_2}(\tau^2_t(y)))=
(P_1^{it}\otimes P_2^{it})\Lvp(x\otimes y),
\end{split}\]
which implies $P=P_1\otimes P_2$. It follows directly from the definition that the Kac-Takesaki operator of $\GG$ is given by $\mrW=(\mrW_1)_{13} (\mrW_2)_{24}$, hence $\mathrm{C}_0(\GG)=\mathrm{C}_0(\GG_1)\otimes \mathrm{C}_0(\GG_2)$ (note that here $\otimes$ denotes the minimal tensor product of \cst-algebras).

\begin{proposition}
The \cst-algebra $\mathrm{C}_0^{u}(\whG)$ is isomorphic to $\mathrm{C}^u_0(\whG_1)\otimes_{\operatorname{max}} \mathrm{C}_0^u(\whG_2)$, where $\otimes_{\operatorname{max}}$ denotes the maximal tensor product of \cst-algebras. Under this identification we have
\[
\WW=(\WW_1)_{13} (\WW_2)_{24}.
\]
\end{proposition}

\begin{proof}
One easily checks that the normal injective $\star$-homomorphism
\[
\gamma_1\colon \LL^{\infty}(\wh{\GG_1})\ni x \mapsto
x\otimes\I \in \LL^{\infty}(\wh{\GG_1\times \GG_2})=
\LL^{\infty}(\wh{\GG_1})\oxx \LL^{\infty}(\wh{\GG_2})
\]
satisfies $\Delta_{\wh{\GG_1\times \GG_2}}\circ \gamma_1=(\gamma_1\otimes\gamma_1)\circ \Delta_{\wh{\GG_1}}$, hence $\GG_1$ is a closed subgroup of $\GG_1\times \GG_2$ in the sense of Vaes (\cite[Theorem 3.3]{closedqsub}). Denote by $\gamma'_1$ the restriction of $\gamma_1$ to $\mathrm{C}_0(\wh{\GG_1})$. It is shown in \cite{closedqsub} that $\gamma'_1$ is a morphism between $\mathrm{C}_0(\wh{\GG_1})$ and $\mathrm{C}_0(\wh{\GG_1\times \GG_2})$. Denote by $\pi_1,\wh{\pi}_1$ the corresponding strong quantum homomorphism $\GG_1\rightarrow \GG_1\times \GG_2$ and its dual. Similar reasoning works for $\GG_2$ -- let us denote the respective objects by $\gamma'_2,\pi_2,\wh{\pi}_2$. Observe that we have $
\gamma_1(x) \gamma_2(y)=x\otimes y$ for all $x\in \LL^{\infty}(\wh{\GG_1}),y\in
\LL^{\infty}(\wh{\GG_2}),$ hence
\[\begin{split}
\mrW&=
((\id\otimes\gamma'_1)\mrW_1)_{134}
((\id\otimes\gamma'_2)\mrW_2)_{234}\\
&=
(((\Lambda_{\GG_1}\circ \pi_1)\otimes\id)\WWl)_{134}\;
(((\Lambda_{\GG_2}\circ \pi_2)\otimes\id)\WWl)_{234}\\
&=
((\Lambda_{\GG_1}\circ \pi_1)\otimes(\Lambda_{\GG_2}\circ \pi_2)\otimes
\id)(\WWl_{134} \WWl_{234})\\
&=
((\Lambda_{\GG_1}\circ \pi_1)\otimes(\Lambda_{\GG_2}\circ \pi_2)\otimes
\id)(\Delta_{\GG}^{u}\otimes\id)\WWl.
\end{split}\]
\\
Take any $\omega\in \LL^1(\whG)$ and slice right leg of the above equation: we get
\[
\Lambda_{\GG}((\id\otimes\omega)\WWl)=
(\id\otimes\omega)\mrW=
((\Lambda_{\GG_1}\circ \pi_1)\otimes(\Lambda_{\GG_2}\circ \pi_2))\Delta_{\GG}^{u}((\id\otimes\omega)\WWl),
\]
consequently it follows that
\begin{equation}\label{eq30}
((\Lambda_{\GG_1}\circ \pi_1)\otimes(\Lambda_{\GG_2}\circ \pi_2))\circ\Delta_{\GG}^{u}=\Lambda_{\GG}.
\end{equation}
Since the images of $\gamma_1,\gamma_2$ in $\LL^{\infty}(\wh{\GG_1\times\GG_2})$ commute we also have
\[\begin{split}
\mrW&=
((\id\otimes \gamma'_1)\mrW_1)_{134}\; ((\id\otimes\gamma'_2)\mrW_2)_{234}=
((\id\otimes \gamma'_2)\mrW_2)_{234}\; ((\id\otimes\gamma'_1)\mrW_1)_{134}\\
&=
((\Lambda_{\GG_1}\circ \pi_1)\otimes(\Lambda_{\GG_2}\circ \pi_2)\otimes
\id)(\WWl_{234} \WWl_{134})\\
&=
((\Lambda_{\GG_1}\circ \pi_1)\otimes(\Lambda_{\GG_2}\circ \pi_2)\otimes
\id)(\Delta_{\GG}^{u,\operatorname{op}}\otimes\id)\WWl.
\end{split}\]
and it follows that 
\begin{equation}\label{eq31}
((\Lambda_{\GG_1}\circ \pi_1)\otimes(\Lambda_{\GG_2}\circ \pi_2))\circ\Delta_{\GG}^{u,\operatorname{op}}=\Lambda_{\GG}.
\end{equation}
Now, having these equations we can show that there is a one to one correspondence between representations of $\GG$ and pairs of commuting representations of $\GG_1,\GG_2$. Let 
\[
U=(\id\otimes\sigma_U)\WW\in \M(\CG\otimes \mc{K}(\msf{H}_U))
\]
be a representation of $\GG$ corresponding to a nondegenerate representation
\[
\sigma_U\in \Mor(\CGDu,\mc{K}(\msf{H}_U)).
\]
Define an element
\[
U\rest_{\GG_1}=((\Lambda_{\GG_1}\circ \pi_1)\otimes \sigma_U)\WWd\in\M(\mathrm{C}_0(\GG_1)\otimes \mc{K}(\msf{H}_U)).
\]
It is clear that $U\rest_{\GG_1}$ is unitary, it is also a representation of $\GG_1$:
\[
(\Delta_{\GG_1}\otimes\id)(U\rest_{\GG_1})=
((\Lambda_{\GG_1}\circ \pi_1)\otimes(\Lambda_{\GG_1}\circ \pi_1)\otimes\sigma_U)
(\WWd_{13} \WWd_{23})=
(U\rest_{\GG_1})_{13} (U\rest_{\GG_1})_{23}.
\]
Define in a similar manner the restriction of $U$ to $\GG_2,\,U\rest_{\GG_2}$. We have
\[\begin{split}
(U\rest_{\GG_1})_{13} (U\rest_{\GG_2})_{23}
&=((\Lambda_{\GG_1}\circ \pi_1)\otimes 
(\Lambda_{\GG_2}\circ \pi_2)\otimes\sigma_U) ( \WWd_{13} \WW_{23})\\
&=
((\Lambda_{\GG_1}\circ \pi_1)\otimes 
(\Lambda_{\GG_2}\circ \pi_2)\otimes\sigma_U) 
(\Delta_{\GG}^{u}\otimes\id)\WWd\\
&=
(\id\otimes\sigma_U)\WW=U.
\end{split}\]
In the penultimate equality we have used equation \eqref{eq30}. Similarly we check
\[\begin{split}
(U\rest_{\GG_2})_{23} (U\rest_{\GG_1})_{13}
&=((\Lambda_{\GG_1}\circ \pi_1)\otimes 
(\Lambda_{\GG_2}\circ \pi_2)\otimes\sigma_U) ( \WWd_{23} \WW_{13})\\
&=
((\Lambda_{\GG_1}\circ \pi_1)\otimes 
(\Lambda_{\GG_2}\circ \pi_2)\otimes\sigma_U) 
(\Delta_{\GG}^{u,\operatorname{op}}\otimes\id)\WWd\\
&=
(\id\otimes\sigma_U)\WW=U
\end{split}\]
using this time equation \eqref{eq31}. On the other hand, if $U,V$ are representations of $\GG_1,\GG_2$ on the same space such that $U_{13} V_{23}=V_{23} U_{13}$ then this element defines a representation of $\GG=\GG_1\times \GG_2$. In particular, if we take $\sigma_U=\id$ then we get
\[
\WW=
((\id\otimes \wh{\pi}_1)\WW_1)_{13}
((\id\otimes \wh{\pi}_2)\WW_2)_{23}
=
((\id\otimes \wh{\pi}_2)\WW_2)_{23}
((\id\otimes \wh{\pi}_1)\WW_1)_{13}
\]
hence images of $\wh{\pi}_1,\wh{\pi}_2$ commutes. Let us now check that $\CGDu=\mathrm{C}_0^{u}(\wh{\GG}_1)\otimes_{\operatorname{max}} \mathrm{C}_0^u(\wh{\GG}_2)$ and $\WW=(\WW_1)_{13} (\WW_2)_{24}$ (note that here leg numbering notation between 'universal' legs $3,4$ corresponds to the maximal tensor product). Since
\[\begin{split}
&\quad\;
(\Delta\otimes\id\otimes\id)((\WW_1)_{13} (\WW_2)_{24})=
(\WW_1)_{15}(\WW_1)_{35}(\WW_2)_{26}(\WW_2)_{46}\\
&=
((\WW_1)_{13} (\WW_2)_{24})_{1256}
((\WW_1)_{13} (\WW_2)_{24})_{3456},
\end{split}\]
by the universal property of $\CGDu$, there exists a nondegenerate $\star$-homomorphism $\rho\in \Mor(\CGDu,\mathrm{C}_0^u(\wh{\GG}_1)\otimes_{\operatorname{max}} \mathrm{C}_0^u(\wh{\GG}_2))$ such that
\[
(\id\otimes\rho)\WW=
(\WW_1)_{13} (\WW_2)_{24}
\]
(see \cite[Proposition 2.4]{Kustermans}). In order to show that $\rho$ is an isomorphism, let us perform the following calculations:
\[\begin{split}
&\quad\;
(\id\otimes\wh{\pi}_1)(\WW_1)_{13}\,
(\id\otimes\wh{\pi}_2)(\WW_2)_{23}=
((\Lambda_{\GG_1}\circ \pi_1)\otimes\id)(\WWd)_{13} \, ((\Lambda_{\GG_2}\circ \pi_2)\otimes\id)(\WWd)_{23}\\
&=
((\Lambda_{\GG_1}\circ\pi_1)\otimes
(\Lambda_{\GG_2}\circ\pi_2)\otimes\id) (\Delta_{\GG}^{u}\otimes\id)\WWd=
\WW,
\end{split}\]
in the last equality we have used equation \eqref{eq30}. Consider the nondegenerate $\star$-homomorphism
\[
\wh{\pi}_1 \cdot
\wh{\pi}_2\colon \mathrm{C}_0^{u}(\wh{\GG}_1)\otimes_{\operatorname{max}}
\mathrm{C}_0^u(\wh{\GG}_2)\rightarrow
\M(\CGDu)\colon x\otimes y \mapsto \wh{\pi}_1(x) \wh{\pi}_2(y)=
\wh{\pi}_2(y)\wh{\pi}_1(x)
\]
given by the universal property of $\otimes_{\operatorname{max}}$ (see e.g. \cite[Proposition 6.3.7]{Murphy}). The above reasoning implies
\[
(\id\otimes\id\otimes (\wh{\pi}_1\cdot \wh{\pi}_2)) ((\WW_1)_{13} (\WW_2)_{24})=\WW,
\]
therefore $\rho$ is an inverse to $\wh{\pi}_1\cdot \wh{\pi}_2$ and $\CGDu$, $\mathrm{C}_0^{u}(\wh{\GG}_1)\otimes_{\operatorname{max}} \mathrm{C}_0^{u}(\wh{\GG}_2)$ are isomorphic. We will henceforth identify these two \cst-algebras. Under this identification we have $\WW=(\WW_1)_{13} (\WW_2)_{24}$.
\end{proof}

Note that the above result implies that if $\GG_1$ and $\GG_2$ are second countable or type I then so is $\GG_1\times \GG_2$ (\cite[Theorem 7]{Guichardet}). We note also that \cst-algebras of type I are nuclear, hence we do not need to distinguish between the minimal and maximal tensor product (\cite[Proposition XV.1.6]{TakesakiIII}). Using this construction we can produce new type I locally compact quantum groups with properties that combine those of $\GG_1$ and $\GG_2$.

\section{Appendix}
\subsection{Lemmas concerning locally compact quantum groups}\label{appendixqg}
In this section we have gathered several lemmas concerning locally compact quantum groups -- density of various subspaces and relations between the Haar integrals on $\GG$ and $\whG$. These lemmas are probably well known to the experts, however we will include their proofs for the convenience of the reader. In this section $\GG$ is an arbitrary locally compact quantum group.\\
Let us introduce a subset of the space of normal functionals on $\Linf$:
\[
\mc{I}=\{\omega\in \LL^1(\GG)\,|\,\exists_{M\in \RR_{\ge 0}}\,
\forall_{
x\in \mf{N}_\vp}\,
|\omega(x^*)|\le M \|\Lvp(x)\|\}
\]
This subset appears in the definition of the left Haar integral on $\whG$ (see \cite{KustermansVaes}): for $\omega\in \mc{I}$ we have $\lambda(\omega)\in \mf{N}_{\hvp}$ and $\ismaa{\Lambda_{\vp}(x)}{\Lhvp(\lambda(\omega))}=\omega(x^*)$. Sometimes one uses notation $\xi(\omega)=\Lhvp(\lambda(\omega))$.\\
Similarly we can introduce a right version of this subset:
\[
\mc{I}_R=\{\omega\in\LL^1(\GG)\,|\,
\exists_{M\in \RR_{\ge 0}}\,
\forall_{
x\in \mf{N}_\psi}\,
|\omega(x^*)|\le M\|\Lambda_{\psi}(x)\|\}.
\]
For $\omega\in\mc{I}_R$ we define a vector $\xi_R(\omega)\in\LdG$ via the equality $\is{\Lvps(x)}{\xi_R(\omega)}=\omega(x^*)\,(x\in\mf{N}_{\psi})$. The next lemma introduces a relation between the set $\mc{I}_R$ and the right Haar integral on $\whG$.

\begin{lemma}\label{lemat25}
$ $
\begin{enumerate}
\item For $\omega\in\Lj$ we have $\omega\in\mc{I}$ if and only if $ \ov{\omega\circ R}\in\mc{I}_R$, and then $\xi_R(\ov{\omega\circ R})=\hat{J}\xi(\omega)$.
\item If $\omega\in \LL^1_{\sharp}(\GG)$ is a functional such that $\ov{\omega^{\sharp}}\in \mc{I}_R$ then $\lambda(\omega)\in\mf{N}_{\wh{\psi}}\textnormal{ and }\xi_R(\ov{\omega^{\sharp}})=\hat{J} J \Lambda_{\widehat{\psi}}(\lambda(\omega))$.
\end{enumerate}
\end{lemma}

\begin{proof}
Let $\omega\in \mc{I}$ and $x\in\mf{N}_{\psi}$, then $R(x)^*\in\mf{N}_{\vp}$. Calculation
\[
\begin{split}
&\quad\;
\ov{\omega\circ R}(x^*)=\ov{\omega(R(x^*)^*)}=
\ov{\ismaa{\Lvp(R(x^*))}{\xi(\omega)}}=
\ov{\ismaa{\hat{J}\Lambda_{\psi}(x)}{\xi(\omega)}}=
\ismaa{\Lambda_{\psi}(x)}{\hat{J}\xi(\omega)}
\end{split}
\]
shows that $\ov{\omega\circ R}\in\mc{I}_R$ and $\xi_R(\ov{\omega\circ R})=\hat{J}\xi(\omega)$.\\
Assume now that $\ov{\omega\circ R}\in \mc{I}_R$ and $x\in \mf{N}_{\vp}$. We get
\[\begin{split} 
&\quad\;
\ismaa{\Lvp(x)}{\hat{J}\xi_R(\ov{\omega\circ R})}=
\ov{\ismaa{\hat{J}\Lvp(x)}{\xi_R(\ov{\omega\circ R})}}=
\ov{\ismaa{\Lambda_\psi(R(x^*))}{\xi_R(\ov{\omega\circ R})}}\\
&=\ov{\ov{\omega\circ R}(R(x^*)^*)}=
\omega(x^*),
\end{split}\]
which shows that $\omega\in\mc{I}$.\\
Let $\omega\in\Ljsharp$ be a functional such that $\ov{\omega^{\sharp}}\in \mc{I}_R$. The previous point gives us $\ov{\ov{\omega^{\sharp}}\circ R}\in \mc{I}$ and
\[
\begin{split}
&\quad\;\xi_R(\ov{\omega^{\sharp}})=
\hat{J}\xi(\ov{(\ov{\omega^{\sharp}})\circ R})=
\hat{J}\xi(\omega^{\sharp}\circ R)=
\hat{J}\Lhvp(\lambda(\omega^{\sharp}\circ R))=
\hat{J}\Lhvp(\hat{R}(\lambda(\omega^{\sharp})))\\
&=
\hat{J} \Lhvp(\hat{R}(\lambda(\omega))^*)=
\hat{J}J \Lambda_{\hpsi}(\lambda(\omega)),
\end{split}
\]
which proves the claim from the second point.
\end{proof}

Now we will prove that $\mc{I}_R\star\Lj\subseteq \mc{I}_R$. This result was used in Proposition \ref{stw12}, were we showed that if $\mu$ is a Plancherel measure and $\mu'$ equals $\mu$ composed with the conjugation, then $\mu'$ is equivalent to $\mu$.

\begin{lemma}\label{lemat21}
Let $\omega,\nu$ be a functionals in $\Lj$, assume moreover that $\omega\in\mc{I}_R$. Then $\omega\star\nu\in\mc{I}_R$ and $\xi_R(\omega\star\nu)=(\id\otimes\nu)(\mrV^*) \xi_R(\omega)$
\end{lemma}

\begin{proof}
Let $x\in\mf{N}_{\psi}$. Due to the right invariance of $\psi$ we know that $(\id\otimes\ov\nu)\Delta(x)\in\mf{N}_\psi$, and by the definition of $\mrV$ we have $(\id\otimes \ov\nu)\mrV\Lambda_{\psi}(x)=\Lambda_{\psi}((\id\otimes\ov\nu)\Delta(x))$. Therefore we can perform the following calculations:
\[\begin{split} 
&\quad\; \omega\star\nu(x^*)=\omega(((\id\otimes\ov\nu)\Delta(x))^*)=
\ismaa{\Lambda_{\psi}((\id\otimes\ov\nu)\Delta(x))}{\xi_R(\omega)}\\
&=
\ismaa{(\id\otimes\ov\nu)\mrV\,\Lambda_\psi(x)}{\xi_R(\omega)}=
\ismaa{\Lambda_\psi(x)}{(\id\otimes\nu)(\mrV^*)\xi_R(\omega)}.
\end{split}\]
\end{proof}

The next lemma tells us how vector functionals behave under basic operations on normal functionals:

\begin{lemma}\label{lemat17}
Let $\eta,\zeta\in\LL^2(\GG)$. We have $\ov{\omega_{\eta,\zeta}}=\omega_{\zeta,\eta}$ and $\omega_{\eta,\zeta}\circ R=\omega_{\hat{J}\zeta,\hat{J}\eta}$. If we assume that $\eta\in\Dom(\nabla_{\hvp}^{-\frac{1}{2}}),\zeta\in\Dom(\nabla_{\hvp}^{\frac{1}{2}})$, then $\omega_{\eta,\zeta}\in\LL^1_{\sharp}(\GG)$ and $(\omega_{\eta,\zeta})^{\sharp}=
\omega_{\hat{J}\nabla_{\hvp}^{-\frac{1}{2}}  \eta,
\hat{J}\nabla_{\hvp}^{\frac{1}{2}}\zeta}$.
\end{lemma}

\begin{proof}
The first part is trivial, therefore let us check the second one. Take $x\in \Dom(S)$. Then
\[
\begin{split}
\ov{\omega_{\eta,\zeta}}(S(x))&=
\omega_{\zeta,\eta}(R\circ\tau_{-i/2}(x))=
\ismaa{\zeta}{R(\tau_{-i/2}(x)) \eta}=
\ismaa{\zeta}{\hat{J} {\tau_{-i/2}(x)}^* \hat{J}\eta}\\
&=
\ismaa{\hat{J}\eta}{\tau_{-i/2}(x)\hat{J}\zeta}=
\ismaa{\hat{J}\eta}{\nabla_{\hvp}^{\frac{1}{2}} x
\nabla_{\hvp}^{-\frac{1}{2}}\hat{J}\zeta}=
\ismaa{\hat{J}\nabla_{\hvp}^{-\frac{1}{2}} \eta}{
x\hat{J} \nabla_{\hvp}^{\frac{1}{2}}\zeta}\\
&=
\omega_{\hat{J}\nabla_{\hvp}^{-\frac{1}{2}} \eta,
\hat{J}\nabla_{\hvp}^{\frac{1}{2}}\zeta}(x),
\end{split}
\]
which proves $\omega_{\eta,\zeta}\in\LL^1_{\sharp}(\GG)$ and $(\omega_{\eta,\zeta})^{\sharp}=
\omega_{\hat{J}\nabla_{\hvp}^{-\frac{1}{2}}  \eta,
\hat{J}\nabla_{\hvp}^{\frac{1}{2}}\zeta}$.
\end{proof}

The next lemma allows us to construct plenty of functionals in $\mc{I}$ and $\mc{I}_R$. Its proof is a straightforward consequence of the definitions.

\begin{lemma}\label{lemat26}$ $
\begin{enumerate}
\item Let $\zeta\in\LdG$ and $y\in \mf{N}_{\vp}\cap \Dom(\sigma^{\vp}_{\frac{i}{2}})$. Then $\omega_{\Lvp(y),\zeta}\in\mc{I}$ and $\xi(\omega_{\Lvp(y),\zeta})=J \sigma^{\vp}_{\frac{i}{2}}(y) J \zeta$.
\item
Let $\zeta\in\LdG$ and $y\in \mf{N}_{\psi}\cap \Dom(\sigma^{\psi}_{\frac{i}{2}})$. Then $\omega_{\Lambda_\psi(y),\zeta}\in\mc{I}_R$ and $\xi_R(\omega_{\Lambda_\psi(y),\zeta})=J^{\psi} \sigma^{\psi}_{\frac{i}{2}}(y) J^{\psi} \zeta$.
\end{enumerate}
\end{lemma}

Having the above lemma we are able to derive a result which tells us about existence of dense subspaces in $\Linf$ and $\Lj$ which consist of elements with desirable properties. We use this result plenty of times in our work.

\begin{lemma}\label{lemat28}
There exists $\mc{X}$, a \ssot-dense subspace in $\Linf$ such that for any $y\in \mc{X}$ and $w\in\CC$ we have
\[
y\in\mf{N}_\vp\cap\mf{N}_\psi\cap\Dom(\sigma^{\vp}_w)\cap\Dom(\sigma^{\psi}_w),\; \Lambda_\vp(y)\in\Dom(\nabla_{\hvp}^w)
\]
and $\sigma^{\vp}_w(y),\,\sigma^{\psi}_w(y)\in\mc{X},\;\nabla_{\hvp}^w\Lvp(y)\in\Lvp(\mc{X})$. Moreover, the operator $y\delta^w$ is bounded and its closure is in $\mc{X}$. We have also equality of dense subspaces $\hat{J}\Lvp(\mc{X})=\Lvp(\mc{X})=\Lambda_\psi(\mc{X})\subseteq\LdG$. The subspace
\[
\mc{Y}=\lin\{\omega_{\xi,\eta}\,|\,\xi,\eta\in\Lvp(\mc{X})=\Lambda_\psi(\mc{X})\}
\]
is dense in $\Lj$. For $\omega\in \mc{Y}$ we have $\omega\in\Ljsharp$ and $\ov{\omega},\;\omega\circ R,\; \omega^{\sharp}\in\mc{Y}$. Moreover $\omega\in \mc{I}\cap\mc{I}_R$, therefore $\lambda(\omega)\in\mf{N}_{\hvp}\cap\mf{N}_{\hpsi}$. Subspaces $\xi(\mc{Y})=\Lhvp(\lambda(\mc{Y}))$ and $\xi_R(\mc{Y}),\,\Lambda_{\hpsi}(\lambda(\mc{Y}))$ are dense in $\LdG$.
\end{lemma}

\begin{proof}
Let us take $x\in \mf{N}_{\vp},A=(a,b,c,d)\in\CC^4, n\in\NN$ and define
\[
x_{n,A}=\sqrt{\tfrac{n^4}{\pi^4}}\int_{\RR^{\times 4}}
e^{-n(l-a,s-b,t-c,p-d)^2} \sigma^{\psi}_{l}\circ \sigma^{\vp}_{s}\circ \tau_t(x) \delta^{-ip}
\md l\md s\md t \md p.
\]
This integral converges in \ssot (Lemma \ref{lemat29}). Above, the expression $(l-a,s-b,t-c,p-d)^2$ means $(l-a)^2+(s-b)^2+(t-c)^2+(p-d)^2$. A standard argument shows that $x_{n,0}$ converges in norm to $x$ as $n\to\infty$. Moreover, we have $x_{n,A}\in \mf{N}_\vp$. Indeed, we have
\[
\begin{split}
&\quad\;
\vp(x_{n,A}^* x_{n,A})=
\sup_{\omega\in \LL^1(\GG)_+:\,\omega\le \vp} \omega(x_{n,A}^* x_{n,A})\\
&=
\tfrac{n^4}{\pi^4}
\sup_{\omega\in \LL^1(\GG)_+:\,\omega\le \vp} \int_{\RR^{\times 4}}\int_{\RR^{\times 4}}
e^{-n((l-a,s-b,t-c,p-d)^2+(l'-a,s'-b,t'-\ov{c}, p'-\ov{d})^2)}\\
&\quad\quad\quad\quad\quad\quad\quad\quad\quad\quad
\omega(
\delta^{ip}\sigma^{\psi}_{l}\circ\sigma^{\vp}_{s}\circ \tau_t(x^*) 
\sigma^{\psi}_{l'}\circ\sigma^{\vp}_{s'}\circ \tau_{t'}(x) \delta^{-ip'}
)
\md l\md s \md t \md p \md l'\md s' \md t' \md p'\\
&\le
\tfrac{n^4}{\pi^4}
\sup_{\omega\in \LL^1(\GG)_+:\,\omega\le \vp} \int_{\RR^{\times 4}}\int_{\RR^{\times 4}}
e^{-n(\Re(l-a,s-b,t-c,p-d)^2+\Re(l'-a,s'-b,t'-\ov{c}, p'-\ov{d})^2)}\\
&\quad\quad\quad\quad\quad\quad\quad\quad\quad\quad
|\omega(
\delta^{ip}\sigma^{\psi}_{l}\circ \sigma^{\vp}_{s}\circ \tau_t(x^*) 
\sigma^{\psi}_{l'}\circ\sigma^{\vp}_{s'}\circ \tau_{t'}(x) \delta^{-ip'}
)|
\md l\md s \md t \md p \md l'\md s' \md t' \md p'.
\end{split}
\]
Next, we have the following inequalities
\[\begin{split}
&\quad\;
|\omega(
\delta^{ip}\sigma^{\psi}_l\circ \sigma^{\vp}_{s}\circ \tau_t(x^*) 
\sigma^{\psi}_{l'}\circ \sigma^{\vp}_{s'}\circ \tau_{t'}(x) \delta^{-ip'})|\\
&\le
\omega(
\delta^{ip}\sigma^{\psi}_l\circ\sigma^{\vp}_{s}\circ \tau_t(x^*) 
\sigma^{\vp}_{s}\circ \tau_{t}(x) \delta^{-ip}
)^{\frac{1}{2}}
\omega(
\delta^{ip'}\sigma^{\psi}_{l'}\circ\sigma^{\vp}_{s'}\circ \tau_{t'}(x^*) 
\sigma^{\vp}_{s'}\circ \tau_{t'}(x) \delta^{-ip'}
)^{\frac{1}{2}}\\
&\le
\vp(
\delta^{ip}\sigma^{\psi}_l\circ\sigma^{\vp}_{s}\circ \tau_t(x^*) 
\sigma^{\vp}_{s}\circ \tau_{t}(x) \delta^{-ip}
)^{\frac{1}{2}}
\vp(
\delta^{ip'}\sigma^{\psi}_{l'}\circ\sigma^{\vp}_{s'}\circ \tau_{t'}(x^*) 
\sigma^{\vp}_{s'}\circ \tau_{t'}(x) \delta^{-ip'}
)^{\frac{1}{2}}\\
&\le
\nu^{\frac{p}{4}+\frac{p'}{4}}\vp(
\sigma^{\psi}_l\circ\sigma^{\vp}_{s}\circ \tau_t(x^*x) 
)^{\frac{1}{2}}
\vp(
\sigma^{\psi}_{l'}\circ\sigma^{\vp}_{s'}\circ \tau_{t'}(x^*a) 
)^{\frac{1}{2}}=
\nu^{\frac{p}{4}+\frac{p'}{4}-s-s'+l+l'}\vp(
x^*x)
\end{split}
\]
(we have used the relation $\vp(\delta^{ir}a^*a \delta^{-ir})\le \nu^{\frac{r}{2}} \vp(a^*a)\;(r\in \RR,a\in \Linf)$ (\cite[Proposition 2.14]{Stratila})), which gives
\[
\begin{split}
&\quad\;
\vp(x_{n,A}^* x_{n,A})\le
\tfrac{n^4}{\pi^4}
\int_{\RR^{\times 4}}\int_{\RR^{\times 4}}
e^{-n(\Re(l-a,s-b,t-c,p-d)^2+\Re(l'-a,s'-b,t'-\ov{c}, p'-\ov{d})^2)}\\
&\quad\quad\quad\quad\quad\quad\quad\quad\quad\quad\quad\quad
\nu^{\frac{p}{4}+\frac{p'}{4}-s-s'+l+l'}\vp(
x^*x)
\md l\md s\md t \md p \md l'\md s' \md t' \md p'<+\infty
\end{split}
\]
for any $n\in\NN,A=(a,b,c,d)\in \CC^4$. Take $w\in \CC$ and $\xi\in\Dom(\delta^w)$. Since
\[
\begin{split}
x_{n,A}\delta^w\xi&=
\sqrt{\tfrac{n^4}{\pi^4}}\int_{\RR^{\times 4}}
e^{-n(l-a,s-b,t-c,p-d)^2} \sigma^{\psi}_l\circ\sigma^{\vp}_{s}\circ \tau_t(x) \delta^{-ip+w}\xi
\md l\md s\md t \md p\\
&=
\sqrt{\tfrac{n^4}{\pi^4}}\int_{\RR^{\times 4}}
e^{-n(l-a,s-b,t-c,p-iw-d)^2} \sigma^{\psi}_l\circ\sigma^{\vp}_{s}\circ \tau_t(x) \delta^{-ip}\xi
\md l\md s\md t \md p\\
&=x_{n,A+(0,0,0,iw)}\xi
\end{split}
\]
we know that $x_{n,A}\delta^w$ is a bounded operator and after closure we have $\ov{x_{n,A} \delta^w}=x_{n,A+(0,0,0,iw)}$ In particular this means that $x_{n,A}\in\mf{N}_{\psi}$.\\
An argument of the same type as above shows that $x_{n,A}\in \bigcap_{w\in \CC} \Dom(\sigma^{\vp}_w)\cap\Dom(\sigma^{\psi}_w)$ and
\[
\sigma^{\vp}_w(x_{n,A})=x_{n,A+(0,w,0,0)},\quad
\sigma^{\psi}_w(x_{n,A})=x_{n,A+(w,0,0,0)},
\]
Thanks to the Hille theorem (Lemma \ref{lemat30}), Lemma \ref{lemat31} and $\ssot\times \|\cdot\|$ closedness of $\Lvp$ we have
\[
\Lambda_{\vp}(x_{n,A})=
\sqrt{\tfrac{n^4}{\pi^4}}\int_{\RR^{\times 4}}
e^{-n(l-a,s-b,t-c,p-d)^2} \Lambda_{\vp}(\sigma^{\psi}_l\circ\sigma^{\vp}_{s}\circ \tau_t(x) \delta^{-ip})
\md l\md s\md t \md p,
\]
therefore it is clear that $\Lambda_{\vp}(x_{n,0})\xrightarrow[n\to\infty]{} \Lambda_{\vp}(x)$ in norm.\\
For any $q\in \RR$ due to the equality $\nabla_{\hvp}^{iq} \Lvp(a)=\Lvp(\tau_q(a)\delta^{-iq})\;(a\in\mf{N}_\vp)$ and $\tau_q(\delta)=\delta$ we have
\[\begin{split} 
\nabla_{\hvp}^{iq}\Lvp(x_{n,A})&=
\sqrt{\tfrac{n^4}{\pi^4}}\int_{\RR^{\times 4}}
e^{-n(l-a,s-b,t-c,p-d)^2} \nabla_{\hvp}^{iq}\Lvp(\sigma^{\psi}_l\circ\sigma^{\vp}_{s}\circ \tau_t(x) \delta^{-ip})
\md l\md s\md t \md p\\
&=
\sqrt{\tfrac{n^4}{\pi^4}}\int_{\RR^{\times 4}}
e^{-n(l-a,s-b,t-c,p-d)^2} \Lvp(\sigma^{\psi}_l\circ\sigma^{\vp}_{s}\circ \tau_{t+q}(x) \delta^{-i(p+q)})
\md l\md s\md t \md p\\
&=
\sqrt{\tfrac{n^4}{\pi^4}}\int_{\RR^{\times 4}}
e^{-n(l-a,s-b,t-c-q,p-d-q)^2} \Lvp(\sigma^{\psi}_l\circ\sigma^{\vp}_{s}\circ \tau_{t}(x) \delta^{-ip})
\md l\md s\md t \md p,
\end{split}\]
therefore $\Lvp(x_{n,A})\in\Dom(\nabla_{\hvp}^{w})$ for any $w\in \CC$ and
\[
\begin{split}
\nabla_{\hvp}^w \Lvp(x_{n,A})&=
\sqrt{\tfrac{n^4}{\pi^4}}\int_{\RR^{\times 4}}
e^{-n(l-a,s-b,t-c+iw,p-d+iw)^2} \Lvp(\sigma^{\psi}_l\circ\sigma^{\vp}_{s}\circ \tau_{t}(x) \delta^{-ip})
\md l\md s\md t \md p\\
&=
\Lvp(x_{n,A+(0,0,-iw,-iw)}).
\end{split}
\]
Define
\[
\tilde{\mc{X}}=\{x_{n,A}\,|\,x\in \mf{N}_\vp,\,n\in\NN,\,A=(a,b,c,d)\in\CC^4\}.
\]
We have shown that $\tilde{\mc{X}}$ is a \ssot dense subspace in $\Linf$ and for any $y\in \tilde{\mc{X}},w\in\CC$ we have
\[
y\in \mf{N}_\vp\cap \mf{N}_{\psi}\cap \Dom(\sigma^{\vp}_w)\cap\Dom(\sigma^{\psi}_w),\quad
\Lvp(y)\in\Dom(\nabla_{\hvp}^{w})
\]
moreover
\[
\sigma^{\vp}_w(y),\sigma^{\psi}_w(y)\in\tilde{\mc{X}},\quad\nabla_{\hvp}^w(\Lvp(y))\in \Lvp(\mc{X}_1).
\]
Operator $y\delta^w$ is bounded and its closure is in $\tilde{\mc{X}}$. Moreover, $\Lvp(\tilde{\mc{X}})$ is a dense subspace in $\LdG$.\\
Assume now that we have $x\in \mf{N}_\vp\cap\mf{N}_\psi$. The same reasoning as above gives us
\[
\Lambda_{\psi}(x_{n,A})=
\sqrt{\tfrac{n^4}{\pi^4}}\int_{\RR^{\times 4}}
e^{-n(l-a,s-b,t-c,p-d)^2} \Lambda_{\psi}(\sigma^{\psi}_l\circ\sigma^{\vp}_{s}\circ \tau_t(x) \delta^{-ip})
\md l\md s\md t \md p,
\]
hence $\Lambda_{\psi}(x_{n,0})\xrightarrow[n\to\infty]{} \Lambda_{\psi}(x)$. Let us see how the operator $\hat{J}$ acts:
\[\begin{split} 
&\quad\;
\hat{J}\Lvp(x_{n,A})=\Lambda_{\psi}(R(x_{n,A})^*)\\
&=
\Lambda_{\psi}\bigl(
\sqrt{\tfrac{n^4}{\pi^4}}\int_{\RR^{\times 4}}
e^{-n(l-\ov{a},s-\ov{b},t-\ov{c},p-\ov{d})^2} R(\sigma^{\psi}_l\circ\sigma^{\vp}_{s}\circ \tau_{t}(x) \delta^{-ip})^*
\md l\md s\md t \md p\bigr)\\
&=
\Lambda_{\psi}\bigl(
\sqrt{\tfrac{n^4}{\pi^4}}\int_{\RR^{\times 4}}
e^{-n(l-\ov{a},s-\ov{b},t-\ov c,p-\ov d)^2} R(\sigma^{\psi}_l\circ\sigma^{\vp}_{s}\circ \tau_{t}(x))^* \delta^{-ip}
\md l\md s\md t \md p\bigr).
\end{split}\]
We have $a\in\mf{N}_\psi$ and $\Lambda_{\psi}(a)=\Lvp(\ov{a \delta^{\frac{1}{2}}})$ for $a\in \Linf$ such that the operator $a\delta^{\frac{1}{2}}$ is closable and its closure is in $\mf{N}_\vp$. It is the case for $a=R(x_{n,A})^*$, so we can continue our calculation:
\[\begin{split} 
\hat{J}\Lvp(x_{n,A})&=
\Lambda_{\vp}\bigl(
\sqrt{\tfrac{n^4}{\pi^4}}\int_{\RR^{\times 4}}
e^{-n(l-\ov{a},s-\ov{b},t-\ov c,p-\ov d-\frac{i}{2})^2} R(\sigma^{\psi}_l\circ\sigma^{\vp}_{s}\circ \tau_{t}(x))^* \delta^{-ip}
\md l\md s\md t \md p\bigr)\\
&=
\Lambda_{\vp}\bigl(
\sqrt{\tfrac{n^4}{\pi^4}}\int_{\RR^{\times 4}}
e^{-n(l-\ov{a},s-\ov{b},t-\ov c,p-\ov d-\frac{i}{2})^2} \sigma^{\psi}_{-l}\circ\sigma^{\vp}_{-s}\circ \tau_{t}(R(x)^*) \delta^{-ip}
\md l\md s\md t \md p\bigr)\\
&=
\Lambda_{\vp}\bigl(
\sqrt{\tfrac{n^4}{\pi^4}}\int_{\RR^{\times 4}}
e^{-n(l+\ov{a},s+\ov{b},t-\ov c,p-\ov d-\frac{i}{2})^2} \sigma^{\psi}_{l}\circ\sigma^{\vp}_{s}\circ \tau_{t}(R(x)^*) \delta^{-ip}
\md l\md s\md t \md p\bigr)\\
&=\Lvp(R(x^*)_{n,(-\ov a, -\ov b, \ov c ,\ov d + \frac{i}{2})}).
\end{split}\]
Since $x\in \mf{N}_\vp\cap\mf{N}_\psi$ then also $R(x^*)\in \mf{N}_\vp\cap\mf{N}_\psi$. Now, define
\[
\mc{X}=\{x_{n,A}\,|\, x\in \mf{N}_\vp\cap\mf{N}_\psi,\,n\in\NN,\,A\in\CC^4\}.
\]
Since we know that $\tilde{\mc{X}}$ is \ssot-dense in $\Linf$, so is $\mc{X}$. We have $\mc{X}\subseteq\tilde{\mc{X}}$, therefore all the "regularity" conditions remain true for the elements of $\mc{X}$. We have also density of the subspaces $\Lvp(\mc{X}),\Lambda_{\psi}(\mc{X})$ in $\LdG$. We have $\Lvp(\mc{X})=\Lambda_{\psi}(\mc{X})$ -- this follows from the equality $\Lambda_{\psi}(a)=\Lvp(\ov{a \delta^{\frac{1}{2}}})$ for nice $a$. Moreover $\hat{J}\Lvp(\mc{X})\subseteq \Lvp(\mc{X})$. This proves all the assertions about $\mc{X}$. \\
Let us define as in the claim a subspace of $\Lj$:
\[
\mc{Y}=\lin\{\omega_{\xi,\eta}\,|\,\xi,\eta\in\Lvp(\mc{X})=\Lambda_\psi(\mc{X})\}
\]
and take $\omega=\omega_{\Lvp(x),\Lvp(y)}\in\mc{Y}$. Due to Lemma \ref{lemat17} and properties of $\mc{X}$ we have $\omega\in\Ljsharp$ and
\[
\ov{\omega}=\omega_{\Lvp(y),\Lvp(x)}
,\;\omega\circ R=\omega_{\hat{J}\Lvp(y),\hat{J}\Lvp(x)}
,\; \omega^{\sharp}= \omega_{\hat{J}\nabla_{\hvp}^{-\frac{1}{2}}\Lvp(x),\hat{J}\nabla_{\hvp}^{\frac{1}{2}}\Lvp(y)}\in\mc{Y}.
\]
Lemma \ref{lemat26} gives us $\omega\in \mc{I}\cap\mc{I}_R$. Next, for $\omega=\omega_{\Lvp(x),\Lvp(y)}$ as above we have
\[
\xi(\omega)=\Lhvp(\lambda(\omega))=J \sigma^{\vp}_{\frac{i}{2}}(x) J \Lvp(y).
\]
Take a net $(x_i)_{i\in I}\in\mc{X}^{I}$ which converges to $\I$ in $\ssot$. Then $\sigma^{\vp}_{-\frac{i}{2}}(x_i)\in\mc{X}$ and
\[
\xi(\omega_{\Lvp(\sigma^{\vp}_{-\frac{i}{2}}(x_i)),\Lvp(y)})=
J \sigma^{\vp}_{\frac{i}{2}}(\sigma^{\vp}_{-\frac{i}{2}}(x_i)) J \Lvp(y)=
J x_i J \Lvp(y)\xrightarrow[i\in I]{\LdG}\Lvp(y),
\]
therefore the space $\xi(\mc{Y})=\Lhvp(\lambda(\mc{Y}))$  is dense in $\LdG$. Similarly we check the density of $\xi_R(\mc{Y})$: we have (Lemma \ref{lemat26})
\[
\xi_R(\omega)=J^{\psi} \sigma^{\psi}_{\frac{i}{2}}(x)J^{\psi}\Lambda_{\psi}(y)
\] 
and we can proceed as before. For any $\omega\in \mc{Y}$ we have $\ov{\omega^{\sharp}}\in\mc{Y}\subseteq \mc{I}_R$ therefore we can use Lemma \ref{lemat25} to get
\[
\Lambda_{\hpsi}(\lambda(\omega))=J\hat{J}\xi_R(\ov{\omega^{\sharp}})
\]
and deduce the density of $\Lambda_{\hpsi}(\lambda(\mc{Y}))$.
\end{proof}

\begin{lemma}\label{lemat37}
Let $\GG$ be a locally compact quantum group. The following conditions are equivalent:
\begin{enumerate}[label=\arabic*)]
\item $\CG$ is a separable \cst-algebra,
\item $\mathrm{C}_0^{u}(\GG)$ is a separable \cst-algebra,
\item $\LL^1(\whG)$ is a separable Banach space,
\item $\LdG$ is a separable Hilbert space.
\end{enumerate}
\end{lemma}

If the above conditions are satisfied we say that $\GG$ is \emph{second countable}. Note that since $\LdG\simeq \LL^2(\whG)$, this result implies that $\GG$ is second countable if and only if $\whG$ is.

\begin{proof}
As the GNS Hilbert spaces of $\vp,\vp^{u}$ can be identified with $\LdG$ we get $1)\Rightarrow 4)$ and $2)\Rightarrow 4)$ (see \cite[Theorem C.2]{MasudaNakagamiWoronowicz}).
Since $\Linfd$ is a von Neumann algebra acting on $\LdG$, point $4)$ implies $3)$. Next, since $\CG$ is the norm closure of $\{(\id\otimes\omega)\mrW^{\GG} \,|\, \omega\in \LL^1(\whG)\}$ we get $3)\Rightarrow 1)$. Implication $3)\Rightarrow 2)$ is analogous.
\end{proof}

\begin{lemma}\label{lemat11}
Assume that $\GG$ is second countable, i.e.~$\LdG$ is separable. There exists a sequence in $\{\lambda(\alpha)\,|\, \alpha\in \mc{I}\}$ which converges in \ssot to the identity operator.
\end{lemma}

\begin{proof}
The norm closure of the space $X=\{\lambda(\alpha)\,|\, \alpha\in \mc{I}\}$ equals $\CGDu$. Consequently, due to the Kaplansky theorem, we can find a bounded net in $\ov{X}$ which converges in $\sot$ to $\I$. However, since $\LdG$ is separable, the strong operator topology is metrizable on bounded subsets (\cite[Proposition I.6.3]{Davidson}) and we can find a sequence $(a_n)_{n\in\NN}$ in $\ov{X}$ which converges to $\I$. Let $d$ be a metric for the strong operator topology restricted to a large enough ball. For each $n\in\NN$ choose $b_n\in X$ such that $d(a_n,b_n)\le \tfrac{1}{n}$. It is clear that $(b_n)_{n\in\NN}$ is a sequence in $X$ which converges in $\sot$ to $\I$.
\end{proof}

\subsection{Other lemmas}
In this part we have gathered lemmas which are not related directly to the theory of locally compact quantum groups. First, we mention two lemmas concerned with existence of Pettis integrals which we use a lot throughout the text. We skip their proofs as they are elementary.

\begin{lemma}\label{lemat31}
Let $(X,\mf{M},\mu)$ be a measure space, $\msf{H}$ a Hilbert space, and $f\colon X\rightarrow \msf{H}$ a function such that for each $\xi\in \msf{H}$ the function $X\ni x \mapsto \is{f(x)}{\xi}\in\CC$ is measurable. Assume moreover that $X\ni x\mapsto \|f(x)\|\in \RR_{\ge 0}$ is measurable and integrable. Then the Pettis integral $\int_X f\md\mu$ exists.
\end{lemma}

\begin{lemma}\label{lemat29}
Let $(X,\mf{M},\mu)$ be a measure space, $\msf{H}$ a separable Hilbert space, $\M\subseteq\B(\msf{H})$ a von Neumann algebra and $ X\ni x \mapsto T_x\in \B(\msf{H})$ a map such that for any $\zeta,\eta\in \msf{H}$ the function $X\ni x \mapsto \is{\zeta}{T_x\eta}\in \CC$ is measurable. Then the function $X\ni x \mapsto \|T_x\|\in \RR$ is also measurable. If moreover we assume that $\int_X \|T_x\|\md\mu(x)<+\infty$ then we get existence of the Pettis integral
\[
\int_X T_x \md\mu(x)\in \M
\]
of the function $X\rightarrow (\M,\swot)$.
\end{lemma}

The next lemma is the Hille theorem for the Pettis integral. We include a proof for the convenience of the reader.

\begin{lemma}\label{lemat30}
Let $X,Y$ be locally convex topological vector spaces and $A\colon X\supseteq \Dom(A)\rightarrow Y$ a closed linear operator. Assume that $(\Omega,\mf{M},\mu)$ is a measure space and $f\colon \Omega\rightarrow X$ is a Pettis integrable function such that $f(\omega)\in \Dom(A)$ for all $\omega\in \Omega$, and the function $A\circ f \colon \Omega\rightarrow Y$ is also Pettis integrable. Then
\[
\int_\Omega f \md\mu\in \Dom(A)\qquad \textnormal{and}\qquad
A(\int_{\Omega} f \md\mu)=\int_{\Omega} A\circ f\md \mu.
\]
\end{lemma}

\begin{proof}
Put on the vector space $X\oplus Y$ the product topology -- with it, $X\oplus Y$ becomes a locally convex topological vector space, and $\operatorname{Gr}(A)=\{(x,Ax)\,|\,x\in \Dom(A)\}$ is its closed subspace.
Define a function $F\colon \Omega\ni \omega\mapsto (f(\omega),A\circ f(\omega) ) \in X\oplus Y$. It is a well defined function such that $F(\omega)\in \operatorname{Gr}(A)$ for all $\omega\in \Omega$. Choose any continuous functional $\phi\in (X\oplus Y)^*$. For $(x,y)\in X\oplus Y$ we have
\[
\phi(x,y)=\phi(x,0)+\phi(0,y)=\phi_X(x)+\phi_Y(y),
\]
where $\phi_X$, $\phi_Y$ are continuous functionals on $X,\,Y$ defined as a composition of $\phi$ with (continuous, linear) inclusions $X,Y\rightarrow X\oplus Y$. We have
\[
\begin{split}
&\quad\;\phi(  \int_\Omega f \md\mu,\int_\Omega A\circ f \md\mu)=
\phi_X (\int_{\Omega} f \md\mu)+\phi_Y ( \int_\Omega A\circ f \md\mu)\\
&=
\int_\Omega \langle\phi_X,f\rangle \md\mu+
\int_\Omega \langle\phi_Y,A\circ f\rangle \md\mu=
\int_{\Omega} ( \langle\phi_X,f\rangle+ \langle\phi_Y,A\circ f\rangle)\md\mu\\
&=
\int_\Omega \langle \phi, F\rangle \md\mu,
\end{split}
\]
which proves that $F$ is Pettis integrable, and shows the equality
\[
\int_{\Omega} F \md\mu=(\int_\Omega f \md\mu,\int_\Omega A\circ f\md\mu).
\]
Since $\operatorname{Gr}(A)$ is a closed subspace we know that $\int_\Omega F\md\mu\in \operatorname{Gr}(A)$, hence
\[
\int_{\Omega}f\md\mu \in \Dom(A),\quad
A(\int_{\Omega}f\md\mu)=\int_{\Omega}A\circ f\md\mu.
\]
\end{proof}

The next two lemmas are concerned with unbounded operators on Hilbert spaces:

\begin{lemma}\label{lemat2}
Let $\msf{H}$ be a Hilbert space and $A,B$ (unbounded) positive self-adjoint operators on $\msf{H}$. If $\Dom(A)\subseteq\Dom(B)$ and $\ismaa{A\xi}{A\xi}=\ismaa{B\xi}{B\xi}$ for $\xi\in\Dom(A)$ then $A=B$
\end{lemma}

\begin{proof}
Due to polarization identity we get
\[
\ismaa{A\xi}{A\eta}=\ismaa{B\xi}{B\eta}
\]
for $\xi,\eta\in \Dom(A)$. Take $\eta\in\Dom(A^2)\subseteq\Dom(B)$. Map
\[
\Dom(B)\ni \xi\mapsto 
\ismaa{B\eta}{B\xi}=
\ismaa{A\eta}{A\xi}=
\ismaa{A^2\eta}{\xi}
\in\CC
\]
is linear and bounded, hence $B\eta\in\Dom(B^*)=\Dom(B)$ and $B^*(B\eta)=B^2\eta=A^2\eta\;(\eta\in\msf{H}_0)$. This proves inclusion $A^2\subseteq B^2$. Operators $A^2,B^2$ are positive and self-adjoint, therefore $A^2=B^2$ (self-adjoint operators do not have nontrivial self-adjoint extensions). Taking square roots gives us the claim.
\end{proof}

The next lemma is almost trivial but we believe it is better to mention this result.

\begin{lemma}\label{lemat24}
Let $\msf{H}$ be a Hilbert space, and $A$ a closed densely defined operator on $\msf{H}$ with domain $\Dom(A)$. If there exists a dense subspace $V\subseteq \Dom(A)$ and number $K\in \RR$ such that $\|A\eta\|\le K \|\eta\|$ for $\eta\in V$ then $\Dom(A)=\msf{H}$, and $A$ is bounded.
\end{lemma}

The next lemma tells us how we can construct a projection onto the intersection of subspaces in a separable Hilbert space.

\begin{lemma}\label{lemat14}
Let $\msf{H}$ be a separable Hilbert space and $P_1,P_2,\dotsc$ orthogonal projections. For $k\in \NN$ define $A_k=P_1 \cdot\cdots\cdot P_k \cdot\cdots\cdot P_1\in\B(\msf{H})$.
\begin{enumerate}
\item 
There exists an increasing sequence $(n_k)_{k\in\NN}$ of natural numbers such that $(A_k^{n_k})_{k\in\NN}$ converges in $\sot$ to the projection onto $\bigcap_{k=1}^{\infty} P_k\msf{H}$.
\item If for some increasing sequences of natural numbers $(m_k)_{k\in\NN}, (n_k)_{k\in \NN}$, the sequence $(A_{m_k}^{n_k})_{k\in\NN}$ converges in $\sot$ to a projection $Q$, then $Q$ is the projection onto $\bigcap_{k=1}^{\infty} P_k\msf{H}$.
\end{enumerate}
\end{lemma}

\begin{proof}
Take $k\in \NN$ and let us show that
\begin{equation}\label{eq45}
P_1\msf{H}\cap \cdots\cap P_k\msf{H}=\ker (A_k-\I)
\end{equation}
It is clear that if $\eta\in P_1\msf{H}\cap \cdots\cap P_k\msf{H}$ then $A_k\eta=\eta$. Assume that $\xi\in\msf{H}$ satisfies $A_k\xi=\xi$. We have
\[
\|\xi\|=\|P_1\cdot\cdots\cdot P_k\cdot \cdots\cdot P_1 \xi\|\le \|P_1\xi\|\le\|\xi\|,
\]
and it follows that $P_1\xi=\xi$. Next,
\[
\|\xi\|=\|P_1 P_2 \cdot\cdots\cdot P_k\cdot \cdots\cdot P_2 P_1 \xi\|\le
\|P_2 \cdot\cdots\cdot P_k \cdot\cdots\cdot P_2 \xi\| \le \| P_2 \xi\| \le\|\xi\|
\]
and we get $P_2\xi=\xi$. If we proceed further in this manner we arrive at $\xi\in P_1\msf{H}\cap \cdots\cap P_k\msf{H}$.\\
Let us apply the spectral theorem to the operator $A_k$. We get a measure space $(\Omega,\mc{M},\mu)$ with $\mu(\Omega)<+\infty$, a unitary operator $U\colon \msf{H}\rightarrow \LL^2(\Omega,\mu)$ and a positive measurable function $f\colon \Omega\rightarrow \RR_{\ge 0}$ such that $|f(x)|\le 1$ for almost all $x\in\Omega$ and $UA_kU^* = \M_{f}$, where $\M_f$ is the operator of multiplication by $f$. Take $\phi\in \LL^2(\Omega,\mu)$ and let $X=\{x\in \Omega\,|\, f(x)=1\}$. Since
\[
\M_f^n \phi=f^n \phi\xrightarrow[n\to\infty]{} \chi_X \phi=\M_{\chi_X} \phi,
\]
the sequence $(\M_f^n)_{n\in\NN}$ converges in $\sot$ to the projection onto $\ker(\M_f-\I)$. Conjugating back with $U$ shows that the sequence $(A_k^n)_{n\in\NN}$ converges in $\sot$ to the projection onto $\ker(A_k-\I)=P_1\msf{H}\cap\cdots\cap P_k\msf{H}$.\\
Let $\{\xi_m\,|\, m\in\NN\}$ be a dense subspace in the closed unit ball of $\msf{H}$. For each $k\in \NN$ choose $n_k\in\NN$ such that
\[
\|A_k^{N} \xi_m-\lim_{n\to\infty} A_k^n \xi_m\|\le\tfrac{1}{k}\quad(N\ge n_k,m\in\{1,\dotsc,k\}),
\]
 assume moreover that the sequence $(n_k)_{k\in\NN}$ is increasing.\\
Take $\eps>0$ and a norm $1$ vector $\xi\in\bigcap_{m=1}^{\infty} P_m\msf{H}$. There exists $m\in\NN$ such that $\|\xi-\xi_m\|\le\tfrac{\eps}{2}$. For $k \ge m$ we have
\[\begin{split}
\|A_k^{n_k} \xi-\xi\|&\le
\tfrac{\eps}{2} + \|A^{n_k}_{k} \xi_m-\xi\|\le
\tfrac{\eps}{2} + \tfrac{1}{k}+\|\lim_{n\to\infty} A^{n}_{k} \xi_m-\xi\|\\
&\le
\eps +\tfrac{1}{k} + \|\lim_{n\to\infty}
A^n_k \xi -\xi\|=\eps+\tfrac{1}{k},
\end{split}\]
hence $A^{n_k}_k \xi\xrightarrow[k\to\infty]{}\xi$. Take now a norm $1$ vector
\[
\eta\in (\bigcap_{m=1}^{\infty} P_m \msf{H})^{\perp}=
\ov{\lin}_{m\in\NN} (P_m \msf{H})^{\perp}
\]
and $\eps>0$. We can find $m\in\NN$ such that $\|\eta- \xi_m\|\le\tfrac{\eps}{3}$ and $\eta_p \in \lin_{q\in\{1,\dotsc,p\}} (P_q\msf{H})^{\perp}$ such that $\|\eta-\eta_p\|\le \tfrac{\eps}{3}$. For $k\ge\max\{ m,p\}$ we have
\[
\begin{split}
\|A_k^{n_k}\eta\|\le
\tfrac{\eps}{3}+
\|A^{n_k}_k \xi_m\|\le
\tfrac{\eps}{3}+\tfrac{1}{k}+
\|\lim_{n\to\infty} A^n_k \xi_m\|\le\eps+\tfrac{1}{k}+
\|\lim_{n\to\infty} A^n_k \eta_p\|=
\eps+\tfrac{1}{k}
,
\end{split}\]
where in the last equality we have used the fact that $(A^n_k)_{n\in\NN}$ converges to the projection onto $P_1\msf{H}\cap \cdots \cap P_k\msf{H}$ and
\[
(P_1\msf{H}\cap \cdots \cap P_k\msf{H})^{\perp}=
\ov{(P_1\msf{H})^{\perp} +\cdots+ (P_k\msf{H})^{\perp}}\supseteq
(P_1\msf{H})^{\perp} +\cdots+ (P_p\msf{H})^{\perp}\ni\eta_p.
\]
The above reasoning implies that $A^{n_k}_k\eta\xrightarrow[k\to\infty]{}0$ and $(A^{n_k}_k)_{k\in\NN}$ converges in $\sot$ to the projection onto $\bigcap_{m=1}^{\infty} P_m\msf{H}$. This proves the first point.\\
Assume now that $A_{m_k}^{n_k}\xrightarrow[k\to\infty]{\sot} Q$ for some projection $Q\in\B(\msf{H})$. Take $\xi\in Q \msf{H}$ and $\eps>0$. There exists $k\in \NN$ such that
\[
\|\xi\|=\|Q\xi\|\approx_{\eps} \|A^{n_k}_{m_k}\xi\|\le
\|P_1 \xi\| \le \|\xi\|,
\]
which implies that $P_1\xi=\xi$. We can proceed as before and we arrive at $Q\msf{H}\subseteq \bigcap_{m=1}^{\infty} P_m\msf{H}$. On the other hand, if $\eta\in \bigcap_{m=1}^{\infty} P_m\msf{H}$ then
\[
Q\eta=\lim_{k\to\infty}
(P_1 \cdot\cdots\cdot P_{m_k}\cdot\cdots\cdot P_1)^{n_k} \eta=\eta
\]
and we get $Q\msf{H}=\bigcap_{m=1}^{\infty} P_m\msf{H}$.
\end{proof}

\subsection{Quasi-containment of representations}
Let $A$ be a \cst-algebra, $\pi\colon A\rightarrow \B(\mscr{H}_\pi)$ nondegenerate representation and $\ov{\pi}\colon A^{**}\rightarrow \B(\mscr{H}_\pi)$ the unique normal extension of $\pi$ to the enveloping von Neumann algebra $A^{**}$. We define the central cover $c(\pi)$ as the projection in $\mc{Z}(A^{**})$ satisfying $\ker \ov{\pi}=A^{**} (\I-c(\pi))$ (\cite[Definition 1.4.2]{BrownOzawa}). We have $c(\pi)A^{**}\simeq\pi(A)''=\ov{\pi}(A^{**})$.

\begin{proposition}
Let $\pi,\sigma$  be nondegenerate representations of $A$ on separable Hilbert spaces. The following are equivalent:
\begin{enumerate}[label=\arabic*)]
\item
There exists an isomorphism $\theta\colon \pi(A)''\rightarrow\sigma(A)''$ satisfying $\theta(\pi(x))=\theta(x)$ for $x\in A$.
\item
There exist cardinal numbers $n,m$ such that $n\pi\simeq m\sigma$ ($\simeq$ denotes unitary equivalence).
\item
$c(\pi)=c(\sigma)$.
\end{enumerate}
\end{proposition}

This is a part of \cite[Theorem 5.3.1]{DixmierC} combined with \cite[Proposition 1.4.5]{BrownOzawa}.  If the following conditions are met, we say that $\pi$ and $\sigma$ are \emph{quasi-equivalent}. We will denote this relation by  $\pi\approx_q \sigma$.\\
Now we would like to introduce a one sided version, which is supposed to be the \emph{quasi-containment} ($\lec_q$). Recall that symbol $\pi\subseteq \sigma$ means that there is subrepresentation $\sigma'$ of $\sigma$ such that $\pi\simeq \sigma'$.

\begin{proposition}\label{stw15}
Let $\pi$, $\sigma$ be nondegenerate representations of $A$. The following are equivalent:
\begin{enumerate}[label=\arabic*)]
\item
There exists $\sigma'$, a subrepresentation of $\sigma$ such that $\pi\approx_q \sigma'$,
\item
$c(\pi)\le c(\sigma)$,
\item
there exist cardinal numbers $n,m$ such that $n\pi\subseteq m\sigma$.
\end{enumerate}
\end{proposition}

\begin{definition}
Whenever conditions $1)-3)$ of the above theorem are satisfied, we will write $\pi\lec_q \sigma$ and say that $\pi$ is \emph{quasi-contained} in $\sigma$.
\end{definition}

\begin{proof}
Equivalence of $1)$ and $2)$ was stated without proof in \cite[Proposition 1.4.6]{BrownOzawa}. It can be proven in the following manner:\\
Assume $\pi\approx_q \sigma'\subseteq \sigma$. Then $c(\pi)=c(\sigma')\le c(\sigma)$, which shows $1)\Rightarrow 2)$.\\
Assume now $c(\pi)\le c(\sigma)$. then $c(\pi)c(\sigma)=c(\pi)$ and $c(\sigma)-c(\pi)$ is a central projection orthogonal to $c(\pi)$.  Operators $\tilde{\sigma}(c(\sigma)-c(\pi)),\,\tilde{\sigma}(c(\pi))$ are central projections in $\sigma(A)''$. Moreover $\tilde{\sigma}(c(\sigma))=\I$. Consider the representation
\[
\sigma'\colon A\ni x \mapsto \tilde{\sigma}(c(\pi)) \sigma(x)=\tilde{\sigma}(c(\pi)x)\in \B(\tilde{\sigma}(c(\pi))\mscr{H}_\sigma).
\]
and its normal extension $\tilde{\sigma'}\colon A^{**}\rightarrow \sigma'(A)''\colon x \mapsto \tilde{\sigma}(c(\pi)x)$. We have 
\[
\begin{split}
\ker \tilde{\sigma'}&=
\{x\in A^{**}\,|\, \tilde{\sigma}(c(\pi)x)=0\}\\
&=
\{x\in A^{**}\,|\, c(\sigma)c(\pi)x=0\}\\
&=
\{x\in A^{**}\,|\, c(\pi)x=0\}=\ker\tilde{\pi},
\end{split}
\]
hence $c(\pi)=c(\sigma')$ and $\pi\approx_q \sigma'$. Define
\[
\sigma''\colon A\ni x\mapsto \tilde{\sigma}((c(\sigma)-c(\pi))x)\in
\B(\,(\tilde{\sigma}(c(\sigma)-c(\pi)) \mscr{H}_\sigma).
\]
Representations $\sigma',\sigma''$ are nondegenerate. Thanks to the unitary operator
\[
\mscr{H}_\sigma\ni\xi\mapsto \tilde{\sigma}(c(\pi))\xi+
\tilde{\sigma}(c(\sigma)-c(\pi))\xi\in
\tilde{\sigma}(c(\pi))\mscr{H}_\sigma\oplus
\tilde{\sigma}(c(\sigma)-c(\pi))\mscr{H}_\sigma,
\]
we have $\sigma\simeq \sigma'\oplus\sigma''$, which proves the implication $2)\Rightarrow 1)$.\\
Implication $3)\Rightarrow 2)$: from the definition of the central cover we have $c(n\pi)=c(\pi)$ for any cardinal number $n$. From $3)$ it follows that we have $m\sigma\simeq n\pi\oplus \sigma'$ for a certain representation $\sigma'$. This gives
\[
c(\pi)=c(n\pi)\le c(n\pi\oplus \sigma')=c(m\sigma)=c(\sigma).
\]
Implication $1)\Rightarrow 3)$: from point $1)$ we get unitary equivalence $\sigma\simeq \sigma'\oplus \sigma''$, with $\sigma'$ satisfying $\pi\approx_q \sigma'$. Therefore there exist cardinal numbers $n,m$ such that $n\pi\simeq m\sigma'$. It follows that
\[
n\pi\simeq m\sigma'\subseteq m\sigma'\oplus m\sigma''\simeq
m(\sigma'\oplus\sigma'')\simeq m\sigma.
\]
\end{proof}

From the second point of the above proposition it follows that $\lec_q$ is a transitive relation, moreover
\[
(\pi\lec_q \sigma,\,\sigma\lec_q \pi)\Leftrightarrow \pi\approx_q \sigma.
\]

Let $A$ be a separable \cst-algebra of type I and $\pi,\sigma$ two nondegenerate representations of $A$ on separable Hilbert spaces with decompositions (see \cite[Theorem 8.6.6]{DixmierC})
\[
\pi\simeq
\bigoplus_{n\in \NN\cup\{\aleph_0\}}
\int_{\hat{A}}^{\oplus}
\zeta\,\md\mu_n(\zeta),\quad
\sigma\simeq\bigoplus_{n\in\NN\cup\{\aleph_0\}}
\int_{\hat{A}}^{\oplus}\zeta\,\md\nu_n(\zeta).
\]
The following connection between quasi-containment and properties of measures $\{\mu_n,\nu_n\,|\,n\in \NN\cup\{\aleph_0\}\}$ holds:

\begin{proposition}\label{stw26}
We have $\sum_{n\in \NN\cup\{\aleph_0\}}\mu_n \ll\sum_{n\in \NN\cup\{\aleph_0\}}\nu_n$ if and only if $\phi_\mu\lec_q\phi_\nu$.
\end{proposition}

This is a combination of \cite[Theorem 8.6.6, Proposition 8.4.5]{DixmierC}.\\

Let us now turn to representations of a locally compact quantum group $\GG$. Assume that we have two representations
\[
U^\pi=(\id\otimes\pi){\WW},\quad
U^\sigma=(\id\otimes\sigma){\WW}
\]
corresponding to nondegenerate representations of $\CGDu$: $\pi,\sigma$. Their tensor product is given by
\[
U^{\pi}\tp U^{\sigma}=
(\id\otimes\pi\otimes\sigma) ({\WW}_{12} {\WW}_{13})=
(\id\otimes\pi\otimes\sigma)(\id\otimes\Delta_{\whG}^{u,\operatorname{op}}){\WW}.
\]
This suggests that we should define tensor product of representations of $\CGDu$ by
\[
\pi\tp\sigma=(\pi\otimes\sigma)\Delta_{\whG}^{u,\operatorname{op}}.
\]
We say that representations $U^\pi$, $U^\sigma$ are quasi-equivalent if there exists a Hilbert space $\bf H$ such that $1_{\bf H}\tp U^\pi\simeq 1_{\bf H}\tp U^\sigma$. Since 
\[
1_{\bf H}\tp U^\pi\simeq \oplus_{\dim \bf H} U^\pi=(\id\otimes\oplus_{\dim \bf H} \pi){\WW}=
(\id\otimes(\dim \bf H)\cdot\pi){\WW},
\]
representations $U^\pi$ i $U^\sigma$ are quasi equivalent if and only if $\pi\approx_q \sigma$. (Quasi-equivalence of representations of $\GG$ appears i.e. in \cite{closedqsub}). Correspondingly we declare $U^{\pi}$ to be quasi-contained in $U^{\sigma}$ (written $U^{\pi}\lec_q U^{\sigma}$) if and only if $\pi\lec_q \sigma$. Next propositon tells us that quasi-equivalence and quasi-containment are respected by tensor products:
\begin{proposition}\label{stw10}
Let $\pi_1,\pi_2,\sigma_1,\sigma_2$ be nondegenerate representations of $\CGDu$.
\begin{enumerate}
\item If $\pi_1\approx_q \sigma_1$ and $\pi_2\approx_q \sigma_2$ then $\pi_1\tp\pi_2 \approx_q \sigma_1\tp\sigma_2$.
\item If $\pi_1\subseteq \sigma_1$ and $\pi_2\subseteq \sigma_2$ then $\pi_1\tp\pi_2 \subseteq \sigma_1\tp\sigma_2$. 
\item If $\pi_1\lec_q \sigma_1$ and $\pi_2\lec_q \sigma_2$ then $\pi_1\tp\pi_2 \lec_q \sigma_1\tp\sigma_2$. 
\end{enumerate}
\end{proposition}

\begin{proof}
Let $U^{\pi_1},\,U^{\pi_2},U^{\sigma_1},U^{\sigma_2}$ be corresponding representations of $\GG$. There exist Hilbert spaces $\bf H_1,\bf H_2$ such that
\[
1_{\bf H_{1}}\tp U^{\pi_1}\simeq 
1_{\bf H_{1}}\tp U^{\sigma_1},\quad
1_{\bf H_{2}}\tp U^{\pi_2}\simeq 
1_{\bf H_{2}}\tp U^{\sigma_2}.
\]
We have
\[
\begin{split}
&1_{\bf H_1\otimes \bf H_2} \tp U^{\pi_1}\tp U^{\pi_2}\simeq
1_{\bf H_1}\tp 1_{\bf H_2} \tp U^{\pi_1}\tp U^{\pi_2}\simeq
(1_{\bf H_{1}}\tp U^{\pi_1})\tp
(1_{\bf H_{2}}\tp U^{\pi_2})\\
\simeq&
(1_{\bf H_{1}}\tp U^{\sigma_1})\tp
(1_{\bf H_{2}}\tp U^{\sigma_2})\simeq
1_{\bf H_{1}}\tp 1_{\bf H_{2}} \tp U^{\sigma_1}
\tp U^{\sigma_2}\simeq
1_{\bf H_{1}\otimes \bf H_{2}} \tp U^{\sigma_1}
\tp U^{\sigma_2},
\end{split}
\]
which shows the quasi-equivalence of $U^{\pi_1}\tp U^{\pi_2}$ and $U^{\sigma_1}\tp U^{\sigma_2}$. Since $U^{\pi_1}\tp U^{\pi_2}=U^{\pi_1\stp \pi_2}$ we get the first statement of the proposition.\\
Let us prove the second point: there exist representations $\sigma'_1,\sigma'_2$ such that
\[
U^{\sigma_1}\simeq U^{\sigma'_1}\oplus U^{\pi_1},\quad
U^{\sigma_2}\simeq U^{\sigma'_2}\oplus U^{\pi_2},
\]
therefore
\[
U^{\sigma_1\stp\sigma_2}=U^{\sigma_1}\tp U^{\sigma_2}\simeq
(U^{\sigma'_1}\tp U^{\sigma'_2})\oplus
(U^{\sigma'_1}\tp U^{\pi_2})\oplus
(U^{\pi_1}\tp U^{\sigma'_2})\oplus
(U^{\pi_1}\tp U^{\pi_2})
\]
and $\pi_1\tp\pi_2\subseteq \sigma_1\tp \sigma_2$.\\
Now we turn to the third point: there exist representations $\sigma'_1,\sigma'_2$ such that $\pi_1\approx_q \sigma'_1\subseteq \sigma_1$ and $\pi_2\approx_q \sigma'_2\subseteq \sigma_2$. From the second point we get $\pi_1\tp \pi_2\approx_q\sigma'_1\tp \sigma'_2\subseteq \sigma_1\tp \sigma_2$, hence $\pi_1\tp\pi_2\lec_q\sigma_1\tp\sigma_2$.
\end{proof}

\section*{Acknowledgements}
The author wish to thank Roland Vergnioux for the invitation to Caen and many inspiring discussions during the stay. Moreover, the author would like to express his gratitude towards Piotr M.~Sołtan and Adam Skalski for their guidance and helpful sugestions during work on this project.\\
The author was partially supported by the Polish National Agency for the Academic Exchange, Polonium grant PPN/BIL/2018/1/00197 and the NCN (National Centre of
Science) grant 2014/14/E/ST1/00525.

\bibliographystyle{plain}
\bibliography{bibliografia}

\end{document}